\pdfoutput=1
\documentclass[12pt]{article} 
\usepackage{etex}
\usepackage{amsmath,amsthm,amsfonts,amssymb}
\usepackage{graphicx, epsfig, graphics}
\usepackage[title]{appendix}

\usepackage[usenames]{color} 
\usepackage{pdfsync}
\usepackage{tikz,pgfplots}
\usetikzlibrary{shapes.geometric}
\usetikzlibrary{arrows.meta,arrows}

\usepackage{enumerate}

\usepackage{pictex,dcpic}

\usepackage{mathrsfs}

\newtheorem{thm}{Theorem}[section]

\newtheorem{lemma}[thm]{Lemma}

\newtheorem{cor}[thm]{Corollary}
\newtheorem{corollary}[thm]{Corollary}
\newtheorem{proposition}[thm]{Proposition}

\newtheorem{prop}[thm]{Proposition}

\newtheorem*{sublem*}{Sublemma}
\theoremstyle{definition}

\newtheorem{definition}[thm]{Definition} 
\newtheorem{construction}[thm]{Construction} 

\newtheorem{notn}[thm]{Notation}

\newtheorem{ass}[thm]{Assumption}
\newtheorem{ex}[thm]{Example}
\newtheorem*{excont*}{Example~\ref{e:main example} (continued)}
\newtheorem{exs}[thm]{Examples}
\newtheorem{remark}[thm]{Remark}
\theoremstyle{remark}

\newcounter{remarks}

\newcounter{enumitemp}
\newenvironment{enumeratecontinue}{
\setcounter{enumitemp}{\value{enumi}}
\begin{enumerate}
\setcounter{enumi}{\value{enumitemp}}
}
{
\end{enumerate}
}

\newcommand\pref[1]{(\ref{#1})}

\def\MW{{\sf MW}}
\def\M{{\sf M}}
\def\W{{\sf W}}

\DeclareMathOperator{\Fix}{{\sf{Fix}}}
\DeclareMathOperator{\FixN}{{\sf{Fix_N}}}
\DeclareMathOperator{\Per}{Per}

\DeclareMathOperator{\rk}{rank}

\newcommand{\Z}{{\mathbb Z}}

\newcommand{\Q}{{\mathbb Q}}
\newcommand{\f}{{F_n}}
\newcommand{\D}{{\mathcal D}}
\newcommand{\E}{{\mathcal E}}
\newcommand{\V}{\mathcal V}

\newcommand{\Out}{\mathsf{Out}}
\newcommand{\Aut}{\mathsf{Aut}}
\newcommand{\Inn}{\mathsf{Inn}}

\newcommand{\Hom}{\mathsf{Hom}}

\newcommand{\Stab}{\mathsf{Stab}}

\newcommand{\ffs}{free factor system}

\newcommand{\pg}{{\sf PG}}

\newcommand{\upg}{{\sf{UPG}}}
\newcommand{\upgn}{{\sf{UPG}}(\f)}
\newcommand{\F}{\mathcal F}

\newcommand{\path}{\sigma}

\def\L{\mathcal L}
\def\B{\mathcal B}
\newcommand{\A}{\mathcal A}

\newcommand{\fG} {f : G \to G}
\newcommand{\ti} {\tilde}
\newcommand{\iNp} {indivisible Nielsen path}

\newcommand{\filt}{\emptyset = G_0 \subset G_1 \subset \ldots  \subset G_N = G}

\newcommand{\eg}{EG}

\newcommand{\ct}{CT}

\newcommand{\cQ}{\mathcal Q}

\newcommand{\cP}{\mathcal P}

\newcommand{\cR}{\mathcal R}
\newcommand{\cS}{\mathcal S}

\newcommand{\stallings}{\Gamma}

\newcommand{\core}{{\sf{core}}}

\def\<{\prec}
\def\bN{\mathbb N}
\def\cC{\mathcal C}
\def\LW{{\sf{L}}}	

\def\arb{\upsilon} 
\def\Arb{\Upsilon} 

\DeclareMathOperator{\Acc}{{\sf Acc}}
\newcommand{\acc}{{\sf{\Omega}}}

\newcommand{\accnr}{{\sf{\acc_{NP}}}}
\newcommand{\both}{{or}}

\newcommand{\X}{\mathcal X}
\newcommand{\ray}{r}

\newcommand{\twistpath}{twist path}
\DeclareMathOperator{\lin}{{\sf Lin}}
\usepackage[usenames]{color} 

\DeclareMathOperator{\Ker}{{\sf Ker}}

\newcommand{\gf}{\Gamma(f)}

 \DeclareMathOperator\Perm{{\sf{Perm}}}

 \DeclareMathOperator{\offset}{{\sf offset}}
 \newcommand{\dle}{<_T}
 \newcommand{\nbhd}{\sH}

\def\ellprime{\mu}

\newcommand{\newrays}{\Delta\cR}
\newcommand{\afterR}{S}

\DeclareMathOperator\objX{{\sf X}}
\DeclareMathOperator\objY{{\sf Y}}
\DeclareMathOperator\objQ{{\sf Q}}

\def\sa{{\sf{SA}}}
\def\lat0{\mathfrak L}

\def\fL{\mathfrak L}
\def\fc{\mathfrak c}

\def\fe{\frak e}

\def\fL{\mathfrak L} 

\def\1edge{{\F^\pm}}
\def\cat{{\sf{IS}}}
\def\atoms{A}

\def\contractible{{\sf contractible}}
\def\cyclic{{\sf infinite cyclic}}
\def\llarge{{\sf large}}

\def\sH{{\sf H}}

\DeclareMathOperator{\order}{or}
\def\seqor{{\mathcal S}_{\order}}
\DeclareMathOperator{\unorder}{un}
\def\sequn{{\mathcal S}_{\unorder}}
\def\catA{{\mathbb A}}

\newcommand{\explicit}{has explicit finite fibers}
\def\compute{we can compute}

\def\ecat{{{\overline{\sf{IS}}}}}
\newcommand{\sbh}{\Sigma_b(H)}
\newcommand{\sbk}{\Sigma_b(K)}
\newcommand{\sh}{\Sigma(H)}
\newcommand{\sk}{\Sigma(K)}
\DeclareMathOperator{\im}{Imm}
\DeclareMathOperator{\rc}{RC}

\def\sHsub{{\sf H}_{\phi,\fc}}
\def\Abig{\catA_\bullet}
\DeclareMathOperator{\alg}{ALG}

\def\dotI{{\sf J}}

\begin{document}


\author{Mark Feighn and Michael Handel}
\title{The conjugacy problem for \upg\ elements of $\Out(\f)$}
\date{\today}
\maketitle
\tableofcontents
\begin{abstract}
An element $\phi$ of the outer automorphism group $\Out(\f)$ of the rank $n$ free group $F_n$ is {\it polynomially growing} if the word lengths of conjugacy classes in $\f$ grow at most polynomially under iteration by $\phi$. It is {\it unipotent} if additionally its action on the first homology of $\f$ with integer coefficients is unipotent. In particular, if $\phi$ is polynomially growing and acts trivially on first homology with coefficients the integers mod 3 then $\phi$ is unipotent and also every polynomially growing element has a positive power that is unipotent. We solve the conjugacy problem in $\Out(\f)$ for the subset of unipotent elements. Specifically, there is an algorithm that decides if two such are conjugate in $\Out(\f)$.
\end{abstract}

\section{Introduction}In this paper, we consider  the conjugacy problem for $\Out(F_n)$, the group of outer automorphisms of the free group of rank $n$.  Namely, given  $\phi, \psi \in \Out(F_n)$,  find  an algorithm that  decides if $\phi$ and $\psi$ are conjugate in $\Out(F_n)$.

The case in which $\phi$ is fully irreducible, also known as iwip, was first solved by Sela \cite{zs:isomorphism} using his solution to the isomorphism problem for torsion-free word hyperbolic groups. This was recently generalized, using a similar approach, by Dahmani \cite{fd:conjugacy} to the case that $\phi$ is hyperbolic, or equivalently, that  every non-trivial element of $F_n$ has exponential growth under iteration by $\phi$.  See also \cite{fd:more}. An alternate approach to the fully irreducible case takes advantage of the fact that the finite set of (unmarked) train track maps  that represent a fully irreducible $\phi$  is a complete invariant for the conjugacy class of $\phi$.   Los \cite{jl:conjugacy} and Lustig \cite{ml:iwip} (see also \cite{hm:axes}) solved the conjugacy problem for fully irreducible $\phi$ by algorithmically constructing the set of (unmarked) train track maps for $\phi$. 

On the other end of the growth spectrum, the conjugacy problem for Dehn twists (equivalently rotationless, linearly growing $\phi$) was solved by Cohen and Lustig \cite{cl:conjugacy} using, among other things, Whitehead's algorithm (see below). Krsti\'c, Lustig and Vogtmann \cite{klv:conjugacy} proved an equivariant Whitehead algorithm and used that to solve the conjugacy problem for all elements with linear growth.

Building on the approach of Sela mentioned above, Dahmani and Touikan \cite{dt:conjugacy} reduce the conjugacy problem for $\Out(\f)$ to a list of problems about mapping tori of polynomial growing elements. This is applied in their solution to the conjugacy problem for outer automorphisms of free groups whose polynomially growing part is unipotent linear \cite{dt:unipotent}.

Dahmani, Francaviglia, Martino, and Touikan \cite{dfmt:conjugacy} solve the conjugacy problem for $\Out(F_3)$.

Lustig \cite{ml:conjugacy1, ml:conjugacy2} posted papers in 2000 and 2001 addressing the general case of the conjugacy problem but these have never been published.

Our main theorem addresses the case that $\phi$ is polynomially growing and rotationless, equivalently $\phi$ is polynomially growing and induces a unipotent action on on $H_1(F_n, \Z)$; we write $\phi \in \upgn$.  Being an element of $\upgn$ is a conjugacy invariant and can be checked  algorithmically.  

It is often the case, when studying $\Out(\f)$, that the techniques required to treat the  $\upgn$ case are very different from those needed for the cases in which there is exponential growth.  For example, the polynomially growing and exponentially growing  cases of the Tits alternative for $\Out(\f)$ are proved in separate papers \cite{bfh:tits2} \cite{bfh:tits1}.

\begin{thm} \label{t:main}  There is an algorithm that takes as input $\phi,\psi \in \upgn$ and outputs {\tt YES} or {\tt NO} depending on whether or not there exists $\theta \in \Out(\f)$  such that $\phi = \psi^\theta:=\theta\psi\theta^{-1}$.  Further, if {\tt YES} then the algorithm also outputs such a $\theta$. 
\end{thm}

\begin{remark}
If one knows that $\phi$ and $\psi$ are conjugate, then a conjugator $\theta$ can be produced by searching a list of the elements of $\Out(\f)$.  This is not what we do.  Rather, the construction of a conjugator, when one exists, is an integral part of the  proof of the main statement of Theorem~\ref{t:main}.
\end{remark}
 
 \begin{remark}  Theorem~\ref{t:main} is not an abstract existence theorem.  It is proved by constructing an explicit algorithm satisfying the conclusions of the theorem. The same is true for other results in this paper that begin with \lq There is an algorithm\rq.
   \end{remark}

A detailed description of the algorithm is given in Section~\ref{s:full algorithm} so we restrict ourselves here to four results/observations that underly our proof.

\medskip
$\bullet$    Each $\phi \in \upg$ is rotationless (Lemma~\ref{lem:rotationless}) and so   can be represented by a particularly nice relative train track map $\fG$ call a {\it \ct} ; see See Section~\ref{sec:ct}.   There is an algorithm (Theorem~\ref{t:ct is algorithmic}) to construct one such $\fG$ and from this   we can compute all of the invariants used in this paper.

 \medskip
$\bullet$  
 A set equipped with an action by a group $G$ is a {\it $G$-set}. A  $G$-set $X$ {\it satisfies property} \W\ (for   Whitehead) if it comes equipped with  an algorithm that takes as input $x,y\in X$ and outputs     {\tt YES} or {\tt NO} depending on whether or not there exists $\theta\in G$ such that $\theta(x)=y$ together with such a $\theta$ if {\tt YES}.
 We call such an algorithm an \W-{\it algorithm.}    The Whitehead/Gersten algorithm is a \W-algorithm for  the $\Out(\f)$-set of  finite lists of conjugacy classes of finitely generated subgroups of $\f$ \cite[Theorems W\&M]{sg:whitehead}, see also \cite{sk:gersten} and \cite{bfh:gersten}.  
 
 This can be applied directly to our problem by finding  subgroups associated to elements of $ \upg$.
  For example, there is a free factor system $\F_0(\phi)$ characterized by the fact that a conjugacy class in $F_n$ is carried by $\F_0(\phi)$ if and only if it grows linearly under iteration by $\phi$.  Since a free factor system is an unordered list of conjugacy classes of free factors,  we can check if   there exists $\theta \in \Out(\f)$ such that  $\F_0(\psi) =  \theta(\F_0(\phi))$.
If no such $\theta$ exists  then $\phi$ and $\psi$ are not conjugate.  If there is a such a $\theta$ then after replacing $\psi$ by $\psi^{\theta^{-1}}$, we may assume, as far as the conjugacy problem is concerned,  that $\F_0(\phi) = \F_0(\psi)$.  Moreover, any conjugator will preserve $\F_0(\phi) = \F_0(\psi)$.

In sections \ref{s:gersten}--\ref{sec:all staple pairs} we show that the Whitehead/Gersten algorithm can be used as the platform on which to build  other useful $\Out(\f)$-sets that satisfy property \W.   The  $\Out(\f)$-set of  finite lists of finitely generated subgroups of $\f$ also satisfies property \M\ (for McCool).  Namely, it is equipped with an algorithm that takes as input $x\in X$ and outputs a finite presentation for $G_x:=\{\theta\in G\mid \theta(x)=x\}$.  Although it is not strictly necessary for solving the conjugacy problem,  property \M\ is important in its own right and   we show that all of the $\Out(\f)$-sets constructed in sections \ref{s:gersten}--\ref{sec:all staple pairs}  satisfy property \M.
 \medskip
 
$\bullet$ Lemma~\ref{l:recognition}, which is an adaptation of the Recognition Theorem \cite[Theorem 5.3]{fh:recognition},  
 gives necessary and sufficient conditions for $\theta \in \Out(\f)$ to conjugate $\phi \in \upg$ to $\psi \in \upg$.  The non-numerical condition is that  $\theta(\L(\phi)) = \L(\psi)$ where $\L(\phi)$ is a  certain set of lines associated to $\phi$ and similarly for $\L(\psi)$.  If $\fG$ is a \ct\ representing $\phi$ then  $\L(\phi)$ is the set of lines carried by a finite type Stallings graph   $\stallings(f)$ called the {\it eigengraph}  for $f$.  $\stallings(f)$ depends on $f$ but the set of lines carried by $\stallings(f)$ depends only on $\phi$.  The numerical condition of Lemma~\ref{l:recognition} concerns the   \lq twist coordinates\rq\ associated to the linear parts of $\phi$ and $\psi$ and is relatively easy to handle; see Lemma~\ref{base case} and  Lemma~\ref{lem:twist index counts}.  Almost all of the paper is concerned with the existence or not of $\theta$ satisfying $\theta(\L(\phi)) = \L(\psi)$.

 \medskip
 
$\bullet$  A \ct \ $\fG$ comes equipped with a filtration $G_{i_0} \subset G_{i_1}  \subset \cdots  \subset G_{i_t}$ where  each $G_{i_j}$ is an $f$-invariant core subgraph and where obtained from $G_{i_{j-1}}$ by adding a single topological arc, possibly divided into two edges, for $j > 0$.    Edges  of  $G_{i_j} \setminus G_{i_{j-1}}$ are said to have height $j$.   $\stallings(f )$ has a compact core to which finitely many rays $\{R_E\}$ are added, one  for each  non-fixed non-linear  edge $E$ of $G$. Understanding the structure of rays is an important step in understanding $\L(\phi)$. Each $R_E$ has initial edge $E$ and $R_E \setminus E$ is a ray that crosses only edges with height strictly less than that of $E$.  (This is most definitely a \upg\ phenomenon.  If $E$ belongs to an  exponentially growing stratum then $E$ occurs infinitely often in  $R_E$.) Thus $R_E$ can be studied inductively, working up through the filtration.  This is carried out in Section~\ref{s:limit lines} and Sections~\ref{sec:all staple pairs}--\ref{s:algorithm}.

\medskip

Example~\ref{e:main example} gives an illustrative element of $\upg(\f)$ and is further developed as we progress through the text.

\section*{Acknowledgements}
We are indebted to an anonymous referee for an exceptionally thorough and useful report.

Feighn was supported by the National Science Foundation under Grant No.~DMS-1406167 and also under Grant No.~DMS-14401040 while the author was in residence at the Mathematical Sciences Research Institute in Berkeley, California, during the Fall 2016 semester.

Handel was supported by National Science Foundation grant  DMS-1308710 and by    PSC-CUNY  grants in Program Years 47 and  49.

\section{The algorithm}\label{s:full algorithm}

The logical structure of our proof of Theorem~\ref{t:main}  is a series of reductions 
$$\text{ Theorem~\ref{t:main}}  \Longleftarrow \text{ Proposition~\ref{p:conjugacy in X}}  \Longleftarrow \text{ Proposition~\ref{last step}} \Longleftarrow \text{ Proposition~\ref{inductive step}} $$
and a proof of Proposition~\ref{inductive step}.   The above theorem and  propositions   produce algorithms that we denote by $\alg_{\ref{t:main}}, \   \alg_{\ref{p:conjugacy in X}}, \ \alg_{\ref{last step}}$ and $\alg_{\ref{inductive step}}$ respectively.  The proof of the implication $\text{Theorem~\ref{t:main}}  \Longleftarrow \text{ Proposition~\ref{p:conjugacy in X}}$ shows how to use $\alg_{\ref{p:conjugacy in X}}$ to construct $\alg_{\ref{t:main}}$ and similarly for the other implications.  Thus $\alg_{\ref{t:main}}$  calls $\alg_{\ref{p:conjugacy in X}}$ which calls $\alg_{\ref{last step}}$ which calls $\alg_{\ref{inductive step}}$.

\subsection{$\text{Theorem~\ref{t:main}}  \Longleftarrow \text{ Proposition~\ref{p:conjugacy in X}}$}
One way to make progress on the conjugacy problem for \upg\   is to find {\it \W-invariants} for \upg; i.e. $\Out(\f)$-equivariant maps $J: \upg \to X$ where $X$  is      an $\Out(\f)$-set with a $\W$-algorithm $W_X(\cdot,\cdot)$.     If $W_X(J(\phi),J( \psi)) = {\tt NO}$ then $\phi$ is not conjugate to $\psi$ in $\Out(\f)$.   If  $W_X(J(\phi),J( \psi)) = ({\tt YES}, \xi)$ then $J(\psi^{\xi^{-1}}) = J(\phi)$.  Replacing $\psi$ by $\psi^{\xi^{-1}}$, we may assume that $J(\phi) = J(\psi)$.   In this case, any   $\theta$ conjugating $\phi$ to $\psi$ is   contained in the subgroup of $\Out(\f)$ that fixes $J(\phi)$. 

In Sections \ref{s:limit lines} through \ref{s:applying wg} we  construct seven such \W-invariants and bundle them into a single invariant ${\sf I}_\fc(\phi)$.  Once this is done, it is easy  to use an algorithm satisfying the conclusions of   Proposition~\ref{p:conjugacy in X}   to produce an algorithm satisfying the conclusions of Theorem~\ref{t:main}.   The details are given in the proof of Lemma~\ref{l:reduction}.

  Items (1) - (4) below outline how our ultimate \W-invariant ${\sf I}_{\fc} : \upg \to \ecat(\Abig)$ is chosen.  Item (5) refers to shrinking the set of potential conjugators from the stabilizer of ${\sf I}_{\fc}(\phi)$ to one of its finite index subgroups $\X_\fc(\phi)$.
     
   \begin{enumerate}
\item  \label{item:dyn inv}  {\bf (  Dynamical invariants of $\phi \in \upg$)} 
\begin{itemize} 
\item 
the finite multi-set $\Fix(\phi)$ of conjugacy classes of   {\it fixed subgroups of $\phi$}. (Definition~\ref{d:fix})
\item 
the  {\it linear free factor system} $\F_0(\phi)$.   (Definition~\ref{notn:F0})
\item 
 the finite set  $\{\fc\}$ of {\it  special $\phi$-chains}. (Section \ref{s:canonical ffs} and in particular Notation~\ref{notn:ffs})
\item  
the finite set ${\sf A}_\both(\phi)$   of {\it axes}  for $\phi$.  (Section~\ref{sec:axes}) 
\item  
the finite set   $ \sa(\phi)$ of {\it strong axes}   for $\phi$.  (Section~\ref{sec:axes}) 
 \item  
the finite set $\accnr(\phi)$  of all non-periodic {\it limit lines} for all eigenrays of $\phi$.  (Section~\ref{s:limit lines})
\ \ \item 
for each one-edge extension $\fe$ of each $\frak c$,   the set $\LW_{\fe}(\phi)$ of {\it added lines with respect to $\fe$}. (Definition~\ref{d:added lines})
\end{itemize}
\end{enumerate}

\noindent The invariants in the first four items  are {\it algebraic} in that they take values in  $\Out(\f)$-sets that can be expressed in terms of conjugacy classes of finitely generated subgroups of $\f$ or more generally are {\it iterated sets} (Section~\ref{s:iterated sets}).  In particular, they take values in    $\Out(\f)$-sets with  \W-algorithms and so can be used   as they are.  The others must be modified.
\begin{enumeratecontinue}
\item {\bf   (Algebraic versions of dynamical invariants)}
 \label{item:alg inv}
For the last three dynamical invariants,  define  corresponding (but weaker) algebraic invariants. The last two depend on a  choice of special  chain $\frak c$.  (Section~\ref{s:special chains})
 
\begin{itemize}
\item
 the finite set of {\it algebraic strong axes}. (Section~\ref{s:algebraic sa}) \item
the finite set 
$\{\sH_\fc(L):L\in\accnr(\phi)\}$ of {\it algebraic limit lines}. (Section~\ref{s:algebraic limit lines})
\item
for each one-edge extension $\fe$ in $\frak c$,   the finite set  $\sH_{\fe\in\fc}(\phi)$ of {\it algebraic added lines with respect to $\fe$}. (Section~\ref{s:algebraic added lines})
\end{itemize}
\end{enumeratecontinue}
\begin{remark}  If the seven dynamical invariants in \pref{item:dyn inv}  take the same values on $\phi$ and $\psi$ then, using Lemma~\ref{l:recognition}, it is easy to check if $\phi$ and $\psi$ are conjugate.  The same is not true for the seven algebraic invariants in \pref{item:dyn inv} and \pref{item:alg inv}.  Too much information was lost in  translation.
\end{remark}
\begin{enumeratecontinue}
 \item {\bf (\W-invariants})      Iterated sets, and in particular $\ecat(\Abig)$,  are defined in Sections~\ref{s:gersten} and \ref{s:ultimate atoms}.  By construction, all of our algebraic invariants take values in the iterated set $\ecat(\Abig)$.  We  construct a \W-algorithm for $\ecat(\Abig)$ (and all other iterated sets).
\item {\bf(The total invariant  ${\sf I}_{\fc}(\phi)$)}   is defined by combining  the   algebraic invariants in \pref{item:dyn inv} and \pref{item:alg inv} into a single  algebraic invariant that takes values in $\ecat(\Abig)$.  (Definition~\ref{d:I}). 
  \item{\bf (Reduce potential conjugators)}
   Elements of $\X_{\fc}(\phi) < \Out(\f)$ not only stabilize the algebraic invariants in \pref{item:dyn inv} and \pref{item:alg inv} making up ${\sf I}_{\frak c}(\phi)$, they also induce trivial permutations on those invariants that are finite sets.   (Definition~\ref{d:X}). 
\end{enumeratecontinue}
 As mentioned above Lemma~\ref{l:reduction}  is proved by constructing $\alg_{\ref{t:main}}$ using    $\alg_{\ref{p:conjugacy in X}}$ and properties of ${\sf I}_{\frak c}(\phi)$. Hence to prove Theorem~\ref{t:main}, we are reduced to proving:
  
  \bigskip

\noindent{\bf Proposition~\ref{p:conjugacy in X}}  
{\it    There is an algorithm that takes as input $\phi,\psi\in\upgn$ and a chain $\fc$ such that
\begin{itemize}
\item
$\fc$ is special for both $\phi$ and $\psi$ and
\item
${\sf I}_\fc(\phi)={\sf I}_\fc(\psi)$
\end{itemize}
and that outputs {\tt YES} or {\tt NO} depending whether or not there is $\theta\in\X_\fc(\phi)$ conjugating $\phi$ to $\psi$. Further, if {\tt YES} then such a $\theta$ is produced.
}

\bigskip

\subsection{$\text{Proposition~\ref{p:conjugacy in X}}  \Longleftarrow \text{ Proposition~\ref{last step}}$}

 $\alg_{\ref{p:conjugacy in X}}$ and $\alg_{\ref{last step}}$ differ only in the subgroup of potential conjugators that must be considered.  In  
Proposition~\ref{p:conjugacy in X} it is $\X_\fc(\phi)$ and in Proposition~\ref{last step} it is an  infinite index subgroup 
$\Ker(\bar Q^\phi) < \X_\fc(\phi)$ defined in    Definition~\ref{d:pair equivalence}. See statement of Proposition~\ref{p:conjugacy in X} below.

The set of (eigen)rays $\cR(\phi)$ (Definition~\ref{d:fix}) is a fundamental dynamical invariant of  $\phi$.  Each $r \in \cR(\phi)$ is the conjugacy class $[\ti r]$ of a point $\ti r \in \partial \f$.    There is no $\W$-algorithm for $\partial \f$ so we work with a weaker algebraic  invariant, the conjugacy class $F_{\fc}(r)$ of a free factor determined by $r$ and a special chain $\fc$; see Section~\ref{s:algebraic rays}. We do not list this in \pref{item:alg inv}  because it is built into  the set of algebraic lines and the set of algebraic  added lines.   The great advantage of $\Ker(\bar Q^\phi)$ over $  \X_\fc(\phi)$ is that in the proof of   Proposition~\ref{inductive step} we need only consider conjugating elements that preserve $r$. (See Lemma~\ref{new rays from lines} and Lemma~\ref{check any theta}.)
 Instead of having to check if two rays are conjugate, we need only check if they are equal.  
 
 The definition of $\bar Q^\phi(\xi)$ for $\xi \in \X_\fc$ is given in Definition~\ref{d:pair equivalence}.  The key result, from the algorithmic point of view,  is 
 
\bigskip 
\noindent{\bf Proposition~\ref{prop:barQ}} \ \ \ 
There is   an algorithm that produces a finite set $\{\eta_i\} \subset  \X$ so that the union of the cosets of $\Ker(\bar Q^\phi)$ determined by the $\eta_i$'s contains each $\theta \in \X$ that conjugates $\phi$ to $\psi$.

\bigskip

The proof of Proposition~\ref{prop:barQ} requires a detailed understanding  of the structure of eigenrays and is the most technical part of the paper.  The proof of Lemma~\ref{p:reduction2} shows how to quickly construct $\alg_{\ref{p:conjugacy in X}}$ using 
 $\alg_{\ref{last step}}$ and the coset representatives produced by the algorithm of Proposition~\ref{prop:barQ}. In other words, to prove Proposition~\ref{p:conjugacy in X}, we are reduced to proving:

\bigskip 
\noindent{\bf Proposition~\ref{last step}} 
{\it There is an algorithm that takes as input  $\phi, \psi \in \upgn$ and a chain $\fc$ such that 
\begin{itemize}
\item
$\fc$ is a special chain for $\phi$ and $\psi$ and 
\item
${\sf I}_\fc(\phi) = {\sf I}_\fc(\psi)$
\end{itemize}
and that outputs {\tt YES} or {\tt NO} depending on whether or not there is $\theta \in \Ker(\bar Q^\phi)$ conjugating  $\phi$ to $\psi$.   Further, if  {\tt YES} then   such a $\theta$ is produced. 
}
\bigskip

 \subsection{$\text{Proposition~\ref{last step}}  \Longleftarrow \text{ Proposition~\ref{inductive step}}$}
 
 This is an easy step.  The details are given in  \lq Proof of Proposition~\ref{last step} (assuming Lemma~\ref{base case} and Proposition~\ref{inductive step})\rq\ following the statement of Proposition~\ref{inductive step}. After this step, we may assume that the restrictions of $\phi$ and $\psi$ to the linear free factor system $\F_0(\phi)=\F_0(\psi)$ are equal. This provides the basis for an inductive argument completed in the next step.
 
 \subsection{Proof of Proposition~\ref{inductive step}}  Proposition~\ref{inductive step} is the inductive step of an argument   up the filtration induced by $\fc$.       There are six items labeled (1)--(5), (7) that are sequentially checked.  If any of these is false then return {\tt NO}.  Otherwise,  construct the desired conjugator following  pages~\pageref{d1 part 1}--\pageref{last lemma}.

\section{Background} \label{s:background}
\subsection{Standard Notation}  The free group on $n$ generators is denoted $F_n$. For $a\in\f$, conjugation by $a$ is denoted $i_a$, i.e.\ $i_a(x)=axa^{-1}$ for $x\in\f$. The group of automorphisms of $F_n$, the group of inner automorphisms of $F_n$ and the group of outer automorphisms of $F_n$ are denoted by $\Aut(F_n)$, $\Inn(F_n):=\{i_a\mid a\in\f\}$ and $\Out(F_n) =  \Aut(F_n)/\Inn(F_n)$ respectively. 

For subgroups $H<\f$, $[H]$ denotes the conjugacy class of $H$ and, for elements $a\in\f$, $[a]$ denotes the conjugacy class of $a$. 
 
An outer automorphism $\phi \in \Out(F_n)$ has {\em polynomial growth}, written $\phi \in \pg$, if for each $a \in F_n$ there is a polynomial $P$ such that reduced word length of $\phi^k([a])$ with respect to  a fixed set of generators of $F_n$  is bounded above by $P(k)$.  Equivalently, the set of attracting laminations for $\phi$ \cite[Section 3]{bfh:tits1} is empty.   The set $\upgn$ of {\em unipotent outer automorphisms} is  the subset of $\Out(F_n)$ consisting of polynomially growing   $\phi$ whose  induced action on $H_1(F_n, \Z)$ is unipotent. We sometimes  write $\phi \in \upg$ instead of $\phi \in \upgn$.     In Section~\ref{upg is rotationless} we show that $\phi \in \upg$ if and only $ \phi \in \pg$ and $\phi$ is rotationless  in the sense of  \cite[Definition~3.13]{fh:recognition} (where it is called forward rotationless). There is $K_n>0$ such that if $\phi \in \pg$ then $\phi^{K_n}\in\upg$ \cite[Corollary~3.14]{fh:CTconjugacy}.

The graph with one vertex $*$ and with $n$ edges is the {\em rose} $R_n$.  Making use of the standard identification of $\pi_1(R_n,*)$ with $F_n$,   there are bijections between $\Aut(F_n)$  and the group of  pointed homotopy classes of homotopy equivalences $f:(R_n,*)  \to (R_n,*)$ and  between $\Out(F_n)$  and the group of  free homotopy classes of homotopy equivalences $f:R_n  \to R_n$.

If  $G$  is a graph without valence one vertices then a homotopy equivalence $\mu:R_n \to G$  is  called  a {\em marking} and $G$, equipped with a marking, is called  a {\em marked graph}. A marking $\mu$ induces  an   identification, well-defined  up to inner automorphism,  of  the fundamental group of $G$ with the fundamental group of $R_n$ and hence with $F_n$.   This in turn induces an identification of the group of homotopy classes of  homotopy equivalences $f : G \to G$ with $\Out(F_n)$.  If $\phi \in \Out(F_n)$ corresponds to the homotopy class of $\fG$ then we say that $\fG$ {\em represents} $\phi$. In Section~\ref{sec:ct} we recall the existence of very well behaved homotopy equivalences $f: G \to G$ representing an element of $\upgn$.

\begin{ex}\label{e:main example}
Here is an example of a homotopy equivalence $\fG$ of a marked graph that represents an element $\phi$ of $\upg(F_5)$. Let $F_5$ be represented as the fundamental group of the rose $R_5$, let $G$ be the subdivision of $R_5$ pictured in Figure~\ref{f:G}
\begin{figure}[h!] 
\centering
\includegraphics[width=.25\textwidth]{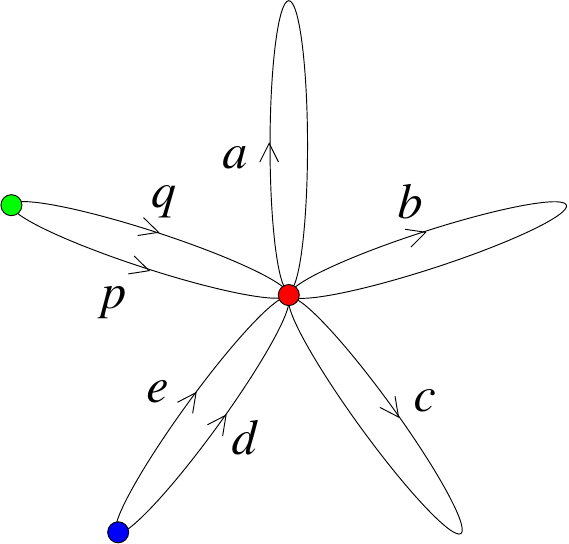}

\caption{$G$
}
\label{f:G}
\end{figure}
and let $\fG$ be given by $a\mapsto a$, $b\mapsto ba$, $c\mapsto cb$, $d\mapsto db^2$, $e\mapsto eb^3$, $p\mapsto pa^2$, $q\mapsto qc$. To see that $\phi$ has polynomial growth, note that
the edge $q$ has cubic growth in that $$|f^k(q)|=|q\cdot c\cdot f(c)\cdot f^2(c)\cdot f^3(c)\cdot\ldots\cdot f^{k-1}(c)|=|q\cdot c\cdot cb\cdot cbba\cdot \ldots \cdot cbba\ldots ba^{k-1}|=\frac{k^3+5k+6}{6}$$ and that no edge grows at a higher rate. In particular conjugacy classes of $\f$ have at most cubic growth.
As we progress through this paper, we will expand upon this example.
\end{ex}

\subsection{Paths, circuits and lines}\label{s:lines}
  
  A {\em path} in a marked graph $G$ is a proper immersion of a closed interval into $G$.  In this paper, we will assume that  the endpoints of a path, if any,  are at vertices. If the interval is degenerate then the path is {\em trivial}; if the interval is infinite or bi-infinite then the path is a {\em ray} or a {\em line} respectively.  We do not distinguish between paths that differ only by a reparameterization of the domain interval.   Thus, every non-trivial path  has a description as a concatenation of oriented edges  and we will use this {\em edge  path} formulation without further mention. Reversing the orientation on a path $\sigma$ produces a path denoted either $\bar \sigma$ or $\sigma^{-1}$.   A {\em circuit} is an immersion of $S^1$  into $G$.  Unless otherwise stated, a circuit is assumed to have an orientation.  Circuits have cylic edge decompositions.  Each conjugacy class in $F_n$ is represented by a unique circuit in $G$. The conjugacy class in $\f$ represented by the circuit $\sigma$ is denoted $[\sigma]$.
  
\begin{notn} \label{n:no partial} Each  $\Phi \in \Aut(F_n)$ induces an equivariant  homeomorphism of $\partial \f$.  To simplify notation somewhat, we refer to this extension as $\Phi$ rather than, say, $\partial \Phi$.    In situations where this might cause confusion, we write $\Phi | \partial \f$ for the induced  homeomorphism of $\partial \f$.  For example, $\Fix(\Phi)$ is the subgroup of $\f$ fixed by $\Phi$ and $\Fix(\Phi | \partial \f)$ is the set of points in $\partial \f$  fixed by the induced homeomorphism.

The action of $F_n$  on $\partial F_n$ is by conjugation, i.e. by $a \cdot P =  i_a(P)$ for $P \in \partial \f$.   For each non-trivial $a \in  F_n$, $ i_a$ fixes two points in $\partial F_n$: a repeller $a^-$ and an  attractor  $a^+$.

A marking $\mu$  induces  an   identification, well-defined  up to inner automorphism,  of  the set of ends of $\ti G$  with $\partial F_n$ and  likewise the group of covering translations of $\ti G$   with  $\Inn(F_n)$. We choose such an identification once and for all.  The covering translation corresponding to $i_a$  is denoted $T_a$ as is the extension of $T_a$ to a homeomorphism of $\partial \f$.  We have  $T_a | \partial \f   = i_a | \f$.   If $\fG$ represents $\phi$ then each lift $\ti f : \ti G \to \ti G$ induces an equivariant  homeomorphism, still called $\ti f$, of $\partial \f$; see, for example, Section~2.3 of \cite{fh:recognition}.     There is a bijection between the set of  lifts $\ti f$ of $\fG$ and the set of   automorphisms $\Phi$ representing $\phi$ defined by $\ti f \leftrightarrow \Phi$ if $\  \ti f | \partial \f=   \Phi | \partial \f$. 
 \end{notn}  
 
A  line $\ti L$ in the universal cover $\ti G$ of a marked graph $G$ is a bi-infinite edge path.  The ends of $\ti L$ determine ends of $\ti G$ and hence points in $\partial F_n$.   In this way,   the  {\em space of oriented lines in the tree $\ti G$}  can be identified with the space $\ti \B$ of ordered pairs of distinct elements of $\partial F_n$.  The {\em space of oriented lines in $G$} is then identified with the space $\B$ of $F_n$-orbits of elements of $\ti \B$.
The topology on $\partial F_n$ induces a topology on $\ti \B$ and hence a topology on $\B$ called the {\em weak topology}.  

\medskip

\subsection{Free factor systems}\label{sec:ffs}
The subgroup system $\F = \{[A_1],\ldots, [A_m]\}$ is a {\em free factor system} if  $A_1,\ldots, A_m$ are non-trivial free factors of $F_n$   and either $F_n =  A_1 \ast \ldots \ast A_m$ or $F_n =  A_1 \ast \ldots \ast A_m \ast B$ for some non-trivial free factor $B$.    The $[A_i]$'s are the {\em components} of $\F$.   If  $G$ is a marked graph and $K$ is a subgraph whose  non-contractible components are $K_1,\ldots,K_m$ then $\F(K, G) := \{[\pi_1(K_1)],\ldots,[\pi_1(K_m)]\}$ is a \ffs \ that is {\it realized by} $K \subset G$.    Every \ffs\ $\A$  can be  realized  by  $K \subset G$ for some marked graph $G$ and some core subgraph $K \subset G$. Recall that a graph is {\it core} if through every edge there is an immersed circuit and that the {\it core of a graph} is the union of the images of its immersed circuits.

We write $\F_1 \sqsubset \F_2$ and say that $\F_1$ {\em is contained in } $\F_2$  if for each component $[A_i]$  of $\F_1$ there is a component $[B_j]$ of $\F_2$ so that $A_i$ is conjugate to a subgroup of $B_j$.  Equivalently, there is a marked graph $G$ with core subgraphs $K_1 \subset K_2$ so that $\F_1 =  \F(K_1,   G)$ and $\F_2 = \F(K_2,   G)$.   If one can choose $K_1$ and $K_2$ so that $K_2 \setminus K_1$ is a single edge then we say that $\F_1 \sqsubset \F_2$ is a {\em one-edge extension}.  For example, $\{[A]\} \sqsubset \{[B]\}$ is a one-edge extension if and only if $\rk(B) =\rk(A) +1$ and   $\{[A_1], [A_2]\} \sqsubset \{[B]\}$ is a one-edge extension if and only if $\rk(B) =\rk(A_1) + \rk(A_2) $. 
 
\begin{ex}  Suppose that $H_1$ is a  subgraph of a marked graph $G$,  that   $H_2 = H_1 \cup E_2  \subset G$  where $E_2$ is an edge that forms a loop that is disjoint from $H_1$ and that $H_3 = H_2 \cup E_3$ where $E_3 \subset G$ is an edge with one endpoint in $H_1$ and the other at the unique endpoint of $E_2$.   Then $\F(H_1,G)\sqsubset \F(H_2,G)$ and  $\F(H_2,G)\sqsubset \F(H_3,G)$ are proper inclusions and $\F(H_1,G)\sqsubset \F(H_3,G)$  is a one-edge extension.     This is essentially the only way in which a one-edge extension can be \lq reducible\rq.   We record a specific consequence of this in the following lemma.
\end{ex} 

\begin{lemma} \label{reducible extensions} Suppose that  $\F(H_1,G)\sqsubset \F(H_2,G)$ and  $\F(H_2,G)\sqsubset \F(H_3,G)$ are proper inclusions and that $\F(H_1,G)$ and $ \F(H_2,G)$ have the same number of components.  Then $\F(H_1,G)\sqsubset \F(H_3,G)$  is a not one-edge extension. 
\end{lemma}
 
\proof This follows from \cite[Part 2, Definition~2.4 and Lemma 2.5]{handelMosher:subgroups}.
\endproof

If  $\F = \{[A_1],\ldots, [A_m]\}$ and $a \in F_n$ is conjugate into some $A_i$ then $[a]$ is {\em carried} by $\F$.  A line $L \in \B$  is {\em carried} by $\F$ if it is a limit of periodic lines corresponding to conjugacy classes that are carried by $\F$.  Equivalently,  $[a]$ or   $L$ is {\em carried} by $\F$  if for some, and hence every, $K \subset G$ realizing   $\F $,  the realization of $[a]$ or $L$ in  $G$ is contained in $K$.   For every collection of conjugacy classes and lines there is a unique minimal (with respect to $\sqsubset$) free factor system that carries each element of the collection \cite[Corollary~2.6.5]{bfh:tits1}.

\begin{notn}\label{n:restriction to free factor}  
$\Out(\f)$ acts on the set of conjugacy classes $[F]$ of free factors $F$.  If $\phi \in \Out(\f)$ fixes $[F]$  then we say that  {\it $[F]$ is $\phi$-invariant} and write $\phi | [F]$  for the {\it restriction of $\phi$ to $[F]$} (which is well defined because $F$ is its own normalizer in  $\f$).  We often say that $F$ is $\phi$-invariant and write $\phi | F$ just to simplify notation.   \cite[Proposition  4.44]{bfh:tits2} implies that if $\phi $ is \upg\   then $\phi|F$ is \upg.
 If $\F = \{[A_1],\ldots,[A_m]\}$ is a free factor system and each $[A_i]$  is $\phi$-invariant then we say that $\F$ is $\phi$-invariant and denote $ \{\phi |A_1,\ldots,\phi|A_m\}$ by $\phi | \F$. 
\end{notn} 
 \medskip

\subsection{$\FixN(\phi)$, principal lifts and $\cR(\phi)$} \label{sec:FixN}
We continue with Notation~\ref{n:no partial}. If $P \in \Fix(  \Phi |\partial \f)$ and if  there is a neighborhood $U$ of $P$ in $  \partial F_n$ such that $ \Phi(U) \subset U$ and such that $\cap_{i =1}^{\infty}  \Phi^i(U) = P$ then $P$ is {\em attracting}.     If $P$ is an attracting fixed point for $ \Phi^{-1} | \partial \f$ then it is a {\em repelling} fixed point for $ \Phi | \partial \f$.   By $\Fix_+(\Phi), \Fix_-(\Phi)$ and $\FixN(\Phi)$  we denote the   {\em set of attracting fixed points} for $ \Phi | \partial \f$,  the {\em set of  repelling fixed points} for $ \Phi| \partial \f$ and the {\em set of non-repelling fixed points} for $ \Phi |\partial \f$ respectively; thus $\FixN(\Phi) =  \Fix(\Phi | \partial \f) \setminus \Fix_-(\Phi)$.  Note that all of these sets are contained in $\partial \f$.
  
If $A < F_n$ is a finitely generated subgroup then the inclusion of $A$ into $F_n$ is a quasi-isometric embedding and so extends to an inclusion of $\partial A$ into $\partial F_n$ with the property that $\{ a^\pm: \mbox{non-trivial }a \in A\}$ is dense in $\partial A$. In particular, since the subgroup $\Fix(\Phi)$ consisting of elements in $F_n$ that are fixed by $\Phi$ is finitely generated \cite{sg:fixed} (see also \cite{bh:tracks} and the references therein), we have $\partial \Fix(\Phi) \subset \partial F_n$. The following  lemma   implies that $\partial \Fix(\Phi) \subset \Fix( \Phi | \partial \f)$ and that $\Fix_+(\Phi), \Fix_-(\Phi)$ and $\FixN(\Phi)$ are $\Fix(\Phi)$-invariant.

\begin{lemma} \label{basic fix}   Let $\Phi\in\Aut(\f)$ and $0\not= a\in\f$. The following are equivalent:
\begin{itemize}
\item
$a \in \Fix(\Phi)$;
\item
either $a^-$ or $a^+$ is  contained in $ \partial \Fix(\Phi)$;
\item
both $a^-$ and $a^+$ are contained in $ \partial \Fix(\Phi)$;
\item
$i_a$ commutes with $\Phi$; and
\item
$ i_a | \partial \f$ commutes with $ \Phi | \partial \f$.
\end{itemize}
\end{lemma}

\proof  This is well known; see, for example, Lemmas~2.3 and 2.4 of \cite{bfh:tits3} and Proposition I.1 of \cite{gjll:index}.
\endproof

\begin{lemma} \label{two automorphisms}  If $P \in \partial F_n$ is fixed by  automorphisms $\Phi \ne \Phi'$ representing $\phi \in \Out(F_n)$ then $P = a^{\pm}$ for some non-trivial $a \in F_n$.
\end{lemma}

\proof  There exists non-trivial $a \in F_n$ such that $i_a  = \Phi^{-1} \Phi'$ fixes $P$.
\endproof

\begin{definition} \label{d:PA}
An automorphism $\Phi$ representing $\phi\in\upgn$ is {\em principal} if $\FixN(\Phi)$ contains at least two points and if $\FixN(\Phi) \ne \{a^-,a^+\}$ for any non-trivial $a \in F_n$.  The {\em set of principal automorphisms representing $\phi$} is denoted $\cP(\phi)$.  See Section~3.2 of \cite{fh:recognition} for complete details.
\end{definition} 

\begin{lemma}  \label{gjll} If $\Phi$ is   principal  then $\FixN(\Phi)$ is the disjoint union of $\partial \Fix(\Phi)$ and $\Fix_+(\Phi)$.  Moreover,   $\Fix_+(\Phi)$ is a union of  finitely many $\Fix(\Phi)$ orbits.
 \end{lemma}
 
 \proof The first assertion follows from Proposition I.1 of \cite{gjll:index}.   The second is obvious if  $\Fix_+(\Phi)$ is finite and follows from Lemma~2.5 of  \cite{bfh:tits3} if $\Fix_+(\Phi)$ is infinite. 
 \endproof
 
\begin{remark} \label{rem:dense non-fixed points} {We sometimes say that $P \in \partial F_n$ is periodic if it is fixed by $i_a$ for some non-trivial $a \in \F_n$.  } Non-periodic points 
 are dense in $\FixN(\Phi)$ for  each $\Phi \in \cP(\phi)$. Lemmas~\ref{basic fix} and \ref{gjll} imply that no element of $\Fix_+(\Phi)$ is periodic.  If $\Fix_+(\Phi) \ne \emptyset$ then $ \Fix_+(\phi)$ is dense in $\FixN(\phi)$ and we are done.   Otherwise, $\Fix(\Phi)$ has rank at least two  and $ \FixN(\Phi) = \partial \Fix(\Phi)$.
\end{remark}

\begin{definition} \label{d:isogredience}
Two automorphisms $\Phi_1$ and $\Phi_2$ are in the same {\em isogredience class} if there exists $a \in F_n$ such that $\Phi_2 = i_a \Phi_1 i_a^{-1}$, in which case $\FixN(\Phi_2) = i_a\FixN(\Phi_1)$ and similarly for $\Fix_-(\Phi_2) ,\Fix_+(\Phi_2)$ and $\Fix(\Phi_2)$.  It follows that if   $\Phi_1$ and $\Phi_2$ are  isogredient then $[\Fix(\Phi_1)] = [\Fix(\Phi_2)]$ and  $[\FixN(\Phi_1)] = [\FixN(\Phi_2)]$ where $[ \ ]$ denotes the orbit  under the action of $F_n$ on sets of points in $\partial F_n$. It is easy to see that   isogredience defines an equivalence relation on  $\cP(\phi)$.   The {\em set of isogredience classes} of $\cP(\phi)$ is denoted $[\cP(\phi)]$.
\end{definition}

  We recall the following result from Remark~3.9 of \cite{fh:recognition}; see also Lemma~\ref{nielsen classes and principal lifts} of this paper. 
  
\begin{lemma}  $\cP(\phi)$ is a finite union of isogredience classes.
\end{lemma}

Our next lemma states that $[\FixN(\Phi)]$ determines the isogredience class of $\Phi \in \cP(\phi)$.
  
\begin{lemma}  Suppose that $\Phi_1, \Phi_2 \in \cP(\phi)$.   Then $\Phi_1$ and $\Phi_2$ are isogredient if and only if $[\FixN(\Phi_1)] = [\FixN(\Phi_2)]$.  More precisely,  $\Phi_2 = i_a \Phi_1 i_a^{-1}$ if and only if $\FixN(\Phi_2) = i_a\FixN(\Phi_1)$.
\end{lemma}

\proof    It is obvious that if  $\Phi_2 = i_a \Phi_1 i_a^{-1}$  then $\FixN(\Phi_2) =  i_a\FixN(\Phi_1)$.  For the converse note that if  $\FixN(\Phi_2) = i_a\FixN(\Phi_1)= \FixN(i_a\Phi_1i_a^{-1})$ then   $ \Phi_2^{-1}   i_a \Phi_1 i_a^{-1}$ is an inner automorphism whose induced action on $\partial \f$ fixes $\FixN(\Phi_2)$ and so  is not equal to $\{a^-,a^+\}$ for any non-trivial $a$.  This proves  that  $ \Phi_2^{-1}   i_a \Phi_1 i_a^{-1}$ is trivial and so  $\Phi_2 = i_a \Phi_1 i_a^{-1}$.
\endproof

\begin{definition}\label{d:fix}
Define sets \[ \FixN(\phi):=\{ [\FixN(\Phi_1)],  \ldots, [\FixN(\Phi_m)]\}\] and \[\cR(\phi) := \{[P]: P \in  \cup_{i=1}^m  \Fix_+(\Phi_i)\} \subset \partial F_n/F_n\] 
 and a multi-set (repeated elements allowed)
\[ \Fix(\phi):=\{[\Fix(\Phi_1)], \dots,[\Fix(\Phi_m)]\}\] where the $\Phi_i$'s are representatives of the isogredience classes in $\cP(\phi)$.    Thus $\FixN(\phi)$ is a finite set of $F_n$-orbits of subsets of $\partial F_n$ and   $\cR(\phi)$ is a finite set of $F_n$-orbits   of points in $\partial F_n$.
\end{definition}

\begin{definition}
For us a {\it natural invariant of a group $G$} is a map $I:G\to X$ where $X$ is a $G$-set and, for all $\phi, \theta\in G$, we have $I(\phi^\theta)=\theta\big(I(\phi)\big)$. 
\end{definition}

The following lemma says that $[\cP(\phi)]$, $\Fix(\phi)$, $\FixN(\phi)$, and $\cR(\phi)$ are natural invariants of $\Out(\f)$.

\begin{lemma} \label{first theta}  Suppose that $\Theta \in \Aut(\f)$ represents $\theta \in \Out(\f)$ and that $\psi = \theta \phi \theta^{-1}$.  Then 
\begin{enumerate}
\item $\Phi \mapsto \Psi:= \Theta \Phi \Theta^{-1}$ defines a bijection between $\cP(\phi)$ and $\cP(\psi)$ and induces a bijection $[\cP(\phi)] \leftrightarrow [\cP(\psi)] $.
\item $\Fix(\Psi) = \Theta(\Fix(\Phi)),\  \FixN(\Psi) =   \Theta(\FixN(\Phi))  $ and $\Fix_+(\Psi) =   \Theta(\Fix_+(\Phi)) $.
\item $\Fix(\psi) = \theta(\Fix(\phi)),\  \FixN(\psi) = \theta(\FixN(\phi))  $ and $ \cR(\psi) =  \theta( \cR(\phi)) $.
\end{enumerate}
\end{lemma}

\proof  The automorphism $\Psi$ represents $\theta \phi \theta^{-1}= \psi \in \Out(F_n)$.     If $\Phi' =  i_c \Phi i_c^{-1}$ then $\Psi' := \Theta \Phi' \Theta^{-1} =   i_{\Theta(c)} (\Theta\Phi\Theta^{-1}) i_{\Theta(c)}^{-1} = i_{\Theta(c)} \Psi i^{-1}_{\Theta(c)} $ so conjugation by $\Theta$ maps isogredience classes of $\phi$ to isogredience classes of $\psi$.    The items in (2) are easy standard facts about conjugation.  Since $ \Theta(a^{\pm})= (\Theta(a))^{\pm}$,  it follows that $\Psi$ is principal if $\Phi$ is principal.    The induced map $\cP(\phi) \to \cP(\psi)$ is obviously invertible and is hence a bijection.   This completes the proof of (1).   If $\Theta$ is replaced by $i_a\Theta$ then $\Psi$ is replaced by $i_a \Psi i_a^{-1}$ and $\Fix(\Psi), \FixN(\Psi)$ and $\Fix_+(\Psi)$ are replaced by $i_a(\Fix(\Psi)),  i_a(\FixN(\Psi))$ and $ i_a(\Fix_+(\Psi))$ respectively.   Thus $\theta([\Fix(\Phi)]) = [\Fix(\Psi)], \  \theta([\FixN(\Phi)]) = [\FixN(\Psi)]$ and $\theta([\Fix_+(\Phi)]) = [\Fix_+(\Psi)]$.  This verifies (3).  
\endproof

{The following lemma is used implicitly throughout the paper.}
\begin{lemma}   {If $A$ is a $\phi$-invariant  free factor   then the inclusion of $\partial A$ into $\partial F_n$ induces an  inclusion of $\cR(\phi | A)$ into $\cR(\phi)$.}
\end{lemma}

\proof  {   An automorphism $\Phi' : A \to A$ representing $\phi | A$ extends to an automorphism $\Phi : F_n \to F_n$ representing $\phi$.  We claim that if $P \in \partial A$ then $P \in \Fix_+(\Phi')$ if and only if $P \in \Fix_+(\Phi)$.  Symmetrically, $P \in \Fix_-(\Phi')$ if and only if $P \in \Fix_-(\Phi)$.   It follows that $\FixN(\Phi') \subset  \FixN(\Phi)$ and hence that $\Phi$ is principal if $\Phi'$ is principal.  This will complete the proof of the lemma.}

{To prove the claim, extend a  basis $\A$ for $A$ to a basis $\B $ for $\f$. Following   \cite{gjll:index}, we view $P \in \partial F_n$ as an infinite word $ P = x_1x_2x_3\ldots$  with each $x_i \in \B$.   For each $i \in \mathbb N$, let $x_1,\ldots, x_{k(i)}$ be the common initial segment of $\Phi(x_1\ldots x_i)$ and $P$.  Then $P \in \Fix_+(\Phi)$ if and only if $k(i) - i \to \infty$ \cite[Proposition I.1]{gjll:index}.  If $P \in \partial A$ then  each $x_i  \in \A$  and each $\Phi(x_1\ldots x_i) = \Phi'(x_1\ldots x_i) \in A$ so $k(i)$ is the same whether we compute using $\Phi$ or $\Phi'$.  }
\endproof

\subsection{UPG is rotationless}  \label{upg is rotationless}Relative train track theory is most effective when  applied to elements of $\Out(F_n)$ that are rotationless as defined in \cite[Definition 3.13 and Remark 3.14]{fh:recognition}.    In this section, we show that for \pg\  elements,  $\phi$ is rotationless if and only if $\phi$ is \upg.  The exact definition of rotationless plays no role in this paper so is not repeated here.

\begin{lemma} \label{lem:rotationless} Each $\phi \in \upg$ is rotationless.
\end{lemma} 

\proof   
By \cite[Proposition 5.7.5]{bfh:tits1}, there is a sequence  $\F_0 \sqsubset \F_1 \sqsubset \ldots \sqsubset \F_K$ of $\phi$-invariant one-edge extensions where $\F_0$ is trivial and $\F_K = \{[F_n]\}$.   We may assume without loss that  $\F_0 \sqsubset \F_1 \sqsubset \ldots \sqsubset \F_K$ is a maximal such chain. 

\cite[Theorem 2.19]{fh:recognition}, which makes no assumptions on $\phi$,   proves the existence of a relative train track map $\fG$  and filtration $\filt$ representing $\phi$ and satisfying a certain list of five properties, two of which are denoted by (P) and  (NEG).  Additionally, the filtration realizes   $\F_1 \sqsubset \F_2 \sqsubset \ldots \sqsubset \F_K$  in the sense that  each $\F_k$   is represented by a core filtration element  $G_{i_k}$.  Since  $\F_0 \sqsubset \F_1 \sqsubset \ldots \sqsubset \F_K$  is maximal, $G_{i_k}$  is obtained from $G_{i_{k-1}}$ by adding either  a topological circle that is disjoint from $G_{i_{k-1}}$  or  a topological arc with both endpoints in $G_{i_{k-1}}$ \cite[Part II Lemma 2.5]{handelMosher:subgroups}.  We denote the closure of $G_{i_k} \setminus G_{i_{k-1}}$, equipped with the simplicial structure inherited from $G_{i_k}$,  by $\hat H_k$.   Since $\phi$ is \pg, there are no EG strata.

We use the following consequences of properties (P) and (NEG).
\begin{enumerate}
\item   The terminal endpoint of a non-periodic edge in $\hat H_k$ is contained in $G_{i_{k-1}}$.
\item  If $\hat H_k$ contains a periodic edge then it is a single periodic stratum \cite[Lemma~2.20(1)]{fh:recognition}. 
\end{enumerate}

 If $\hat H_k$ is a circle that is disjoint from $ G_{i_{k-1}}$, then its conjugacy class is fixed by some iterate of $\phi$ and so is fixed by $\phi$ \cite[Proposition 3.16]{bfh:tits2}.  There are two possibilities;  $\hat H_k$ is a single fixed edge; or $\hat H_k$  has more than one   edge and $f|\hat H_k$ is a non-trivial rotation with one orbit of edges.  In the latter case, we say that $\hat H_k$ is a {\it rotating circle}.
 
 If $\hat H_k$ intersects $ G_{i_{k-1}}$ then it is a topological arc $E_k$ whose ends may or may not be identified.  Either $f(E_k) =  v_k E_k u_k$ or $f(E_k) = v_k \bar E_k u_k$ for some paths $u_k,v_k \subset G_{i_{k-1}}$ \cite[Corollary 3.2.2]{bfh:tits1}.   Since $\phi$ is \upg, the latter is ruled out by \cite[Proposition 5.7.5(2)  - see the second paragraph on page 595 ]{bfh:tits2}.  
If both $u_k$ and $v_k$ are trivial then $E_k$  is a single fixed edge.  If exactly one of $u_k$ and $v_k$ is trivial then $E_k$ is a single non-periodic NEG edge. If neither $u_k$ and $v_k$ are trivial then  $E_k$ consists of two non-periodic NEG edges with a common fixed initial endpoint.  In all three cases, the directions determined by $E_k$ and $\bar E_k$ are either non-periodic or fixed.

An easy induction argument on $k$ shows that  
\begin{enumerate}[(a)] 
\item If a vertex $v$ is not contained in a rotating circle then $v$  is fixed by $f$ and each periodic direction based at $v$ is fixed by $f$.
\item Each rotating circle is a component of $\Per(f)$, the set of periodic points for $f$, and each point in a rotating circle has exactly two periodic directions.
\end{enumerate}
These are exactly the conditions needed to verify that $\fG$  is rotationless in the sense of \cite[Definition 3.18]{fh:recognition}. \cite[Proposition 3.19]{fh:recognition} states that the existence of a rotationless $\fG$ satisfying the conclusions of \cite[Theorem 2.19]{fh:recognition} is equivalent to   $\phi$ being rotationless.
\endproof 

\begin{remark} The converse of Lemma~\ref{lem:rotationless}, that every rotationless \pg\ $\phi$ is \upg,  is also true.  We make no use of this fact but include a proof for completeness.  See Section~\ref{sec:ct} for a review of \ct s. Since $\phi$ is rotationess and $\pg$, it   is represented by a \ct\ $\fG$ without \eg\ or zero strata.  For any such  $\fG$, there is a filtration $\filt$ by $f$-invariant core subraphs  such that $G_{i}$ is obtained from $G_{i-1}$ by adding a single topological edge $E_{i}$ whose image $f(E_i) \subset G_i$ crosses $E_i$ exactly once and crosses $\bar E_i$ not at all.    \cite[Proposition 5.7.5]{bfh:tits1} therefore implies that $\phi$ is \upg.   
\end{remark}

\subsection{\ct s}  \label{sec:ct}  
A {\em \ct}\ is a particularly nice kind of homotopy equivalence $\fG$ of a marked directed graph.   Every rotationless $\phi$, and in particular every $\phi \in \upgn$, is represented by  a \ct; see  \cite[Theorem 4.28]{fh:recognition} or \cite[Theorem 5.1.8]{bfh:tits1}.  Moreover, \ct s  are considerably simpler in the $\upgn$-case than in the general case. 

For the remainder of the section we assume that $\fG$ is a \ct\ representing an element of $\upgn$ and  review its properties. Complete details can be found in \cite{fh:recognition}   (see in particular Section 4.1) and in \cite{fh:CTconjugacy}.  The latter introduces the (Inheritance) property for a \ct \ $\fG$, which states that the restriction of $f$ to each component of each core filtration element is also a \ct, and contains an algorithm to produce \ct s satisfying  (Inheritance).  We say that $\fG$ {\it realizes} a chain $\F_0 \sqsubset \F_1 \sqsubset \ldots \sqsubset \F_K$ of  $\phi$-invariant free factor systems if each $F_j$ is realized by an $f$-invariant core subgraph of $G$; see Section~\ref{sec:ffs}.

 \begin{thm}[{\cite[Theorem~1.1]{fh:CTconjugacy}}]\label{t:ct is algorithmic}
There is an algorithm whose input is a rotationless $\phi\in\Out(\f)$ and whose output is a \ct\ $f:G\to G$ that represents $\phi$ and satisfies (Inheritance). Moreover, for any   chain $\mathcal C$ of $\phi$-invariant free factor systems, one can choose $f:G\to G$ to realize $\mathcal C$.
\end{thm}

We assume throughout this paper that our chosen \ct s  satisfy (Inheritance).

The marked graph $G$ comes equipped with an $f$-invariant  filtration $\emptyset = G_0 \subset G_1 \subset \ldots \subset G_N =  G$ by subgraphs $G_i$  in which each $G_{i}$ is obtained from $G_{i-1}$ by adding a single oriented edge $E_i$. For each $E_i$ there is a  (possibly trivial) closed path   $u_i \subset G_{i-1}$ such that  $f(E_i) = E_i u_i$; if $u_i$  is  non-trivial then it forms a circuit.     A path or circuit has {\em height i} if it crosses $E_i$, meaning that either $E_i$ or $\bar E_i$ occurs in its edge decomposition, but does not cross  $E_j$ for any $j > i$.

\begin{excont*} \label{ex.a}
$\fG$ is a \ct\ with $f$-invariant filtration $\emptyset=G_0\subset G_1\subset\dots\subset G_7=G$ given by adding one edge at a time in alphabetical order.
\end{excont*}

Every map $\alpha$ into $G$ with domain a closed interval or $S^1$  and with endpoints, if any, at vertices   is properly homotopic rel endpoints  to a path or circuit $[\alpha]$;  we say that $[\alpha]$ is obtained from $\alpha$ by {\em tightening}.   If $\sigma$ is a   path or circuit then  we usually denote $[f(\sigma)]$ by $f_\#(\sigma)$.   A decomposition into subpaths $\sigma = \sigma_1 \cdot \sigma_2 \cdot \ldots$ is a {\em splitting} if $f^k_\#(\sigma) =   f^k_\#(\sigma_1)\cdot  f^k_\#(\sigma_2) \cdot \ldots$ for all $k \ge 1$.   In other words, $f^k(\sigma)$ can be tightened by tightening each $f^k(\sigma_i)$.

A finite path $\sigma$  is a {\em Nielsen path} if  $f_\#(\sigma) = \sigma$; it is an {\em \iNp} if it is not a fixed edge and does not split into a non-trivial concatenation of Nielsen paths.  Every Nielsen path has a splitting into fixed edges and \iNp s.   If $f^k_\#(\sigma) =  \sigma$ for some $k \ge 1$ then $\sigma$ is a {\em periodic Nielsen path}.  In a \ct, every periodic Nielsen path is a Nielsen path.   
 
 An edge $E_i$ is {\em linear} if $u_i$ is a non-trivial Nielsen path. 
 The {\em set of oriented linear edges} is denoted $\lin(f)$ and the set obtained from $\lin(f)$ by reversing orientation is denoted $\lin^{-1}(f)$.   In our example, $\lin(f)=\{b,p\}$. 
 
 Associated to a CT $\fG$ is a finite set  of non-trivial closed Nielsen paths called {\em \twistpath s}. This set is well-defined up to a change of orientation on each path. In the remainder of this paragraph we recall some useful properties of twist paths. Each twist path $w$ determines a circuit $[w]$ in $G$ {representing a root-free}\footnote{Non-trivial $a\in\f$ is {\it root-free} if $x\in\f$ and $x^k=a$ implies $k=\pm 1$.} conjugacy class in $\f$ and distinct twist paths determine distinct unoriented circuits; i.e.\  circuits whose cyclic edge decompositions differ by more than a change of orientation.   For each  \twistpath\ $w$,  the set $\lin_w(f)$ of  (necessarily linear) edges $E_i$ such that $f(E_i) = E_iw^{d_i}$ for some $d_i \ne 0$ is non-empty  and is called the {\em linear family associated to $w$};   note that  $f^k_\#(E_i) =  E_i w^{kd_i}$ grows linearly in $k$.   Every linear edge belongs to one of these linear families. If $E_i \in \lin_w(f)$  and $p\ne 0$ then   $E_i w^p \bar E_i$ is an \iNp.  All   \iNp s have this form.  If $E_i$ and $E_j$ are distinct edges in $ \lin_w(f)$ then  $d_i \ne d_j$;  if       $d_i$ and $d_j$ have the same sign, then paths of the form $E_i w^p \bar E_j$ are {\em exceptional paths} associated to $w$.   Note that $f^k_\#(E_i w^p \bar E_j) =  E_i w^{p+k(d_i - d_j)} \bar E_j$ so these paths also grow linearly under iteration.   Exceptional paths have no non-trivial splittings (which would not be true if we allowed $d_i$ and $d_j$ to have the opposite sign).  
 
 \begin{excont*}\label{ex.b}
 In our example, we choose our set of twist paths to be $\{a\}$, as opposed to $\{a^{-1}\}$. 
 \end{excont*}
    
A splitting $\sigma = \sigma_1 \cdot \sigma_2 \cdot \ldots$ is a {\it complete splitting} if each $\sigma_i$ is either a single edge or an \iNp\ or an exceptional path.   If $\sigma_i$ is not a Nielsen path then it is {\it a growing term}; if at least one $\sigma_i$ is growing then {\it $\sigma$} is growing.  We say that $\sigma_i$ is   {\it a linear term} if  it is exceptional or equal to $E$ or $\bar E$ for some $E \in \lin(f)$.   Complete splittings are unique when they exist  (Lemma~4.11 of \cite{fh:recognition}).  A path with a complete splitting is said to be {\it completely split}.    For each edge $E_i$  there is a complete splitting of $f(E_i)$  whose first term is $E_i$ and whose remaining terms define a complete splitting of $u_i$.   The image under $f_\#$ of a completely split path or circuit is completely split.  For each path or circuit $\sigma$, the image $f_\#^k(\sigma)$ is completely split for all sufficiently large $k$ \cite[Lemma 4.25]{fh:recognition}.
  
The {\em set of oriented non-fixed non-linear edges} is denoted $\E_f$  and the set obtained from $\E_f$ by reversing orientation is denoted $\E_f^{-1}$.  We say that an edge in $\E_f$ or $\E_f^{-1}$ has  {\em  higher order}.    An easy induction argument shows that, for each $E _i\in \E_f$ and  $k \ge 1$,   $f^{k}_\#(E_i)$ is completely split and 
   \[ f^k_\#(E_i) = E_i \cdot u_i \cdot f_\#(u_i) \cdot \ldots \cdot f^{k-1}_\#(u_i)\]
Thus $f^{k-1}_\#(E_i)$ is an initial segment of $f^k_\#(E_i)$ and the union  
\[ R_{E_i} = E \cdot u_i \cdot f_\#(u_i) \cdot f^2_\#(u_i) \cdot \ldots \]
of this nested sequence is an $f_\#$-invariant ray called the {\em eigenray associated to $E_i$}.  The complete splittings of the individual $f^{k}_\#(u_i)$'s  define a complete splitting of $R_{E_i}$.  

\begin{excont*}\label{ex.c}
In our example, the edge $a$ is fixed, the edges $b$ and $p$ are linear, and the other edges have higher order, i.e.\ $\E_f=\{c,d,e,q\}$.
As an example of an eigenray, $$R_q=q\cdot c\cdot cb\cdot cbba\cdot \ldots \cdot cbba\ldots ba^{k-1}\cdot\ldots$$
\end{excont*}

\begin{lemma}  \label{not crossed}  If $E \in \E_f \cup \E_f^{-1}$  then $E$  is   not crossed by any Nielsen path or exceptional path.   In particular, each crossing of $E$ by a completely split path is a term in the complete splitting of that path.
\end{lemma}
  
\proof
Suppose that some Nielsen path $\sigma$ crosses $E$. Since $\sigma$ is a concatenation of fixed edges and \iNp s and since every \iNp\ has the form $E_iw^p \bar E_i$ for some linear edge and \twistpath\ $w$, $E$ must be crossed by some Nielsen path $w$ with height lower than $\sigma$.   The obvious induction argument completes the proof.  
\endproof

\begin{remark}   One can define $R_E$ for a linear edge $E$ in the same way that one does for a higher order edge.   If $f(E) = Ew^d$ for some \twistpath\ $w$ then $R_E = Ew^\infty$ if $d > 0$ and $R_E = Ew^{-\infty}$ if $d < 0$.  These rays play a different role in the theory than eigenrays. 
\end{remark}

\subsection  {Principal lifts from the \ct\ point of view}  \label{sec:principal}
Suppose that $\fG$ is a \ct\ representing $\phi$.
If $\Phi$ is a principal lift for $\phi$ then we say that the corresponding  {\em $\ti f$  is a principal lift of $\fG$}.
 
\begin{lemma} \label{principal lifts have fixed points} A lift $\ti f : \ti G \to \ti G$ is principal if and only if $\Fix(\ti f) \ne \emptyset$ in which case $\Fix(\ti f)$ contains a vertex.
\end{lemma}
 
\proof  This follows from Corollary~3.17, Corollary~3.27 and Remark 4.9 of \cite{fh:recognition} and the fact that  $\Fix(f)$ is a union of vertices and fixed edges.
\endproof
 
\begin{lemma} \label{basic fix  2}  Suppose that  $\ti f : \ti G \to \ti G$ is the  lift of $\fG$ corresponding to $\Phi \in \cP(\phi)$ and that $a \in F_n$. The following are equivalent:
\begin{itemize}
\item   $a \in \Fix(\Phi)$; 
\item   $ i_a | \partial \f =  T_a | \partial \f$ commutes with $ \Phi | \partial \f=  \ti f | \partial \f$; 
\item  $T_a | \ti G$ commutes with $\ti f | \ti G$;
\item $T_a (\Fix(\ti f | \ti G)) = \Fix(\ti f | \ti G)$.
\end{itemize}
 
\end{lemma}

\proof This is well known.  All but the  equivalence of the third and fourth bullets can be found in  Lemma~2.4 of \cite{bfh:tits3}.  If $T_a | \ti G$ commutes with $\ti f| \ti G$ then it  preserves the fixed point set of $\ti f | \ti G$.  Conversely, if $\ti x, T_a(\ti x) \in \Fix(\ti f| \ti G)$ then $\ti f | \ti G$ and $T_a \ti f T_a^{-1} | \ti G$ both fix $T_a(\ti x)$ and so must be equal.  This proves the  equivalence of the third and fourth bullets.
\endproof

We say that lifts  $\ti f_1$ and $\ti f_2$ are {\em isogredient} if they correspond to isogredient automorphisms $\Phi_1$ and $\Phi_2$.  Equivalently,   $\ti f_2 = T_a \ti f_1 T_a^{-1}$ for some covering translation $T_a$.  Recall that $x,y \in \Fix(f)$ are {\em Nielsen equivalent} if they are the endpoints of a Nielsen path or equivalently, if for each lift $\ti x$, the unique lift $\ti f$ that fixes $\ti x$ also fixes some lift $\ti y$ of $y$.  

\begin{lemma}\label{nielsen classes and principal lifts}    The map which assigns to each principal lift $\ti f : \ti G \to \ti G$ the projection into $G$ of $\Fix(\ti f)$ induces  a bijection between the set of  isogredience classes of principal lifts and the set of  Nielsen classes for  $f$.   In particular, there are only finitely many   isogredience classes of principal lifts.
\end{lemma}

\proof This follows from \cite[Lemma 3.8]{fh:recognition} and Lemma~\ref{principal lifts have fixed points}.
\endproof

Recall from Section~\ref{sec:ct} that for each  $E \in \E_f$ there is a closed completely split path $u$ such that $f(E) = E \cdot u$ is a splitting and such that  the eigenray $R_E = E \cdot u \cdot f_\#(u) \cdot f^2_\#(u)\cdot \ldots$ is   $f_\#$-invariant.

The following lemma is similar to \cite[Lemma 3.10]{fh:CTconjugacy}, which applies more generally but has a weaker conclusion.
\begin{lemma}  \label{identifying Fix+} Suppose that $\ti f$ corresponds to $\Phi \in \cP(\phi)$.  If $\ti E$ is a lift of $E \in \E_f$ and if the initial endpoint of $\ti E$ is contained in $\Fix(\ti f)$ then the lift $\ti R_{\ti E}$ of $R_E$   that begins with $\ti E$ converges to a point in $\Fix_+(\Phi)$.  This defines a    bijection between $\Fix_+(\Phi)$ and  the set of all such $\ti E$ and also a bijection between $\cR(\phi)$ and $\E_f$. 
\end{lemma}

\proof     $\ti R_{\ti E}$ converges to some $P \in\FixN(\Phi)$ by   Lemma~4.36-(1) in \cite{fh:recognition}.    Since $E$ is not linear, $u$ is not a Nielsen path and hence not a periodic Nielsen path.  The length of   $f^k_\#(u)$ therefore  goes to infinity with $k$. Proposition I.1 of \cite{gjll:index} implies that $P \in\Fix_+(\Phi)$.  

Suppose that $\ti E_1$ and $\ti E_2$ are distinct edges that project into $\E_f$, that the initial endpoint $\ti x_i$  of $\ti E_i$ is fixed  by $\ti f$ and that, for $i=1,2$,   $\ti R_i$  is the lift of $R_{E_i}$ with initial edge $\ti E_i$.   The path  that connects  $\ti x_1$ to $\ti x_2$ projects to a Nielsen path   $\sigma \subset G$. If  $\ti R_1$ and $\ti R_2$ converge to the same point in $\Fix_+(\Phi)$ then $\sigma$ crosses $ E_1$ or $E_2$   in contradiction to Lemma~\ref{not crossed}.   This proves that the map  $\{\ti E\} \mapsto \Fix_+(\Phi)$ is injective;  surjectivity  follows from Lemma~4.36-(2) in \cite{fh:recognition}.  The second bijection is obtained from the first by projecting to the sets of $F_n$-orbits.
\endproof

\begin{excont*}\label{ex.d}
Since $\E_f=\{c,d,e,q\}$, $\cR(\phi)$ has four elements, denoted $\{r_c,r_d,r_e,r_q\}$; cf.\ Figure~\ref{f:eigengraph}.
\end{excont*}

\section{Recognizing a Conjugator} 
Associated to each \ct\ $\fG$ representing a rotationless element $\phi \in \Out(\f)$ is a finite type labeled graph $\gf$ that realizes $\FixN(\phi)$.   We refer to $\gf$ as the {\em eigengraph} for $\fG$.    In Section~\ref{sec:stallings} we  recall the construction and relevant properties of $\gf$ in the case that $\phi\in\upgn$.   Every $\phi$-invariant conjugacy class $[a]$ is represented by an oriented circuit in $\gf$. There is a finite set $\A_\both(\phi)$ of such $[a]$ that are root-free and that are represented by more than one oriented circuit in $\gf$.  The \lq extra\rq\ circuits correspond to the linear edges in $\fG$.  In Section~\ref{sec:axes}, we describe how the extra circuits can be incorporated into  an invariant $\sa(\phi)$ of $\phi$ that is independent of the choice of $\fG$.  Moreover, certain pairs of elements of  $\sa(\phi)$ have   {\em twist coordinates} that can be read off from the twist coordinates on linear edges in $\fG$.    Section~\ref{recognition} is an application of  the Recognition Theorem of \cite{fh:recognition}.     Assuming that  $\phi, \psi \in \upgn$, we use eigengraphs, $\sa(\phi)$  and twist coordinates to give necessary and sufficient conditions for a given  $\theta \in \Out(F_n)$ to conjugate $\phi$ to $\psi$.
   
\subsection  {The Eigengraph $\gf$} \label{sec:stallings}
In this section, we recall  a finite type labeled graph  that captures many of the invariants of $\phi$ that are essential to our algorithm. For further details and more examples, see \cite[Sections 9, 10 and 12]{fh:CTconjugacy}.  

A graph $\stallings$ without valence one vertices and equipped with a simplicial immersion $p :\stallings \to G$ to a marked graph $G$ will be called a {\em Stallings graph}.  We label the vertices and edges of $\Gamma$ by their $p$-images in $G$.  Two Stallings graphs $p_1 : \Gamma_1 \to G$ and $p_2 : \Gamma_2 \to G$ are equivalent if there is a  label preserving simplicial homeomorphism $h : \Gamma_1 \to \Gamma_2$.    We will not distinguish between equivalent Stallings graphs.      Since all vertices have valence at least two,  every edge in $\stallings$ is crossed by a line. We say that $\stallings$ has {\em finite type} if its core is finite and if the complement of the core is a finite union of rays.

Given a \ct\ $\fG$ representing $\phi$, we construct  a finite type Stallings graph $\gf$, called the  {\em   eigengraph for $\fG$},    as follows.   Let $\Gamma^0(f)$   be  $G$ with the interiors of all non-fixed edges removed. 
 In particular, $\Gamma^0(f)$ may contain isolated vertices. The labeling on $\Gamma^0(f)$ is the obvious one.  For each $E \in \lin(f)$, first attach  an edge, say $E'$,  to $\Gamma^0(f)$ by identifying  the initial endpoint of $E'$  with  the initial endpoint of $E$, thought of as a vertex in $\Gamma^0(f)$.  The label on $E'$ is  $E$.     Then add  a  path $\omega$ by attaching both of its endpoints  to the  terminal endpoint of $E'$, which now has valence three.  If $w$ denotes the \twistpath\ associated to $E$ then we label $\omega$ by $w$, thought of as an edge path, and subdivide $\omega$ so that each edge in $\omega$ is labeled by a single edge in $G$.  The net effect is to add a {\em lollipop} to $\Gamma^0(f)$ for each edge in $\lin(f)$.    This labeling defines an immersion because $w$ determines a circuit that does not cross the edge $E$.   Finally, for each $E \in \E_f$,  attach a ray  labeled  $R_E$ (as defined in Section~\ref{sec:ct}) by identifying the initial endpoint of this ray with  the initial endpoint of $E$, thought of as a vertex in $\Gamma^0(f)$.  We will also use the term {\it eigenray} for this added ray. The resulting graph is denoted $\gf$. This labeling maintains the  immersion property because $R_E$ is an immersed ray in $G$ and because no other edge labeled $E$ has initial vertex in $\Gamma^0(f)$.
 
 \begin{excont*}\label{ex.e}
 The eigengraph for our running example is pictured in Figure~\ref{f:eigengraph}.
 \begin{figure}[h!] 
\centering
\includegraphics[width=.3\textwidth]{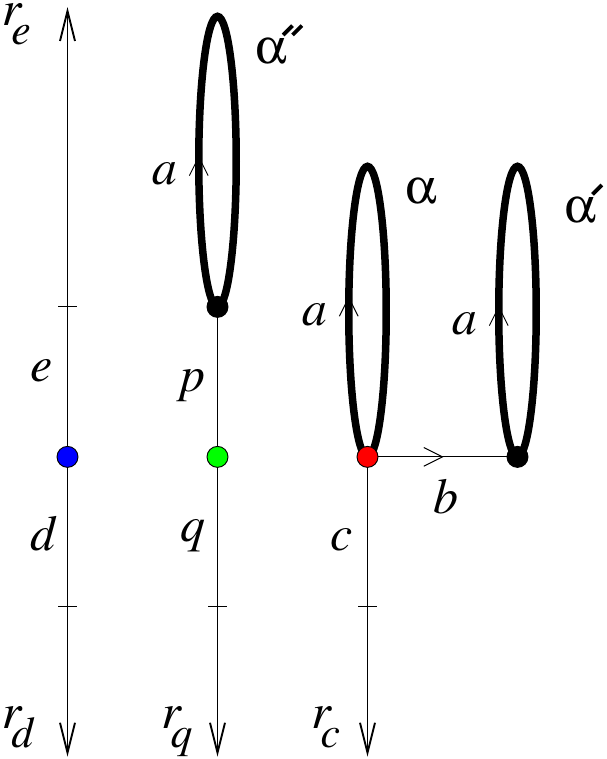}
\caption{The eigengraph $\Gamma(f)$ of our example. 
Only the first edge of the eigenrays is labeled here. For example, the eigenray $R_c=cbbaba^2ba^3\dots$ starts at the red vertex and only its first edge is labeled (by $c$). Other aspects of this figure are explained later.
}
\label{f:eigengraph}
\end{figure}
  \end{excont*}
  
    The vertices of $\stallings(f)$ that are not in $\Gamma^0(f)$ have valence either two or three by construction.   The valence of $v \in \Gamma^0(f)$ in $\stallings(f)$ is equal to the number of fixed directions based at $v$ in $G$.  If $v$, thought of as a vertex in $G$, is not the terminal endpoint of an edge  in $\E_f \cup \lin(f)$, then $v$ has the same valence in $\stallings(f)$ that it does in $G$.  If $v$ is the terminal endpoint of an edge  in $\E_f \cup \lin(f)$,  let $E_i$ be the lowest such edge.   Then $f(E_i) =E_i \cdot u_i$ where $u_i$ is a closed path based at $v$ whose ends determine distinct fixed directions at $v$ by \cite[Lemma 4.21]{fh:recognition}.  This proves that $v$ has valence at least two in   $\stallings(f)$  and hence that $\stallings(f)$ is a Stallings graph.  It has finite type by construction.
    
     As noted in Section~\ref{sec:ct}, each Nielsen path in $G$ decomposes as a concatenation of fixed edges and \iNp s and each \iNp\ is a closed path.  It follows that two vertices in $\Fix(f)$ are in the same Nielsen class if and only if they are connected by a sequence of fixed edges.   In particular, the vertices in each component of $\Gamma^0(f)$ form exactly one Nielsen class in $\Fix(f)$.   Since the inclusion of $\Gamma^0(f)$ into $\gf$ induces a bijection of components, there is a bijection between the set of components of $\gf$ and the set of Nielsen classes in $\Fix(f)$ and hence (Lemma~\ref{nielsen classes and principal lifts}) a bijection between the set of components of $\gf$ and the set of isogredience classes in $\cP(\phi)$.  We denote the component of $\gf$ corresponding to the isogredience class $[\Phi]$  by $\Gamma_{[\Phi]}(f)$ or by $\stallings(\ti f)$ where $\ti f$ is the lift of $f$ that corresponds to $\Phi$.

We say that {\em a line   is carried by $\stallings_{[\Phi]}(f)$} if its realization $L \subset G$ lifts into $\stallings_{[\Phi]}(f)$ and is {\em carried by $\gf$} if it is carried by some component $\stallings_{[\Phi]}(f)$. The following lemma shows that the set of lines carried by $\stallings(f)$ is independent of the choice of  $\fG$. We will sometimes refer to these as {\em principal lines.}
    
\begin{lemma}   \label{lem:lifting}
The  following are equivalent for any \ct\ $\fG$ representing $\phi$, any $\Phi \in \cP(\phi)$ and any line $L \subset G$.
\begin{enumerate}
\item $L$ is carried by $\Gamma_{[\Phi]}(f)$ [resp.\ the core of $\Gamma_{[\Phi]}(f)$]. 
\item There is a lift $\ti L \subset \ti G$ such that $\{\partial_\pm \ti L\}\ \subset \FixN(\Phi)$ [resp. $\{\partial_\pm \ti L\} \subset \partial \Fix(\Phi)$].
\end{enumerate} 
\end{lemma}

\proof It suffices to prove the unbracketed statement.   Let $q :\ti G \to G$  and $\ q_\Gamma :  \tilde\stallings_{[\Phi]}(f) \to \stallings_{[\Phi]}(f)$ be the universal covering maps.   The labeling map $p : \stallings_{[\Phi]}(f) \to G$ is an immersion and so lifts to an embedding $\ti p : \ti \stallings_{[\Phi]}(f) \hookrightarrow \ti G$.    If a  line $L \subset G$ lifts to a line $L_\Gamma \subset \stallings_{[\Phi]}(f)$  and  if $\ti L_\Gamma \subset \ti \stallings_{[\Phi]}(f)$ is a lift of $L_\Gamma$ then $\ti L := \ti p (\ti  L_\Gamma) \subset \ti G$ is a lift of $L$.   Conversely, if $L \subset G$ lifts to $\ti L \subset \ti G$ and there exists $\ti L_\Gamma \subset \ti  \stallings_{[\Phi]}(f)$ with $\ti p(\ti L_\Gamma) = \ti L$, then $  L_\Gamma :=  q_\Gamma(\ti L_\Gamma)$ is a lift of $L$.  The Lemma therefore follows from \cite[Lemma~12.4]{fh:CTconjugacy} which states that $L \subset  G$ lifts to $\ti p (\ti  L_\Gamma) $ if and only if (2) is satisfied.
\endproof

An end of an immersed line in $\gf$ can either be an end of $\gf$ or can wrap infinitely around one of the lollipop circuits or can cross a vertex in $\Gamma^0(f)$ infinitely often.  This gives the following description of lines that lift into  $\gf$.

\begin{lemma} \label{lem:lines in gf}
A line $\sigma \subset G$ lifts into $\gf$ if and only if it contains a (possibly trivial) subpath $\beta$ that is  a concatenation of fixed edges and \iNp s and such that the complement of $\beta$ is $0,1$ or $2$ rays, each of which is either $R_e$ for some higher order edge $e$  or $Ew^{\pm\infty}$ for some \twistpath\ $w$ and some linear edge $E \in \lin_w(f)$.\qed
\end{lemma}

The following lemma, in conjunction with Lemma~\ref{lem:lifting}, implies that the set of   conjugacy classes  determined by twist paths and their inverses is an invariant of $\phi$.  This set is explored further in Section~\ref{sec:axes}.

\begin{lemma} \label{axes in stallings graph}
Suppose that  $\fG$ is a \ct\ representing $\phi$ and that $[a]$ is a root-free conjugacy class that is fixed by $\phi$   and that $\sigma_a$ is the circuit in $ G$  representing  $[a]$.  
\begin{enumerate}
\item If  $[a] = [w]$ (resp.\ $[a] = [\bar w]$)   for some \twistpath\ $w$,  then for each edge $E \in \lin_w(f)$  there is a lift of $\sigma_a$ to the loop $\omega$ (resp.\ $\bar \omega$) in the lollipop associated to $E$.  Additionally there is a unique  lift of $\sigma_a$ to a circuit in $\gf$ that is not contained in any lollipop.
\item Otherwise, there is a unique lift of $\sigma_a$ to a circuit in $\gf$. 
\end{enumerate}
\end{lemma}  

\proof
We claim that if $\tau'$ is a path in the core of $\gf$  that   projects to a Nielsen path $\tau$ in $  G$  and if the initial vertex $v''$ of $\tau'$ is not in $\Gamma^0(f)$ then $\tau'$ is contained in the loop $\omega$ of some lollipop.     We may  assume without loss that the claim holds for paths with height less than that of $\tau$ and that  $\tau$ is either a single fixed edge or an \iNp.   Since the core of $\gf$ is contained in the union of $\Gamma^0(f)$ with the lollipops associated to the linear edges of $G$,   there is a lollipop composed of an edge $E_1'$   projecting to a linear edge $E_1 \subset G$ and a  loop $\omega$ projecting to its \twistpath\ $w_1$ such that $v''\in \omega$.  If $\tau$ is a fixed edge then $\tau'$ is disjoint from the interior of $E_1'$ and so is contained in $\omega$.  If $\tau$ is an \iNp\ then it has the form $E_2{w_2}^p \bar E_2$ for some linear edge $E_2$ with \twistpath\ $w_2$.    There is an induced decomposition $\tau' = X'{w'_2}Y'$ where $X', w'_2$ and $Y'$ project to $E_2, w_2^p$ and $\bar E_2$ respectively. Since $E_2 \ne \bar E_1$, we have $X' \subset \omega$ and the initial vertex of   ${w'_2}^p$ is contained in $\omega$.  Since $w_2$ has height less than $E_2$ and so height less than $\tau$,  \    ${w'_2}^p\subset \omega$.  Finally, since $X' $ is contained in $\omega$, \ $E_2$ has height less than $E_1$ and in particular, $\bar E_2 \ne\bar  E_1$.  Thus $  Y' \subset \omega$.  This  completes the proof of the claim.

Each $\sigma = \sigma_a$ as in the statement of the lemma has  a cyclic splitting    $\sigma = \sigma_1 \cdot \ldots \cdot \sigma_m$ into fixed edges and \iNp s  $\sigma_i$.  The above claim shows that if $\sigma' = \sigma'_1 \cdot \ldots \cdot \sigma'_m$ is a lift to $\stallings(f)$ in which an endpoint of some $\sigma_i'$ is not contained in $\Gamma(f)^0$ then $\sigma'$ is entirely contained in the loop $ \omega$ associated to one of the lollipops.  
      
      To complete the proof we need only show   each $\sigma$ has a unique lift in which the endpoints of each $\sigma_i$  lift into $\Gamma^0(f)$.  Since the vertices of $G$ have unique lifts into $\Gamma^0(f)$, it suffices to show that each $\sigma_i$ has a unique lift with endpoints in $\Gamma^0(f)$ and this is immediate from the construction of $\stallings^0(f)$.   
\endproof

\subsection  {Strong axes and twist coordinates} \label{sec:axes}
The following lemma describes the extent to which   fixed subgroups fail to be malnormal.

\begin{lemma}  \label{malnormal} For  distinct automorphisms $\Phi_1$ and $\Phi_2$ representing the same outer automorphism and for  any $c \in F_n$:
\begin{enumerate}
\item\label{i:cap}
$\Fix(\Phi_1) \cap \Fix(\Phi_2)$ is either trivial or a maximal cyclic subgroup.
\item  If $c \not \in \Fix(\Phi_1)$ then $\Fix(\Phi_1) \cap (\Fix(\Phi_1))^c =  \Fix(\Phi_1) \cap \Fix(i_c\Phi_1 i_c^{-1})$  is either trivial or a maximal cyclic subgroup.
\item\label{i:normalizer}
$\Fix(\Phi_1)$ is its own normalizer. 
\end{enumerate}
\end{lemma}

\proof    If $\Phi_1$ and $\Phi_2$ are distinct automorphisms representing the same outer automorphism  then $\Phi_1^{-1} \Phi_2$ is a non-trivial inner automorphism and $\Fix(\Phi_1) \cap \Fix(\Phi_2)$ is a subgroup of the cyclic group $\Fix(\Phi_1^{-1} \Phi_2)$.  Maximality of $\Fix(\Phi_1) \cap \Fix(\Phi_2)$ follows from Lemma~\ref{basic fix}  and the fact that $(a^k)^\pm = a^\pm$ for all non-trivial $a \in \f$ and all $k \ge 1$.   This proves (1).

For (2) note that if $c \not \in \Fix(\Phi_1)$ then $i_c \Phi_1 i_c^{-1}  = i_{c \Phi_1(c^{-1})} \Phi_1 \ne \Phi_1$. 
Note also that $ \Fix(i_c \Phi_1 i_c^{-1})=(\Fix(\Phi_1))^c $.  Item (2) therefore follows from (1) applied with   $\Phi_2 = i_c \Phi_1 i_c^{-1}$.  In  proving (3) we may assume  by (2)   that   $\Fix(\Phi_1) =  \langle a \rangle$ for some root-free $a\in\f$ and in this case (3) is obvious. \endproof

The conjugacy class of a  cyclic subgroup is   determined by  the conjugacy class of either of its generators.  As we have no way to canonically choose a generator, we work, for now, with unoriented conjugacy classes.   The following definition appeared  as Definition~4.6 of \cite{bfh:tits3} under slightly different hypotheses and in the paragraph before Remark 4.39 of \cite{fh:recognition} in the \ct\ context. 

\begin{definition}  \label{def:axes} Elements $a, b \in \f$ are in the same {\em unoriented conjugacy class} if $a = i_c(b)$ or $a=i_c(b^{-1})$ for some $c \in \f$.   An unoriented conjugacy class $[a]_u$ of a non-trivial root-free $a \in \f$  is an {\em axis} for $\phi$ if $\langle a \rangle = \Fix(\Phi_1) \cap \Fix(\Phi_2)$ for distinct $\Phi_1, \Phi_2 \in \cP(\phi)$.  The {\em multiplicity of an axis $[a]_u$} is the number of distinct $\Phi_i \in \cP(\phi)$ that fix $a$. The {\em set of axes for $\phi$} is denoted $\A(\phi)$. The set $\{ [a]: [a]_u \in \A(\phi)\}$ is denoted $\A_{\both}(\phi)$. 
\end{definition}

There is a very useful description of $\A(\phi)$ in terms of a \ct\ $\fG$.

\begin{lemma} \label{first axes in ct} If $\fG$ is a \ct\ representing $\phi$ and $\{w_i\}$ is the set of twist paths for $f$  then $\A(\phi) = \{[w_i]_u\}$.    In particular, $\A(\phi)$ is finite.
\end{lemma}

\proof  This follows from Lemma~4.40 of \cite{fh:recognition}.
\endproof

 \begin{notn}  \label{notn:base lift}   
If  $[a]_u = [w]_u$ for some twist path $w$ then, up to a reversal of orientation, the axis of the covering translation $T_a : \ti G \to \ti G$ can be viewed as an infinite concatenation  $\ldots \ti w_{-1} \ti w_0 \ti w_1 \ldots $ of paths that project to $w$.     There is a principal lift $\ti f_{a,0}:\ti G \to \ti G$, called the {\em base principal lift for $a$},  that fixes the endpoints of each $\ti w_l$.  The principal automorphism $\Phi_{a,0}$ corresponding to $\ti f_{a,0}$ is called the {\em base principal automorphism} for $a$.   If  $a$ is an element of some basis  for $\f$ then the base principal lift for $a$   depends on the choice of   $\fG$,  and not just on $\phi$.  

For each edge $E^j \in \lin_w(f)$, there is a principal lift $\ti f_{a,j} : \ti G \to \ti G$ that fixes the initial endpoint of each lift $\ti E^j$ with terminal endpoint equal to the initial endpoint of some $\ti w_l$. (We write $E^j$ rather than $E_j$ to emphasize that $j$ is not an indicator of height in $G$.) The principal automorphism corresponding to $\ti f_{a,j}$ is denoted $\Phi_{a,j}$.  Note that $\Phi_{a,0} = \Phi_{a^{-1},0}$ and $\Phi_{a,j} = \Phi_{a^{-1},j}$.    Further details can be found in Lemma~4.40 of \cite{fh:recognition} and the paragraph that precedes it.
 \end{notn}

\begin{lemma}  \label{lem:ct principal lifts} Suppose that  $\fG$ is a \ct\ representing  $\phi$, that $w$ is a twist path for $f$ and that $a \in F_n$  satisfies $[a]_u =  [w]_u$.  Suppose also that $E^1,\ldots,E^{m-1}$ are the edges in $\lin_{w}(f)$.
\begin{enumerate}
\item   $\{\Phi_{a,0}, \Phi_{a,1},\ldots,  \Phi_{a,m-1}\}$ is the set of  principal automorphisms that fix $a$.  In particular, the multiplicity of each element of $\A(\phi)$ is finite.
\item  If   $f(E^j) = E^j{w}^{d_j}$ then $\Phi_{a,j} = i_a^{d_j} \Phi_{a,0}$ if $[a] = [w]$ and $\Phi_{a,j} = i_a^{-d_j} \Phi_{a,0}$ if $[a] = [\bar w]$.
\end{enumerate}
\end{lemma} 

\proof  This follows from Lemma~4.40 of \cite{fh:recognition}.
\endproof

\begin{definition}\label{d:conjugacy pairs}
Suppose that the group $G$ acts on the sets $X_i$, $i=1,\dots, k$, and that $x_i\in X_i$. The orbit of $(x_1,\dots,x_k)$ under the diagonal action of $G$ on $\prod_{i=1}^k X_i$, denoted $[x_1, \dots, x_k]_G$, is a {\it conjugacy $k$-tuple}. If $k=2$ then we say $[x_1,x_2]_G$ is a {\it conjugacy pair}. We sometimes suppress the subscript, in which case $G=\f$.
\end{definition}

\begin{exs}\label{e:conjugacy pairs}
Here are some examples of conjugacy pairs where $G=\f$.
\begin{itemize}
\item
We will often take $X_i$ to be the set of finitely generated subgroups of $\f$ or $\f$ itself with the action of $\f$ given by conjugation. If $H<\f$ (resp.\ $x\in \f$) then $[H]_\f$ (resp.\ $[x]_\f$) is the conjugacy class of $H$ in $\f$ (resp.\ $x$ in $\f$). Conjugacy pairs formed with these $X_i$'s will play an important role in this paper, especially in Section~\ref{s:more atoms}. 
\item
If $X=\partial\f$ and if $x\not= y\in X$ then $(x,y)\in X\times X$ is an oriented line. The conjugacy pair $[x,y]_\f$ is represents an oriented line in any marked graph. 
\item 
If $X$ is the power set of $\partial\f$ and $A$ and $B$ are disjoint subsets of  $\partial \f$, then $(A,B)\in X\times X$ denotes the set of lines $L$ with $\partial_-L\in A$ and $\partial_+L\in B$. The conjugacy pair $[A,B]_\f$ represents a set of oriented lines in any marked graph.
\end{itemize}
\end{exs}

We now define strong axes, the first of our invariants that is expressed as a conjugacy pair.

\begin{definition}  \label{def:strong axis} Let $\A_\both (\phi)$ be the set of conjugacy classes representing elements of $\A(\phi)$, i.e.\ $[a]\in\A_\both(\phi)$ if $\{[a],[a^{-1}]\}\in\A(\phi)$. $\f$ acts on pairs $(\Phi, a)$ where $\Phi\in\cP(\phi)$, $a\in\Fix(\Phi)$, and $[a]\in \A_\both(\phi)$ via $(\Phi,a)^g=(\Phi^{i_g}, a^g)$. The $\f$-orbit, equivalently conjugacy pair, $[\Phi,a]$ is a {\em strong axis for $\phi$}. If $\alpha_s = [\Phi,a]$ then  we let $\alpha^{-1}_s :=  [\Phi,a^{-1}]$. The set of all strong axes for $\phi$ is denoted $\sa(\phi)$.    $\Aut(F_n)$ acts on pairs $(\Phi,a)$ by $\Theta\cdot(\Phi,a) = (\Theta \Phi \Theta^{-1}, \Theta(a))$.    This   descends to an action of $\Out(\f)$ on $\sa(\phi)$.
 
 We can   partition $\sa(\phi)$ according to the second coordinate: for each $ \mu \in \A_\both(\phi)$ let $\sa(\phi,\mu)$ be    the subset of $\sa(\phi)$ consisting of elements  in which some, and hence every, representative $(\Phi,a)$ satisfies $[a] = \mu$.   
\end{definition}

\begin{lemma} \label{lem:strong axis}
Suppose that $a \in \f$,   that $[a] \in \A_\both(\phi)$ and that $\Phi_{a,0},\ldots, \Phi_{a,m-1}$ are as in Notation~\ref{notn:base lift}.   For each $\alpha_s \in \sa(\phi,[a])$,  there is a unique $\Phi_{a,j}$ such that   $\alpha_s=[\Phi_{a,j},a]$.   Thus $\sa(\phi,[a]) = \{[\Phi_{a,0},a], [\Phi_{a,1},a],\ldots, [\Phi_{a,m-1},a]\}$.
\end{lemma}

\proof
Each $\alpha_s \in \sa(\phi,[a])$ is represented by $( \Phi,a^c)$ and hence  by $(  i_c \Phi i_c^{-1}, a)$, for some $c \in \f$ and some $\Phi  \in \cP(\phi)$.  Since $a \in \Fix(i_c \Phi i_c^{-1})$, there exists $j$ such that $i_c \Phi i_c^{-1} = \Phi_j$.

    For uniqueness, note that  if $ (  \Phi_j, a) = c \cdot (\Phi_i,a)$ for some $c \in \f$ then $c = a^p$ for some $p$ so  $i_c$ commutes with $\Phi_j$ and   $  c \cdot (\Phi_i, a) =  (\Phi_i, a)$.     
\endproof

\begin{remark}\label{r:sa}
There is another useful description of  $[\Phi_{a,j},a]\in\sa(\phi,[a])$ in terms of a \ct\ $\fG$.  Let $w$ be the \twistpath\ satisfying   $[a]_u = [w]_u$ and let $v$ be the initial vertex of $w$.  There is an automorphism $f_{v\#} : \pi_1(G,v) \to \pi_1(G,v)$ that sends the homotopy class of the closed path $\sigma$ with basepoint $v$ to the homotopy class of the closed path $f(\sigma)$ with basepoint $v$.  Let $\tau$ be the element of $\pi_1(G,v)$ determined by $w$ if $[a] = [w]$ and by $\bar w$ if $[a] = [\bar w]$.  In both cases, $\tau$   is fixed by $f_{v\#}$.  There is an isomorphism from $\pi_1(G,v)$ to $F_n$ that is well defined up to post-composition with an inner automorphism of $F_n$.    The pair $(f_{v\#},\tau)$   determines a well defined element (namely $[\Phi_{a,0},a]$)  of $\sa(\phi,[a])$.   Similarly if $v_j$ is the initial endpoint of $E^j \in \lin_w(f)$, let $\tau_j$ be the element of $\pi_1(G,v_j)$ determined by $E^jw\bar E^j$ if $[a] = [w]$ and by $E^j\bar w\bar E^j$ if $[a] = [\bar w]$. Then $(f_{v_j\#},\tau_j)$ determines   $[\Phi_{a,j},a]$.
 
 Continuing with this notation, we can relate $\sa(\phi,[a])$ to circuits in the eigengraph $\Gamma(f)$ that are lifts of $[w]_u$.  For $j \ne 0$,  $[\Phi_{a,j},a]$ corresponds to the loop at the end of the lollipop in $\Gamma(f)$ determined by $E^j$.  By Lemma~\ref{axes in stallings graph} there is one more lift of $[w]_u$ into $\Gamma(f)$ and this corresponds to $[\Phi_{a,0},a]$.    
 \end{remark}

\begin{definition}  \label{def:twist} Suppose that $\mu \in \A_\both(\phi)$ and that $\alpha_s, \alpha_s' \in \sa(\phi,\mu)$. Choose $a \in \f$ such that $[a] =\mu$ and let $\Phi, \Phi' \in \cP(\phi)$ be the unique elements such that $\alpha_s=[\Phi,a]$ and $\alpha'_s=[\Phi',a]$.   Since $\Phi$ and $ \Phi'  $ both fix $a$  there exists $\tau \in \Z$ such that  $\Phi'  = i_a^\tau \Phi$; equivalently,  $ \Phi' \Phi^{-1} = i_a^\tau$. 
 We say that $\tau = \tau(\alpha'_s,\alpha_s)$ is the {\em twist coordinate} associated to $\alpha'_s$ and $\alpha_s$.  
\end{definition}

\begin{excont*}\label{ex.f}
In our example, $\sa(\phi,[a])$ is represented in Figure~\ref{f:eigengraph} by the three circles $\alpha$, $\alpha'$, and $\alpha''$ labeled $a$ and drawn with thicker lines. $\sa(\phi)=\sa(\phi,[a])\cup\sa(\phi,[a^{-1}])$. We have for example $\tau(\alpha',\alpha)=1$.
\end{excont*}

\begin{lemma} \label{lem:twist} Twist coordinates are well-defined.
\end{lemma} 

\proof  We have to show that   $ \tau(\alpha'_s,\alpha_s)$ is independent of the choice of $a$ representing  $\mu$.  If $a$ is replaced by $a^c$ then $\Phi$ and $\Phi'$ are replaced by $i_c \Phi i^{-1}_c$ and $i_c \Phi' i^{-1}_c$  respectively and so  $  i_a^\tau = \Phi' \Phi^{-1}  $ is replaced by $i_c \Phi'  \Phi^{-1}i^{-1}_c = i_c i_a^\tau  i^{-1}_c  ={i_{a^c}}^\tau  $.
\endproof

The following lemma allows us to compute twist coordinates for strong axes from a \ct\ $\fG$. It is  an immediate consequence of Lemma~\ref{lem:ct principal lifts} and  the definitions. 
  
  \begin{lemma}  \label{twist facts} 
  \begin{enumerate}
   \item If $[a] = [w]$ and $E^j \in \lin_w(f)$ satisfies $f(E^j) = E^j w^{d_j}$ then $\tau([\Phi_{a,j},a], [\Phi_{a,0},a]) = d_j$.

  \item Suppose that $\mu \in \A_\both(\phi)$ and that  $\alpha_s, \beta_s,\gamma_s \in \sa(\phi,\mu)$.  Then
  \begin{enumerate}
  \item $ \tau( \alpha_s,\beta_s)= -   \tau(\beta_s, \alpha_s)  $ 
 \item 
   $\tau( \alpha_s,\gamma_s) = \tau(\alpha_s, \beta_s) + \tau( \beta_s, \gamma_s) $  
\item  
 $ \tau( \alpha_s,\beta_s)= -   \tau(\alpha^{-1}_s, \beta^{-1}_s) $
 \end{enumerate}
 \end{enumerate}
 \end{lemma}

The next lemma shows that $\A(\phi)$, $\sa(\phi,[a])$, and $\sa(\phi)$ are natural invariants.
  
  \begin{lemma}  \label{theta and axes} Assume that $\psi = \theta \phi \theta^{-1}$ and that $\Theta$ represents $\theta$.  
\begin{enumerate}
\item $[a]_u \leftrightarrow (\theta[a])_u$ defines a bijection  $\A(\phi) \leftrightarrow \A(\psi)$. 
\item $( \Phi,a) \leftrightarrow (\Theta \Phi \Theta^{-1},\Theta(a))$ induces a bijection $ \sa(\phi, [a]) \leftrightarrow \sa(\psi,\theta([a]))$ that preserves twist coordinates.
\end{enumerate}
\end{lemma}

\proof   If $\Phi_1, \Phi_2 \in \cP(\phi) $ fix $a\in \f$  then  $\Psi_1 := \Theta \Phi_1 \Theta^{-1},  \Psi_2 := \Theta \Phi_2 \Theta^{-1}  \in \cP(\psi)$ fix $\Theta(a)$. This proves  (1). 

For (2),  let $\Psi = \Theta \Phi \Theta^{-1}$ and note that if $\Phi \in \cP(\phi)$ and $a \in \Fix(\Phi)$ then $\Psi \in \cP(\psi)$ and $\Theta(a) \in \Fix(\Psi)$.  Moreover
\[  c\cdot ( \Phi,a) =   (  i_c \Phi i_c^{-1},a^c) \mapsto  (i_{\Theta(c)}\Psi  i_{\Theta(c)}^{-1}, \Theta(a)^{\Theta(c)})) = \Theta(c) \cdot ( \Psi,\Theta(a)) \]
This proves that $(\Phi,a ) \mapsto (\Theta \Phi \Theta^{-1},\Theta(a))$ induces a well defined map  $ \sa(\phi,[a]) \to  \sa(\psi,\theta([a]))$ that is obviously invertible and is hence a bijection.    If $\Phi_i$ and $\Psi_i$ are as in the proof of (1) and if $\Phi_2 = i_a^\tau \Phi_1$ then $\Psi_2 = i^\tau_{\Theta(a)} \Psi_1$.  This proves that twist coordinates are preserved.
\endproof

We conclude this section with a conjugacy class of pairs construction that is better suited to the techniques in Section~\ref{s:gersten} than the one in Definition~\ref{def:strong axis} but is only applicable when the fixed subgroups in question have rank at least two.

\begin{definition}  \label{def:subgroup element pairs}
Given $\phi \in \Out(F_n)$, consider pairs $(\Fix(\Phi), a)$ where   $\Phi \in \cP(\phi)$,   $a \in \Fix(\Phi)$ and $[a] \in \A_\both(\phi)$.  Using $\Fix(i_c\Phi i_c^{-1}) =  i_c(\Fix(\Phi))$, the action of $F_n$ on such pairs is given by $c\cdot  (\Fix(\Phi), a) =  (i_c(\Fix(\Phi)), i_c(a))$, giving a conjugacy pair $[\Fix(\Phi),a]$. Similarly, $\Aut(F_n)$ acts on pairs $(\Fix(\Phi),a)$ by $\Theta\cdot(\Fix(\Phi),a) = (\Theta (\Fix( \Phi)),  \Theta(a))$. This descends to an action of $\Out(\f)$ on the set of such conjugacy pairs.   
\end{definition}

\begin{remark}  Since $\Fix(\Phi)$ is its own normalizer (Lemma~\ref{malnormal}\pref{i:normalizer}), $[\Fix(\Phi),a] = [\Fix(\Phi),a']$ if and only if $a' = i_c(a)$ for some $c \in \Fix(\Phi)$; equivalently $a$ and $a'$ are conjugate as elements of  $\Fix(\Phi)$.
\end{remark}

\begin{lemma}  \label{l:alg sa = sa} Suppose that $\Fix(\Phi)$ and $\Fix(\Phi')$ have rank at least two. Then 
$$[\Phi,a] = [ \Phi', a'] \Longleftrightarrow [\Fix(\Phi),a] = [\Fix(\Phi'), a'] $$
\end{lemma}

\proof   By definition, $[\Phi,a] = [ \Phi', a'] $ if and only if there exists $c \in F_n$ such that $ \Phi' = i_c \Phi i_c^{-1}$ and $a' = i_c(a)$.  Similarly,  $[\Fix(\Phi),a] = [ \Fix(\Phi'), a'] $ if and only if there exists $c \in F_n$ such that $ \Fix(\Phi') = i_c \Fix(\Phi)$ and $a' = i_c(a)$.   As we are assuming that  $\Fix(\Phi)$ and $\Fix(\Phi')$ have rank at least two,   $ \Phi' = i_c \Phi i_c^{-1}$ if and only if $\Fix(\Phi')=i_c \Fix(\Phi)$.
\endproof

\subsection{Applying the Recognition Theorem}\label{recognition}
The Recognition Theorem \cite[Theorem 5.1]{fh:recognition} gives invariants that completely determine rotationless elements of $\Out(\f)$. In this paper, via the following lemma, we use it to give a sufficient condition for two elements of $\upgn$ to be conjugate in $\Out(\f)$.

\begin{lemma} \label{l:recognition} Suppose that $\fG$ and $g :G' \to G'$ are \ct s representing $\phi$ and $\psi$ respectively,   that $\theta \in \Out(F_n)$ and that a line $L$  lifts into $\Gamma(f)$ (meaning that the realization of $L$ in $G$ is the image of a line in $\Gamma(f)$) if and only if $\theta(L)$  lifts into $\Gamma(g)$.    Then for each $\Theta \in \Aut(F_n)$ representing $\theta$:
\begin{enumerate}
\item  \label{item:principal}
there is a bijection $B_{\cP} :\cP(\phi) \to \cP(\psi)$    such that $\FixN(B_\cP(\Phi)) =  \Theta(\FixN(\Phi)) = \FixN(\Phi^\Theta)$.  In particular, $\Fix(B_\cP(\Phi)) =  \Theta \Fix(\Phi)  =  \Fix(\Phi^\Theta)$.
\item \label{item:strong axes}
$ [\Phi,a] \mapsto   [B_{\cP}(\Phi), \Theta(a)] $ defines   a bijection $B_{\sa}: \sa(\phi) \to \sa(\psi)$, independent of the choice of $\Theta$, such that $\phi^\theta = \psi$ if and only if  $B_{\sa}$ preserves twist coordinates.  \end{enumerate}
\end{lemma}

 \proof  
Given $\Phi \in \cP(\phi)$, choose a  line $\ti L_1 \subset  \ti G$ with both ends non-periodic and both ends in $\FixN(\Phi)$.  (This is possible by Remark~\ref{rem:dense non-fixed points}).  By Lemma~\ref{lem:lifting}, the projection $L \subset G$ lifts to  the component $\Gamma_{[\Phi]}(f)$ of $\Gamma(f)$ that corresponds to $[\Phi]$.  By hypothesis, the line $L_1'  \subset G'$ corresponding to $\theta(L)$ lifts to a component of $\Gamma(g)$     and so  by a second application of Lemma~\ref{lem:lifting}   there is a unique  $\Psi \in \cP(\psi)$ such that  $\FixN(\Psi)$ contains the endpoints $\{\Theta(\partial_\pm\ti L_1)\}$ of $\Theta(\ti L_1)$; moreover, $L'_1$ lifts into $\Gamma_{[\Psi]}(g)$.     To see that  $\Psi$     is independent of the choice of $\ti L_1$, suppose that we are given some other $\ti L_2$  with both ends non-periodic and both ends in $\FixN(\Phi)$. Let $\ti L_3$ be the line connecting the terminal endpoint of $\ti L_1$ to the initial endpoint of $\ti L_2$.  Since $\ti L_1$ and $\ti L_3$ have a common endpoint, replacing $\ti L_1$ with $\ti L_3$ does not change  $\Psi$.  For the same reason,  replacing $\ti L_3$ with $\ti L_2$ does not change $\Psi$.  We conclude that  $B_\cP(\Phi) = \Psi$ is well-defined.  This argument also shows that $\Theta$ maps each non-periodic element of $\FixN(\Phi)$ to a non-periodic element of $\FixN(\Psi)$.  Since non-periodic points in $\FixN(\Phi)$ are dense in  $\FixN(\Phi)$,  $ \Theta(\FixN(\Phi)) \subset \FixN(\Psi)$.  Reversing the roles of $\phi$ and $\psi$ and replacing $\theta$ with $\theta^{-1}$, we see that $ \Theta^{-1}(\FixN(\Psi)) \subset \FixN(\Phi)$, which completes the proof of \pref{item:principal}.  Note that if $\Psi = B_{\cP}(\Phi)$ then for all $c \in F_n$, 
$$B_\cP(i_c \Phi i_c^{-1}) =  i_{\Theta(c)} \Psi i_{\Theta(c)}^{-1}$$
  because 
  $ \Theta(\FixN(i_c \Phi i_c^{-1})) =  \Theta( i_c \FixN(\Phi)) = i_{\Theta(c)} \Theta( \FixN(\Phi)) = i_{\Theta(c)}\FixN(\Psi)  = \FixN(i_{\Theta(c)}\Psi i_{\Theta(c)}^{-1}) $. 
  
  \medskip

For \pref{item:strong axes}, suppose that $[a] \in \A_\both(\phi)$, that $\Phi \in \cP(\phi)$ fixes $a$ and that $\Psi = B_{\cP}(\Phi)$.
Define 
$$B_{\sa}([\Phi,a]) = [B_{\cP}(\Phi), \Theta(a)] =  [\Psi, \Theta(a)]$$
Then for all $c \in F_n$, 
$$B_{\sa}([\Phi,a]^c) = B_{\sa}([i_c\Phi i_c^{-1}, i_c(a)]) = [i_{\Theta(c)}\Psi i_{\Theta(c)}^{-1}, i_{\Theta(c)}(\Theta(a))]  = [\Psi,a]^{\Theta(c)}$$
so    $ B_{\sa}$ is   well defined.  By symmetry,  $ B_{\sa}$ is a bijection.  If $\Theta$ is replaced by $i_b\Theta$ for some $b \in F_n$
then $ (\Psi, \Theta(a))$ is replaced by $(i_b \Psi i_b^{-1}, i_b \Theta(a)) =  b \cdot (\Psi, \Theta(a))$.  This shows that  $ B_{\sa}$ is independent of the choice of $\Theta$.
 It remains to show that $\phi^\theta = \psi$ if and only if    $B_{\sa}$ preserves twist coordinates.
  \medskip

 Let $\upsilon = \phi^\theta$.  By Lemma~\ref{first theta} and  Lemma~\ref{theta and axes}, conjugation by $\Theta$ induces:
 \begin{itemize}
\item    a bijection $B'_{\cP} :\cP(\phi) \to \cP(\upsilon)$   defined by $\Phi \mapsto \Theta \Phi \Theta^{-1}$ and satisfying  $\FixN(B'_\cP(\Phi)) =  \Theta\FixN(\Phi)$. 
 \item     a bijection $B'_{\sa}: \sa(\phi) \to \sa(\upsilon)$ defined by $[\Phi,a] \mapsto [\Theta \Phi \Theta^{-1}, \Theta(a)]$ that preserves twist coordinates.
\end{itemize}

 The bijections   $B''_{\cP} = B_\cP {B'}_\cP^{-1}:\cP(\upsilon) \to \cP(\psi)$ and $B''_{\sa} = B_{\sa} {B'}_{\sa}^{-1}:\sa(\upsilon) \to \sa(\psi)$ satisfy:
 \begin{description}
\item [(a)] \label{item:prin}    $\FixN( B''_\cP(\Upsilon)) =  \FixN(\Upsilon)$ for all $\Upsilon \in \cP(\upsilon)$. 
 \item  [(b)]   $B''_{\sa} $    preserves twist coordinates if and only if $B_{\sa}$ does.
\end{description}
Applying  (b), it suffices to show that $\phi^\theta = \psi$ if and only if  $B''_{\sa}$ preserves twist coordinates.

\medskip

Suppose that $[b] \in \A_\both(\upsilon)$, that $b \in \Fix(\Upsilon)$ and that  $ \Upsilon, i_{b^d}\Upsilon \in \cP(\upsilon)$.     Let $a = \Theta^{-1}(b)$ and
   $\Phi= \Theta^{-1}\Upsilon \Theta$.  Then 
  $$B''_{\sa}[\Upsilon,b] = B_{\sa}[\Phi, a] = [B_\cP(\Phi),b]$$
  and likewise
  $$B''_{\sa}[i_b^d\Upsilon,b] = B_{\sa}[i_a^d\Phi, a] = [B_\cP(i_a^d\Phi),b] $$
By definition, the twist coordinate for $[i_{b^d}\Upsilon,b]$ and  $[\Upsilon,b]$ is $d$. It follows that $B''_{\sa} $ preserves twist coordinates if and only if 
$$ B_\cP(i_a^d\Phi) = i_{b^d}B_\cP(\Phi)$$
Since  
$$  B_\cP(i_a^d\Phi) = B''_\cP(i_b^d\Upsilon) \qquad \text{and} \qquad  B_\cP(\Phi)= B''_\cP(\Upsilon)   $$
we conclude that $B''_{\sa} $ preserves twist coordinates if and only if 
\begin{description}
\item[(c)]  $B''_\cP(i_b^d\Upsilon) = i_b^dB''_\cP(\Upsilon)$
\end{description}
By the Recognition Theorem \cite[Theorem 5.3]{fh:recognition}, (a) and (c) are equivalent to $\upsilon = \psi$.
\endproof
 
\section{Limit lines $\acc(\ray) \subset {\cal B}$}\label{s:limit lines}

{Each point $P \in \partial \f$ determines a closed set of lines; see for example \cite[Section 2.4]{fh:recognition}, where the closed set of lines is called the accumulation set of $P$.  In this section we focus on the case that $P \in \cR(\phi)$ and analyze these lines using \ct s.}

\begin{definition} \label{defn:limit line downstairs} 
For each $r \in \partial F_n / F_n$, we define the set  $\acc(r) \subset \cal B$ of {\em limit lines of $r$} as follows. Choose a lift $\ti r \in \partial F_n$, a marked graph $K$  and a ray $\ti R \subset \ti K$ with terminal end $\ti r$.  Let $R \subset K$ be the projected image of $\ti R$.  Then    $L \in \acc(r)$  (thought of as a line in $K$) if and only if  the following equivalent conditions are satisfied. 
\begin{enumerate}
\item
Each finite subpath of $L$ occurs as a subpath of $R$.
\item
For each lift $\ti L \subset \ti K$ of $L \subset K$ there are translates $\ti R_j$ of $\ti R$ such that the initial endpoints of $\ti R_j$ converge to the initial endpoint of $\ti L$ and the terminal endpoints of $\ti R_j$ converge to the terminal endpoint of $\ti L$.
\end{enumerate}
Let $\accnr(r)$ be the set of non-periodic elements of $\acc(r)$.
\end{definition}

\begin{lemma} \label{acc is well defined} $\acc(r)$ and $\accnr(r)$ are well-defined.  Moreover, for each $\theta \in \Out(F_n)$, $\theta(\acc(r)) = \acc(\theta(r))$ and  $\theta(\accnr(r)) = \accnr(\theta(r))$. 
\end{lemma}

\proof If $R'$ is another ray with terminal end   $r$ then $R$ and $R'$ have a common terminal subray $R''$.  Let $R = \alpha R''$ and $R' = \alpha' R''$.  Given a finite subpath $\tau_2 \subset K$ of a line $\ell$,  extend it to  a finite subpath $\tau_1\tau_2 \subset  K$ of $\ell$  where $\tau_1$ is longer than both $\alpha$ and $\alpha'$.    If $\tau_1 \tau_2$ occurs in $R$ then $\tau_2$ occurs in $R''$.  Since $\tau_2$ was arbitrary, every finite subpath of $\ell$ occur in $R$ if and only if every finite subpath of $\ell$ occurs in $R''$.  The same holds for $R'$ and $R''$.  This proves that $\acc(r)$ is independent of the choice of $R$.   Independence of the choice of $\ti R$ is obvious as is the equivalence of (1) and (2).   

Suppose that $K'$ is another marked graph and that $g: K \to K'$ is a homotopy equivalence that preserves markings and so represents the identity outer automorphism. Let $\ti g : \ti K \to \ti K'$ be a lift of $g$.  If $\ti L \subset \ti K$ is a lift of $L$  and $\ti R_j\subset \ti K$ is a sequence of translates of   ray $\ti R$  such that the   initial and terminal endpoints of $\ti R_j$ converge to those of $\ti L$, then the same is true of $\ti L ' = \ti g_\#(\ti L) \subset \ti K'$ and $\ti R'_j = \ti g_\#(\ti R_j) \subset \ti K'$.  This proves  that  $\acc(r)$ is independent of the choice of $K$.

For the moreover statement, choose a homotopy equivalence $h : K \to K$ that represents $\theta$ and lifts $\ti L \subset \ti K$ and $\ti h : \ti K \to \ti K$.   If   $\ti R_j\subset \ti K$ is a sequence of translates of   $\ti R$  whose   initial and terminal endpoints  converge to those of $\ti L$, then the   initial and terminal endpoints of $\ti h_\#(\ti R_j)$ converge to those of    $\ti h_\#(\ti L) \subset \ti K$.  This proves that $\theta(\accnr(r)) \subset \accnr(\theta(r))$.  The reverse inclusion follows by symmetry.
\endproof    

We now specialize to $r \in \cR(\phi)$.
\begin{notn}
For $\phi \in \upgn$, let    
 $$\acc(\phi) = \cup_{r \in \cR(\phi)}\acc(r) \qquad \text{and} \qquad \accnr(\phi) = \cup_{r \in \cR(\phi)}\accnr(r)$$
 \end{notn}
 
 As an immediate consequence of  Lemma~\ref{first theta} and the moreover statement of Lemma~\ref{acc is well defined} we have
 \begin{corollary} \label{acc is natural}
 Suppose that $\theta \in \Out(F_n)$ and that $\psi = \theta \phi \theta^{-1} \in \upgn$.  Then $\theta(\acc(\phi)) = \acc(\psi)$ and  $\theta(\accnr(\phi)) = \accnr(\psi)$.\qed
 \end{corollary}
 
\medskip

 {{\bf For the remainder of the section we assume that  $\fG$ is a \ct\ representing $\phi\in \upgn$.}}  Our goal  is to describe  $\acc(\phi)$ and $\accnr(\phi)$ in terms of  $\fG$.
See in particular Corollary~\ref{cor:limit lines}.

One advantage of working in a \ct\ is that we can work with finite paths and not just with lines and rays. 
\begin{definition} \label{defn:accumulation set} Given a path $\sigma\subset G$, we say that a line $L \subset G$ is contained in the {\em accumulation set  $\Acc(\sigma)$ of $\sigma$ with respect to $f$} if   every finite subpath of $L$ occurs as a subpath of $f^k_\#(\sigma)$ for arbitrarily large  $k$.
\end{definition}

\begin{notn}  For each twist path $w$, we write $w^\infty$ for both the ray that is an infinite concatenation of copies of $w$ and the line that is a bi-infinite concatenation of copies of $w$, using context to distinguish between the two.  We use either $\bar w^{\infty}$ or $w^{-\infty}$ for the ray or line obtained from $w^{\infty}$ by reversing orientation on $w$.
\end{notn}
\begin{exs} \label{simple limits}  \begin{enumerate}
\item If $\sigma$ is a Nielsen path then $\Acc(\sigma) = \emptyset$. 
 \item Suppose that  $ E\in \lin(f)$ and $f(E) = Ew^{d} $.  
  \begin{enumerate}
 \item  If  $d  > 0$ then $ \Acc(E) =  \{w^{\infty}\}$ and $\Acc(\bar E) = \{\bar w^{\infty}\}$.
 \item  If  $d  < 0$ then $\Acc( E) =  \{\bar w^{\infty}\}$  and $\Acc(\bar E) =  \{w^{ \infty}\}$.
 \end{enumerate}
\item If $E_i,E_j \in \lin(f)$ satisfy $f(E_i) = E_iw^{d_i}$ and $f(E_j) = E_jw^{d_j}$ for $d_i \ne d_j$  then for all $p \in \Z$, $\Acc(E_iw^p\bar E_j ) =  \{w^{\infty}\}$ if  $d_i  > d_j$ and $\Acc(E_iw^p\bar E_j ) =  \{\bar w^{\infty}\}$  if  $d_i  < d_j$.
\end{enumerate}
\end{exs}

Recall  from Lemma~\ref{identifying Fix+} that  there is a bijection between $\cR(\phi)$ and the set $\E_f$ of non-fixed non-linear  edges of $G$ and that  if  $\ray \in \cR(\phi)$ corresponds to $E \in \E_f$ then the   eigenray $R_E = E\cdot u_E\cdot [f(u_E)]\cdot [f^2(u_E)] \cdot \ldots \subset G$ has terminal end  $r$.   Thus, a line $L\subset G$ is an element of  $ \acc(r)$ if and only if  each finite subpath of $L$ occurs as a subpath of   $R_E$.

Limit lines of eigenrays  are connected to accumulation sets as follows. 

\begin{lemma}\label{two definitions} If $\ray \in \cR(\phi)$ corresponds to $E \in \E_f$  and $f(E) = E\cdot u_E$ then   $\acc(\ray) = \Acc(E) = \Acc(u_E\cdot f_\#(u_E)) = \Acc(u_E\cdot f_\#(u_E)\cdot \ldots \cdot  f^k_\#(u_E))$ for any $k \ge 1$.
\end{lemma}

\proof    The first equality  is an immediate consequence of the definitions and  the fact that $E \subset f(E) \subset f^2_\#(E) \subset \ldots$ is an increasing sequence whose union is $R_E$.   Likewise,  
$$ \Acc(u_E\cdot f_\#(u_E)) \subset  \Acc(u_E\cdot f_\#(u_E)\cdot \ldots \cdot  f^k_\#(u_E)) \subset \Acc(f^{k+1}_\#(E)) = \Acc(E)$$
  is immediate.  It therefore suffices to show that $ \acc(r) \subset \Acc(u_E\cdot f_\#(u_E))$. 
  
If $L \in \acc(\ray)$ then every finite subpath $\sigma$ of $L$ occurs as a subpath of every subray of $R_E$.  Since the length of $f^k_\#(u_E)$ tends to infinity with $k$,   each occurence of $\sigma$ that is sufficiently far away from the initial endpoint of $R_E$ is contained in some $ f^k_\#(u_E)\cdot f^{k+1}_\#(u_E) = f^k_\#(u_E \cdot f_\#(u_E))$.  As the occurence of $\sigma$ moves farther down the ray, $k \to \infty$.   
\endproof

\begin{notn}  \label{notn: strong partial order}  Define a partial order on the set $\E_f \cup \E_f^{-1}$ by $E_1 \gg E_2$  if $E_1 \ne E_2$ and if,  for some $k \ge 0$, $E_2$  is  crossed by $f^k_\#(E_1)$ and so by Lemma~\ref{not crossed} is    a term  in the complete splitting of  $f^k_\#(E_1)$.   (In Notation~\ref{weaker po}  we define a partial order $>$  on $\E_f$ that does not distinguish between $E$ and $\bar E$.)     
\end{notn}

As an immediate consequence of the definition, we have

\begin{lemma} \label{lem:strong po}  If $E, E' \in \E_f \cup \E_f^{-1}$ and $E\gg E'$ then the height of $E'$ is less than the height of $E$ and $\Acc(E') \subset \Acc(E)$.  \qed
\end{lemma}

The terms $\mu_i$ in the complete splitting of   $u_E\cdot f_\#(u_E)$ are   Nielsen paths, exceptional paths  and single edges with height strictly less than that of $E$.   Each $\Acc(\mu_i)$ is a subset of  $\Acc(E) = \acc(r)$.    If  $\mu_i \in \E_f \cup \E_f^{-1}$ then  $\Acc(\mu_i)$ can be understood inductively.  The remaining $\Acc(\mu_i)$'s   are given in Examples~\ref{simple limits}.     The work in identifying $\acc(\ray) = \Acc(u_E\cdot f_\#(u_E))$  is to  determine what additional lines  must be added to $\bigcup \Acc(\mu_i)$. 

\begin{notn} \label{notn: f infinity} For  a path $\alpha \subset G$, we say that {\em $ f^k_\#(\alpha)$  converges to a ray $R \subset G$} if for all $m$ there exists $K$ such that the initial $m$-length segments of $ f^k_\#(\alpha)$ and of $R$  are equal for all $k \ge K$.  Note that $R$ is necessarily unique and $f_\#$-invariant.  We sometimes write $R = f^\infty_\#(\alpha)$.
\end{notn}

\begin{exs} \label{example:converges}
\begin{enumerate}
\item  Suppose that $E \in \lin(f)$ and that $f(E) = Ew^{d}$.  
\begin{enumerate}
\item  If $d > 0$ then  $f_\#^k(E)$ converges to $ Ew^{\infty}$ and $f_\#^k(\bar E)$ converges to $ \bar w^{\infty}$.
\item If $d < 0$ then    $f_\#^k(E)$ converges to $ E\bar w^{\infty}$ and $f_\#^k(\bar E)$ converges to $ w^{\infty}$.
\end{enumerate}
\item  If $E \in \E_f$ then $f_\#^k(E)$ converges to $R_E$.
\item If $E_i,E_j \in \lin(f)$ satisfy $f(E_i) = E_iw^{d_i}$ and $f(E_j) = E_jw^{d_j}$ for $d_i \ne d_j$   then  for all $p \in \Z$, $f_\#^k(E_i w^p\bar E_j) $ converges to $E_iw^{  \infty}$ if $d_i > d_j$ and to $ E_i\bar w^{ \infty}$ if $d_i < d_j$.  
\end{enumerate}
\end{exs}

\begin{notn} \label{notn:stabilization} If $E \in \E_f$, then the first growing term of $f(\bar E)$ has height less than that of $E$.  It follows that there exists    $M > 1$ so that if $\sigma_i$ is a  growing term in the complete splitting of a path $\sigma$    and if $m \ge M$, then the first growing term in the complete splitting of $f_\#^m(\sigma_i)$ is not an element of  $\E_f^{-1}$ and the last growing term in the complete splitting of $f_\#^m(  \sigma_i)$ is not an element of $\E_f$.   We refer to $M$ as the {\em stabilization constant} for $f$.   
\end{notn}

\begin{lemma} \label{paths converge} Let $M$ be the stabilization constant for $f$.  If $\sigma$ is a completely split growing path   then $ f^k_\#(\sigma)$ converges to a ray  $f_\#^\infty(\sigma) = \rho R $   where 
\begin{enumerate}
\item \label{item:ray types} $\rho$ is a  (possibly trivial) Nielsen path and  one of the following holds:  
\begin{enumerate} \item \label{ray type a}$R =R_E$ for some $E \in \E_f$.
\item   \label{ray type b} $R = Ew^{\pm \infty}$ for some $E \in \lin_w(f)$.
\item   \label{ray type c} $R = w^{\pm\infty}$ for some twist path $w$.  
\end{enumerate} 
\item  \label{item:first growing}
If $\sigma =\mu_1\cdot \nu_1\cdot \mu_2 \cdot   \ldots$ is the coarsening of the complete splitting of $\sigma$  into maximal (possibly trivial) Nielsen paths $\mu_i$ and single growing terms $\nu_i$, then 
$f_\#^\infty(\sigma)  = \mu_1 f_\#^\infty(\nu_1) $.
\item  \label{item:the bar E case} In case 1(c) there exists  $E \in \lin_w(f)$ and a smallest $k_\sigma \le M$ such that  the first growing term in the coarsened complete splitting of  $f^k_\#(\sigma)$ is $\bar E$ for all $k \ge k_\sigma$.  Moreover, if  the first growing term in the coarsened complete splitting of $\sigma$ is not an edge in $\E_f^{-1}$  then $k_\sigma =1$.
\end{enumerate}
\end{lemma}

\proof   There is no loss in replacing  $\sigma$ with its first growing term.     The only case that does not follow   from Example~\ref{example:converges} is  that $\sigma = \bar E  \in \E_f^{-1}$.   This case follows from the definition of $M$ and the obvious induction argument.
\endproof  

\begin{remark}  \label{finite determination} The rays in \ref{paths converge} are finitely determined: in case \pref{ray type a} $R$ is determined by the edge $E$, in case \pref{ray type b}  $R$ is determined by $E, w$ and a choice of $\pm$ and in case \pref{ray type c} $R$ is determined by  $w$ and a choice of $\pm$.  From this data one can write down any finite initial subpath of $R$.
\end{remark}

\begin{lemma} \label{adjacency term}  Suppose that $\sigma \subset G$ is a completely split path  and that $\sigma = \alpha\cdot \beta$ is a coarsening of the complete splitting in which both $\alpha$ and $\beta$ are growing.  Let $R^- = f_\#^\infty(\bar \alpha)$, let $R^+ = f_\#^\infty(\beta)$ and let $\ell = (R^-)^{-1}R^+$.       Then $\Acc(\sigma) = \Acc(\alpha) \cup \Acc(\beta) \cup \{\ell\}$.
\end{lemma}

\proof  The inclusion  $\Acc(\alpha) \cup \Acc(\beta)  \subset \Acc(\sigma)$ follows from the fact that   $\alpha$ and $\beta$ occur as concatenation of  terms in  a splitting of $\sigma$.  It is an immediate consequence of the definitions that   $\ell \in \Acc(\gamma)$.   It therefore suffices to assume that $L \in \Acc(\sigma)$ is not contained in $\Acc(\alpha) \cup \Acc(\beta) $ and prove that $L = \ell$.

      Choose a finite subpath $L_1$  of $L$ and $K > 0$ so  that $L_1$ does not occur as a subpath of $f_\#^k(\alpha)$ or of $  f_\#^k(\beta)$ for $k \ge K$. Extend $L_1$ to an increasing sequence      $L_1 \subset L_2 \subset \ldots$  of finite subpaths of $L$ whose union is $L$. For each  $j \ge 1$, let  $C_j$ be the length of $L_j$.     There exist arbitrarily large $k$ so that $L_j$ includes as a subpath of   $f^{k}_\#(\sigma) = f^{k}_\#(\alpha)\cdot  f^{k}_\#( \beta) $.   The induced inclusion of  $L_1$ in $f^{k}_\#(\sigma)$  must  intersect both $ f^{k}_\#(\alpha)$  and $ f^{k}_\#(\beta)$  and so $L_j$ is included as  a subpath of the concatenation of the terminal segment of $ f^{k}_\#(\alpha)$ of length $C_j$ with the  initial segment of $ f^{k}_\#(\beta)$  of length $C_j $.    If $k$ is sufficiently large then the length $C_j$ initial segments of $R^-$ and of $ f^{k}_\#(\bar \alpha)$ agree and the length $C_j$ initial segments of $R^+$ and of $ f^{k}_\#(\beta)$ agree.   Thus  each $L_j$ can be included as a subpath of $\ell$.  Since the induced inclusion of $L_1$  contains the juncture point between $R^-$ and $R^+$,  we may   pass to a subsequence of $L_j$'s and choose inclusions of  $L_j$ into $\ell$ so that induced inclusion of $L_1$ in $\ell$  is indpendent of $j$.    It follows that if $i < j$ then the inclusion of $L_i$ into $\ell$ is the restriction of the inclusion of $L_j$ into $\ell$ and hence that there is a well-defined inclusion of $L$ into $\ell$.  This inclusion is necessarily onto and so $L = \ell$. 
\endproof

 \begin{cor} \label{cor:limit lines}  For each $r \in \cR(\phi)$,    
 \begin{enumerate}
\item \label{item:ell decomposes}  Each $L  \in \acc(r)$ decomposes as   $L = (R^-)^{-1}\rho R^+$  where $\rho$ is a (possibly trivial) Nielsen path and   $ R^+$ and $R^-$ satisfy 1(a), 1(b) or 1(c) of Lemma~\ref{paths converge}.   In particular, each $L$ is $\phi$-invariant and is finitely determined in the sense of Remark~\ref{finite determination}  and each periodic    $L$ equals $ w^{\pm \infty}$ for some   \twistpath\ $w$.

\item \label{item:algorithmic} $\acc(r)$ is a finite set and the finite data that determines each of its elements   can be read off from     $\fG$. 

\item \label{item:accnr is nonempty}    $\accnr(r) \ne \emptyset$.

\item \label{item:endpoints in FixN}
For each $L \in \acc(r)$ and each lift $\ti L$ there exists $\Phi \in \cP(\phi)$ such that $\partial_- \ti L, \partial_+\ti L
\in \FixN(\Phi)$.  Equivalently  $L$ lifts into $\gf$.
\end{enumerate}
\end{cor}

\proof
By  Lemma~\ref{two definitions} we can replace   $\acc(r)$ with $\Acc(E)$ where $E \in \E_f$ corresponds to $r$.    
  Lemma~\ref{two definitions} also implies that  $\Acc(E) = \Acc(u \cdot f_\#(u))$  where $f(E)= E\cdot u$.  Let \[u \cdot f_\#(u) = \rho_0\cdot \sigma_1 \cdot \rho_1 \cdot \sigma_2\ldots \sigma_q \cdot \rho_{q}\]   be a coarsening of the  complete splitting of $u \cdot f_\#(u)$ so that each $\sigma_i$ is a single growing term and so that the $\rho_i$'s are (possibly trivial) Nielsen paths.      For $1 \le i \le q-1$, let $R_i^- =  f^{\infty}_\#(\bar \sigma_{i})$ and for  $2 \le i \le q$, let $R_i^+ = f^{\infty}_\#(\sigma_{i})$.  For  $ 1 \le i \le q-1$, define $\ell_i =   (R_i^-)^{-1}\rho_{i} R_{i+1}^+$. Lemma~\ref{adjacency term} and the obvious induction argument imply that
  \[\acc(r) = \Acc(u \cdot f_\#(u)) = \Acc(\sigma_1) \cup \ell_1 \cup \Acc(\sigma_2) \cup \ldots \cup \ell_{q-1} \cup\Acc(\sigma_{q}) \]   

Lemma~\ref{paths converge}\pref{item:ray types}  implies that each $\ell_i$ satisfies \pref{item:ell decomposes}.  If $\sigma_i$ is linear then $\Acc(\sigma_i) = w^{\pm \infty}$ for some \twistpath\ $w $ by Examples~\ref{simple limits}.   
  The remaining $\sigma_i$ have the form $E'$ or $\bar E'$ for some $E' \in \E_f$  with height less than that of $E$.    Downward induction on the height of $E$   completes the proof of \pref{item:ell decomposes}  and \pref{item:algorithmic}.   
    
  We now turn to \pref{item:accnr is nonempty}, assuming at first that $q > 2$.   If $\sigma_2$ is exceptional or an element of $\E_f \cup \lin(f)$ then $\ell_1$ is non-periodic.  Otherwise, $\bar \sigma_2 \in \E_f  \cup \lin(f)$ and $\ell_2$ is non-periodic. Both of these statements follow from Lemma~\ref{paths converge}.  If $q =2$  then  $\sigma_1$ is linear.  One easily checks that $\ell_1$ is non-periodic in the various cases that can occur. 
   For example if $\sigma_1 = E_1 \in \lin(f)$ then $\sigma_2 = E_1$ and $ \ell_1 =   \bar w^{\pm \infty} \rho_1 E_1 w^{\pm \infty}$.  The remaining cases are left to the reader.

The equivalence of  the two conditions in \pref{item:endpoints in FixN} follows from Lemma~\ref{lem:lifting}. To prove that $L$ lifts into $\Gamma(f)$, we make use of the fact that each vertex in $G$ lifts uniquely to $\stallings^0(f)$ and the fact that each Nielsen path in $G$ lifts uniquely into $\stallings(f)$ with one, and hence both, endpoints in $\stallings^0(f)$.   These facts follow immediately from the construction of $\stallings(f)$ and the fact that every Nielsen path is a concatenation of fixed edges and (necessarily closed) \iNp s.  Given these facts, we may assume that $L = (R^-)^{-1}\rho R^+$ is not a concatenation of Nielsen paths and hence that the initial edge $E_j$ of either $R^-$ or $R^+$ is an element of $\lin(f) \cup \E_f$.  The two cases are symmetric so we may assume that $E_j$ is the initial edge of $R^-$.  Let  $E'_j \subset \stallings(f)$ be the unique lift of $E_j$ with initial vertex $v' \in \stallings^0(f)$ and then extend this to a lift of $R^-$ into $\stallings(f)$. The Nielsen path $\rho$ lifts to a path  $\rho' \subset \stallings^0(f)$ with initial vertex $v'$ and terminal vertex, say  $w'$.  If $R^+$ is a concatenation of Nielsen paths then it lifts into $\stallings^0(f)$ with initial vertex $w'$.  Otherwise we lift $R^+$ in the same way that we lifted $R^-$. 
\endproof

\begin{excont*}\label{ex.g}
Recall $R_q=q\cdot c\cdot cb\cdot cbba\cdot \ldots \cdot cbba\ldots ba^{k-1}\cdot\ldots$ and so $\acc(r_q)=\{a^\infty R_c, a^\infty ba^\infty, a^\infty\}$ and $\accnr(r_q)=\{a^\infty R_c, a^\infty ba^\infty\}$.
\end{excont*}

\section{Special free factor systems}

\subsection{A canonical collection of free factor systems}\label{s:canonical ffs}
In this section, we define a canonical partial  order $<$ on $\cR(\phi)$ and then associate a nested sequence $\vec\F(\phi, \dle) = \F_0 \sqsubset \F_1 \sqsubset \ldots \sqsubset \F_t$ of $\phi$-invariant free factor systems to each total order $\dle$ on $\cR(\phi)$ that extends  $<$.  The bottom \ffs\ $\F_0$ 
is  the smallest free factor system that carries all  conjugacy classes that grow at most linearly and is independent of $\dle$.  The inclusions $\F_{i-1} \sqsubset \F_i$ are all one-edge extensions.  The \ct s that represent $\phi$ with filtrations that realize $\vec\F(\phi, \dle)$ are easier to work  with than  generic \ct s - see Lemma~\ref{lemma:three extension types}.

\begin{notn}  \label{weaker po}   Suppose that $\fG$ is a \ct\ representing $\phi$ and  that $E_1$ and $ E_2$ are distinct elements of  $\E_f$.  If $E_1$ or $\bar E_1$ is a term of the complete splitting of   $f_\#^k(E_2)$ for some   $k \ge 1$ then we write $E_1< E_2$.  Lemma~\ref{not crossed} implies that $<$ is a partial order on $\E_f$.  If $E_1 <E_2$ are consecutive elements in the partial order then we write $E_1<_cE_2$.  Note that if we define $E_1 <' E_2$   to mean $E_1$ or $\bar E_1$ is a term of the complete splitting of   $f(E_2)$ then $<$ is the partial order determined from $<'$ by  extending transitively.  Thus $<$ can be computed.

If $r_1, r_2 \in \cR(\phi)$ and $r_1$ is an  end  of some element of $\accnr(r_2)$ then we write $r_1 < r_2$.  Lemma ~\ref{preserves partial order} below implies that $<$ defines a partial order on $\cR(\phi)$.  If $r_1 < r_2$ are consecutive elements in the partial order then we write $r_1<_cr_2$.
\end{notn}

\begin{excont*}\label{ex.h}
In our example, the only relation is $r_c<r_q$.
\end{excont*}

Recall from Lemma~\ref{identifying Fix+} that the map that sends $E$ to the end of $R_E$ defines a bijection between $ \E_f$ and  $\cR(\phi)$. 

\begin{lemma} \label{preserves partial order} For any \ct\ $\fG$, the  bijection between $\E_f$  and $\cR(\phi)$   preserves $<$.
\end{lemma}

\begin{proof}  Suppose that $E_1,E_2 \in \E_f$ correspond to $r_1,r_2 \in \cR(\phi)$ respectively. 

 If $E_1 < E_2$ and $f(E_2) = E_2 \cdot u_2$ then $E_1$ or $\bar E_1$ is a term in the complete splitting of $f^k_\#(u_2)$ for some, and hence all sufficiently large, $k$.  By Lemma~\ref{two definitions}, there exists a completely split path $\gamma$ such that $\acc(r_2) = \Acc(\gamma)$ and such that the complete splitting of $\gamma$ has a  coarsening  $\gamma = \gamma_1 \cdot \gamma_2 \cdot \gamma_3$ into three growing terms with $\gamma_2$ equal to  either $\bar E_1$ or $E_1$.  Lemma~\ref{adjacency term} therefore implies that  $R_{E_1}$ is  a terminal ray of $L$ or $L^{-1}$ for some   $L \in \acc(r_2)$.  Thus   $r_1 < r_2$.
 
 If $r_1 < r_2$ then $R_{E_1}$ is  a terminal ray of $L$ or $L^{-1}$ for some   $L \in \acc(r_2)$ by Corollary~\ref{cor:limit lines}\pref{item:ell decomposes}.  It follows that $f^k_\#(E_2)$ crosses $E_1$ or $\bar E_1$ for all sufficiently large $k$.  Lemma~\ref{not crossed} implies that $E_1 < E_2$.
\end{proof}

\begin{lemma} \label{weaker po conjugacy} If $\psi = \theta \phi \theta^{-1}$ then the bijection $\cR(\phi) \to \cR(\psi)$ induced by $\theta$ (see Lemma~\ref{first theta}) preserves partial orders.
\end{lemma}

\proof  By Lemmas~\ref{first theta} and \ref{acc is natural}, we have  $\theta(\cR(\phi)) =  \cR(\psi)$ and  $\theta(\accnr(r)) = \accnr(\theta(r))$ for each $r \in \cR(\phi)$.  The fact that $\theta$ preserves the partial order now follows from the definition of the partial order.
\endproof 

\begin{notn} \label{notn:total order}    Extend the partial order $<$ on $\cR(\phi)$ to a total  order $\dle$ and write $\cR(\phi) = \{r_1,\ldots, r_s\}$ where the elements are listed in increasing order. Given a \ct\ $\fG$ representing $\phi$, transfer the total order $\dle$ on $\cR(\phi)$ to a total order (also called) $\dle$ on $\E_f  = \{E_1,\ldots, E_s\}$ using the bijection between $\E_f$ and $\cR(\phi)$ given in Lemma~\ref{identifying Fix+}.

Recall from Section~\ref{sec:stallings} that each component $C$ of the eigengraph $\Gamma(f)$ is constructed  from a component $C_0$ of $\Fix(f)$ by first adding \lq lollipops\rq, one for each $E \in \lin(f)$ with initial vertex in $C_0$, to form $C_1$, and then adding rays labeled   $R_E$, one for each $E \in \E_f$ with initial vertex in $C_0$.  Each $E \in \E_f$ contributes exactly one ray to $\Gamma(f)$ and we identify that ray with the eigenray $R_E$; it is the unique lift of $R_E$ to $\stallings(f)$.   Each $E \in \lin(f)$  contributes exactly one lollipop to $\Gamma(f)$.    Note that $C$ is contractible if and only if $C_1$ is contractible if and only if $C_0$ is contractible and there are no $E \in \lin(f)$ with initial vertex in $C_0$.   In this contractible case, $C$ is obtained from a (possibly trivial) tree in $\Fix(f)$ by adding eigenrays and we  single out the ray $R_E \subset C$ whose associated edge $E$ is lowest with respect to $\dle$.    These edges define subsets $\cR^*(\phi) \subset \cR(\phi)$ and $\E^*_f \subset \E_f$ that correspond under the bijection between $\cR(\phi)$ and $\E_f$.  
\end{notn}

 \begin{definition}  \label{notn:F0}  A conjugacy class $[a]$ {\em grows at most linearly} under iteration by $\phi$  if for some, and hence every, set of generators  there is a linear function $P$ such that word length of $\phi^k([a])$ with respect to those generators is bounded by $P(k)$.  If $\fG$ represents $\phi$, then word length of $\phi^k([a])$ can be replaced by edge length of $f_\#^k(\sigma)$ in $G$  where $\sigma \subset G$ is the circuit representing $[a]$.  The {\em linear growth \ffs}  $ \F_0(\phi)$ is the  minimal free factor system that carries all  conjugacy classes that grow at most linearly under iteration by $\phi$.  
\end{definition}  

\begin{lemma}   \label{lemma:one edge extension} Suppose that $\fG$ is a \ct\ representing $\phi$ and that $\dle$ and $\E_f = \{E_1,\ldots,E_s\}$ are as in Notation~\ref{notn:total order}.  Let  $K_0 \subset G$  be the subgraph consisting  of all fixed and linear edges for $\fG$.   For $1 \le j \le s$, inductively define $K_j = K_{j-1} \cup  E_j$. Then:
 \begin{enumerate}
 \item $ \F_0(\phi)= \F(K_0, G)$ (as defined at the beginning of Section~\ref{sec:ffs}).
 \item Each $K_j$ is $f$-invariant.  
\item  If  $E_j \in \E^*_f$   then   $ \F(K_j,G) =  \F(K_{j-1},G)$; otherwise $ \F(K_{j-1},G) \sqsubset \F(K_j,G) $ is a proper one-edge extension.  
\end{enumerate} 
\end{lemma}

\begin{proof}   In proving (1), we work with circuits $\sigma \subset G$ and edge length in $G$ rather than conjugacy classes $[a]$  and word length with respect to a set of generators of $F_n$.  
 If $E$ is an edge of $K_0$, then $f(E)=E\cdot u$ for some (possibly trivial) closed Nielsen path $u$. Lemma~\ref{not crossed} implies that $u \subset K_0$ and  hence that $K_0$ is $f$-invariant. Each circuit in $K_0$ grows at most linearly under iteration by $f_\#$ since every edge in $K_0$ does.  Thus $\F(K_0, G) \sqsubset \F_0(\phi)$.  
 
  After replacing $\sigma$ with   $f^m_\#(\sigma)$ for some $m \ge 0$, we may assume by \cite[Lemma 4.25]{fh:recognition} that $\sigma$ is completely split. Lemma~\ref{not crossed} implies that if $\sigma$ is not contained in $K_0$ then, up to reversal of orientation,  some term in the complete splitting of $\sigma$ is an edge $E \in \E_f$.  In this case,  $\sigma$ grows at least as fast as $E$ does.  If $f(E) = E\cdot u$ then   the length of $f_\#^k(u)$ goes to infinity with $k$ and so the   length of $f_\#^k(E)$  grows faster than any linear function.      This proves that $K_0$ contains every circuit that grows at most linearly so $\F_0(\phi) \sqsubset \F(K_0, G)$.  This completes the proof of (1).

For the remainder of the proof we may assume that $j \ge 1$.
For $E_j \in \E_f$, the terms in the complete splitting of $f(E_j)$, other than $E_j$ itself, are   exceptional paths,  Nielsen paths and  single edges $E_i$ or $\bar E_i$ that are either linear or satisfy $E_i < E_j$. Lemma~\ref{not crossed} implies that the exceptional paths and Nielsen paths are contained in $K_0$.  The single edge terms other than $E_j$ are contained in $K_{j-1}$ by construction.    Thus  $f(E_j) \subset K_j$ and   $K_j$ is $f$-invariant. This proves (2).  
 
  The  terminal endpoint of each $E_j \in \E_f$ is contained in a non-contractible component of $K_{j-1}$ because $f(E_j) = E_j u_j$ for a non-trivial closed path  $u_j \subset K_{j-1}$.    If $E_j \in \E^*_f$ with initial vertex $v_j$ then the component of $K_{j-1}$ that contains  $v_j$ is a contractible component of $\Fix(f)$.    In this case every line in $K_j$ is contained in $K_{j-1}$ so $ \F(K_j,G) =  \F(K_{j-1},G)$.  Otherwise,   $v_j$ is contained in a non-contractible  component of $K_{j-1}$ so   $\F(K_{j-1},G) \sqsubset \F(K_j,G)$ is a proper inclusion.  Obviously  $K_j$ is obtained from $K_{j-1}$ by adding a single edge.
\end{proof}

Recall from Lemma~\ref{lem:lifting} that  the set of lines that lift to $\stallings(f)$ is independent of the choice of \ct\  $\fG$ representing $\phi$. The next lemma shows that the $\F(K_j,G)$'s defined in Lemma~\ref{lemma:one edge extension} depend only on $\phi$ and $<_T$ and not on the choice of \ct\ $\fG$.

\begin{lemma} \label{lines generate} Continue with the notation of Lemma~\ref {lemma:one edge extension}.  For each $r_j \notin \cR^*(\phi)$, there exists at least one line $ L(r_j)$ that lifts to $\stallings(f)$, whose terminal end is  $r_j$ and whose initial end is  not $r_l$ for any $l\ge j$.   Moreover, for any such choice of lines, $\F(K_j,G)$ is the smallest free factor system that contains $\F_0(\phi)$ and carries $\{L(r_l):  l \le j \text { and } r_l \notin\cR^*(\phi)\}$.
\end{lemma}

\proof
Let $C = C(j)$ be the component of $\Gamma(f)$ that contains $R_{E_j}$, let $C_0 \subset C_1 \subset C$ be as in Notation~\ref{notn:total order} and, for each $1 \le q \le s$, let $A_q\subset C$  be the union of $C_1$ with the rays  $R_{E_l}$ in $C$ with $E_l \le E_q$.  By construction, and by Lemma~\ref{lemma:one edge extension}, $R_{E_l}$ is included in $A_q$ if and only if $R_{E_l} \subset K_q$.      Since $r_j \notin \cR^*(\phi)$, either $C_1  $ is non-contractible or   $A_{j-1}$ contains at least one  ray $R_{E_l}$.  In both cases,  the ray $R_{E_j} \subset C$ extends by a ray in $A_{j-1}$ to a line in $A_j$.  The projection $L(r_j)$ of this line into $K_j$ satisfies the conclusions of the main statement of the lemma.

The moreover part of the lemma is proved by induction on $j$ with the base case $j = 0$ following from Lemma~\ref{lemma:one edge extension}(1).  For the inductive case, let $\F'_j$ be   the smallest free factor system   that carries $K_{j-1}$  and $L(r_j)$.  Then   $\F(K_{j-1},G) \sqsubset \F'_j \sqsubset \F(K_j,G)$ with the first inclusion being proper and  $\F_j' $ does not have more components than $\F(K_{j-1},G)$. Lemma~\ref{reducible extensions} implies that    $\F'_j = \F(K_j,G)$. 
\endproof

\def\sHH{{\sf HH}}
\def\sLH{{\sf LH}}

\begin{notn} \label{notn:ffs}   Let $K_0 \subset K_1 \subset \ldots \subset K_s = G$ be as in Lemma~\ref{lemma:one edge extension} and let $\vec\F(\phi, \dle) = \F_0 \sqsubset \F_1 \sqsubset \ldots \sqsubset \F_t$ be  the increasing sequence of distinct free factor systems determined by the $K_j$'s.   (Equivalently, $\vec\F(\phi, \dle)$ is the sequence determined by those $K_j$'s with $r_j \not \in \cR^*(\phi)$.)    We say that $\vec\F(\phi, <_T)$   is the {\em sequence  of free factor systems determined by $\phi$ and $\dle$}. Lemma~\ref{lines generate} justifies this description by showing that   $\vec\F(\phi, \dle)$ depends only on $\phi$ and $\dle$.  To simplify notation a bit, we write $L_k$ for $L(r(j))$ where $r(j)$ is the $k^{th}$ lowest element of $\cR(\phi) \setminus \cR^*(\phi)$.   Thus $\F_k$ is filled  by $\F_0$ and $L_1,\ldots, L_k$.

We sometimes refer to a  nested sequence of free factor systems as a {\it chain}. A chain $\frak c=(\F_0\sqsubset \dots\sqsubset\F_t)$ is {\it special for $\phi$} if $\frak c=\vec\F(\phi,<_T)$ for some extension $<_T$ of $<$ to a total order on $\cR(\phi)$. A free factor system $\F$ is {\it special for $\phi$} if $\F$ is an element of some special chain for $\phi$. The set of special free factor systems for $\phi$ is denoted $\lat0(\phi)$. A free factor $F$ or its conjugacy class is {\it special for $\phi$} if $[F]$ is an element of some special \ffs\ for $\phi$. A pair $\fe=(\F^-\sqsubset \F^+$) of \ffs s is a {\it special one-edge extension for $\phi$} if its appears as consecutive elements of some special chain for $\phi$.
\end{notn}

By applying the existence theorem for \ct s given in \cite[Theorem 1.1]{fh:CTconjugacy},  we can  choose a \ct\ whose filtration  realizes  $\vec \F(\phi, <_T)$ for any given $<_T$.  The following lemma shows that the case analysis for a  \ct\ with this property is simpler than that of a random \ct.    

\begin{lemma}  \label{lemma:three extension types}   Suppose that $\vec\F(\phi, \dle) = \F_0 \sqsubset \F_1 \sqsubset \ldots \sqsubset \F_t$ and  that  $\fG$  and $\filt$   are a \ct\ and filtration representing $\phi$ and realizing $\vec\F(\phi, <_T)$; i.e.   for all $0 \le k \le t$ there is an $f$-invariant core subgraph $G_{i_k}$ such that $\F_k= \F(G_{i_k}, G)$.     Then  $G_{i_k} \setminus G_{i_{k-1}}$ is a single topological arc $A_k$ with both endpoints in $G_{i_{k-1}}$.  Moreover, letting $D_k$ be the element of $\E_f$ corresponding to $\partial_+ L_k \in \cR(\phi)$ (as in Lemma~\ref{lines generate}), $A_k$ can be oriented so that one of the following is satisfied.
\begin{description}
\item \label{c:higher higher} $[\sHH]:$\ \    $A_k =  \bar C_k D_k$ where $C_k\in \E_f$.
\item\label{c:linear-higher order}$[\sLH]:$\ \   $A_k =  \bar C_k D_k$ where $C_k\in \lin(f)$. 
\item\label{c:higher order}$[\sH]:$\ \ $A_k = D_k$.
\end{description}
\end{lemma}

\proof   By Lemma~\ref{lemma:one edge extension}, each  $\F_{j-1} \sqsubset \F_j$  is a one-edge extension.  \cite[Part II, Lemma 2.5]{handelMosher:subgroups}  therefore implies that  $G_{i_k}$ is constructed from $G_{i_{k-1}}$ in one of three ways: add    a single topological edge with both endpoints in $G_{i_{k-1}}$; add a single topological edge that forms a circuit that is disjoint from $G_{i_{k-1}}$;   add an edge forming a disjoint circuit and then add an edge connecting that circuit to $G_{i_{k-1}}$.   In the second and third cases the circuit is $f$-invariant   in contradiction to the fact that  $K_0$ contains all $\phi$-invariant conjugacy classes.  Thus $G_{i_k} $ is obtained from $G_{i_{k-1}}$ by adding a single topological arc $A_k$ with both endpoints in $G_{i_{k-1}}$. 
 
 The arc $A_k$ consists of either one or two edges of $G$.  Indeed, a \lq middle\rq\ edge can not be fixed by the (Periodic Edge) property (\cite[Definition 4.7(5)] {fh:recognition}) of a \ct\  and cannot be non-fixed because in that case its  terminal end would be contained in a core subgraph of  $G_{i_{k-1}}$ by \cite[Lemma 4.21]{fh:recognition}. The (Periodic Edge) property also implies that if $A_k$ consists of two edges, then neither is fixed.    To complete the proof, it suffices to show that $A_k$ crosses $D_k $.  We will do so by showing that $D_k$ is not contained in $G_{i_{k-1}}$ and is contained in $G_{i_k}$.
 
 A line $L \subset G$ that lifts to $\stallings(f)$ but is not contained in $K_0$ either decomposes as the concatenation of a ray  in $K_0$ and an eigenray $R_{E'}$ or decomposes as the concatenation of a finite path in $K_0$ and a pair of eigenrays $R_{E'}$ and $R_{E''}$.  In the former case, each $E  \in \E_f$ crossed by $L$ satisfies $E \le  E'$ and in the latter case each $E\in \E_f$ satisfies $E \le  E'$ or  $E \le  E''$.    It follows that every edge $E \in \E_f$ crossed by $ \bigcup_{q=1}^{k-1} L_q$ satisfies $E <_T D_k$.     Since $K_0$ and $L_1,\ldots,L_{k-1}$ fill $\F_{k-1}$, we conclude that $D_k$ is not contained in  $G_{i_{k-1}}$.   Since $L_k$ lifts to $\stallings(f)$ and $\partial_+ L_k =  \partial R_{D_k}$, it follows that $R_{D_k}$ is   a terminal ray of $L_k$.  In particular,   $L_k$ crosses $D_k$.  Lemma~\ref{lines generate} implies that  $L_k  \subset G_{i_k}$  and we are done.
\endproof

\begin{lemma} \label{l:types}
Let $\fe=(\F^-\sqsubset\F^+)$ be special for $\phi$.
\begin{enumerate}
\item
The types \sHH, \sLH, or \sH\ of $\fe$ as in Lemma~\ref{lemma:three extension types} are mutually exclusive  and  independent of the special chain $\vec\F(\phi,<_T)$ containing $\fe$ and the choice of \ct\ $\fG$ realizing $\vec\F(\phi,<_T)$. 
\item
Suppose that $\fe$ appears as consecutive elements $\F_{k-1}\sqsubset\F_k$ in $\vec\F(\phi,<_T)$, which is realized by the \ct\ $\fG$. Using terminology as in Lemma~\ref{lemma:three extension types}, say that $\fe$ is respectively {\em \contractible, \cyclic, or \llarge} depending on whether the component of the eigengraph $\Gamma_{f|\F_k}$ containing the eigenray $R_{D_k}$ is contractible, has infinite cyclic fundamental group, or has fundamental group with rank at least two. The types \contractible, \cyclic, or \llarge\ of $\fe$  are mutually exclusive and independent of the choices of $\vec\F(\phi,<_T)$ and $f$.
\end{enumerate}
\end{lemma}

\proof 
(1): Suppose $\fe=(\F_{k-1}\sqsubset \F_k)$ in $\vec\F(\phi,<_T)$. The difference between the cardinality of $\cR(\phi | \F_{k})$ and the cardinality of $\cR(\phi | \F_{k-1})$ is $2$ in the [\sHH] case and  $1$ in the [\sLH] and [\sH] cases.  In case [\sLH],   either the number of axes for $\phi | \F_k$ is strictly larger than the number of axes for $\phi | \F_{k-1}$ or there is a common axis of $\phi | \F_{k}$ and $\phi | \F_{k-1}$ whose multiplicity in the former is strictly larger than in the latter.  Neither of these happens in case [\sH].
 
(2): Here is an invariant description. Let $r\in\Delta:=\cR(\phi | \F_{k})\setminus\cR(\phi | \F_{k-1})$, for example we could take $r$ to be determined by $R_{D_k}$. Either $\Delta=\{r\}$ or $\Delta=\{r, s\}$ and there is a $\phi|\F_k$-fixed line $L$ whose ends represent $r$ and $s$. Let $\ti r$ be a lift of $r$ to $\partial\f$ and let $\ti L= [\ti r, \ti s]$ be a lift of $L$ if $\Delta=\{r,s\}$. By definition, $\fe$ is  \contractible, \cyclic, or \llarge\  iff $\Fix(\Phi_{\ti r})$ is trivial, infinite cyclic, or of rank at least two where $\Phi_{\ti r}$ is the unique representative of $\phi$ fixing $\ti r$. We are done by noting that $\Phi_{\ti r}=\Phi_{\ti s}$ if $\Delta=\{r,s\}$.
 \endproof
 
\begin{excont*}\label{ex.i}
If we extend the partial order $r_c< r_q$ on $\cR(\phi)$ to the total order $r_c  \dle r_d \dle r_e \dle r_q$ we get the special chain $\frak c$ represented by the sequence of graphs in Figure~\ref{f:vec F}. See the notation in the {examples on pages~\pageref{e:main example} and \pageref{ex.a}.}
\begin{figure}[h!]
\centering
\includegraphics[width=.5\textwidth]{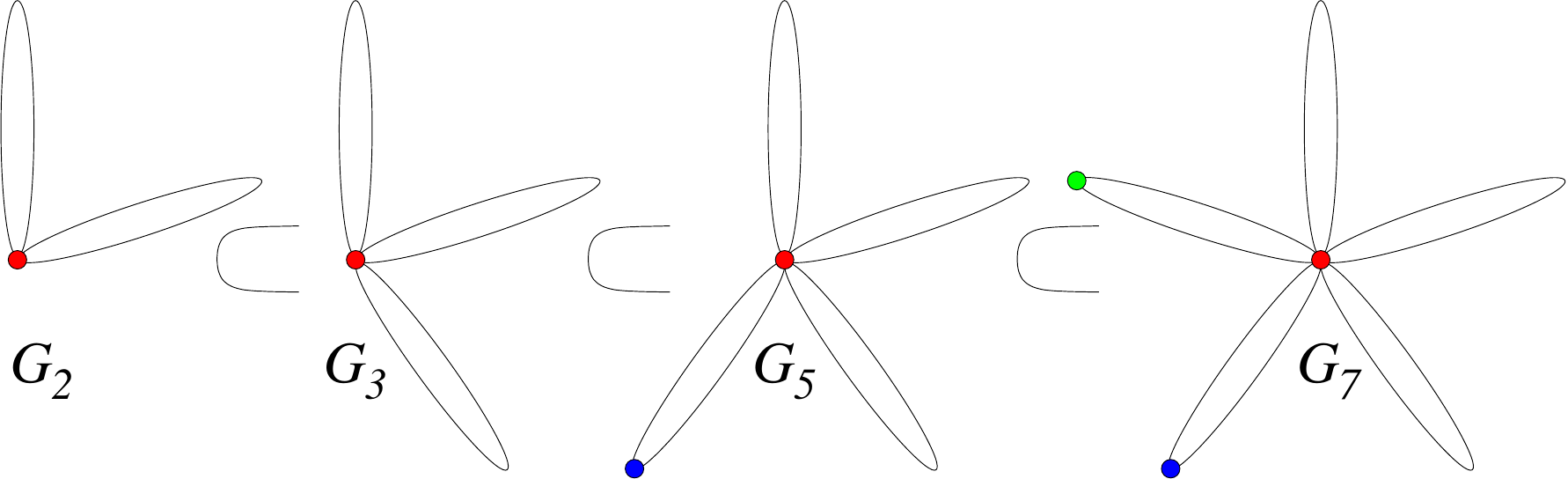}
\caption{$\frak c=\{[G_2]\}\sqsubset \{[G_3]\}\sqsubset \{[G_5]\}\sqsubset \{[G_7]\}$}
\label{f:vec F}
\end{figure}
\end{excont*}
 
 \begin{ex}\label{e:fixed edge}
 Consider the \ct\ $\fG$ given as follows: start with a rose with edges $a$ and $b$. Define $f(a)=a$ and $f(b)=ba$. Add a new vertex $v$ with adjacent edges $c, d$ and define $f(c)=cb$ and $f(d)=db^2$. Add another new vertex $v'$ with adjacent edges $c'$ and $d'$ with $f(c')=c'b^3$ and $f(d')=d'b^4$ finally add an $f$-fixed edge $e$ with endpoints $v$ and $v'$. The $\phi$-fixed free factor system $\F$ represented by the complement of $e$ in $G$ is not in $\lat0(\phi)$. Indeed, $\F\sqsubset\{\f\}$ is 1-edge, but not of type \sH, \sHH, or \sLH, contradicting Lemma~\ref{lemma:three extension types}.
\end{ex}

\begin{ex}
Suppose $\fG$ is a \ct\ containing a circle $C$ with only one vertex $x$ and such that $x$ is  the initial endpoint of an \sH-edge, the terminal  endpoint of a linear edge in an $\sLH$ extension  (so that $C$ is an axis), and there are no other edges containing $x$. Then $\Gamma(f)$ has no   components of rank at least two containing an axis corresponding to $C$.
\end{ex}

\begin{lemma}\label{l:special is natural}
Referring to Notation~\ref{notn:ffs}, suppose $\F, [F], \frak c$, and $\fe$ are special for $\phi$. If $\theta\in\Out(\f)$ then:
\begin{itemize}
\item
$\theta(\F), \theta([F]), \theta(\frak c)$, and $\theta(\fe)$ are special for $\phi^\theta$; and
\item
the types \sH, \sHH, or \sLH\ and the types \contractible, \cyclic, or \llarge\ of $\fe$ and $\theta(\fe)$ are the same.  
\end{itemize}
\end{lemma}

\begin{proof}
This is an immediate consequence of the fact (Lemma~\ref{weaker po conjugacy}) that conjugation preserves partial orders and the invariant description of types given in the proof of Lemma~\ref{l:types}.
\end{proof}

\begin{definition}[added lines]\label{d:added lines}
Suppose that $\frak c$ is a special chain for $\phi$ that is realized by $\fG$ and that 
$\fe=(\F^-\sqsubset\F^+) \in \frak c$.  Then $\cR(\phi| \F^+) \setminus \cR(\phi | \F^-)$ contains two elements if $\fe$ has type $\sHH$ and one element otherwise.  These elements are said to be {\it new} with respect to $\fe$. Similarly 
$\Gamma({f|\F^+})$ carries more lines than $\Gamma({f|{\F^-}})$. The set of {\it added lines with respect to $\fe$}, denoted $\LW_{\fe}(\phi)$, is a $\phi$-invariant subset of these lines. In case $\fe$ is \contractible, $\LW_{\fe}(\phi)$ consists of all lines $L$ in $\Gamma({f|{\F^+}})$ with $\partial_+L$ new. If $\fe$ is not \contractible\ then we also require that $\partial_-L$ is not in $\cR(\phi)$. $\LW_{\fe}(\phi)$ has an equivalent invariant description as follows. Set $\Phi:=\Phi_{\ti r^+}|F^+$ where $r^+$ is new, $\ti r^+\in\partial F^+\subset\partial\f$, and $[F^+]$ is the component of $\F^+$ carrying $r^+$. $\LW_{\fe}(\phi)$:  $=[\partial\Fix(\Phi),\ti r^+]$ if $\Fix(\Phi)$ is non-trivial; else  $=[\FixN(\Phi)\setminus \{\ti r^+\}), \ti r^+]$ if there is only one new eigenray; else = the set consisting of the two lines with lifts with endpoints in $\FixN(\Phi)$. This invariant description shows that $\LW_\fe(\phi)$ is independent of the special chain $\fe\in\frak c$, which is why $\fc$ does not appear in the notation.
\end{definition}

\begin{excont*}\label{ex.j} 
Referring to Figures~{\ref{f:G},} \ref{f:eigengraph} and \ref{f:vec F},  if $\fe_1:=(\{[G_2]\}\sqsubset \{[G_3]\})$ then $\LW_{\fe_1}(\phi)$ consists of the infinitely many lines 
$L$ in the third listed component of $\Gamma(f)$ in Figure~\ref{f:eigengraph} that cross the oriented edge $c$ exactly one and $c^{-1}$ not at all. If $\frak e_2:=(\{[G_3]\}\sqsubset \{[G_5]\})$ then $\LW_{\fe_2}(\phi)$ consists of two lines; they are represented by $(R_d)^{-1}R_e$ and its inverse. If $\fe_3:=(\{[G_5]\}\sqsubset \{[G_7]\})$ then $\LW_{\fe_3}(\phi)$ consists of two lines; they are represented by $a^\infty p^{-1}R_q$ and $a^{-\infty} p^{-1}R_q$. 
\end{excont*}

\begin{lemma} \label{LW is natural} For $\theta \in \Out(F_n)$, $\theta\big(\LW_{\fe}(\phi)\big) = \LW_{ \theta(\fe)}(\phi^\theta)$.
\end{lemma}

\begin{proof}
By Lemma~\ref{l:special is natural}, $\theta(\fe)\in\theta(\fc)$ are special for $\phi^\theta$. The set of new elements of $\cR(\phi)$ with respect to $\fe=(\F^+\sqsubset\F^-)$ has the invariant description $\cR(\phi|\F^+)\setminus\cR(\phi|\F^-)$. In particular, $\theta$ takes the new elements with respect to $\phi$ to those with respect to $\phi^\theta$; see Lemma~\ref{first theta}. The equation in the lemma then follows from the invariant definition of added lines in Definition~\ref{d:added lines} and the naturality results of Lemma~\ref{first theta}.
\end{proof}

In the next lemma, we record some consequences of Lemmas~\ref{weaker po conjugacy}, \ref{lemma:one edge extension}, and \ref{lines generate}.   
\begin{lemma}\label{l:F0 is natural}
\begin{enumerate}

\item

A conjugacy class grows at most linearly under iteration by $\phi$ if and only if it is carried by $\F_0(\phi)$.
\item
$\F_0(\phi)=\F(\Fix(\phi))$, i.e.\ $\F_0(\phi)$ is the smallest \ffs\ carrying $\Fix(\phi)$.
\end{enumerate}
\end{lemma}

\begin{proof}

(1): By definition $\F_0(\phi)$ carries all conjugacy classes that grow at most linearly. Conversely, by Lemma~\ref{lemma:one edge extension}, $\F_0(\phi)$ is represented a graph $K_0$ consisting of linear and fixed edges. Hence every conjugacy class carried by $\F_0(\phi)$ grows at most linearly.

(2):  By (1), $\F(\Fix(\phi))\sqsubset\F_0(\phi)$. Suppose $\sqsubset$ is proper. By \cite[Theorem~1.1]{fh:CTconjugacy}, there is a \ct\ $\fG$ realizing $\phi$ with $f$-invariant core subgraphs $G_k\subsetneq G_l$ representing these two \ffs s and such that $f|G_l$ is a \ct. By Lemma~\ref{lemma:one edge extension}, every edge of $G_l$ is fixed or linear.  Let $E$ be an edge of $G_l\setminus G_k$. There is a Nielsen circuit $\rho$ in $G_l$ containing $E$. (Indeed, by the construction of eigengraphs in Section~\ref{sec:stallings}:
\begin{itemize}
\item
every edge of a \ct\ is the label of some edge in its eigengraph;
\item
the eigengraph of a linear growth \ct\ is a compact core graph; and
\item
every circuit in an eigengraph is Nielsen.
\end{itemize}
The existence of $\rho$ now follows from the defining property of a core graph that there is a circuit through every edge.)
The fixed  conjugacy class represented by $\rho$ is not in $\F\big(\Fix(\phi)\big)$, contradiction.
\end{proof}

\subsection{The lattice of special free factor systems}

This section is not needed for the rest of the paper and so could be skipped by the reader. Recall (second paragraph of Notation~\ref{notn:ffs}) that $\fL(\phi)$ denotes the set of special free factor systems for $\phi$.   {The main results, Lemmas~\ref{l:random special} and \ref{l:natural lattice}, are that $(\lat0(\phi), \sqsubset)$ is a lattice that is natural with respect to $\Out(\f)$ in the sense that, for $\theta\in\Out(\f)$, $$\theta(\lat0(\phi),\sqsubset)=(\lat0(\phi^\theta),\sqsubset)$$
In this section, $\fG$ will always denote a \ct\ for $\phi$. We will conflate an element of $\cR(\phi)$ and its image in $\E_f$ under the bijection $\cR(\phi)\leftrightarrow \E_f$; see Lemma~\ref{identifying Fix+}. A subset $S$ of $\cR(\phi)$ is {\it admissible} if it satisfies $$(q\in S)\wedge (r<q)\implies r\in S$$
If $S\subset\cR(\phi)$ then mimicking Notation~\ref{lemma:one edge extension} we let $K(S)$ denote the union of $K_0$ and the edges in $S$. Recall $K_0$ is the union of the fixed and linear edges of $G$.

\begin{lemma}\label{l:admissible}
The following are equivalent.
\begin{enumerate}
\item
$\F$ is special for $\phi$.
\item
$\F=\F(K(S),G)$ for some admissible $S\subset\cR(\phi)$.
\item
$\F=\F(H,G)$ for some $f$-invariant $H\subset G$ containing $K_0$.
\end{enumerate}
\end{lemma}

\begin{proof}
(1)$\implies$(2): By definition, a free factor system $\F$ is special iff there is a total order $<_T$  extending $ <$ and an initial interval $[r_1,\dots, r_k]$ of $(\cR,<_T)$ such that $\F=\F\big(K(\{r_1,\dots,r_k\}),G\big)$. Since an initial interval is admissible, we may take $S=\{r_1,\dots,r_k\}$.

(2)$\implies$(3): We may take $H=K(S)$.

(3)$\implies$(2): Let $S$ be the set of edges in $H$ that are not in $K_0$. It is enough to show that $S$ is admissible. Let $q\in S$ and let $r\in\cR(\phi)$ satisfy $r<q$. By definition of $<$, there is $k>1$ so that the edge $r$ or its inverse is a term in the complete splitting of $f^k_\#(q)$. Since the edge $q$ is in $H$ and $H$ is $f$-invariant, the edge $r$ is also in $H$.

(2)$\implies$ (1): We claim that if $S$ is admissible then there is an extension $<_T$ of $<$ so that $S$ is an initial segment of $(\cR(\phi),<_T)$. Indeed, start with any total order $<_T$ extending $<$ and iteratively interchange $r$ and $s$ if $r<_T s$ are consecutive, $s\in S$, and $r\notin S$. For such a $<_T$, $S$ represents an element of $\vec \F(\phi,<_T)$.

\end{proof}

\begin{lemma}\label{l:random special}
Let $\F$ be a special \ffs\ for $\phi$.
\begin{enumerate}
\item
The set of admissible subsets of $\cR(\phi)$ is a sublattice of $2^{\cR(\phi)}$.
\item
There is a minimal admissible $S\subset \cR(\phi)$ such that $\F=\F\big(K(S),G\big)$. We say that such an admissible $S$ is {\em efficient for $\F$}. In fact, if $S'$ is admissible then the set of edges of $\core\big(K(S')\big)$ not in $K_0$ is efficient for $\F\big(K(S'),G\big)$.
\item
$S=\bigcap\{S'\mid \F=\F(K(S'),G)\}$ is efficient for $\F$.
\item
$S$ is efficient for $\F$ iff $\F= \F\big(K(S),G\big)$ and through every edge representing an element of $S$ there is a circuit in $K(S)$.
\item
If $S_1$ and $S_2$ are efficient admissible and $\F(K(S_1),G)\sqsubset\F(K(S_2), G)$ then $S_1\subset S_2$.
\item
$(\lat0(\phi),\sqsubset)$ is a lattice.
\item
Every maximal chain in $\lat0(\phi)$ is special.
\item
Every minimal pair $\fe=(\F^-\sqsubset\F^+)$ in $\lat0(\phi)$ (i.e.\ if $\F$ special and $\F^-\sqsubset\F\sqsubset\F^+$ then $\F^-=\F$ or $\F=\F^+$) is special.
\end{enumerate}
\end{lemma}

\begin{proof}
(1) follows directly from the definition of admissible.

(2): Suppose $S_1, S_2$ are admissible and $K(S_1)$ and $K(S_2)$ each represent $\F$. Hence $C:=\core\big(K(S_1)\big)=\core\big(K(S_2)\big)$ and, since $G$ is a \ct, $C$ is $f$-invariant. (Indeed, by \cite[Lemma~4.21]{fh:recognition}, the removal of an edge with a valence one vertex from an $f$-invariant subgraph results in an $f$-invariant subgraph.) It follows that $K_0\cup C$ is the minimal $f$-invariant subgraph of $G$ representing $\F$ and containing $K_0$; see Lemma~\ref{l:admissible}. Hence $S$ is the set of edges of $K_0\cup C$ not in $K_0$.

(3), (4), and (5) follow easily from (2).

(6): Suppose $S_1$ and $S_2$ are efficient. Then using (4), $S_1\cup S_2$ is efficient. It follows that $\F\big(K(S_1\cup S_2), G\big)$ is the smallest (with respect to $\sqsubset$) special free factor system for $\phi$ containing $\F_1:=\F\big(K(S_1),G\big)$ and $\F_2:=\F\big(K(S_2),G\big)$. Suppose $S$ is efficient and $\F\big(K(S),G\big)\sqsubset F_i$, $i=1,2$. By (5), $S\subset S_1\cap S_2$. Since $K(S_1\cap S_2) = K(S_1)\cap K(S_2)$, the largest special \ffs\ $\F$ for $\phi$ in each of $\F_1$ and $\F_2$ is represented by $K(S_1\cap S_2)$, i.e.\ by $K(S)$ where $S$ is efficient for $\F\big(K(S_1\cap S_2), G\big)$.

(7): Let $\frak c$ be represented by $K(S_1)\subset \dots \subset K(S_N)$ with each $S_i$ efficient. By (5), $S_i\subset S_{i+1}$. An argument similar to that in the proof of Lemma~\ref{l:admissible}\big((2)$\Rightarrow$(1)\big) shows that there is $<_T$ extending $<$ so that each $S_i$ is an initial interval in $\big(\cR(\phi), <_T\big)$. Hence $\frak c$ is special.

(8) follows from (7) by enlarging $\fe$ to a maximal chain.
\end{proof}

\begin{remark}
$\lat0(\f)$ is not a sublattice of the lattice of all $\phi$-invariant free factor systems. For example, reconsider Example~\ref{e:fixed edge}. $S=\{c,d\}$ and $S'=\{c',d'\}$ are efficient. If $\F=\F\big(K(S),G\big)$ and  $\F'=\F\big(K(S'),G\big)$ then the smallest $\phi$-invariant free factor system containing $\F$ and $\F'$ is represented by the complement of the fixed edge $e$ whereas the smallest element of $\lat0(\phi)$ containing $\F$ and $\F'$ is $\f$.

In the proof of Lemma~\ref{l:random special} we noted that the union of efficient sets is efficient. The intersection need not be efficient. For example, suppose highest order edges $a$, $b$, and $c$ share an initial vertex of valence three. Consider the complement $S$ of $a$ and the complement $S'$ of $b$. The edge $c$ is in $S\cap S'$ and has initial vertex of valence one in $K(S\cap S')$.
\end{remark}

\begin{lemma}\label{l:natural lattice}
\begin{itemize}
\item
If $\fe=(\F^-\sqsubset\F^+)$ is minimal in $\lat0(\phi)$ then $\fe$ has a well-defined type $\sH$, $\sHH$, or $\sLH$ and a well-defined type \contractible, \cyclic, or \llarge. 
\item
For  $\theta\in\Out(\f)$, the map $\F\mapsto \theta(\F)$ induces a lattice isomorphism $$(\fL(\phi), \sqsubset)\to(\fL(\phi^\theta),\sqsubset)$$ that preserves the above types.
\end{itemize}
\end{lemma}

\begin{proof}
This follows Lemmas~\ref{l:random special} and \ref{l:special is natural}.
\end{proof}

\section{More on conjugacy pairs}  \label{s:conjugacy pairs}
Recall that conjugacy pairs were introduced in Definition~\ref{d:conjugacy pairs}. In this section we define some conjugacy pairs that will be used to define invariants of elements of $\upgn$ and describe their properties.
\subsection{$[\partial H,\partial K]$}

We will want to compare conjugacy pairs of subgroups $[H,K]$ with the set of lines $[\partial H,\partial K]$; see Example~\ref{e:conjugacy pairs}. For this we will use the next lemma which is a corollary of \cite[Lemma~3.9]{KS:Greenberg}.

\begin{lemma}\label{l:stab of boundary}
Suppose that $H<\f$ is finitely generated. Then the stabilizer $G$ in $\f$ of $\partial H\subset \partial \f$ is the  maximal $K<\f$ in which $H$ has finite index.
\end{lemma}

\begin{cor} \label{c:rootfree}
Suppose that finitely generated $H<\f$ is root-closed, i.e.\ $a^k\in H$, $k\not=0$ implies $a\in H$. Then $H$ is the stabilizer in $\f$ of $\partial H$.
\end{cor}

\begin{proof}
If $H<G$ has finite index and $H\not= G$, then $H$ is not root-closed.
\end{proof}

\begin{cor}\label{c:ff and fix good} 
\begin{enumerate}
\item
If $H<\f$ is a free factor, then $H$ is the stabilizer of $\partial H$.
\item
If $H=\Fix(\Phi)$ for $\Phi\in\Aut(\f)$  then $H$ is the stabilizer of $\partial H$.
\item
If $a\in\f$ is root-free then $A=\langle a\rangle$ is the stabilizer of the two-point set $\partial A$.
\end{enumerate}
\end{cor}

\begin{proof}
Free factors and the group generated by a root-free element are clearly root-closed. For (2), $\Phi(a^k)=a^k$ for $k\not=0$ implies that $\Phi(a)$ is a $k$th root of $a^k$ and so equals $a$. 
\end{proof}

\begin{remark}
\label{r:our setting}
Corollary~\ref{c:ff and fix good}, which contains the only cases that we need in this paper, does not require the  generality of Lemma~\ref{l:stab of boundary}. Item (3) is elementary.  Items (1) and (2) follow from (3)  and \begin{itemize}
\item
For $H, K<\f$ finitely generated, $\partial H\cap\partial K=\partial(H\cap K)$ (See \cite[Theorem~12.2(9)]{kb:boundaries} in the setting of hyperbolic groups or, for the case at hand, \cite[Fact~1.2]{handelMosher:subgroups}).
\item
If $H$ is a non-trivial free factor then $H\cap H^g=\{1\}$ unless $g\in H$.
\item
$\Fix(\Phi)\cap\Fix(\Phi)^g$ is cyclic unless $g\in\Fix(\Phi)$ (Lemma~\ref{malnormal}\pref{i:cap}).
\end{itemize} 
\end{remark}

\begin{cor}\label{c:lines=conj pairs}
Suppose that $H, K<\f$ are finitely generated and root-closed.  Then $[H,K]$ determines $[\partial H, \partial K]$ and vice versa.
\end{cor}

\begin{proof}  
Suppose that $H', K'<\f$ are finitely generated and root-closed.

If $[\partial H,\partial K]=[\partial H',\partial K']$ then there is $x\in\f$ such that $(x\partial H,x\partial K)=(\partial H^x, \partial K^x)=(\partial H', \partial K')$. Hence $$(H^x,K^x)=(\Stab(\partial H^x),\Stab(\partial K^x))=(\Stab(\partial H'),\Stab(\partial K'))=(H',K')$$ So, $[H,K]=[H',K']$.

Conversely, if $[H,K]=[H',K']$ then there is $x\in\f$ such that $(H^x,K^x)=(H',K')$. Hence $(\partial H^x,\partial K^x)=(x\partial H, x\partial K)=(\partial H',\partial K')$, and so $[\partial H,\partial K]=[\partial H',\partial K']$.
\end{proof}

\begin{remark} \label{rem:lines=conj pairs}
If we are in the setting of Corollary~\ref{c:lines=conj pairs} and $\partial H\cap\partial K=\emptyset$ we will sometimes abuse notation and think of $[H,K]$ as the set of lines $[\partial H,\partial K]$ and vice versa.
\end{remark}

\subsection{Some Stallings graph algorithms} \label{sec:NewNewfolding}

In this section we assume that $G$ is a marked graph with marking $m :(R_n,\ast)  \to (G,b)$ where $*$ is the unique vertex of the rose $R_n$ and $b = m(*)$ is the basepoint for $G$.  There is an induced identification of $\f$ with $\pi_1(G,b)$.

For each finitely generated subgroup $H < F_n$, Stallings   \cite[5.4]{js:folding} constructs a finite graph $\sbh$ with basepoint $b_H$ and an immersion $p_H :(\sbh,b_H) \to (G,b)$ such that the image of the injection $\pi_1(\sbh, b_H) \to \pi_1(G,b)$  induced by $p_H$ equals $H$. The base point $b_H$ may have valence one but all other vertices of $\sbh$ have valence at least two.  We equip $\sbh$ with the CW-structure whose vertex set is the preimage  of the vertex set of $G$.   The resulting edges of $\sbh$, sometimes called {\it edgelets},  are labeled by their image edges in $G$.  The core of $\sbh$ is denoted $\sh$. The minimal edgelet-path from $b_H$ to $\sh$  is denoted $\beta_H$. The terminal endpoint of $\beta_H$ is denoted $c_H \in \sh$. 

For finitely generated subgroups $K,H < \f$,  we let  $\im(K,H)$ be the set of immersions $J :\sk \to \sh$ that maps edgelets to edgelets and preserves labels; we say that $J$ {\it preserves labels}.  
 We do not distinguish between elements of $\im(K,H)$ that induce the same map on the set of edgelets.  Thus $\im(K,H)$     is finite and can be computed by inspection.   An {\it equivalence} is an element of $\im(K,H)$ that  is a homeomorphism. Note that  elements of $ \im(K,H)$ that agree on a vertex of $K$ are equal.    
 
\begin{lemma}  \label{l:inclusion immersion}If $K< H$ are finitely generated subgroups of $\f$ then there is a (necessarily unique) label-preserving immersion $J_{K,H} : (\sbk,  b_K) \to (\sbh,b_H)$.
\end{lemma}
 
 \proof  We recall Stallings'  construction of $\sbk$ from  \cite[5.4]{js:folding}.  Choose closed paths  $\rho_1,\ldots,\rho_m \subset G$ based at $b$ that represent generators of $K < \pi_1(G,b)$.  Define $\Gamma_1$ to be a rose of rank $m$ with unique vertex $b'$ and define $p_1:(\Gamma_1,b_1') \to (G,b)$ to be an immersion on edges, mapping the $i^{th}$ edge to $\rho_i$. Subdivide $\Gamma_1$ into edgelets labeled by edges of $G$ to obtain $p_2:(\Gamma_2,b_2') \to (G,b)$. The map $p_2$ factors into a sequence of edgelet  folds $(\Gamma_2,b_2') \to (\Gamma_3,b_3') \to \ldots \to (\Gamma_k,b'_k)$ followed by an immersion $p_k : ( \Gamma_k, b'_k) \to (G,b)$.   Define  $ (\Sigma_b(K), b_K) = (\Gamma_k, b'_k)$ and $p_H = p_k$. 
 
   Since $K <H$, each $\rho_i$ lifts to a closed edgelet-path $\rho_i'$ in $\sbh$ based at $b_H$.  Since the $i^{th}$ edge of $\Gamma_2$ and   $\rho_i'$ agree as labeled edgelet-paths,   there is an induced label-preserving map $q_2: (\Gamma_2, b') \to   (\Sigma_b(H),b_H) $ satisfying $p_2 = p_H q_2$.  Since $p_H$ is an immersion, the edgelets that are identified  by the  folding maps $\Gamma_2 \to \Gamma_3 \to \ldots \to \Gamma_k$ are also identified by $q_2$.  Thus, there exists a map $J_{K,H} :   (\Gamma_k, b'_k)\to (\Sigma_b(H),b_H)$ such that $p_K = p_H q_k$.  Since $p_H$ and $p_K$ are immersions, the same is true for $J_{K,H}$.                                                                                                                                                     
\endproof
 
Note that if $a \in \f$ and $K^a < H$ then $K^{ha}< H$ for all $h \in H$.  Let $\rc(K,H)$ be the set of right cosets of $H$ in $\f$ such that $K^a < H$ for some (each) $a$ representing that coset.

\begin{lemma} \label{l:stallings bijection} There is an algorithm with output a    bijection  $\im(K,H) \longleftrightarrow \rc(K,H)$.  In particular, there is an algorithm that produces coset representatives for the elements of $\rc(K,H)$.
\end{lemma}

\proof  ($\longrightarrow$) We associate a coset $Ha \in \rc(K,H)$ to  $J \in \im(K,H)$  as follows. Choose a path $\xi \subset \sh$ from $c_H$ to $J(c_K)$ and note that $p_K(c_K) = p_H(J(c_K))$.  Let  $a \in \pi_1(G,b)$ be    represented by the closed path $[p_H( \beta_H \xi) p_K(\bar \beta_K)] \subset G$  where $[\cdot]$ indicates tightening.  Each $x \in K$ is represented in $\pi_1(G,b)$ by $p_K(\beta_K\gamma  \bar \beta_K)$  for some closed path $\gamma \subset \sk$ based at $c_K$.  It follows that $x^a$ is represented in $\pi_1(G,b)$ by 
$[ p_H( \beta_H \xi)p_K(\gamma) p_H( \bar\xi \bar \beta_H )] = [p_H( \beta_H \xi)p_H(J(\gamma)) p_H( \bar\xi \bar \beta_H )] = p_H(\beta_H[ \xi J(\gamma) \bar\xi]\bar \beta_H)$ which represents an element in $H$.  This proves that $K^a < H$.   If $\xi$ is replaced by another path $\xi'$ connecting $c_H$ to $J(c_K)$ then $a$ is replaced by $ha$ where $h \in H$ is represented by $[p_H(\beta_H \xi' \bar\xi \bar \beta_H)]$. Thus, $Ha$ is independent of the choice of $\xi$.  If $J$ is replaced with $J' \ne J$ and if $\eta\subset \Sigma(H)$ is a path connecting $J(c_K)$ to $J'(c_K)$,  then $\xi$ is replaced with $\xi \eta$ and  $a$ is replaced with $a' = da$  where $d$ is represented in $\pi_1(G,b)$ by $[p_H(\beta_H \xi  \eta\bar \xi \bar \beta_H)]$.  Since $\eta$ is not a closed path,  $d$ does not lift into $\sh$ and $a'$ does not belong to the same right coset of $H$ as $a$.  This shows that  $J \mapsto Ha$ defines an injection from $\im(K,H)$ to $\rc(K,H)$.

\bigskip
\noindent ($\longleftarrow$) We begin the proof of surjectivity   by   constructing $\Sigma_b(K^a)$ from $\sbk$. Represent $a$ in $\pi_1(G,b)$  by a closed edge-path $\alpha \subset G$  based at $b$ and let $\beta'$ be the edgelet path  labeled by the path in $G$ obtained by tightening $\alpha p_K(\beta_K)$. Define $\Sigma'$ from the disjoint union of $\beta'$ and a copy $(\Sigma'(K),c')$ of $(\Sigma(K),c_K)$   by identifying the terminal endpoint of $\beta'$ with $c'$.  The labeling on edgelets induces $p' : (\Sigma', b') \to (G,b)$ where $b'$ is the initial vertex of $\beta'$.   The image of the injection $\pi_1(\Sigma', b') \to \pi_1(G,b)$  induced by $p'$ equals $K^a$.  If $p'$ is an immersion then $(\Sigma(K^a), \beta_{K^a} ,c_{K^a}) = (\Sigma',\beta',  c')$. Otherwise, $\Sigma(K^a)$ is obtained from $\Sigma'$ by folding  a maximal initial edgelet-subpath of $\bar \beta'$ with an edgelet-subpath $\mu \subset \Sigma'(K) $ that begins at $c'$.     In this case, $b_{K^a}$ is the folded  image of $b'$  and $c_{K^a}$ is the terminal endpoint of $\mu$.   
    
    Continuing with the above notation, define the equivalence $J_{K,a} \in \im(K,K^a)$  to be the identifying homeomorphism from $(\sk,c_K)$ to $(\Sigma'(K),c')$.  Assuming that $K^a < H$, apply Lemma~\ref{l:inclusion immersion} and define $J_{a,K,H} = J_{K^a,H}|\Sigma(K^a)\circ J_{K,a}\in \im (K,H)$.  By construction, $[\alpha p_K(\beta_K)] \subset G$ lifts to a the path in $\sbh$ from $b_H$ to $J_{a,K,H}(c_K)$.  Writing this path as $\beta_H  \xi$, we have that $[p_H( \beta_H \xi) p_K(\bar \beta_K)] = \alpha$ and hence that (in the notation of the first paragraph of this proof)  $J_{a,K,H} \mapsto Ha$.  
    \endproof
    
    We will need the following well-known result.
    
    \begin{corollary} \label{cor:normalizer} If $H< \f$ is a finitely generated and $H^a < H$ for $a \in \f$ then $H^a = H$.
    \end{corollary}
    
    \proof The obvious induction argument shows that $H^{a^p} < H^{a^{p-1}} < \ldots < H^a < H$ for all $p \ge 1$.  Each $H{a^s}$, $s \ge 1$, is therefore an element  $\rc(H,H)$, which is finite by Lemma~\ref{l:stallings bijection}.  It follows that $H{a^s} = H{a^t}$ for some $s \ne t$ and hence that $a^p \in H$ for some $p \ge 1$.  Thus $H^{a^p} =  H$ which implies that  $H  <  H^a  < H$  and hence that $H = H^a$.
\endproof

The following three algorithms   are easy consequences of  Lemma~\ref{l:stallings bijection}.
   
\begin{lemma} \label{l:folding conjugacy}There is an algorithm that decides if a given pair $H$ and $K$ of finitely generated subgroups of $F_n$ are conjugate and if so produces an  element $a \in F_n$ satisfying $K^a = H$.
\end{lemma}

\proof  We continue with notation from the proof of Lemma~\ref{l:stallings bijection}.  
   If $K^a  = H$ then $J_{K^a,H}$ is the identity and  hence $J_{a,K,H} \in \im(K,H)$ is an equivalence.  This shows that if   $\im(K,H)$ does not contain an equivalence then $K$ and $H$ are not conjugate.  If $\im(K,H)$ does  contain an equivalence $J$,  apply Lemma~\ref{l:stallings bijection} to $J$ and $J^{-1} \in \im(H,K)$ to produce $a,b \in \f$ such that $K^a < H$ and $H^b <K$.  From $H^{ab} < H$ and   Corollary~\ref{cor:normalizer} it follows that  that $H^{ab} = H$  and hence that  $H = (H^b)^a  < K^a <H$ which implies that $K^a = H$.
\endproof

\begin{lemma}  \label{l:folding normalizer} The normalizer $N(H)$ of a finitely generated subgroup $H < F_n$ is finitely generated.  We have  an algorithm that produces coset representatives $\{a_1,\ldots,a_p\}$ of $H$ in $N(H)$.
\end{lemma}

\proof    Corollary~\ref{cor:normalizer} implies that $N(H)/H = \rc(H,H)$.     Lemma~\ref{l:stallings bijection} therefore completes the proof.
\endproof
  
\begin{lemma} \label{l:folding supergroup}  If  $K < H< F_n$ are finitely generated subgroups then the set of subgroups of $H$ that are $F_n$-conjugate to $K$ determine finitely many $H$-conjugacy classes.  There is an algorithm that produces representatives $K_1,\ldots,K_p$ of these $H$-conjugacy classes. \end{lemma}

\proof If $K^{a_1}, K^{a_2} < H$ then  $K^{a_1}$ and $ K^{a_2}$ determine the same $H$-conjugacy class if and only if $ha_1 = a_2$ for some $h \in H$.  Lemma~\ref{l:stallings bijection}, which produces  representatives of the elements of $\rc(K,H)$, therefore completes the proof. \endproof

\subsection{Good conjugacy pairs}
In addition to conjugacy classes of finitely generated subgroups of $\f$, our adaptation of Gersten's algorithm will also take conjugacy pairs as input; see Notation~\ref{ultimate atoms}. If $H_1$ and $H_2$ are subgroups of $H$ and the natural map $H_1*H_2\to H$ is an isomorphism then we say that $H$ is the {\it internal free product} of $H_1$ and $H_2$. If $A<B<\f$ then $[A]_B$ denotes the conjugacy class of $A$ in $B$. If $B=\f$ then we sometimes suppress the subscript.

\begin{definition}[good conjugacy pairs]\label{d:good}
For $H_1, H_2<\f$, the conjugacy pair $[H_1,H_2]_\f$ is {\it good} if $\langle H_1, H_2\rangle$ is the internal free product of $H_1$ and $H_2$.
\end{definition}

The next lemma collects some facts about good pairs.

\begin{lemma} \label{l:good facts}
Let $H_1, H_2<\f$ be finitely generated. 
\begin{enumerate}
\item\label{i:good iff}
$[H_1,H_2]$ is good iff $\rk(\langle H_1, H_2\rangle)=\rk(H_1)+\rk(H_2)$.
\item\label{i:good disjointness}
If $[H_1,H_2]$ is good then $\partial H_1$ and $\partial H_2$ are disjoint.
\end{enumerate}
\end{lemma}

\begin{proof}
The natural map $H_1*H_2\to\langle H_1, H_2\rangle$ is surjective. Since finitely generated free groups are Hopfian  (surjective endomorphisms are isomorphisms) \cite[Theorem~2.13]{mks:book}, the only if direction of (1) follows. The if direction of (1) is obvious.

(2) follows from $\partial H_1 \cap \partial H_2 = \partial (H_1 \cap H_2)$; cf.\ the first item in Remark~\ref{r:our setting}.
\end{proof}

 Our next goal is necessary and sufficient conditions for two good conjugacy pairs to be equal.    We begin with an important special case.
 
\begin{lemma} \label{l:pairs special case}
Suppose that $K_1, K_2, L_1, L_2<\f$ are finitely generated, that $[K_1, K_2]_\f$ and $[L_1,L_2]_\f$ are good conjugacy pairs and that  $\langle K_1,K_2\rangle =\langle L_1,L_2\rangle = \f $. Then  the following are equivalent.
 \begin{enumerate}
 \item $[K_1, K_2]_\f = [L_1,L_2]_\f$
 \item $[K_1]_\f=[L_1]_\f$ and $[K_2]_\f=[L_2]_\f$. 
 \end{enumerate}
 \end{lemma} 
 
 \proof  
 $\pref{i:=}\Rightarrow\pref{i:restrict}$: If $[K_1, K_2]_\f=[L_1,L_2]_\f$ then by definition there is $g\in\f$ such that $(K_1^g, K_2^g)=(L_1,L_2)$.

 $\pref{i:restrict}\Rightarrow\pref{i:=}$:   By hypothesis there are $g_i\in\f$ such that $K_i^{g_i}=L_i$. In particular, $\Delta :=(i_{g_1}|K_1)*(i_{g_2}|K_2)\in\Aut(\f)$ represents an element $\delta \in \Out(F_n)$. Let $r_i$ be the rank of $K_i$, let $A_i$ be  a rose with rank $r_i$ whose petals are labeled by a basis for $K_i$ and let $A$ be the rose of rank $n$ obtained from $A_1$ and $A_2$ by identifying their unique vertices $v_1$ and $v_2$ to a single vertex $v$.  Blow up  $v$ to an arc.   More precisely,   let $X$ be the graph obtained from the disjoint union of $A_1, A_2$  and a vertex $w$ by adding  oriented edges $E_1$ and $E_2$ connecting $w$ to   $v_1$ and $v_2$ respectively.  Denote the arc $\bar E_1 E_2 \subset X$ by $E$ and the subgraph $A_i \cup E_i \subset X$ by $X_i$.  Identify $\pi_1(X_i, w)$ with $\pi_1(A_i, v_i) = K_i$ via the   map $q_i : X_i  \to A_i$ that collapses $E_i$ to $v$.   Let  $q : X \to A$  be the map that collapses  $E$ to $v$.     If $\alpha_i \subset A$ is the closed  path   based at $v$ that represents $g_i$ then there is a unique closed path $\beta_i \subset X$ based at $w$ that satisfies $q_\#(\beta_i) = \alpha_i$.     The map $f : X \to X$ defined by $f|A_i = $ identity and  $f(E_i) = [\beta_i E_i]$ induces the automorphism $\Delta$ and so is a homotopy equivalence.  Homotop $f$ rel  $A_1\cup A_2$  to a map $f' : X \to X$ whose restriction to  $E$ is an immersion.  Then $f'$ is a topological representative of $\delta$ and  \cite[Corollary~3.2.2]{bfh:tits1} implies that $f'(E) = \bar \gamma_1E \gamma_2$ for some (necessarily closed) paths $\gamma_i  \subset A_i$.   If $k_i \in K_i$ is represented by the homotopy class of $\gamma_i$, then $f'$ induces the automorphism $\Delta' :=(i_{k_1}|K_1)*(i_{k_2}|K_2)$.    There exists $h \in \f$ such that $\Delta = i_h \Delta'$.   We have  $i_{g_i} = i_{hk_i}$  and hence  $i_{g_1 k_1^{-1}}  = i_{g_2 k_2^{-1}}$. Thus 
\begin{align*}
[L_1,L_2]_\f&=[K_1^{g_1},K_2^{g_2}]_\f=[K_1^{g_1}, (K_2^{g_1^{-1}g_2})^{g_1}]_\f\\
&=[K_1,K_2^{g_1^{-1}g_2}]_\f=[K_1,K_2^{k_1^{-1}k_2}]_\f=[K_1^{k_1},K_2^{k_2}]_\f\\&
=[K_1,K_2]_\f
\end{align*} 
\endproof

For each finitely generated   $L<\f$, we define a function $f_L$ as follows.  The domain of $f_L$ is the set of good conjugacy pairs $[K_1,K_2]$ with $K:=\langle K_1, K_2\rangle$ conjugate to $L$.   Any $g\in\f$ such that $K^g=L$ is well-defined up to the normalizer $N(L)$ of $L$ in $\f$. That is, if $L=K^{g}=K^{g'}$ then $g'g^{-1}\in N(L)$. Hence $ ([K_1^g]_L,[K_2^g]_L)$  is well-defined up to the diagonal action of $N(L)$ (equivalently $N(L)/L$) on the   set of pairs of conjugacy classes of subgroups of $L$.  We define $f_L([K_1,K_2])$ to be the  $N(L)/L$ orbit of  $([K_1^g]_L,[K_2^g]_L)$.   Note that if $K = L$ and $\xi_1,\ldots,\xi_r$ are coset representatives of $L$ in $N(L)$  then  $f_L([K_1,K_2]) = \{([K_1^{\xi_1}]_L,[K_2^{\xi_1}]_L),\ldots, ([K_1^{\xi_r}]_L, [K_2^{\xi_r}]_L)\}$. 
\begin{remark} \label{r:fL vs fL'}
Suppose $x\in \f$ and $L^x=L'$. Then ${\sf Domain}(f_L)= {\sf Domain}(f_{L'})$ and conjugation $i_x$ by $x$ induces a bijection (which we give the same name) $i_x:{\sf Codomain}(f_L)\to {\sf Codomain}(f_{L'})$ given by the $N(L)$-orbit of $([L_1]_L,[L_2]_L)$ maps to the $N(L')$-orbit of $([L_1^x]_{L'}, [L_2^x]_{L'})$. It is an easy check that $f_{L'}=i_x\circ f_{L}$.
\end{remark}

\begin{lemma}\label{l:conjugacy pair iff}  Suppose that
$K_1, K_2, L_1, L_2<\f$ are finitely generated and that $[K_1, K_2]_\f$ and $[L_1,L_2]_\f$ are good conjugacy pairs. Set $K:=\langle K_1,K_2\rangle$ and $L:=\langle L_1,L_2\rangle$. Then the following are equivalent:
\begin{enumerate}
\item\label{i:=}
$[K_1, K_2]_\f=[L_1,L_2]_\f$
\item\label{i:restrict}
There is $g\in\f$ such that $K^g=L$, $[K_1^g]_L=[L_1]_L$, and  $[K^g_2]_L=[L_2]_L$.
\item\label{i:fL}
$[K]=[L]$ and $f_L([K_1,K_2])=f_L([L_1,L_2])$
\item  \label{i:fH}
$[K]=[L]$ and, for some (any) $H<\f$ with $[H]=[L]=[K]$, $f_H([K_1,K_2])=f_H([L_1,L_2])$ 
 \end{enumerate}
\end{lemma}

\begin{proof}
$\pref{i:=}\Rightarrow\pref{i:restrict}$: If $[K_1, K_2]_\f=[L_1,L_2]_\f$ then by definition there is $g\in\f$ such that $(K_1^g, K_2^g)=(L_1,L_2)$.

 $\pref{i:restrict}\Rightarrow\pref{i:=}$: By Lemma~\ref{l:pairs special case} applied to $K_i^g$ and $L_i$ with $L$ playing the role of $\f$ we have $[K_1^g,K_2^g]_L=[L_1,L_2]_L$. In particular, $[K_1,K_2]_\f=[L_1,L_2]_\f$.

$\pref{i:restrict}\Rightarrow\pref{i:fL}$ is clear from the definition of $f_L$.

$\pref{i:fL}\Rightarrow\pref{i:restrict}$: Suppose $[K]=[L]$ and $f_L([K_1,K_2])=f_L([L_1,L_2])$.  By the former there is $g'\in \f$ such that $K^{g'}=L$ and by the latter there is $n\in N(L)$ such that $([K_1^{g'}]_L,[K_2^{g'}]_L)=([L_1^n]_L,[L_2^n]_L)$. Take $g=n^{-1}g'$.

$\pref{i:fL}\Leftrightarrow\pref{i:fH}$ follows directly from Remark~\ref{r:fL vs fL'}.
\end{proof}

\begin{cor}\label{c:equal pairs}
There is an algorithm with input two good conjugacy pairs $[K_1, K_2]$ and $[L_1,L_2]$ of finitely generated subgroups of $\f$ and output {\tt YES} or {\tt NO} depending on whether or not $[K_1,K_2]=[L_1,L_2]$.
\end{cor}

\begin{proof}  Apply Lemma~\ref{l:folding conjugacy} to decide if $K$ and $L$ are $F_n$-conjugate.  If not, then output {\tt NO}.  Otherwise, Lemma~\ref{l:folding conjugacy} gives   $x\in\f$ such that $K^x=L$ and we replace $K_1$ and $K_2$ by $K_1^x$ and $K_2^x$ so that now $K=L$.  Apply Lemma~\ref{l:folding normalizer} to produce coset representatives $\xi_1,\ldots, \xi_r$ of $L$ in $N(L)$.
According to Lemma~\ref{l:conjugacy pair iff}, $[K_1,K_2]=[L_1,L_2]$ iff $[K_1^{\xi_i}]_L = [L_1]_L$ and $[K_2^{\xi_i}]_L = [L_2]_L$ for some $1 \le i \le r$. This can be checked by applying Lemma~\ref{l:folding conjugacy} with $F_n$ replaced by $L$.
\end{proof}

The following lemma is used in Lemma~\ref{l:pairs algorithm} to determine which pairs of conjugacy classes correspond to good conjugacy pairs.
 
\begin{lemma}  \label{l:image}There is algorithm with input two  finitely generated subgroups $K_1,K_2 < F_n$ and output {\tt YES} or {\tt NO} depending on whether or not there exist $K'_i< F_n$ such that $[K'_i] = [K_i]$  and such that  $\f$ is the internal free product of $K'_1$ and $K'_2$.  If {\tt YES} then one such $K'_1$ and $K'_2$ are produced.
\end{lemma} 

\proof  Choose any finitely generated subgroups $A_i$ such $\rk(A_i)=\rk(K_i)$ and such that  $\f$ is the internal free product of $A_1$ and $A_2$.    $K'_1$ and $K'_2$  exist if and only if there is a $\theta\in\Out(\f)$ such that $\theta([K_i])=[A_i]$. The existence of such a $\theta$ can be checked using Gersten's generalization of Whitehead's theorem  \cite{sg:whitehead},\cite{bfh:gersten}, which appears as Theorem~\ref{t:original} in this paper.
 Additionally, the algorithm produces such a $\theta$ if one exists; we take  $K_i' = \Theta^{-1}(A_i)$ where $\Theta\in\theta$. 
\endproof

\begin{notn}\label{n:C}
 $\cC(\f)$ denotes the set of conjugacy classes of finitely generated subgroups of $\f$.
 \end{notn}

To aid in working with good conjugacy pairs, we relate them to ordered triples in $\cC(\f)$. Consider the following map from good conjugacy pairs to ordered triples in $\cC(\f)$
\begin{equation}\label{e:pairs}[H_1,H_2]\mapsto ([H_1], [H_2], [H])\end{equation} where $H:=\langle H_1, H_2\rangle$.

\begin{lemma}\label{l:pairs algorithm}
We have an algorithm with input a good conjugacy pair $[H_1,H_2]$ and output a finite enumeration of the fiber $F$ of the above map \pref{e:pairs} containing $[H_1,H_2]$.
\end{lemma}

\begin{proof}
Consider the map induced by $f_H$ from $F$ to $\{([K_1]_H,[K_2]_H)\}/N(H)$ where $K_i$ ranges over finitely generated subgroups of $H$ such that $H_i$ and $K_i$ are conjugate in $\f$. By Lemma~\ref{l:conjugacy pair iff}, this map is injective. So, it remains to produce an element of $F$ for each element of the image. The $[K_i]_H$'s  can be finitely enumerated by Lemma~\ref{l:folding supergroup}.   By  Lemma~\ref{l:image} we can decide if $([K_1]_H,[K_2]_H)$ represents an element of the image.  Applying Lemma~\ref{l:folding normalizer} and then Lemma~\ref{l:folding conjugacy} (with $\f$ replaced by $H$) we can decide if two pairs $([K_1]_H,[K_2]_H)$ and $([K'_1]_H,[K'_2]_H)$ are in the same $N(H)$-orbit and so remove redundancy from our list. 
\end{proof}

\begin{lemma} \label{l:fiber of 1}
We have an algorithm with input an ordered triple $([H_1], [H_2], [H])$ of elements of $\cC(\f)$ and output {\tt YES} or {\tt NO} depending on whether or not the fiber $F$ of the above map \pref{e:pairs} is empty. Further, if {\tt NO} the algorithm also outputs an element of $F$.
\end{lemma}

\proof
Our goal is to either find subgroups $K_i$ in the same $\f$ conjugacy class as $H_i$  such that $K_i < H$ and such that $H$ is  the internal free product of $K_1$ and $ K_2$ or to conclude that no such $K_i$ exist.

By Lemma~\ref{l:stallings bijection}, we can compute coset representatives for the elements of   $\rc(H_i,H)$.  If $\rc(H_1,H) = \emptyset$ then no element of the $\f$-conjugacy of $H_1$ is a subgroup of  $H$ and we  return {\tt YES}.  Similarly, return {\tt YES} if  $\rc(H_2,H) = \emptyset$.     Otherwise, choose $b_i$ representing a coset in $\rc(H_i,H)$.  Replacing $H_i$ by $H_i^{b_i}$ we may assume that $H_i < H$. 
   
Lemma~\ref{l:folding normalizer} produces coset representatives $\{a_1,\ldots,a_p\}$ of $H$ in $N(H)$.   Thus a subgroup $K_i< H$ is in the same $\f$ conjugacy class as $H_i$ if and only if it is in the same $H$-conjugacy class as $H_i^{a_j}$ for some $1 \le j \le p$. Order the pairs $(H_1^{a_j}, H_2^{a_k})$ lexicographically on $ 1 \le j,k \le p$.  Apply Lemma~\ref{l:image}  with $\f$ replaced by $H$ and with $(K_1,K_2 )$ replaced by the first pair on the list $ (H_1^{a_1},H_2^{a_1})$ to either produce $K_1, K_2 < H$ such that:
\begin{itemize}
\item
$K_1$ is in the same $H$ conjugacy class as  $H_1^{a_1}$
\item
$K_2$ is in the same $H$ conjugacy class as  $H_2^{a_1}$
\item
$H$ is the internal free product of $K_1$ and $K_2$
\end{itemize}
or  to conclude that no such $K_1$ and $K_2$ exist.  In the former case return {\tt NO} and $[K_1, K_2]$.  In the latter case proceed on to the next pair on the list.  Continue until you either return {\tt NO} and the desired $[K_1,K_2]$ or reach the end of the list, in which case return {\tt YES}.   
\endproof

\begin{cor}\label{c:fiber of 1}
We have an algorithm with input an ordered triple $([H_1], [H_2], [H])$ of elements of $\cC(\f)$ and output a finite enumeration of the fiber $F$ of the above map \pref{e:pairs} over $([H_1], [H_2], [H])$.
\end{cor}

\begin{proof}
Use Lemma~\ref{l:fiber of 1} to determine if $F$ is empty or not and obtain an element $[H'_1, H'_2]\in F$ if not. Input $[H'_1, H'_2]$ into the algorithm of Lemma~\ref{l:pairs algorithm} to enumerate $F$.
\end{proof}

We will also use conjugacy pairs that aren't necessarily good.

\begin{lemma}\label{l:bad pairs}
Consider the set of conjugacy pairs of the form $[H,A]$ with $A<H<\f$ all finitely generated and non-trivial. (In particular this pair is not good.) 
\begin{enumerate}
\item
Two such $[H,A]$ and $[H',A']$ are equal iff there is $g\in\f$ such that $H^g=H'$ and $[A^g]_{H'}=[A']_{H'}$. In particular, {$[H,A]=[H,A']$} iff $A$ and $A'$ are in the same orbit of the action of $N(H)$ on $H$.
\item
The map from the set of such pairs to ordered sequences in $\cC(\f)$ given by $[H, A]\mapsto ([H], [A])$ has fibers that can be finitely enumerated.
\end{enumerate}
\end{lemma}

\begin{proof}
(1) \ \ The only if direction is obvious.  The if direction follows from the fact that if $[A^g]_{H'}=[A']_{H'}$ then there exists $h' \in H'$ such that $A^{h'g} =  A'$ for some $h' \in H'$.

(2) \ \ Given finitely generated non-trivial subgroups $K,L < \f$,  compute  $\rc(L,K)$ by applying Lemma~\ref{l:stallings bijection}.  If $\rc(L,K) = \emptyset$, then $L$ is not conjugate into $K$ so the fiber over $([K],[L])$ is empty.    Otherwise, choose $b$ representing a coset in $\rc(L,K)$.   Replacing $L$ by $L^b$ we may assume that $L < K$.   Apply Lemma~\ref{l:folding normalizer} to produce coset representatives $\xi_1,\ldots,\xi_p$ of $K$ in $N(K)$.  The fiber containing $[K,L]$ equals $\{ [K,L^{\xi_1}],\ldots,  [K,L^{\xi_p}]\}$ by (1).
\end{proof}

\section{Computable $G$-sets}\label{s:computable}  
The ultimate goal of this paper is to provide an algorithm solving the conjugacy problem for $\upgn$, i.e.\ Theorem~\ref{t:main}. We will need other algorithms as part of our solution. In this section and the following two (Sections~\ref{s:computable}, \ref{s:finite index}, and \ref{s:gersten}), we formalize some of the algorithmic aspects present in the $\Out(\f)$-setting. In particular, we provide what could be viewed as a ``data structure" for the input and output of our algorithms.  These sections require no knowledge of $\f$ and are independent of the rest of the paper.

\begin{definition}[Computable]\label{d:computable}
\begin{itemize}
\item
A function $f:X\to Y$ is {\it computable} if it comes equipped with an algorithm with input $x\in X$ and output $f(x)\in Y$.
\item
An {\it enumeration} of a set $X$ is a surjection $\bN\to X$. A {\it finite enumeration} of $X$ is a surjection $\{1, 2, \dots, N\}\to X$. The {\it index} of $x\in X$ is the minimal $n$ such that $n\mapsto x$.
\item
A set $X$ is {\it computable} if it comes equipped with a computable   enumerataion $\bN\to X$ and an algorithm with input $x,x'\in X$ and output {\tt YES} or {\tt NO} depending on whether or not $x=x'$. By default, the empty set is computable. We sometimes write $X=(x_1, x_2, \dots)$ to indicate the enumeration. See Lemma~\ref{l:index algorithm}.
\item
A group $G$ is {\it computable} if the underlying set is computable and  it comes equipped with a third algorithm with input $\theta,\theta', \theta''\in G$ and output {\tt YES} or {\tt NO} depending on whether or not $\theta\theta'=\theta''$.
\item
A $G$-set $X$ is {\it computable} if $G$ and $X$ are computable and  it comes equipped with yet another algorithm with input $\theta\in G$, $x, x'\in X$ and output {\tt YES} or {\tt NO} depending on whether or not $\theta(x)=x'$.
\end{itemize}
\end{definition}

\begin{lemma}\label{l:index algorithm}
If $X=(x_1, x_2, \dots)$ is a computable set then we have an algorithm with input $x\in X$ and output the index of $x$.
\end{lemma}

\begin{proof}
Starting with $i=1$, iteratively check if $x=x_i$.
\end{proof}

To see how $\Out(\f)$ and our $\Out(\f)$-sets are enumerated and that they are computable, see Section~\ref{s:ultimate atoms}. 

 A set $Y$ of interest is often the quotient of a computable set $X$, i.e.\ $Y=X/\sim$ for some equivalence relation $\sim$. We want to use $X$ to give $Y$ the structure of a computable set. We view elements of $X$ as {\it representatives} of elements of $Y$ and always give elements $y\in Y$ as $y=[x]$ where $x\in X$ and $[x]$ is the equivalence class of $x$.

\begin{lemma}\label{l:rep}
Suppose $X$ is a computable set and $Y=X/\sim$ is a quotient of $X$. If we have an algorithm with input $x, x'\in X$ and output {\tt YES} or {\tt NO} depending on whether or not $[x]=[x']$, then $Y$ is computable.

There are the obvious generalizations for groups, etc.
\end{lemma}

\begin{proof}
The computable enumeration of $Y$ maps $i\in\bN$ to $[x_i]\in Y$. Given input $y=[x]$, $y'=[x']$, we can use the algorithm in the hypothesis to output {\tt YES} or {\tt NO} depending on whether or not $[x]=[x']$, i.e.\ whether of not $y=y'$. 
\end{proof}

\begin{ex}\label{e:fp}
Suppose we are given a finite generating set for a group $G$. Elements of $G$ are represented as finite words in the generators and their inverses. This set $X$ of finite words can be computably enumerated, say using length, and $X$ is computable. The composition of the enumeration for $X$ and the evaluation map $e:X\to G$ computably enumerates $G$. If we have an algorithm with input $x, x'\in X$ and output {\tt YES} or {\tt NO} depending on whether or not $e(x)=e(x')$ then $G$ is computable.  This is the case, for example, if we are given a finite presentation for $G$ and an algorithm solving the word problem for this presentation.
\end{ex}

\begin{lemma} \label{l:subset}
\begin{enumerate}
\item\label{i:subset}
Suppose $Z$ is a subset of the computable set $X$. If we have an algorithm with input $x\in X$ and output {\tt YES} or {\tt NO} depending on whether or not $x\in Z$, then $Z$ is computable.
\item\label{i:product}
  If $X$ and $Y$ are computable sets then $X\times Y$ is a computable set.
\end{enumerate}
There are the obvious generalizations for groups, etc.
\end{lemma}

\begin{proof}
\pref{i:subset}: If $Z$ is empty then it is computable by definition. Suppose $Z\not=\emptyset$. The computable enumeration $f:\bN\to Z$ is given as follows. Applying the algorithm in the hypothesis a finite number of times, we can find the  minimal $i\in\bN$ such that $x_i\in Z$.   Define $f(j)=x_i$, for $1\le j\le i$. If $n>i$, then $f(n)=x_n$ if $x_n\in Z$ and $f(n)=f(n-1)$ otherwise.

 The proof of \pref{i:product} is left to the reader.
\end{proof}

In the setting of Lemma~\ref{l:subset}\pref{i:subset}, we view elements of $Z$ as given to us as elements of $X$ that are in $Z$. One reason for the rather odd looking enumeration $f$ in the proof is that we have to make sure that $f$ is defined on all of $\bN$. (Consider the case where $Z$ is finite.)

We now collect some basic properties of computable groups.

\def\cG{\mathcal G}

\begin{lemma}\label{l:computing basics} 
Let $G=(g_1, g_2, \dots)$ and $G'=(g'_1, g'_2, \dots)$ be computable groups.
\begin{enumerate}
\item\label{i:compute 2}
We have algorithms:
\begin{enumerate}
\item
with input $g\in G$ and output the index of $G$;
\item \label{i:product index}
with input $g,h\in G$ and output the index of $gh$;
\item
with output the index of $1$; and
\item \label{i:inverse index}
with input $g$ and output the index of $g^{-1}$.
\end{enumerate}
\item\label{i:word problem}
We have an algorithm with input a finite word $w$ in $\{g_1,g_2,\dots\}^{\pm 1}$ and output the index of $w$ in $G$. In particular, we have an algorithm to solve the word problem in $G$.
\item\label{i:word}
Suppose we are given a finite generating set $\cG=\{h_1,\dots,h_N\}\subset \{g_1, g_2,\dots\}$ for $G$. Then we have an algorithm with input $g\in G$ and output a word $w$ with letters in $\cG$ such that $g=w$ in $G$.
\item\label{i:morphism}
Suppose $f:G\to G'$ is a homomorphism. If we are given a finite generating set $\cG\subset \{g_1, g_2,\dots\}$ for $G$ and $f(\cG)$ then $f$ is computable (with algorithm given in the proof).
\item\label{i:kernel}
If $f:G\to G'$ is a computable homomorphism, then $\Ker(f)$ is computable.
\end{enumerate}
\end{lemma}

\begin{proof}
\pref{i:compute 2}: For (a),  see Lemma~\ref{l:index algorithm}. For (b), (c), and (d): starting with $i=1$, iteratively use the algorithm that comes with a computable group to respectively check: (b) if $gh=g_i$; (c) if $g_ig_i=g_i$; and (d): if $g g_i=1$.

\pref{i:word problem}: Use \pref{i:inverse index} to remove all negative exponents in $w$. Then, use \pref{i:product index}  to iteratively reduce the length of $w$ by replacing consecutive letters $g_ig_j$ with a single letter $g_k$. 

\pref{i:word}: Enumerate the words in $\cG$ (say using length). Iteratively check if $g$ is the $N$th word.

\pref{i:morphism}: To compute $f(g_i)$, use \pref{i:word} to write $g_i$ as a word $w$ in $\cG$. Then $f(w)$ is a word in $f(\cG)$. Finally, use \pref{i:word problem} to find the index of $f(w)$.

\pref{i:kernel} follows from Lemma~\ref{l:subset} since we can algorithmically check if $f(g_i)=1$ in $G'$.
\end{proof}
\begin{remark}  
Using Lemma~\ref{l:computing basics}\pref{i:compute 2}, we may rewrite a given finite subset of $\{g_1, g_2, \dots\}^{\pm 1}$ as a finite subset of $\{g_1, g_2, \dots\}$. In particular, if $\cG$ generates $G$ then we may algorithmically compute a finite subset of $\{g_1, g_2, \dots\}$ that is a symmetric generating set.\end{remark}

\begin{ex}
A computable group need not be finitely generated; a finitely generated computable group need not be finitely presented; etc. Indeed, the kernel of $f:(F_2)^n\to\Z$ which sends each basis element to $1\in \Z$ has varying finiteness properties depending on $n$; see \cite{bb:morse}. By Lemma~\ref{l:computing basics}\pref{i:kernel}, $\Ker(f)$ is computable.
\end{ex}

\begin{lemma}\label{l:stab is computable}
If the $G$-set $X$ is computable and $x\in X$, then the stabilizer $G_x$ in $G=(g_1, g_2, \dots)$ of $x$ is computable.
\end{lemma}

\begin{proof}  
This follows from Lemma~\ref{l:subset} applied to $G_x<G$ since we can algorithmically check if $g_i (x)=x$.
\end{proof}

\section{Finite presentations and finite index subgroups}\label{s:finite index}
The following lemma is useful for finding presentations of finite index subgroups of a finitely presented group. It is well-known (see, for example \cite[Chapter~2, Section~4.~The Reidemeister-Schreier Method]{ls:book}) but for completeness we provide a proof.

\begin{lemma} \label{finite actions}  There is an algorithm that takes as input:
\begin{itemize}
\item
a finite presentation for a computable group $G=(g_1, g_2,\dots)$;
\item 
the multiplication table for a finite group $Q$;
\item 
a computable surjection  $P: G \to Q$; and 
\item
a subgroup $Q'$ of $Q$ given as a finite list of elements of $Q$.
\end{itemize}
and outputs:
\begin{enumerate}
\item
a finite presentation for the subgroup $G':=P^{-1}(Q')$ of $G$; and
\item
finite sets of   left and right coset representatives for $G'<G$.
\end{enumerate}
\end{lemma}

\begin{remark}\label{r:compute 1}
 In some applications, $G$ will act   on a finite set $S$ and so we have a homomorphism $G\to \Perm(S)$ to the permutation group of $S$. $Q$ will be the image of this map and $P:G\to Q$ the induced map. This  will allow us to compute the multiplication table for $Q$.
\end{remark}

\begin{proof}[Proof of Lemma~\ref{finite actions}] 
(1): Let $G=\langle h_1, \dots, h_i\mid r_1, \dots , r_j\rangle$ be the given finite presentation for $G$ where the generators are in $\{g_1, \dots\}$ and the relations are words in the generators; 
 see Lemma~\ref{l:computing basics}\pref{i:word problem}.   Let $X_G$ denote its presentation 2-complex. We assume the reader is familiar with obtaining a finite presentation for a group $H$ from a finite based 2-complex with fundamental group $H$. Hence (1) is reduced to constructing the finite based cover $X_{G'}$ of $X_G$ whose fundamental group has image in the fundamental group of $X_G$ equal to $G'$.

Set $k:=[G:G']=[Q:Q']$ and note $k\le |Q|$. Then every index $k$ based cover of $X_G$ has $k\cdot i$ 1-cells and   $k\cdot j$ 2-cells. Further, if $|r_q|$ denotes the length of $r_q$ as a word in $\{h_1, \dots, h_i\}$, then each 2-cell in the cover has boundary of length at most $k\cdot \max\{|r_q|\mid 1\le q\le j\}$. Hence we can construct all based covers of $X_G$ of index $k$. Examine each in turn to check whether the image $K$ of its fundamental group in $X_G$ is $G'$.
This can be done by checking whether the image in $X_G$ of a generating set for the fundamental group of the cover has $P$-image contained in $Q'$. (Indeed, if so then $K<G'$ but both $K$ and $G'$ have the same index in $G$.) Since $G'$ has index $k$ in $G$, we are guaranteed that one of these covers satisfies $K=G'$. This completes the proof of (1).

(2): We find left coset representatives, the other case being symmetric. Using the hypotheses on $Q$, choose a set $S_Q$ of left cosets representatives for $Q'<Q$. Then $S_G \subset G$ is a set of left coset representatives for $G'<G$ 
if the restriction $P|S_G$ is injective with image $S_Q$. To find $g\in G$ with $P$-image   $q\in Q$, the $P$-image of a set of generators for $G$ generates $Q$ and in $Q$ we may write  $q$ in terms of these generators for $Q$.
\end{proof}

Given a short exact sequence $1\to N\to G\overset f\to Q\to 1$, we are interested in finding a finite presentation for $G$ from finite presentations for $N$ and $Q$.

\begin{lemma}  \label{l:bw}
Suppose $f:G\to Q$ is a computable surjection between computable groups $G$ and $Q$. 

Suppose we are given:
\begin{itemize}
\item
 a finite presentation for $N:=\Ker(f)$ and 
 \item
a finite presentation for $Q$ (for example if $Q$ is finite and we are given the multiplication table for $Q$ then this item is satisfied).  
\end{itemize}
Then we may find a finite presentation for $G$. (In fact, one is constructed in the proof.)  
\end{lemma}

\proof 
Suppose that the finite presentation for $N$ has generating set $\{g_{N,i}\mid i\in I\}\subset G$ and set of relators $\{r_{N,j}\mid j\in J\}$ and suppose that the finite presentation for $Q$ has generating set  $\{g_{Q,k}\mid k\in K\}$ and set of relators $\{r_{Q,l}\mid l\in L\}$.

For each $g_{Q,k}$, find an element $\hat g_{Q,k} \in G$ with image $g_{Q,k}$ in $Q$. This can be done algorithmically by iteratively searching for $g_i$ such that $f(g_i)=g_{Q,k}$.

Set $\hat r_{Q,l}=r_{Q,l}(\hat g_{Q,k}\mid k\in K)$, i.e.\ $r_{Q,l}$ is a word in $\{g_{Q,k}\mid k\in K\}$ and $\hat r_{Q,l}$ denotes the same word in $\{\hat g_{Q,k}\mid k\in K\}$.  The image of $\hat r_{Q,l}$ in $Q$ represents $1_Q$ and so there is a word $n_{Q,l}$ in $\{g_{N,i}\mid i\in I\}$ such that $\hat r_{Q,l}=n_{Q,l}$ in $G$.  Since $N$ is normal, $\hat g_{Q,k}(g_{N,i})\hat g_{Q,k}^{-1}=n_{N,i,Q,k}$ for some word $n_{N,i,Q,k}$ in $\{g_{N,i}\mid i\in I\}$.  Since $G$ is computable, $n_{Q,l}$ and  $n_{N,i,Q,k}$ can be found algorithmically; see Lemma~\ref{l:computing basics}\pref{i:word}. By \cite[Lemma~2.1]{bw:fp}, there is a finite presentation for $G$ with:
\begin{itemize}
\item
generating set $\{g_{N,i}, \hat g_{Q,k}\mid i\in I, k\in K\}$ and
\item
set of relators that is the union of:
\begin{itemize}
\item
$\{r_{N,j}\mid j\in J\}$;
\item
$\{\hat r_{Q,l}=n_{Q,l}\mid l\in L\}$; and
\item
$\{\hat g_{Q,k}(g_{N,i})\hat g_{Q,k}^{-1}=n_{N,i,Q,k}\mid i\in I,  k\in K\}$.
\end{itemize}
\end{itemize}
\endproof

\section{\MW-algorithms} \label{s:gersten}
Our solution of the conjugacy problem for $\upgn$ in $\Out(\f)$ will use a generalization of an algorithm of Gersten that in turn generalizes algorithms of Whitehead and McCool. This section is devoted to describing these generalizations.

A set equipped with an action by a group $G$ is a {\it $G$-set}.   We will only consider {\it computable} $G$-sets; see Definition~\ref{d:computable}.

\begin{definition}[Property {\MW}]\label{d:weak MW} A computable $G$-set $X$ {\it satisfies property} \MW\ (for McCool and Whitehead) if it comes equipped with  an algorithm that takes as input $x,y\in X$ and outputs:   
\begin{description}
\item[{\M}:]
Finite presentations for $G_x:=\{\theta\in G\mid \theta(x)=x\}$ and for $G_y:=\{\theta\in G\mid \theta(y)=y\}$; and
\item[{\W}:]
 {\tt YES} or {\tt NO} depending on whether or not there exists $\theta\in G$ such that $\theta(x)=y$ together with such a $\theta$ if {\tt YES}.
\end{description}
We call such an algorithm an \MW-{\it algorithm.}  Sometimes we refer to an algorithm that satisfies item \M\ [resp.\ item\ \W] as an  \M-{\it algorithm} [resp.\  \W-{\it algorithm}]. Recall that  $G_x$ and $G_y$ are computable by Lemma~\ref{l:stab is computable} .
\end{definition}

Of course our interest here is in the case $G=\Out(\f)$ where the second bullet is associated with J.H.C.~Whitehead \cite{jhcw:certain, jhcw:equivalent} whose algorithm  decides if there is $\theta$ taking one finite ordered set of conjugacy classes in $\f$ to another and produces such a $\theta$ if one exists. The first bullet is associated with McCool \cite{jm:fp} whose algorithm produces a finite presentation for the stabilizer of a finite ordered set of conjugacy classes of elements of $\f$. S.~Gersten generalized the algorithms of Whitehead and McCool to  finite sequences in the $\Out(\f)$-set $\cC(\f)$; see Notation~\ref{n:C}.  ($\cC(\f)$ is shown to be a computable $\Out(\f)$-set at the beginning of Section~\ref{s:ultimate atoms}.) We state Gersten's result in a slightly weakened form.

\begin{thm}[{\cite[Theorems W\&M]{sg:whitehead}}, see also \cite{sk:gersten} and \cite{bfh:gersten}]\label{t:original}
The action of $\Out(\f)$ on the set of finite ordered sets in $\cC(\f)$ satisfies property \MW.
\end{thm}

We will define algebraic invariants for elements of $\upgn$; see Section~\ref{s:algebraic invariants}. An obstruction to $\phi,\psi\in\upgn$ being conjugate in $\Out(\f)$ is the existence of a $\theta\in\Out(\f)$ taking the algebraic invariants of $\phi$ to those of $\psi$. More specifically, if such a $\theta$ does not exist then $\phi$ and $\psi$ are not conjugate. If such a $\theta$ does exist, then we replace $\phi$ by $\phi^\theta$ (or $\psi$ by $\psi^{(\theta^{-1})}$) and reduce to the case where the algebraic invariants of $\phi$ and $\psi$ agree. One step  in our algorithm for the conjugacy problem for $\upgn$ in 
$\Out(\f)$ will be to check whether such a $\theta$ exists and to produce one if so. 

Some of our invariants are elements of $\cC(\f)$ and so fit nicely into the setting of Gersten's theorem.  We will extend Gersten's theorem so that it applies to all our invariants. These invariants   are best described in terms of {\it iterated sets} of elements of $\cC(\f)$, or more generally   in terms of sets with finite-to-one maps to iterated sets. Roughly,  an iterated set in a $G$-set $\catA$ is a finite set consisting of elements of $\catA$ and previously produced iterated sets. The set may be ordered or not. We commonly take $\catA$ to be $\cC(\f)$.

There are two main results. The first, Proposition~\ref{p:enhanced 1}, promotes \MW-algorithms for finite ordered sets in $\catA$ to \MW-algorithms for the set $\ecat(\catA)$ (of equivalence classes) of iterated sets in $\catA$. More specifically, it states that if the $G$-action on finite ordered subsets of $\catA$ satisfies property \MW\ then so does the $G$-action on $\ecat(\catA)$. The second, Corollary~\ref{cor:hat MW}, is a method for enlarging $\catA$.

After reviewing our invariants in Section~\ref{s:review of the invariants} and defining the algebraic invariants in Section~\ref{s:algebraic invariants}, we apply our generalized Gersten's algorithm to obtain a reduction of the conjugacy problem for $\upgn$ in $\Out(\f)$ to Proposition~\ref{p:conjugacy in X} in Lemma~\ref{l:reduction}.

\subsection{Iterated sets and their equivalence classes}\label{s:iterated sets}

\def\atom{Atom}
\def\nonatom{NonAtom}
\def\bX{\mathbb X}
\def\bY{\mathbb Y}
\def\leaves{\mathcal L}
\def\exit{\mathcal E}

\begin{definition} \label{d:rooted tree}
A {\it rooted tree} $(T, *)$ is a finite, simplicial, directed tree $T$ with a base point $*$ called the {\it root}. A valence 0 vertex  (i.e.\ $T=\{*\}$) or a valence one vertex that is not the root is a {\it leaf}. The set of leaves in $T$ is denoted $\leaves(T)$. All edges are oriented away from the root. The set of edges exiting a vertex $x\in T$ is denoted $\exit_T(x)$. Paths are directed. We also may give some vertices an extra structure: a vertex $x$ that is not a leaf is {\it ordered} if an order has been imposed on $\exit_T(x)$. A vertex that is not a leaf and that is not ordered is {\it unordered}.

We view rooted trees $(T,*)$ as combinatorial objects. In particular, edges are specified by ordered pairs of vertices.  
For technical reasons having to do with computability, we require that all vertices of the trees we consider lie in a  set $V$ that we fix once and for all.  (For our purposes, one can take  $V$ to be  $\mathbb N$.) 
 \end{definition}

\begin{ex}
We will draw rooted trees with the root at the top. Ordered vertices are indicated by using dashed lines for its exiting edges. The imposed ordering is displayed from left to right. In the rooted tree below only the root is ordered. There are 4 leaves.
$$
\begindc{\undigraph}[150]
\obj(0,0)[00]{}
\obj(-1,-2)[-1-2]{}
\obj(0,-2)[0-2]{}
\obj(1,-2)[1-2]{}
\obj(1,-4)[1-4]{}
\obj(2,-4)[2-4]{}
\mor{00}{-1-2}{}[\atright,\dashline]
\mor{00}{0-2}{}[\atright,\dashline]
\mor{00}{1-2}{}[\atright,\dashline]
\mor{1-2}{1-4}{}[\atright,\solidline]
\mor{1-2}{2-4}{}[\atright,\solidline]
\enddc
$$
\end{ex}

\begin{definition}[Iterated set]
An {\it iterated set in a set $\catA$} is a rooted tree such that each leaf is labeled by an element of $\catA$. Specifically, an iterated set in $\catA$ is a pair $\big((T,*), \chi\big)$ where $(T,*)$ is a rooted tree and $\chi:\leaves(T)\to\catA$ is a function. We do not assume that $\chi$  is one-to-one. We will often use sans serif capital letters for iterated sets and write, for example, $\objX=\big((T,*),\chi \big)$. The {\it set of atoms of $\objX$} is $\chi(\L(T))$. We sometimes refer to $\catA$ as {\it the set of atoms}. For $l\in\L(T)$, we sometimes refer to $\chi(l)$ as the {\it label} or {\it atom} of $l$. 
 The set of iterated sets in $\catA$ is denoted $\cat(\catA)$. 
 \end{definition}

\begin{ex}\label{e:ffs seq}  
If we take $\catA=\cC(\f)$, then a nested sequence $\vec \F = \F_1\sqsubset \dots \sqsubset \F_N$ of free factor systems determines an iterated set in $\catA$ as follows. First we identify each free factor system $\F_i=\{[F_{i,1}],\dots,[F_{i,m_i}]\}$  with an iterated set:
$$
\begindc{\undigraph}[125]
\obj(0,0)[00]{}
\obj(-3,-3)[-3-3]{$[F_{i,1}]$}[\south]
\obj(-1,-3)[-1-3]{$[F_{i,2}]$}[\south]
\obj(3,-3)[3-3]{$[F_{i,m_i}]$}[\south]
\mor{00}{-3-3}{}[\atright,\solidline]
\mor{-1-3}{3-3}{$\dots$}[\atleft,\nullarrow]
\mor{00}{-1-3}{}[\atleft,\solidline]
\mor{00}{3-3}{}[\atright,\solidline]
\enddc
$$
Then $\vec \F$ determines the ordered set $\{\F_1,\ldots,\F_N\}$:
$$
\begindc{\undigraph}[140]
\obj(0,0)[00]{}
\obj(-3,-3)[-3-3]{$[F_{1,1}]$}[\south]
\obj(-1,-3)[-1-3]{$[F_{1,2}]$}[\south]
\obj(3,-3)[3-3]{$[F_{1,m_1}]$}[\south]
\mor{00}{-3-3}{}[\atright,\solidline]
\mor{00}{-1-3}{}[\atleft,\solidline]
\mor{00}{3-3}{}[\atright,\solidline]
\obj(10,0)[100]{}
\obj(7,-3)[7-3]{$[F_{N,1}]$}[\south]
\obj(9,-3)[9-3]{$[F_{N,2}]$}[\south]
\obj(13,-3)[13-3]{$[F_{N,m_N}]$}[\south]
\mor{100}{7-3}{}[\atright,\solidline]
\mor{100}{9-3}{}[\atleft,\solidline]
\mor{100}{13-3}{}[\atright,\solidline]
\obj(5,3)[53]{}
\mor{-1-3}{3-3}{$\dots$}[\atright,\nullarrow]
\mor{9-3}{13-3}{$\dots$}[\atright,\nullarrow]
\mor{00}{100}{$\dots$}[\atright,\nullarrow]
\mor{53}{00}{}[\atleft,\dashline]
\mor{53}{100}{}[\atright,\dashline]
\enddc
$$
\end{ex}

\def\catX{{\mathbb X}}
\def\batoms{\mathbb {\atoms}}
\def\morf{{\sf f}}
\def\morg{{\sf g}}
\def\morh{{\sf h}}
\def\boldf{\mathbf f}
\def\boldg{\mathbf g}
\def\morH{\sf  H}

\begin{definition}
\label{d:morphism}
Let $\objX=((T,*),\chi)$ and $\objX'=((T',*'),\chi')$ be iterated sets in $\catA$. 
\begin{itemize}
\item
An {\it order-preserving simplicial isomorphism} $f:(T,*)\to (T',*')$ is a simplicial isomorphism that satisfies:
\begin{enumerate}
\item
a vertex $x$ of $T$ is ordered iff $f(x)$ is ordered and the induced map $\exit_T(x)\to\exit_{T'}\big(f(x)\big)$ is order-preserving.
\end{enumerate}
\item
An {\it equivalence} $\morf:\objX\to\objX'$ is an order-preserving simplicial isomorphism $f:(T,*)\to (T',*')$ that additionally satisfies:
\begin{enumerate}[(2)]
\item
for $l \in\L(T)$, $\chi(l)=\chi'(f(l))$.
\end{enumerate}
Clearly equivalence induces an equivalence relation on $\cat(\catA)$. 
\item
$\ecat(\catA)$ denotes the set of equivalence classes of iterated sets.
\end{itemize}
\end{definition}

\begin{remark}
We will not need this, but if $\catA$ is the set of objects of a category $\hat\catA$, then naturally so are $\cat(\catA)$ and $\ecat(\catA)$. A {\it morphism} $((T,*),\chi))\to ((T',*'),\chi')$ is an order-preserving simplicial isomorphism $f:(T,*)\to (T',*')$ together with a function $m:\chi(\L(T))$ into the morphisms of $\hat \catA$ such that, for $l\in\L(T)$, $m(\chi(l))\in\Hom(\chi(l),\chi'(f(l)))$. An earlier version of this paper used a simplified, but more restrictive variant of this category, which was ultimately not needed.
\end{remark}

\subsection{Promoting property $\MW$}
\begin{definition} \label{d:A}
 Suppose $G$   is a group and $\catA$ is a $G$-set. Then $\cat(\catA)
$ and $\ecat(\catA)$ are $G$-sets with actions given as follows. If $\theta\in G$ and $\objX=((T,*),\chi)$ then $\theta(\objX):=\big((T,*), \theta\circ\chi\big)$, i.e.\ $\theta(\objX)$ is obtained by relabeling $\leaves(T)$ according to $\theta$. The $G$-action descends to $\ecat(\catA)$. 
\end{definition}

We want $\cat(\catA)$ and $\ecat(\catA)$ to be computable. This is the case if our set $V$ of vertices and $\catA$ are computable.

\begin{lemma}\label{l:computable}
\begin{enumerate}
\item
If $V$ is a computable set then the set of rooted trees with vertices in $V$  is computable. 
\item
If additionally $\catA$ is a computable set then the sets $\cat(\catA)$ and $\ecat(\catA)$ are computable.
\item\label{i:computable 3}
If additionally $\catA$ is a computable $G$-set then the $G$-sets $\cat(\catA)$ and $\ecat(\catA)$ are computable.
\end{enumerate}
\end{lemma}

\begin{proof}
(1): We view rooted trees $(T,*)$ as combinatorial objects.  In particular, edges are specified by ordered pairs of vertices.  To completely specify $(T,*)$  we also choose a root vertex, designate some vertices as ordered and  choose an order on exiting edges of those  vertices. Rooted trees  $T$ with vertices in $V$  can be 
 computably enumerated using, say, the sum $|T|$ of the number of vertices and largest index among the vertices of $T$. That is, to enumerate the set of rooted trees, first list all those with $|T|=1$, then 2, etc. Two rooted trees are equal iff the they have the same vertices, edges, root, ordered vertices and same order on outgoing edges of ordered vertices.
      
(2): For an iterated set $\objX=((T,*), \chi)$, let $|\objX|$ denote the sum of  the number of vertices of $|T|$ and the largest index of a label \big(an element of $\chi(\L(T))$\big).  $\cat(\catA)$ may be countably enumerated using $|\objX|$. To enumerate $\cat(\catA)$, for example, list all $\objX$ with $|\objX|=1$, then 2, etc. Two iterated sets are equal iff the underlying rooted trees are equal and the functions on the leaves are equal. The first condition can be checked by (1) and the second can be checked since $\catA$ is computable. We can use the same enumeration for $\ecat(\catA)$. Here if $\objX_1:=((T_1,*_1),\chi_1)$ and $\objX_2:=((T_2, *_2),\chi_2)$ represent respectively $\objQ_1$ and $\objQ_2$ in $\ecat(\catA)$ then $\objQ_1=\objQ_2$ iff $\objX_1$ and $\objX_2$ are equivalent and this is a finite check. Indeed, finitely enumerate the set $S$ of order-preserving simplicial isomorphisms $f:(T_1,*_1) \to (T_2,*_2)$. If $S$ is empty then $\objQ_1\not=\objQ_2$. Otherwise, if some $f\in S$ is an equivalence then $\objQ_1=\objQ_2$ and if not then $\objQ_1\not=\objQ_2$. 

(3): Since $\catA$ is a computable $G$-set, it is a finite check whether $\theta(\objX_1)=\objX_2$ and also whether $\theta(\objX_1)$ and $\objX_2$ are equivalent.
\end{proof}

\begin{ass}
Going forward, we assume that our fixed set $V$ of vertices is computable; see Definition~\ref{d:rooted tree}.  In all applications, $\catA$ will be computable.
\end{ass}

\begin{notn}\label{n:cat conventions}
Unless otherwise specified:
\begin{itemize}
\item
$\catA$ denotes  a computable $G$-set;
\item
$\objX$, $\objX'$, ... denote elements $((T,*),\chi)$, $((T',*'),\chi')$, ... of $\cat(\catA)$;
\item
$\objQ$, $\objQ'$, ... denote elements of $\ecat(\catA)$ and are represented by $\objX$, $\objX'$, ... ; and
\item
an equivalence $\morf:\objX\to\objX'$ is given by $f:(T,*)\to (T',*')$.
\end{itemize}
\end{notn}

\begin{notn}\label{n:seq}
\begin{itemize}
\item
The map $\catA\to\ecat(\catA)$ determined by $a\mapsto ((*,*), *\mapsto a)$ is a $G$-equivariant inclusion. In other words, map $a$ to the trivial tree with vertex labeled $a$. Thus we may think of $\catA$ as a subset of $\ecat(\catA)$.
\item
$\seqor(\catA)$ denotes the subset of $\ecat(\catA)$ represented by iterated sets in which $*$ is ordered and $*$ is the initial endpoint of every edge of $T$. $\seqor(\catA)$ is $G$-invariant. There is an obvious $G$-invariant bijection between the set of non-empty, finite, ordered sequences in $\catA$ and $\seqor(\catA)$. We pass back and forth between these two points of view whenever convenient.
\item
The $G$-set $\sequn(\catA)$ is defined analogously where $*$ is unordered. There is an obvious $G$-invariant bijection between the set of non-empty, finite, multi-sets in $\catA$ and $\sequn(\catA)$.
\end{itemize}
\end{notn}

\begin{prop}[Promoting {$\MW$}]\label{p:enhanced 1}
Let $\catA$ be a computable $G$-set. If $\seqor(\catA)$ satisfies property $\MW$ then so does 
$\ecat(\catA)$.
\end{prop}

\proof
We follow the conventions in Notation~\ref{n:cat conventions}. 
First we provide a \W-algorithm for $\ecat(\catA)$; i.e.\ an algorithm that either finds $\theta \in G$ satisfying $\theta(\objQ) = \objQ'$ or concludes that there is no such $\theta$. 
Finitely enumerate the set $S$ of order-preserving simplicial isomorphisms $f:(T,*) \to (T',*')$. If $S$ is empty then return {\tt NO}. Otherwise, start with the first element $f$ of $S$. By hypothesis there is a \W-algorithm that either finds $\theta \in G$ 
such that $\theta(\chi(l))=\chi'(f(l))$ for each $l\in\L(T)$ or concludes that no such $\theta$ exists.   If $\theta$ is found then $f$ gives an equivalence $\theta(\objX)\to\objX'$ and our \W-algorithm returns {\tt YES} and $\theta$.  If no such $\theta$ exists, move on to the next element of $S$ and try again.  If after considering each element of the finite set $S$ we have not returned {\tt YES} then return {\tt NO}. (Equivalently we could make the queries indexed by $S$ in parallel. Note that a different choice of representatives $\objX$, $\objX'$ for $\objQ$, $\objQ'$, would give the same queries.)

 The stabilizer $G_{\objQ}$ in $G$ of $\objQ$ is computable by Lemma~\ref{l:stab is computable}. For the \M-algorithm, we will produce a finite presentation for $G_{\objQ}$ by applying Lemma~\ref{l:bw} to the short exact sequence induced by $\pi:G_{\objQ}\to \Perm(A)$ where $A$ denotes the set $\chi(\L(\objX))$ of atoms of $\objX$. Since the kernel of $\pi$ is the subgroup of $G$ fixing each element of $A$, we can use the \M-algorithm for $\seqor(\catA)$ to produce a finite presentation for this kernel. 

To apply Lemma~\ref{l:bw}, it remains to produce an element of $G_{\objQ}$ realizing each element of the image of $\pi$. This is done as follows. Given $\sigma\in\Perm(A)$, use the \W-algorithm for $\seqor(\catA)$ to produce $\theta\in G$ realizing $\sigma$ if such exists. If there is no such $\theta$ then $\sigma$ is not in the image of $\pi$. Finally, use the computability of $\ecat(\catA)$ (Lemma~\ref{l:computable}(\ref{i:computable 3})) to check if $\theta(\objX)$ is equivalent to $\objX$. If it is then $\sigma$ is in the image of $\pi$ (and is realized by $\theta$) and otherwise it is not. 
(As above with the \W-algorithm, it is easy to see that the choice of representative $\objX$ for $\objQ$ is immaterial. Also, the choice of $\theta$ does not matter.)
\endproof

For reference  we record the following consequence of the previous proof (really just definitions). 
\begin{cor}\label{c:finite index}
If $\objX=((T,*),\chi)$ represents $\objQ\in\ecat(\catA)$ then the subgroup of $G$   fixing each $\chi(l)$, $l\in\leaves(T)$, has finite index in the stabilizer $G_{\objQ}$ of $\objQ$. \qed
\end{cor}
 
\subsection{More atoms}\label{s:more atoms}
Proposition~\ref{p:enhanced 1} concludes, under conditions on $\catA$, that $\ecat(\catA)$ has property \MW. In this section, conclusions have the form $\ecat(\catA')$ satisfies property \MW\ where $\catA'$ is a $G$-set constructed from $\catA$ in various ways. Intuitively, we are enlarging our collection of useful sets of atoms.

\begin{notn}  \label{notn:explicit fibers}
Suppose $p:\hat Y\to Y$ is an equivariant map of   $G$-sets. For $y\in Y$ [resp. $\hat y\in \hat Y$], let $G_y$ [resp. $G_{\hat y}$] denote the stabilizer of $y$  [resp. $\hat y$] with respect to the action of $G$ on $Y$ [resp. $\hat Y$]. Let $F_{y}$ denote the fiber $p^{-1}(y)$. If $p(\hat y)=y$ then by $p$-equivariance $G_{\hat y}<G_y$ and $G_y$ acts on $F_{y}$ inducing a homomorphism $\rho_y: G_y \to \Perm(F_{y})$. (We declare the permutation group of the empty set to be trivial.)

In this setting, we say that {\it $p$ \explicit}\ if the $G$-sets $Y$ and $\hat Y$ are computable and $p$ comes equipped with an algorithm with input $y\in Y$ and output a finite enumeration of $F_y$.
\end{notn}

\begin{lemma}\label{l:explicit algorithms}
Suppose the $G$-map $p:\hat Y\to Y$ \explicit\ and that $Y$ satisfies property \M. Then
\begin{enumerate}
\item
there is an algorithm with input $y\in Y$, $\theta\in G_{y}$ and output $\rho_{y}(\theta)$; 
\item
there is an algorithm with input $y\in Y$ and output the multiplication table for $\rho_{y}(G_y)$; and
\item\label{i:explicit algorithms 3}
$\hat Y$ satisfies property \M.
\end{enumerate}
\end{lemma}

\begin{proof}
Let $y\in Y$. By Lemma~\ref{l:stab is computable}, $G_y$ is computable. Since $p$ \explicit, \compute\ the finite list of elements of $F_y$. Since $Y$ satisfies property \M\ we can produce a finite presentation for $G_y$.

(1): Since $\hat Y$ is computable, \compute\ the action of $\theta$ on $F_y\subset\hat Y$.

(2) follows by applying (1) to our generators of $G_y$.

(3): Since $G_{\hat y}$ is the $\rho_y$-preimage of the stabilizer $S$ of $\hat y$ in $\Perm(F_y)$, we can find a finite presentation for $G_{\hat y}$ by applying Lemma~\ref{finite actions}, taking $P$ to be the surjective homomorphism $G_y\to \rho_y(G_y)$ and $Q':=\rho_y(G_y)\cap S$.
\end{proof}

\begin{lemma}\label{l:composition}
Suppose $f:Z\to Y$ and $g:Y\to X$ each \explicit. Then the composition $h=g\circ f:Z\to X$ \explicit.
\end{lemma}

\begin{proof}
Since $f$ and $g$ each \explicit, the $G$-sets $X$, $Y$, and $Z$ are each computable.
Given $x\in X$, since $g$ \explicit, we can list the elements of $g^{-1}(x)$. Since $f$ \explicit, we can list the elements of  $f^{-1}(y)$ for each $y\in g^{-1}(x)$. We are done by noting $h^{-1}(x)=\sqcup \{f^{-1}(y)\mid y\in g^{-1}(x)\}$.
\end{proof}

\begin{lemma} \label{l:finite fibers} Suppose that $p:\hat Y\to Y$ is a $G$-equivariant map of $G$-sets such that:
\begin{itemize}
\item
$p$ \explicit\ and
\item
$Y$ satisfies property \MW.
\end{itemize}
Then $\hat Y$ satisfies property \MW. 
\end{lemma}

\proof
$\hat Y$ satisfies property \M\ by Lemma~\ref{l:explicit algorithms}\pref{i:explicit algorithms 3}.

For the $\W$-algorithm, we use Notation~\ref{notn:explicit fibers}. Suppose that    $\hat z \in \hat Y$ and that $z = p(\hat z)$. Since $Y$ satisfies property \W\,  we can check whether or not there is $\theta_0$ such that $theta_0(y)=z$ and \compute\ such a $\theta_0$ if it exists.  If not, then return {\tt NO}. If yes, then $\hat y$ and $\theta_0^{-1}(\hat z)$ are in $F_{y}$ and there is an element of $G$ taking $\hat y$ to $\hat z$ iff there is an element $\theta\in G$ (necessarily in $G_y$) taking $\hat y$ to $\theta_0^{-1}(\hat z)$. Our goal becomes to check whether there is $\theta\in G_{y}$ such that $\rho_y(\theta) \in \Perm(F_{y})$ takes  $\hat y$ to $\theta_0^{-1}(\hat z)$ and to produce such a $\theta$ if so. This can be done using our finite generating set for $G_{y}$  and its action on $F_{y}$. If there is no such $\theta$ then return {\tt NO}. Otherwise return {\tt YES} and $\theta_0\cdot\theta$.
\endproof

\begin{construction}[Canonical extension]\label{c:I}
Suppose that $\catA$ and $\catA'$ are $G$-sets and that $I:\catA'\to\ecat(\catA)$ is $G$-equivariant. We now define a $G$-equivariant map $I_{\ecat}:\ecat(\catA')\to \ecat(\catA)$ that restricts to $I$ on $\catA'$ (see Notation~\ref{n:seq}). We call $I_{\ecat}$  the {\it canonical extension of $I$.}
$$
\begindc{\commdiag}[50]
\obj(0,0)[00]{$\catA'$}
\obj(20,0)[20]{$\ecat(\catA)$}
\obj(10,-10)[1-1]{$\ecat(\catA')$}
\obj(1,-2)[1-2]{}
\obj(1,-4)[1-4]{}
\obj(2,-4)[2-4]{}
\mor{00}{20}{$I$}[\atleft,\solidarrow]
\mor{00}{1-1}{}[\atright,\injectionarrow]
\mor{1-1}{20}{$I_{\ecat}$}[\atright,\dashArrow]
\enddc
$$
From $\objQ'\in\ecat(\catA')$, we construct $\objQ:=I_{\ecat}(\objQ')\in\ecat(\catA)$. We do this by constructing a representative $\objX=((T,*),\chi)$ for $\objQ$ from a representative $\objX'=((T',*'),\chi')$ for $\objQ'$ and, for each $l'\in\L(T')$, a representative $\objX_{l'}=((T_{l'},*_{l'}), \chi_{l'})$ for $\objQ_{l'}:=I(\chi'(l'))\in\ecat(\catA)$. Let $T$ be the tree obtained from $T'\sqcup(\sqcup\{T_{l'}\mid l'\in\L(T')\})$ by identifying $l'\in T'$ and $*_{l'}\in T_{l'}$. We declare the image of $*'$ in $T$ to be the root $*$ of $T$. The leaves of $T$ biject naturally with $\sqcup\{\L(T_{l'})\mid l'\in\L(T')\}$ and we define $\chi:\L(T)\to \catA$ by $\chi |\L(T_{l'}):=\chi_{l'}$.

We next show that $\objQ$ is independent of our choices of representatives, i.e.\ that $\objQ'\mapsto \objQ$ is well-defined. Let $\objY'=((S',\star'), \eta')$ also represent $\objQ'$ and so we have a simplicial isomorphism $f':T'\to S'$ inducing an equivalence $\objX'\to\objY'$. In particular, $I(\chi'(l')))=I(\eta'(f(l')))$ in $\ecat(\catA)$. Thus if $\objY_{f(l')}$ is a representative of $I(\eta'(f(l')))$, then we have equivalences $\objX_{l'}\to\objY_{f(l')}$ induced by simplicial isomorphisms $f_{l'}:(T_{l'},*_{l'})\to (S_{f(l')},\star_{f(l')})$ between the underlying trees. The map $f'$ and the $f_{l'}$'s induce a map $$T'\sqcup(\sqcup\{T_{l'}\mid l'\in\L(T')\})\to S'\sqcup(\sqcup\{S_{f_{l'}(l')}\mid l'\in\L(T')\})$$ which descends to a simplicial isomorphism $T\to S$ that induces an equivalence $\objX\to\objY$. Hence the map $I_{\ecat}:\ecat(\catA')\to\ecat(\catA)$ given by $\objQ'\mapsto\objQ$ is well-defined.

If we start the above construction with $\theta(\objX')$ instead of $\objX'$, the only difference is that $\chi$ is replaced by $\theta\circ \chi$, i.e.\ $I_{\ecat}$ is $G$-equivariant. Recall (Notation~\ref{n:seq}) that we identify $\catA'$ with the subset of elements of $\ecat(\catA')$ with underlying tree consisting of only the root. Thus $I$ and $I_{\ecat}$ agree on  $\catA'$.
\end{construction}

\begin{lemma} \label{l:general}
Let $\catA$ and $\catA'$ be $G$-sets and suppose the $G$-map $I:\catA'\to\ecat(\catA)$ \explicit. Then $I_\ecat:\ecat(\catA')\to\ecat(\catA)$ \explicit.
If additionally $\seqor(\catA)$ satisfies property \MW\ then $\ecat(\catA')$ satisfies property \MW.

\end{lemma}

\proof
\def\cQ{\mathcal Q}
Since $I$ \explicit, the $G$-set $\catA'$ is computable. The $G$-set $\ecat(\catA')$ is therefore computable by Lemma~\ref{l:computable}\pref{i:computable 3}. Let $\objQ\in\ecat(\catA)$ be given and let $\objX=((T,*),\chi)\in\cat(\catA)$ represent $\objQ$. Using notation as in Construction~\ref{c:I}, each $\objQ'$ in the fiber $F$ of $I_{\ecat}$ over $\objQ$ has a representative of the form $((T',*),\chi')$ where $(T',*)$ is a rooted subtree of $(T,*)$. Further, each leaf $l'$ of $T'$ then determines a rooted tree $(T_{l'},l')$ where $T_{l'}$ is the subtree of $T$ spanned by $l'$ and all leaves $l$ of $T$ with a directed path from $l'$ to $l$. We then get a representative $((T_{l'}, l'),\chi_{l'})$ for an element $\objQ_{l'}$ of $\ecat(\catA)$ where $\chi_{l'}$ is the restriction of $\chi$ to the leaves of $T_{l'}$. If there are elements $a'_{l'}\in\catA'$ such that $I(a'_{l'})=\objQ_{l'}$ and if we define $\chi'(l'):=a'_{l'}$ then $((T',*),\chi')$ represents an element of $F$ and all elements of $F$ have this form (for some choice of $T'$). Since $I$ \explicit, we can finitely enumerate the fiber of $I$ over $\objQ_{l'}$ and so also find a finite list in $\cat(\catA')$ of representatives for the elements of $F$. (It is easy to see that a different choice of representative $\objX$ for $\objQ$ produces $F$ with a perhaps different enumeration.)

If additionally $\seqor(\catA)$ satisfies property \MW,  then by Proposition~\ref{p:enhanced 1} so does $\ecat(\catA)$.
 That $\ecat(\catA')$ satisfies property \MW\ is now a direct consequence of Lemma~\ref{l:finite fibers}.
\endproof

\begin{cor}\label{cor:hat MW}
Let $\catA$ and $\catA_i$, $i=1,\dots,k$, be $G$-sets with $\catA$ computable. Suppose that $\seqor(\catA)$ satisfies property \MW\ and that $I_i:\catA_i\to\ecat(\catA)$ is $G$-equivariant and \explicit. Then the induced map $\sqcup_i \catA_i\to\ecat(\catA)$ \explicit, as does $\ecat(\sqcup_i \catA_i)\to\ecat(\catA)$, and $\ecat(\sqcup_i\catA_i)$ satisfies property \MW.
\end{cor}

\begin{proof}
It is apparent that, since $I_i$ \explicit, so does $\sqcup_i \catA'_i\to\ecat(\catA)$. The rest of the corollary then follows directly from Lemma~\ref{l:general}.
\end{proof}

\begin{cor}\label{cor:hatMW 2}
Suppose $\catA'$ is a computable $G$-set and $\seqor(\catA')$ satisfies property \MW. Then:
\begin{enumerate}
\item
For $k=2, 3, \dots$, inductively define $\ecat_k(\catA'):=\ecat(\ecat_{k-1}(\catA'))$ where $\ecat_1(\catA'):=\ecat(\catA')$. The $G$-set $\ecat_k(\catA')$ satisfies property \MW.
\item\label{i:cor:hatMW 2 2}  
Here we use Notation~\ref{n:seq}. For $i=1, 2, \dots$, let $s_i\in \{\seqor, \sequn\}$. Set $S_1{(\catA')}:=s_1(\catA'))\subset \ecat(\catA')$ and inductively define $S_k(\catA'):=s_k(S_{k-1}(\catA'))\subset \ecat_k(\catA')$. The $G$-set $S_k(\catA')$ satisfies property \MW. The natural map $S_k(\catA')\subset \ecat_k(\catA')\to\ecat_{k-1}(\catA')\to\dots\to\ecat(\catA')$ is injective where $\ecat_i(\catA')\to\ecat_{i-1}(\catA')$ is the canonical extension (Construction~\ref{c:I}) of the identity map of $\ecat_{i-1}(\catA')$.  
\end{enumerate}
\end{cor}

\begin{proof}
(1): An application of Proposition~\ref{p:enhanced 1} gives that $\ecat(\catA')$ satisfies property \MW.  Suppose that, for $k\ge 2$, we have that $\ecat_{k-1}(\catA')$ satisfies property \MW. Since the identity map of $\ecat_{k-1}(\catA')$ \explicit, so does   $\ecat_{k}(\catA')\to\ecat_{k-1}(\catA')$ by Lemma~\ref{l:general}. By Lemma~\ref{l:finite fibers},
$\ecat_{k}(\catA')$ also satisfies property \MW.

The first conclusion in (2) follows from (1) since $S_k(\catA')\subset\ecat_k(\catA')$, and an \MW-algorithm for $\ecat_k(\catA')$ provides an \MW-algorithm for $S_k(\catA')$.  To prove the injectivity statement by induction, we show that if $\catA''$ is a $G$-set and $f:\catA''\to\ecat(\catA')$ is injective, then the restriction of the canonical extension $s_i(\catA'')\to\ecat(\catA')$ is also injective for $s_i\in \{\seqor, \sequn\}$.

Suppose that $s_i=\seqor$ and that $(a_1,\ldots, a_r), (b_1,\ldots, b_s) \in  s_i(\catA'')$.   If the images $(f(a_1),\ldots,f(a_r))$ and $(f(b_1),\ldots f(b_s))$ under the canonical map are equal, then $f(a_k)=f(b_k)$, $k=1, \dots, m=n$. Since $f$ is injective, $a_k=b_k$. The case where $s_i=\sequn$ is similar. The inductive proof is left to the reader.
\end{proof}

\begin{remark}\label{r:these are iterated sets}
Via the natural map in Corollary~\ref{cor:hatMW 2}\pref{i:cor:hatMW 2 2}, we sometimes view $S_k(\catA')$ as a subset of $\ecat(\catA')$.   
 \end{remark}

\begin{ex} \label{ex:identify with catA}  Let $\catA=\cC(\f)$. A free factor system is a finite unordered set of free factors and so may be interpreted as an element of $\sequn(\catA)$; cf.~Example~\ref{e:ffs seq}. A special chain is an ordered set of free factor systems and so gives an element of $\seqor\sequn(\catA)$.  The set of special chains associated to $\phi$ gives an element of $\sequn\seqor\sequn(\catA)$. We may view $\sequn(\catA)$, $\seqor\sequn(\catA)$, and $\sequn\seqor\sequn(\catA)$ all as subsets of $\ecat(\catA)$ (Remark~\ref{r:these are iterated sets}).
\end{ex}

\section{Our atoms}\label{s:ultimate atoms} 
In this section we apply Section~\ref{s:more atoms} to enlarge the set $\cC(F_n)$ (Notation~\ref{n:C}) to the sets of atoms that we will need for the remainder of the paper.  See  Lemma~\ref{l:hat MW} and also Section~\ref{s:applying wg}.

To start, we need to know that some familiar sets are computable.
Some  proofs refer to the Stallings graph associated to a finitely generated subgroup of $\f$; see Stallings   \cite[5.4]{js:folding} and Section~\ref{sec:NewNewfolding} of this paper.  The proofs also use Lemma~\ref{l:rep} and  Lemma~\ref{l:subset}, sometimes without explicit mention. Let $\B$ be a basis for $\f$.  The following objects are computable.
\begin{itemize}
\item
The group $\f$: Elements of $\f$ are represented by words in $\B\sqcup\B^{-1}$ and can be enumerated using length. Multiplication is given by concatenation. An element is represented uniquely by a reduced word, which can be found by iteratively cancelling a letter and its inverse. See Example~\ref{e:fp}.
\item
The set of conjugacy classes of elements of $\f$: The enumeration of $\f$ serves also as an enumeration of the set of conjugacy classes. An element is represented uniquely by a cyclically reduced word.
\item
The set of finitely generated subgroups of $\f$: An enumeration is given by an enumeration of finite sets of representatives of elements of $\f$. Two such sets are equal iff they determine the same based Stallings graph with labels in $\B$, which can be found using the Stallings folding algorithm.
\item
$\cC(\f)$: The enumeration of the set of finitely generated subgroups of $\f$ also serves as an iteration for $\cC(\f)$. Two representatives are equal in $\cC(\f)$ iff they determine the same (unbased) Stallings graph with labels in $\B$; see Lemma~\ref{l:folding conjugacy}.

\item
The group $\Aut(\f)$: The set of endomorphisms of $\f$ is bijective with $(\f)^{\B}$ under the map $\Theta\mapsto (b\mapsto \Theta(b))$.  Using the Hopf property of $\f$ \cite[Theorem~2.13]{mks:book}, an endomorphism $\Theta$ is an isomorphism iff it is surjective iff the based Stallings graph of $\langle \Theta(\B)\rangle$ is the rose with petals labeled by the elements of $\B$. Thus an enumeration of $\Aut(\f)$ can be obtained using the enumeration for endomorphisms. 
\item
The group $\Out(\f)$: Our enumeration of automorphisms serves also as an enumeration of outer automorphisms. Two automorphisms represent the same outer automorphisms iff they have the same action on conjugacy classes of words of length at most two in $\B\sqcup \B^{-1}$; see \cite{se:trees}.
\item
The $\Aut(\f)$-set $\f$: Representative endomorphisms act on representative words.
\item
The $\Out(\f)$-set of conjugacy classes of elements of $\f$: Representative endomorphisms act on representative words.
\item
The $\Aut(\f)$-set of finitely generated subgroups of $\f$: Representative endomorphisms act on representative finite subsets of $\f$.
\item
The $\Out(\f)$-set $\cC(\f)$: Representative endomorphisms act on representative finite subsets of $\f$.
\item
The $\Out(\f)$-set $\ecat(\cC(\f))$: See Lemma~\ref{l:computable}\pref{i:computable 3}.
\item
The $\Out(\f)$-sets $\seqor(\cC(\f))$ and $\sequn(\cC(\f))$: See Corollary~\ref{cor:hatMW 2}.
\end{itemize}

\def\Asubgroup{\catA_0}
\def\Agoodsubgroupsubgroup{\catA_1}
\def\Agoodsubgroupelement{\catA_2}
\def\Agoodelementsubgroup{\catA_3}
\def\Agoodelementelement{\catA_4}
\def\Aelement{\catA_5}
\def\Abadgroupsubgroup{\catA'_6}
\def\Abadgroupelement{\catA_6}

 \begin{notn} \label{ultimate atoms}
 In the remainder of the paper $G = \Out(F_n)$. We now define $\Out(\f)$-sets $\Asubgroup, \Agoodsubgroupsubgroup, \ldots, \Abadgroupelement$ that will be used to express our algebraic invariants for elements of $\upgn$. At the same time we show each $\catA_i$ is  computable and admits a map that \explicit\ to a previously defined set.
 \begin{itemize}
 \item 
 $\Asubgroup$ denotes $\cC(\f)$, is computable (see the above itemized list), and has been identified (Notation~\ref{n:seq}) as an $\Out(\f)$-subset of $\ecat(\Asubgroup)$.
 \item 
 $\Agoodsubgroupsubgroup$ denotes the set of good conjugacy pairs of non-trivial finitely generated subgroups of $\f$.  We saw above that the set of finitely generated subgroups of $\f$ is computable, hence (Lemma~\ref{l:subset}\pref{i:product}) so is its square. By Lemma~\ref{l:good facts}\pref{i:good iff}, a pair $(K_1, K_2)$ represents a good $[K_1,K_2]$ iff $\rk(\langle K_1, K_2\rangle)=\rk(K_1)+\rk(K_2)$, and this can be checked using the Stallings graphs of $[K_1], [K_2], \mbox{and} [\langle K_1,K_2\rangle]$. Hence (Lemma~\ref{l:subset}\pref{i:subset}) the subset of pairs representing good conjugacy pairs is computable. By Corollary~\ref{c:equal pairs}, there is an algorithm deciding whether good conjugacy pairs are equal. Hence, by Lemma~\ref{l:rep}, $\Agoodsubgroupsubgroup$ is computable.

By Corollary~\ref{c:fiber of 1}, the map $\Agoodsubgroupsubgroup\to\seqor(\Asubgroup)\hookrightarrow\ecat(\Asubgroup)$ \explicit\ where $\Agoodsubgroupsubgroup\to\seqor(\Asubgroup)$ is given by   $[H_1,H_2]\mapsto ([H_1], [H_2], [\langle H_1, H_2\rangle]) $.

\item 
$\Agoodsubgroupelement$ denotes the set of conjugacy pairs $[H,a]$ where $H$ is a non-trivial finitely generated subgroup of $\f$, $a\in \f$ is non-trivial, and $[H,\langle a\rangle]$ is good. It is clear that $[H,a]=[H',a']$ iff $[H,\langle a\rangle]=[H',\langle a'\rangle]$ and $a$ and $a'$ are conjugate. In particular, $\Agoodsubgroupelement$ is computable and the map $\Agoodsubgroupelement\to\Agoodsubgroupsubgroup$ given by $[H,a]\mapsto [H,\langle a\rangle]$ \explicit\ with fibers of size zero or two.
 
 \item 
 $\Agoodelementsubgroup$ denotes the set of conjugacy pairs $[a,H]$ where $H$ is a non-trivial finitely generated subgroup of $\f$, $a\in \f$ is non-trivial, and $[H,\langle a\rangle]$ is good. That $\Agoodelementsubgroup$ is computable and $\Agoodelementsubgroup\to\Agoodsubgroupsubgroup$ given by $[a,H]\mapsto [\langle a\rangle, H]$ \explicit\ with fibers of size zero or two follows exactly as with $\Agoodsubgroupelement$.

 \item 
 $\Agoodelementelement$ is the set of conjugacy pairs $[a,b]$ of elements of $\f$ where $\langle a, b\rangle$ has rank 2. We have $[a,b]=[a',b']$ iff $[\langle a\rangle, \langle b\rangle]=[\langle a'\rangle, \langle b'\rangle]$, $a$ and $a'$ are conjugate, and $b$ and $b'$ are conjugate. It follows that $\Agoodelementelement$ is computable and $\Agoodelementelement\to\Agoodsubgroupsubgroup$ given by $[a,b]\mapsto [\langle a\rangle, \langle b\rangle]$ \explicit\ with fibers of size zero or four.
\item
$\Aelement$ is the set of conjugacy classes $[a]$ of non-trivial elements $a\in\f$. We saw earlier in this subsection  that $\Aelement$ is computable. The map $\Aelement\to\Asubgroup$ given by $[a]\mapsto [\langle a\rangle]$ \explicit\ with fibers of size zero or two.
\item

$\Abadgroupsubgroup$ is the set of conjugacy pairs $[H,A]$ with $A<H<\f$ all finitely generated and non-trivial. (In particular, $[H,A]$ is  not good.) Using Lemma~\ref{l:bad pairs}(1), the proof that $\Abadgroupsubgroup$ is computable is similar to the proof that $\Agoodsubgroupsubgroup$ is computable. By Lemma~\ref{l:bad pairs}(2), $\Abadgroupsubgroup\to \seqor(\Asubgroup)\hookrightarrow \ecat(\Asubgroup)$ given by $[H, A]\mapsto ([H], [A])$ \explicit. $\Abadgroupsubgroup$ is only used to define $\Abadgroupelement$.

\item
$\Abadgroupelement$ is the set of conjugacy pairs $[H,a]$  where $H$ is a non-trivial finitely generated subgroup of $\f$ and $a\not=1$ is in $H$. (In particular, $[H,\langle a\rangle]$ is not good.) $[H,a]=[H',a']$ iff $[H,\langle a\rangle]=[H',\langle a'\rangle]$ and $a$ is conjugate to $a'$. Hence $\Abadgroupelement\to\Abadgroupsubgroup$ \explicit\ with fibers of size zero or two.
\item
$\Abig:=\Asubgroup\sqcup\Agoodsubgroupsubgroup\sqcup\Agoodsubgroupelement\sqcup\Agoodelementsubgroup\sqcup \Agoodelementelement\sqcup\Aelement\sqcup\Abadgroupelement$
 \end{itemize}
 \end{notn}
 
 \begin{lemma} \label{l:hat MW}  Using Notation~\ref{ultimate atoms}, the $\Out(\f)$-set $\ecat(\Abig)$ 
 satisfies property \MW.
 \end{lemma}
 
\begin{proof}
Since $\seqor(\Asubgroup)$ satisfies property \MW\ by Theorem~\ref{t:original}, it follows from Corollary~\ref{cor:hat MW} that it is enough to show that, for each $i$, we have that $\catA_i$ admits $G$-equivariant map to $\ecat(\Asubgroup)$ that \explicit. Using Notation~\ref{ultimate atoms}, we see that each $\catA_i$ admits a map to $\ecat(\Asubgroup)$ that is a composition of two maps, each of which \explicit. We are done by Lemma~\ref{l:composition}.
\end{proof}

\section{List of dynamical  invariants} \label{s:review of the invariants} 
In Section~\ref{s:algebraic invariants} we define algebraic invariants of $\phi\in\upgn$ that are derived from the dynamical invariants of $\phi$ established in the first five sections of this paper. For the convenience of the reader, we list those dynamical invariants here and provide pointers to the relevant sections of the paper. Here $\fG$ always denotes a \ct\ for $\phi$ and  $\Gamma(f)$ its eigengraph; see Section~\ref{sec:stallings}.
We also use the notation of conjugacy pairs; recall Definition~\ref{d:conjugacy pairs},  Examples~\ref{e:conjugacy pairs} and Section~\ref{s:conjugacy pairs}.
\begin{itemize}
\item
$\cP(\phi)$ denotes the set of principal automorphisms for $\phi$  (Definition~\ref{d:PA}) and $[\cP(\phi)]$ denotes the set of isogredience classes in $\cP(\phi)$ (Definition~\ref{d:isogredience}). $[\cP(\phi)]$ parametrizes the components of $\Gamma(f)$.  
$\Fix(\phi) = \{[\Fix(\Phi)]\mid [\Phi]\in [\cP(\phi)]\}$. Since $[\cP(\phi)]$ is finite, $\Fix(\phi)$ is a finite multi-set  of (possibly trivial) conjugacy classes of finitely generated subgroups of $\f$. 
Geometrically it is the core of $\Gamma(f)$. 
$\Fix_{\ge 2}(\phi):=\{[\Fix(\Phi)]\mid [\Phi]\in [\cP(\phi)], \rk(\Fix(\Phi))\ge 2\}$.
See Sections~\ref{sec:FixN}, \ref{sec:principal}, and \ref{sec:axes}. 

\item
We use $\fc=\vec\F(\phi,<_T)$ to denote a special chain for $\phi$ as in Notation~\ref{notn:ffs}. It is a set of free factor systems naturally ordered by $\sqsubset$. We usually work with a pre-chosen $\fc$. For example, the filtration of our \ct\ $\fG$ will usually realize $\fc$.
If $\F\in\fc$ (resp.\ $[F]\in\F\in\fc$) and if the filtration of $\fG$ realizes $\fc$ then $f|\F$ (resp.\ $f|[F]$) denotes the restriction of $f$ to the core filtration element representing $\F$ (resp.\ the component of the core filtration element representing $[F]$). The corresponding eigengraph is denoted $\Gamma({f|\F})$ (resp.\ $\Gamma({f|[F]})$).

\item
A free factor system is special if it is in some special chain. $\fL(\phi)$ denotes the set of special free factor systems of $\phi$; see Notation~\ref{notn:ffs}. Each element of $\fL(\phi)$ is a free factor system and so is a set of conjugacy classes of free factors in $\f$. If $[F]\in\F\in\fL(\phi)$ then $F$ and $[F]$ are also said to be special. The unique minimal (with respect to $\sqsubset$) element of $\fL(\phi)$, denoted $\F_0(\phi)$, is the linear free factor system of $\phi$. It is represented by the core of the subgraph of $G$  that is the union of fixed and linear edges. An invariant description of $\F_0(\phi)$ is $\F(\Fix(\phi))$, i.e.\ the smallest free factor system carrying $\Fix(\phi)$; see Lemma~\ref{l:F0 is natural}.

\item
\[\cR(\phi) := \{[P]: P \in  \cup_{i=1}^m  \Fix_+(\Phi_i)\} \subset \partial F_n/F_n\] 
where the $\Phi_i$'s are representatives of the isogredience classes in $\cP(\phi)$. 
In other words, $\cR(\phi)$ is the set of conjugacy classes of points in $\partial \f$ that are isolated fixed points for some principal lift of $\phi$. See Section~\ref{sec:FixN}.  In any \ct\  $\fG$ representing $\phi$ there is a bijection $r \longleftrightarrow E$ between $\cR(\phi)$ and the set $\E_f$ of higher order edges of $G$. The eigenray $R_E$ has terminal end $r$.

\item
$\fe\in\fc$ denotes a special 1-edge extension in $\fc$, i.e.\ $\fe=(\F^-\sqsubset\F^+)$ is a pair of consecutive elements  of $\fc$. Suppose $\fG$ realizes $\fc$. $\Gamma({ f|\F^+})\setminus\Gamma({f|\F^-})$ has one or two ends; these represent the {\it new (with respect to $\fe$)} elements of $\cR(\phi)$. 
The 1-edge extension $\fe$ has type {\sf H}, {\sf HH}, or {\sf LH}. There are two new elements iff $\fe$ has type {\sf HH}. A new element is often denoted $r^+$. Further, $\fe$ can be {\contractible}, {\cyclic}, or {\llarge}. See Section~\ref{s:canonical ffs}. 

\item
Continuing the previous bullet, if the filtration of $\fG$ realizes $\fc$ then $\Gamma({f|\F^+})$ carries more lines than $\Gamma({f|{\F^-}})$. The set of {\it added lines with respect to $\fe$}, denoted $\LW_{\fe}(\phi)$, is a $\phi$-invariant subset of these lines. See Definition~\ref{d:added lines}.

\item
$\acc(\phi)=\cup_{r\in\cR(\phi)}\acc(r)$ denotes the finite set of limit lines for $\phi$. See Section~\ref{s:limit lines}. Here $\acc(r)$ denotes the accumulation set of $r$ or equivalently of the eigenray in $\Gamma(f)$ representing $r$. The elements of $\acc(\phi)$ are all represented as lines in $\Gamma(f)$. $\accnr(\phi)\subset\acc(\phi)$ is the subset of non-periodic lines. 

\item
$\A_\both(\phi)$ denotes the set of oriented axes of $\phi$ where a root-free conjugacy class $[a]$ of an element of $\f$ is an axis if it has more than one representation in $\Gamma(f)$. An axis has an invariant description: $[a]$ is an axis if there are $\Phi_1,\Phi_2\in\cP(\phi)$ such that $a\in\Fix(\Phi_1)\cap \Fix(\Phi_2)$ and $\Phi_1 \ne \Phi_2$. In this case, we say the conjugacy pair $[\Phi, a]$ is a strong axis; see Definition~\ref{def:strong axis}. It is represented geometrically as a lift to $\Gamma(f)$ of $[a]$. The set of strong axes is denoted $\sa(\phi)$. Associated to each pair of strong axes $\alpha_1=[\Phi_1,a], \alpha_2=[\Phi_2,a]$ is a twist coordinate $\tau(\alpha_1,\alpha_2)$ in $\Z$.   See Definition~\ref{def:twist}.
\end{itemize}

\section{Algebraic data associated to invariants}\label{s:algebraic invariants}
In this section we define algebraic versions of some of our dynamical invariants. We also explain how the algebraic versions can be computed and viewed as an element of $\ecat(\Abig)$; see Notation~\ref{ultimate atoms}. The algebraic invariants are typically weaker than their dynamic versions. However they have the advantage that they are iterated sets and so fit into the framework of Section~\ref{s:gersten}. Some of our invariants, for example chains, are already algebraic in nature and so need no modification. 

All of our algebraic invariants for $\phi\in\upg$ will be computed using a \ct\ $\fG$ for $\phi$ (see Section~\ref{sec:ct}).  Additionally,  the core of the eigengraph $\Gamma(f)$ can  be computed from $\fG$; see Section~\ref{sec:stallings}. In fact, since $\Gamma(f)$ is obtained from its core by adding the eigenrays of $f$ and the eigenrays have a simple form (Section~\ref{sec:ct}), we can compute arbitrarily large neighborhoods of the core in $\Gamma(f)$.

\subsection{Special chains}\label{s:special chains}
Recall from Notation~\ref{weaker po} and Lemma~\ref{preserves partial order} that there is a canonical partial order $(\cR(\phi),<)$ that can be computed from any \ct\ for $\phi$. Hence all extensions of $<$ to a total order $\dle$ can also be computed. The special chain 
$\vec\F(\phi,\dle)$ for $\phi$ can also be computed from any \ct\ for $\phi$; see Notation~\ref{notn:ffs}. A special chain is an element of $\seqor\sequn(\Asubgroup)\subset\ecat(\Asubgroup)\subset\ecat(\Abig)$; see Example~\ref{ex:identify with catA} and Corollary~\ref{cor:hatMW 2}\pref{i:cor:hatMW 2 2}. Similarly, the set of all special chains for $\phi$ and the set $\fL(\phi)$ of all special free factor systems for $\phi$ can be computed from any \ct\ for $\phi$. Note that the former set is in $\sequn\seqor\sequn(\Asubgroup)\subset\ecat(\Abig)$ and $\fL(\phi)\in \sequn\sequn(\Asubgroup)\subset\ecat(\Abig)$. We will tacitly use Corollary~\ref{cor:hatMW 2}\pref{i:cor:hatMW 2 2} throughout the rest of Section~\ref{s:algebraic invariants}.

\begin{itemize}
\item
Throughout the rest of Section~\ref{s:algebraic invariants}, $\phi\in\upg$, $\fc$ denotes a special chain for $\phi$, and $\fG$ denotes a \ct\ that represents $\phi$, satisfies (Inheritance), and   realizes $\fc$. 
\end{itemize}

\subsection{$\Fix(\phi)$}
The multi-set $\Fix(\phi):=\{[\Fix(\Phi)]\mid [\Phi]\in [\cP(\phi)]\}$ is already algebraic and is an element of $\sequn(\Asubgroup)\subset\ecat(\Abig)$. As reviewed in Section~\ref{s:review of the invariants}, $\Fix(\phi)$ is represented by the core of $\Gamma(f)$ and so can be computed.

\subsection{Axes}
The set $\A_\both(\phi)$ of oriented axes of $\phi$ (Definition~\ref{def:axes}) is already algebraic and is an element of $\sequn(\Aelement)\subset\ecat(\Abig)$. In terms of $\fG$, $[a]\in\A_\both(\phi)$ iff either $[a]$ or $[a^{-1}]$ is represented by a twist path, which can be found by inspecting the linear edges of $G$; see Section~\ref{sec:ct}.   In terms of $\Gamma(f)$, $[a]\in\A_\both(\phi)$ iff $a$ is root-free and represented by more than one circuit in the core of $\Gamma(f)$ with at least one representative embedded. 

\subsection{Algebraic rays}\label{s:algebraic rays}
\begin{remark} \label{ray carried} If $F$ is a free factor and $\ti r \in \partial F$ then we say that {\it  $[F]$ carries $r$}.  Equivalently, if $G$ is a marked graph and $H \subset G$ is a core subgraph   representing $[F]$, then there is  a   ray in $\ti G$ that converges to $\ti r$ and projects into $H$.   In the case that concerns us, $r \in \cR(\phi)$ corresponds to some $E \in \E_f$ (see Lemma~\ref{identifying Fix+}) and $[F]$ is a component of a free factor system in $\fc$.  If $C$  is the component  of the core filtration element of $G$ corresponding to $[F]$ then $[F]$ carries $r$ if and only if some subray of $R_E$ is contained in  $C$.  By construction, $R_E = E \cdot u \cdot f_\#(u) \cdot \ldots$ where the closed path $u$ satisfies $f(E)=E\cdot u$.  Since the height of $f^k_\#(u)$ is independent of $k$, $r$ is carried by $[F]$ if and only if $u \subset  C$.
\end{remark}
    
\begin{itemize}
\item
(Algebraic rays): 
For $r\in\cR(\phi)$, $F_\fc(r)$ denotes the minimal special free factor $[F]\in\F\in\fc$ carrying $r$. If $\ti r \in \partial \f$ is a lift of $r$ then we  also write $F_\fc(\ti r):=F$ where $F$ is the unique representative of $F_\fc(r) $  that contains $\ti r$. An algebraic ray $F_\fc(r)$ is an element of $\Asubgroup\subset\ecat(\Abig)$.
\end{itemize}

\begin{remark}\label{r:F(r)}
Continuing Remark~\ref{ray carried}, $F_\fc(r)$ is represented by the minimal component $C$ of a core filtration element of $G$ containing $u$.  In particular, we can compute $F_\fc(r)$ from our \ct\ $f$. In our running example (see pages \pageref{e:main example}, \pageref{ex.a}, \pageref{ex.h}, and \pageref{ex.i}), $\cR(\phi)=\{r_c, r_d, r_e, r_q\}$, the only relation is $r_c<r_q$, and  the choice of total order is $r_c\ \dle\ r_d\ \dle\ r_e\ \dle\ r_q$. We have $[F_\fc(r_q)]=[\langle a, b, c\rangle]$.
\end{remark}

\begin{remark}
We could work with all chains and define $F(r)$ to be the minimal special free factor $[F]\in\F\in\fL(\phi)$ carrying $r$. This would cause some extra work later in Lemma~\ref{d in Ker Q}.  
\end{remark}

\subsection{Algebraic lines}

Recall that $[\cdot,\cdot]$ denotes a conjugacy pair (see Definition~\ref{d:conjugacy pairs},    Examples~\ref{e:conjugacy pairs} and Section~\ref{s:conjugacy pairs}) and, for non-trivial $a\in\f$, $a^+\in\partial\f$ (resp.\ $a^-$)  denotes the attractor (resp.\ repeller) of $ i_a | \partial \f$; see the beginning of Section~\ref{sec:FixN}.  Recall also that a line $L$ is principal with respect to $\phi$ if there is a lift $\ti L$ whose endpoints are contained in $\FixN(\Phi)$ for some  $\Phi \in \cP(\phi)$.  Equivalently, $L$ lifts into the eigengraph $\Gamma(f)$; see   Lemma~\ref{lem:lifting}.  In this section, we define an algebraic version $\sHsub(L)$ for certain principal line $L$ and  associate to $\sHsub(L)$ a set of lines containing $L$ that in turn determines $\sHsub(L)$. \begin{definition}[Algebraic lines]
Suppose that $L$ is a non-periodic principal line for $\phi$ and that the non-periodic ends of $L$ are contained in $\cR(\phi)$. There are four possibilities.
\begin{description}
\item[{[{\sf{P-P}}]:}]
$L$ has type {\sf P-P} if some (hence every) lift $\ti L$ has the form $(a^-, b^+)$ for some root-free $a, b\in\f$ with  $a\not=b^{\pm 1}$ in $\f$. In particular, $a$ and $b$ are non-trivial. $\sHsub(L):=[a^{}, b]$. To $[a,b]$ we associate $\{L\}$. \ We also define $\sHsub(\ti L):=(a,b)$ and associate to it $\{(a^-,b^+)\}$. In this case $\sHsub(L)$ determines $L$. 

\item[{{[\sf{P-NP}}]:}]
$L$ has type {\sf P-NP} if some (hence every) lift $\ti L$ has the form $(a^-, \ti r)$ for some root-free $a\in\f$ and a lift $\ti r$ of some $r\in\cR(\phi)$. $\sHsub(L):=[a^{}, F_\fc(\ti r)]$. To $\sHsub(L)$ we associate the set of lines $[a^-, \partial F_{\fc}(\ti r)]$. $\sHsub(\ti L):=\big(a, F_\fc(\ti r)\big)$ and has the associated set of lines $\big(a^-,\partial F_\fc(\ti r)\big)$.

\item[{{[\sf{NP-P}}]:}]
$L$ has type {\sf NP-P} if some (hence every) lift $\ti L$ has the form $(\ti r,b^+)$ for some root-free $b\in\f$ and a lift $\ti r$ of some $r\in\cR(\phi)$. $\sHsub(L):=[F_\fc(\ti r), b]$. To $\sHsub(L)$ we associate the set of lines $[\partial F_{\fc}(\ti r), b^+]$. $\sHsub(\ti L):=(F_\fc(\ti r),b)$ with associated  set of lines $(\partial F_\fc(\ti r),b^+)$.

\item[{[{\sf{NP-NP}}]:}]
$L$ has type {\sf NP-NP}  if some (hence every) lift  $\ti L$ has the form $(\ti r, \ti s)$ for lifts $\ti r$ of $r\in\cR(\phi)$ and  $\ti s$ of $s\in\cR(\phi)$. $\sHsub(\ti L):=[F_\fc(\ti r), F_\fc(\ti s)]$. To $\sHsub(\ti L)$ we associate the set of lines $[\partial F_{\fc}(\ti r), \partial F_{\fc}(\ti s)]$. $\sHsub(\ti L):=(F_\fc(\ti r),F_\fc(\ti s))$ with associated set of lines $\big(\partial F_\fc(\ti r),\partial F_\fc(\ti s)\big)$.
\end{description}
\end{definition}

\begin{lemma} \label{l:algebraic lines are good pairs}
Suppose that  $L$ lifts to $\Gamma(f)$ and has one of the types   {\sf P-P}, {\sf P-NP}, {\sf NP-P}, or {\sf NP-NP}.  Then with notation as above:   
\begin{description}
\item  [{{[\sf{P-P}}]:}]
  $ [\langle a^{}\rangle, \langle b^{}\rangle]$ is a good conjugacy pair.
\item  [{{[\sf{P-NP}}]:}] 
  $ [\langle a^{}\rangle, F_\fc(\ti r)]$ is a good conjugacy pair.
  \item  [{{[\sf{NP-P}}]:}] 
  $ [ F_\fc(\ti r),\langle b^{}\rangle]$ is a good conjugacy pair.
 \item  [{{[\sf{NP-NP}}]:}]  
  $  [F_\fc(\ti r),  F_\fc(\ti s)]$ is a good conjugacy pair.
\end{description}
\end{lemma}

\begin{proof}[Proof.]
[{\sf P-P}]: Since $a\not=b^{\pm 1}$ and $a$ and $b$ are root-free,  $\langle a, b\rangle$ is a free group of rank~2.  In particular, $ [\langle a^{}\rangle, \langle b^{}\rangle]$ is good by Lemma~\ref{l:good facts}\pref{i:good iff}. 

[{\sf P-NP}]:
Suppose $L$ has the lift $\ti L=(a^-,\ti r)$. Set $A=\langle a\rangle$. Since $L$ lifts to $\Gamma(f)$, $L=\alpha^{\infty}\sigma R_E$ for some $E\in\E_f$ where $\alpha$ is a circuit in the core of $\Gamma(f)$ representing $[a]$ (Lemma~\ref{lem:lines in gf}). We choose $\sigma$ to have minimal length. Let $\star$ be the terminal vertex of $E$ and let $\ti \star$ be the terminal vertex of the unique lift $\ti E$ of $E$ in $\ti L$.

The based labeled graph $(C,\star)$ as in Remark~\ref{r:F(r)} immerses to $(G,\star)$ and similarly the based labeled graph $(G_A,\star)$ that is a lollipop formed by the union of a circle labeled $\alpha$ and a segment labeled $\sigma E$ immerses to $(G,\star)$. If we define $(H,\star)$ to be the one point union of $(G_A,\star)$ and $(C,\star)$ then by construction, the immersions of $(G_A,\star)$ and $(C,\star)$ to   $(G_A,\star)$ induce a map of  $H\to G$ that  does not admit any Stallings folds and so, by \cite[Proposition 5.3]{js:folding}, induces an injection on the level of fundamental groups. We now have an identification of $\big(A, F_\fc(\ti r), \langle A, F_\fc(\ti r)\rangle\big)$ and $\big(\pi_1(G_A,\star), \pi_1(C,\star), \pi_1(H,\star)\big)$. By VanKampen, we see that $\langle A, F_\fc(\ti r)\rangle$ is the internal free product of $A$ and $F_\fc(\ti r)$.

The cases  [{\sf NP-P}] and [{\sf NP-NP}] are similar.
\end{proof}

\begin{remark} \label{r:lines=conj pairs}
$\sHsub(L)$ is an element of $\Agoodsubgroupsubgroup\sqcup\Agoodsubgroupelement\sqcup \Agoodelementsubgroup\sqcup\Agoodelementelement\subset\ecat(\Abig)$.
Not all lines that lift to $\Gamma(f)$ are assigned a type.\ For each $L$ that has a type, $\sHsub(L)$ can be recovered from its associated set of lines. In the {\sf NP-NP} case, this is a direct application of Corollary~\ref{c:lines=conj pairs}, Remark~\ref{rem:lines=conj pairs} and Lemma~\ref{l:good facts}\pref{i:good disjointness}.  The obvious modification needed for the other cases where $\langle a\rangle$ is replaced by $a$ is left to the reader. We often conflate $\sHsub(L)$ with its associated set of lines.
\end{remark}

\begin{excont*}\label{ex.k}
If $L$ is the upward line represented by the contractible component of $\Gamma(f)$ in Figure~\ref{f:eigengraph} then  $L$ has type {\sf NP-NP} and $\sHsub(L)=[\langle a, b\rangle, \langle a,b\rangle^{d^{-1}e}]$ consists of the set of lines in the graph in Figure~\ref{f:good} that cross $d^{-1}e$ once and $e^{-1}d$ not at all. 
\begin{figure}[h!]
\centering
\includegraphics[width=.3\textwidth]{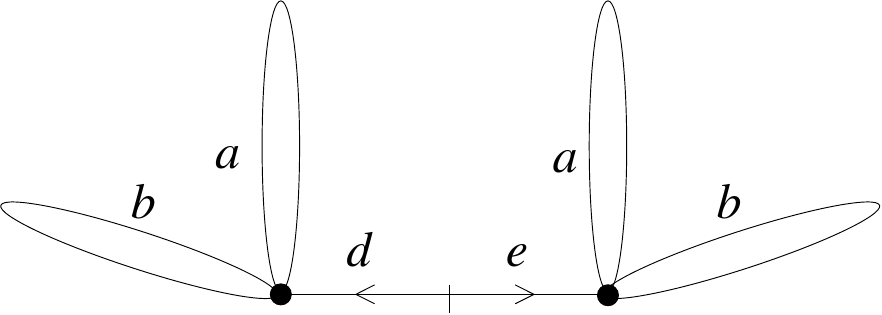}
\caption{}
\label{f:good}
\end{figure}
\end{excont*}

\begin{lemma}\label{l:algebraic lines are natural} Suppose that that $\fc$ is a special chain for $\phi$, that $L$ is a non-periodic principal line for $\phi$ whose non-periodic ends are contained in $\cR(\phi)$  and that $\theta \in \Out(\f)$.  Then $\theta(\fc)$ is a special chain for $\phi^\theta$, $\theta(L)$ is a non-periodic principal line for $\phi^\theta$ whose non-periodic ends are contained in $\cR(\phi^\theta)$ and $\theta\big(\sH_{\phi, \fc}(L)\big)=\sH_{\phi^\theta,\theta(\fc)}\big(\theta(L)\big)$. \end{lemma}

\begin{proof}
We will do the case {\sf P-NP}; the others are similar. Lemma~\ref{l:special is natural} and  Lemma~\ref{first theta} imply that  $\theta(\fc)$ is a special chain for $\phi^\theta$ and that $\theta(L)$ is a principal line for $\phi^\theta$.  If $\Theta\in\theta$ and $\ti L = (a^-,\ti r)$ then $\Theta(\ti L) = (\Theta(a), \Theta(\ti r))$ and 
$$\Theta (\sHsub(\ti L)) = [\Theta(a), \Theta( F_\fc(\ti r))] = [\Theta(a),  F_{\Theta(\fc)}(\Theta(\ti r))]  = \sH_{\phi^\theta,\theta(\fc)}(\Theta(\ti L))$$
\end{proof}

\begin{remark}
For applications, we note that from definitions it follows that if $\theta(\fc)=\fc$ and $\fe\in\fc$ then $\theta(\fe)=\fe$.
\end{remark}

In the next lemma we abuse notation and identify $\sHsub (L)$ with its associated set of lines; see Remark~\ref{r:lines=conj pairs}.

\begin{lemma}  \label{just one eigengraph line} Suppose $\fG$ is a \ct\ for $\phi$ that realizes $\fc$. If $L$ is either an element of $\accnr(\phi)$ or an element of $\LW_\fe(\phi)$ such that $\fe$ is not \llarge, then $L$ is the only line in $\sHsub(L)$ that lifts to $\Gamma(f)$.
\end{lemma}

\begin{proof}
We assume at first that  $L \in \accnr(\phi)$. By Corollary~\ref{cor:limit lines}\pref{item:endpoints in FixN}, all elements of $\accnr(\phi)$ lift to $\Gamma(f)$ and further, by Corollary~\ref{cor:limit lines}\pref{item:ell decomposes}, have one of the types {\sf P-P}, {\sf P-NP}, {\sf NP-P}, or {\sf NP-NP}. In particular, {$\sHsub(L)$} is defined. We will do the case that $L=[a^-,\ti r]$ is {\sf P-NP}, the others being similar. By Corollary~\ref{cor:limit lines}\pref{item:ell decomposes}, $L$ has the form $(R_1)^{-1}R_E$ where $R_1$ consists of only linear and fixed edges and where $E\in\E_f$. In particular, $E$ is the highest edge of $L$.

Every line $L' \in \sHsub(L)$ has a lift of the form $\ti L'=(a^-,\ti s)$ where $\ti s\in F_\fc(\ti r)$.
 Since $\fG$ realizes $\fc$, $F_\fc(r)$ is represented by a subgraph of $G$ whose edges are all lower than $E$. If $\ti r=\ti s$ then $L'=L$ and we are done. We may therefore assume $\ti L'':=(\ti r, \ti s)$ is a line with both endpoints in $\partial F_\fc(\ti r)$. Thus $L''$ only crosses lifts of edges that are lower than $E$. It follows that $L'=(R_1)^{-1}E R_2$ where the ray $R_2$ only crosses edges that are  lower than $E$. If $L'$ lifts to $\Gamma(f)$ then $E$ is the first higher order edge it crosses and so $ER_2=R_E$ and $L'=L$. This completes the proof when $L \in \accnr(\phi)$.

We now assume that $L\in\LW_\fe(\phi)$ and that $\fe$ is not \llarge. If $\fe$ is \contractible\ then, because $L$ lifts to $\Gamma(f)$, $L$ has the form $[\ti r,\ti s]=R_{E_1}^{-1}\rho R_{E_2}$ where $\rho$ is a Nielsen path. By definition $L'\in\sHsub(L)$ has a representative $\ti L'=(\ti r_1, \ti r_2)$ such that either $\ti r_1=\ti r$ or $(\ti r_1, \ti r)$ is a line with both endpoints in $\partial F_\fc(\ti r)$ and such that either $\ti r_2=\ti s$ or $(\ti s, \ti r_2)$ is a line with both endpoints in $\partial F_\fc(\ti s)$. We argue as above to conclude that $L'=L$ if $L'$ lifts to $\Gamma(f)$. The case where $\fe$ is \cyclic\ is similar.
\end{proof}

\begin{lemma} \label{l:unique lift}
Suppose that $L$ has one of the types {\sf P-P}, {\sf P-NP}, {\sf NP-P}, or {\sf NP-NP} and that $\theta\big(\sHsub(L)\big)=\sHsub(L)$. If $\ti L$ is a lift of $L$ then there is a unique $\Theta\in\theta$ such that $\Theta\big(\sHsub(\ti L)\big)=\sHsub(\ti L)$.
\end{lemma}

\begin{proof}
The existence of at least one such $\Theta$ follows from the definitions.

{\sf P-P}: Suppose $\ti L=(a^-, b^+)$. If $\Theta_1\not=\Theta_2$ represent $\theta$ and leave $(a^-, b^+)$ invariant then the difference $\Theta_1\Theta_2^{-1}$ has the form $i_x$ for some $x\not=1$ in $\f$ and leaves $(a^-,b^+)$ invariant. It follows that $a$, $b$, and $x$ share a power. This is impossible since $L$ is not periodic.

{\sf NP-P}: Suppose $\ti L=(\ti r, b^+)$. If $\Theta_1\not=\Theta_2$ in $\theta$ leave $(F_\fc(\ti r), b^+)$ invariant then $\Theta_1\Theta_2^{-1}=i_x$ for some $x\not=1$ in $\f$, $i_x\big(F_\fc(\ti r)\big)=F_\fc(\ti r)$, and $i_x(b)=b$. Hence $x\in F_\fc(\ti r)\cap \langle b\rangle$ which is impossible since $\partial F_\fc(\ti r)\cap\partial\langle b\rangle=\emptyset$ by Lemma~\ref{l:algebraic lines are good pairs}.  

The cases {\sf P-NP} and {\sf NP-NP} are similar.
\end{proof}

\subsection{Algebraic strong axes}\label{s:algebraic sa}
\begin{itemize}
\item (Algebraic strong axes): If $[\Phi, a]$ is a strong axis  and $\Fix(\Phi)$ is not cyclic (equivalently  $\rk(\Fix(\Phi))>1$), then  $[\Fix(\Phi), a]$ is the associated algebraic strong axis.
\end{itemize}

\begin{remark}\label{r:asa}    In (Algebraic strong axes) above we are focusing only on strong axes for the restriction of $\phi$ to its linear free factor system $\F_0(\phi)$; see the proof of Lemma~\ref{base case}. By definition of strong axes, $a\in\Fix(\Phi)$. In particular, $[\Fix(\Phi), a]$ is an element of $\Abadgroupelement\subset\ecat(\Abig)$ and is not good. The set of algebraic strong axes of $\phi$ is an element of $\sequn(\Abadgroupelement)\subset\ecat(\Abig)$. By Lemma~\ref{l:alg sa = sa}, $[\Fix(\Phi), a]$ determines $[\Phi, a]$.

Remark~\ref{r:sa} can be used to compute the algebraic strong axes from our \ct\ $\fG$. Using the notation there, the set $\sa(\phi)$ of strong axes is in a 1-1 correspondence with the set of non-trivial circuits in the core of $\Gamma(f)$ representing elements of $\A_\both(\phi)$.
The strong axis $[\Phi_{a,0},a]$ corresponding to $(f_{v_\#}, \tau)$ has algebraic invariant represented by $[\Fix(f_{v\#}),a]$. $\Fix(f_{v\#})$ is the image in $\pi_1(G,v)$ of $\pi_1(\Gamma(f), v)$ where $v\in\Gamma^0(f)$ is viewed as an element of $G$ (see the construction of $\Gamma(f)$ in Section~\ref{sec:stallings}). The case for $(f_{v_j\#}, v_j)$ is similar.

In our running example (see page~\pageref{ex.f}), the algebraic invariants of $\alpha$ and $\alpha'$ are respectively represented by the pairs $(\langle a, a^b\rangle, a)$ and $(\langle a, a^b\rangle, a^b)$. The strong axis $\alpha''$ doesn't have an algebraic invariant since the corresponding fix set is cyclic. 
This ends Remark~\ref{r:asa}.

\end{remark}

\subsection{Algebraic added lines}\label{s:algebraic added lines}
\begin{itemize}
\item
(Algebraic added lineswith respect to $\fe$): We use notation as in Definition~\ref{d:added lines}. 
$\sH_{\fe\in\fc}(\phi)$ is defined to be $\{\sHsub(L)\mid L\in\LW_\fe(\phi)\}$ if $\LW_\fe(\phi)$ is finite and is defined to be the singleton $\{[\Fix(\Phi),F_\fc(\ti r^+)]\}$ otherwise. The proof that $[\Fix(\Phi),F_\fc(\ti r^+)]$ is good is similar to the proof of Lemma~\ref{l:algebraic lines are good pairs}. $\sH_{\fe\in\fc}(\phi)$ is an element of $\sequn(\Agoodsubgroupsubgroup\sqcup\Agoodsubgroupelement\sqcup\Agoodelementsubgroup)\subset\ecat(\Abig)$. 
\end{itemize}

In terms of our \ct\ $\fG$, algebraic added lines can be computed as follows. We use the notation of Definition~\ref{d:added lines}. Since we already know how to compute algebraic lines, in the cases where the number of added lines is finite, it is enough to describe the added lines. This is done in Definition~\ref{d:added lines} using the eigengraph $\Gamma(f|F^+)$. The case where there are infinitely many added lines is the intersection of [$\sH$] and \llarge\ in Lemma~\ref{l:types}. Let $v$ be the initial vertex of the edge $E\in\E_f$ in $\Gamma(f|F^+)$ corresponding to $r^+$. Then $[\Fix(\Phi),F_\fc(\ti r^+)]$ is represented by $(\Fix(f_{v\#}|F^+), \pi_1(C^E,v))$ where $C$ is the connected, core subgraph of $G$ representing $F_\fc(r^+)$ and $C^E$ denotes the one-point union at the terminal endpoint of $E$ of $C$ and an edge labeled $E$ (so that $C^E$ immerses to $G$). See our running example on page~\pageref{ex.j}  where  $\fe_2$ is \contractible, ${\fe_3}$ is \cyclic, $ \fe_1$ is \llarge\ and 
\begin{align*}
\sH_{\fe_1\in\fc}(\phi)&=\{[\langle a, a^b\rangle,  \langle a, b\rangle^c]\}\\
\sH_{\fe_2\in\fc}(\phi)&=\{[\langle a, b\rangle,  \langle a, b\rangle^{d^{-1}e}], [\langle a, b\rangle,  \langle a, b\rangle^{e^{-1}d}]\}\\
\sH_{\fe_3\in\fc}(\phi)&=\{[a^{-1}, \langle a, b, c\rangle^{p^{-1}q}], [a, \langle a, b, c\rangle^{p^{-1}q}]\}
\end{align*}

\subsection{Algebraic limit lines}\label{s:algebraic limit lines}
\begin{itemize}
\item
$\{\sH_\fc(L):L\in\accnr(\phi)\}$ is the set of algebraic limit lines.
\end{itemize}
 The set of algebraic limit lines is an element of $\sequn(\Agoodsubgroupsubgroup\sqcup\Agoodsubgroupelement\sqcup\Agoodelementsubgroup\sqcup\Agoodelementelement)\subset\ecat(\Abig)$.  As in the case of added lines, to compute algebraic limit lines, we only need to compute $\accnr(\phi)$ from our \ct\ $\fG$. This is done in Section~\ref{s:limit lines}; see Corollary~\ref{cor:limit lines}. Referring to our running example (see page~\pageref{ex.g}), $\accnr(\phi)=\accnr(r_q)=\{a^\infty R_c, a^\infty ba^\infty\}$ and so $\sH_\fc(\phi)=\{[a^{-1},\langle a, b\rangle^c], [a^{-1},a^b]\}$.

\subsection{Naturality}
We will need the following naturality statements.

\begin{lemma} \label{consistency} 
  Suppose $\fe\in\fc$ and  $\Theta \in\theta \in \Out(F_n)$. Then:
\begin{enumerate}

\item
$\theta(\{\sHsub(L):L\in\accnr(\phi)\})= \{\sH_{\phi^\theta, \theta(\fc)}(L'):L'\in\accnr(\phi^\theta)\}$

\item
$\theta(\sH_{\fe\in\fc}(\phi))= \sH_{\theta(\fe)\in\theta(\fc)}(\phi^\theta)$
\item
$  [\Fix(\Phi),a]   \leftrightarrow  \theta\big(  [\Fix(\Phi),a]\big) = ([\Fix(\Phi^\Theta), \Theta(a)]$ defines a bijection between the  algebraic strong axes for $\phi$ and the algebraic strong axes for $\phi^{\theta}$.
\end{enumerate}
\end{lemma}

\begin{proof}
\begin{align*}
\theta(\{\sH_{\fc}(L):L\in\accnr(\phi)\}) 
\overset{}=&\{\theta(\sH_{\fc}(L)):L\in\accnr(\phi)\})\\
\overset{Lemma~\ref{l:algebraic lines are natural}}=&\{\sH_{\theta(\fc)}(\theta(L))):L\in\accnr(\phi)\}\\
\overset{Cor~\ref{acc is natural}}=&\{\sH_{\theta(\fc)}(L')):L'\in\accnr(\phi^\theta)\}\\
\end{align*}

(2) if $\LW_\fe(\phi)$ is finite: 

\vspace{-20pt}\begin{align*}
\theta(\sH_{\fe\in\fc}(\phi))\overset{def}=&\theta(\{\sH_{\fc}(L):L\in\LW_\fe(\phi)\})\\
\overset{}=&\{\theta(\sH_{\fc}(L)):L\in\LW_\fe(\phi)\})\\ 
\overset{Lemma~\ref{l:algebraic lines are natural}}=&\{\sH_{\theta(\fc)}(\theta(L))):L\in\LW_\fe(\phi)\}\\
\overset{Cor~\ref{LW is natural}}=&\{\sH_{\theta(\fc)}(L')):L'\in\LW_{\theta(\fe)}(\phi^\theta)\}\\
\overset{def}=&\sH_{\theta(\fe)\in\theta(\fc)}(\phi^\theta)\\
\end{align*}

(2) if $\LW_\fe(\phi)$ is infinite: 

\vspace{-20pt}

\begin{align*}
\theta(\sH_{\fe\in\fc}(\phi))\overset{def}=&\theta(\{[\Fix(\Phi),F_\fc(\ti r^+)]\})\\
\overset{}=&\{[\Theta(\Fix(\Phi)),\Theta(F_\fc(\ti r^+))]\})\\ 
\overset{Lemma~\ref{l:special is natural}}=&\{[\Fix(\Phi^\Theta),F_{\theta(\fc)}(\Theta(\ti r^+))\}\\
\overset{def}=&\sH_{\theta(\fe)\in\theta(\fc)}(\phi^\theta)\\
\end{align*}

(3):  Lemmas~\ref{theta and axes}(2) and  ~\ref{first theta}(2) imply that $( \Phi,a) \leftrightarrow ( \Phi ^\Theta,\Theta(a))$ induces a bijection $ \sa(\phi, [a]) \leftrightarrow \sa(\psi,\theta([a]))$ and that the ranks of $\Fix(\Phi)$ and $\Fix(\Phi^\Theta)$ are equal.  (3) therefore follows from 
\vspace{-10pt}
\begin{align*}
\theta([\Fix(\Phi),a]) 
\overset{def}=&[\Theta(\Fix(\Phi)), \Theta(a)])\\
\overset{}=&[\Fix(\Phi^\Theta), \Theta(a)]\\
\end{align*}
\end{proof}

\subsection{The algebraic invariant of $\phi$ rel $\fc$}
In this subsection we collect our algebraic invariants into a single master algebraic invariant.

\begin{definition}\label{d:I} 
Fix a special chain $\fc$ for $\phi\in\upg$. 
\begin{enumerate}
\item\label{i:I}
The {\it algebraic invariant of $\phi$ rel $\fc$} is the element of $\ecat(\Abig
)$ that is the ordered set ${\sf I}_\fc(\phi)$  consisting of:
\begin{itemize}
\item
$\fc$
\item
$\Fix(\phi)$;
\item
$(\sH_{\fe\in\fc}(\phi):\fe\in\fc)$ where the special 1-edge extensions $\fe$ are ordered using $\fc$;
\item
$\{\sHsub(L) :L\in\accnr(\phi)\}$;
\item
$\A_{\both}(\phi)$; and 
\item
the set of algebraic strong axes for $\phi$.
\end{itemize}
The six elements in the ordered list ${\sf I}_\fc(\phi)$ are the {\it components of ${\sf I}_\fc(\phi)$}.
\item \label{i:J}
Order (non-canonically) the elements of the union of the six sets defining  ${\sf I}_\fc(\phi)$. The resulting element of $\ecat(\Abig)$ is denoted $\dotI$.
\end{enumerate}
\end{definition}

\begin{remark}\label{r:compute}
In light of Theorem~\ref{t:ct is algorithmic}, we have seen in this section that the set of special chains, ${\sf I}_\fc(\phi)$, and $\dotI$ can be computed. We stress that they   take values in $\ecat(\Abig)$, which satisfies property \MW\ by Lemma~\ref{l:hat MW}.
\end{remark}

 \section{Stabilizers of algebraic invariants }\label{s:applying wg} 
At this point, it is reasonably straightforward to reduce the conjugacy problem for $\upgn$ in $\Out(\f)$ to the problem of deciding whether $\phi, \psi \in \upgn$ with ${\sf I}_\fc(\phi) = {\sf I}_\fc(\psi)$  are conjugate by some $\theta$ in the stabilizer of  ${\sf I}_\fc(\phi)$.  We want a little more.  Namely, we want to restrict the set of potential conjugators  to those elements that stabilize ${\sf I}_\fc(\phi)$ and induce trivial permutations on the components of ${\sf I}_\fc(\phi)$.    Continuing  with the notation of the previous section, we make this precise as follows.

\begin{definition}[$\X_\fc(\phi)$]\label{d:X}
Let $\phi\in\upgn$, $\fc$ be a special chain for $\phi$, and let $\dotI$ be in Definition~\ref{d:I}\pref{i:J}. Let $\X_\fc(\phi)$ denote the stabilizer $\Out_{\dotI}(\f)$ of $\dotI$ in $\Out(\f)$. 
\end{definition}

Unravelling definitions, $\X_\fc(\phi)$ also has a description as the subgroup of $\Out(\f)$ fixing each element in the union of the following six sets.
\begin{enumerate}
\item
$\{[F] : [F]\in\F\in\fc\}$;
\item
$\Fix(\phi)$;
\item
$\cup_{\fe\in\fc}\sH_{\fe\in\fc}(\phi)$;
\item
$\{\sHsub(L) :L\in\accnr(\phi)\}$;
\item
$\A_\both(\phi)$; and
\item\label{i:last}
$\{[\Fix(\Phi),a] :[\Phi,a]\in\sa(\phi), \rk\Fix(\Phi)\ge 2\}$.
\end{enumerate}

\begin{remark} 
As noted in Definition~\ref{d:I}\pref{i:J}, the construction of $\dotI$ was non-canonical. That is, there were choices in its construction. Every choice has the same stabilizer and so $\X_\fc(\phi)$ is independent of choice.
\end{remark}

 In passing and for future use, we have the expected: 

\begin{lemma}\label{l:phi in X} 
$\phi\in\X_\fc(\phi)$
\end{lemma}

\begin{proof}
We have to check that $\phi$ fixes each of the sets (1--\ref{i:last}) above elementwise. In (1), (2), (5), and (6) this is immediate from definitions.   For set (4), this is because $\phi(L)=L$ for all $L\in  \accnr(\phi)$ \big(Corollary~\ref{cor:limit lines}\pref{item:endpoints in FixN}\big) and Lemma~\ref{l:algebraic lines are natural}. That the elements of set (3) are fixed follows from definitions,   Lemma~\ref{LW is natural} and   Lemma~\ref{consistency}(2).  \end{proof}

\begin{definition}\label{d:vf} 
A group $G$ {\it is of type }{\sf F} if it has a finite Eilenberg-MacLane space. $G$ {\it is of type} {\sf VF} if it has a finite index subgroup of type {\sf F}.
\end{definition}

\begin{proposition}\label{p:vf} 
The stabilizer $\Out_{\sf Y}(\f)$ of an element $\sf Y\in \ecat(\Abig)$ is of type {\sf VF}.
\end{proposition}

\begin{proof}
We saw in the proof of Lemma~\ref{l:hat MW} that there is a  map $\ecat(\Abig)\to\ecat(\catA_0)$ that \explicit\ (Notation~\ref{notn:explicit fibers}).   If $\overline{\sf Y}$ is the image of $\sf Y$, then $\Out_{\sf Y}(\f)$ has finite index in $\Out_{\overline{\sf Y}}(\f)$. By Corollary~\ref{c:finite index}, the subgroup $G$ of $\Out(\f)$ fixing each label of $\overline{\sf Y}$ has finite index in $\Out_{\overline{\sf Y}}(\f)$. Also, the subgroup $G$ has type $\mathsf{VF}$ by \cite[Theorem~1.1]{bfh:gersten}. $\Out_{\sf Y}(\f)$, being commensurate with a group of type $\mathsf{VF}$, also has type $\mathsf{VF}$.
\end{proof}

As usual naturality will be important.

\begin{lemma}\label{l:X is natural}
Suppose $\xi\in\Out(\f)$.
\begin{enumerate}
\item\label{i:X is natural}
$(\X_\fc(\phi))^\xi=\X_{\xi(\fc)}(\phi^\xi)$.
\item\label{i:I is natural}
$\xi(\sf I_\fc(\phi))=\sf I_{\xi(\fc)}(\phi^\xi)$
\end{enumerate}
\end{lemma}

\begin{proof}
\pref{i:X is natural}: By Lemma~\ref{l:special is natural}, $\xi(\fc)$ is special for $\phi^\xi$ and so the statement makes sense. The lemma follows easily from the naturality of the quantities appearing in Definition~\ref{d:X}. For example, we verify that if $\theta\in\X_{\fc}(\phi)$ then $\theta^\xi(\sH_{\phi^\xi,\xi(\fc)}(L)\big)=\sH_{\phi^\xi,\xi(\fc)}(L)$ for all $L\in  \accnr(\phi^\xi)$. Indeed, 
\begin{align*}\xi\theta\xi^{-1}\big(\sH_{\phi^\xi,\xi(\fc)}(L)\big)
&=\xi\theta\big(\sH_{\phi,\fc}(\xi^{-1}(L))\big)\\
&=\xi\big(\sH_{\phi,\fc}(\xi^{-1}(L))\big)\\
&=\sH_{\phi^\xi,\xi(\fc)}(L)
\end{align*}
where the first and third equalities use Lemma~\ref{l:algebraic lines are natural} and the second uses Corollary~\ref{acc is natural}. The remainder of the proof consists of similar checks and is left to the reader.

The proof of \pref{i:I is natural} is similar.
\end{proof}

We next reduce the proof of the main result (Theorem~\ref{t:main}) of this paper, i.e.\  the conjugacy problem for $\upgn$ in $\Out(\f)$, to the proof of Proposition~\ref{p:conjugacy in X} stated immediately below. 

\begin{prop}\label{p:conjugacy in X}  
  There is an algorithm that takes as input $\phi,\psi\in\upgn$ and a chain $\fc$ such that
\begin{itemize}
\item
$\fc$ is special for both $\phi$ and $\psi$ and
\item
${\sf I}_\fc(\phi)={\sf I}_\fc(\psi)$
\end{itemize}
and that outputs {\tt YES} or {\tt NO} depending whether or not there is $\theta\in\X_\fc(\phi)$ conjugating $\phi$ to $\psi$. Further, if {\tt YES} then such a $\theta$ is produced.
\end{prop}

Proposition~\ref{p:conjugacy in X} is proved in Sections~\ref{s:XD} and \ref{s:algorithm} below. 

\begin{lemma} \label{l:reduction} Proposition~\ref {p:conjugacy in X} implies Theorem~\ref{t:main}.  That is,  an algorithm that  satisfies the conclusions of Proposition~\ref{p:conjugacy in X} can be used to produce an  algorithm that     satisfies the conclusions of Theorem~\ref{t:main}.
\end{lemma}

\begin{proof}
Assume Proposition~\ref{p:conjugacy in X} holds and $\phi,\psi\in\upg(\f)$.

View the multiset $\sf I(\phi):=\{\sf I_\fc(\phi)\mid \fc$  is a special chain for $\phi$\} as an element of $\ecat(\Abig)$ as described in Section~\ref{s:special chains}. By Lemma~\ref{l:X is natural}\pref{i:I is natural}, $\sf I(\phi)$ is a conjugacy invariant of $\phi$. That is, if $\phi^\theta=\psi$ then $\theta(\sf I(\phi))=\sf I(\psi)$. We may compute $\sf I(\phi)$ and $\sf I(\psi)$; see Remark~\ref{r:compute}.

Since $\ecat(\Abig)$ satisfies property \MW\ \ (Lemma~\ref{l:hat MW}), we can algorithmically check if there is $\theta'\in\Out(\f)$ such that $\theta'(\sf I(\phi))=\sf I(\psi)$. If there is no such $\theta'$ then $\phi$ and $\psi$ are not conjugate; return {\tt NO}. If there is such $\theta'$, then one is produced by the \M-algorithm for $\ecat(\Abig)$. Note that $\phi$ and $\psi$ are conjugate in $\Out(\f)$ iff $\phi$ and $\psi':=\theta'^{-1}\psi \theta'$ are conjugate in $\Out(\f)$ 
 iff $\phi$ and $\psi'$ are conjugate in the stabilizer $G:=\Out_{\sf I(\phi)}(\f)$ of $\sf I(\phi)$.

$G$ acts by permutation on the set of labels of ${\sf I(\phi)}$. Let $G'<G$ denote the subgroup fixing each label. By the \M-algorithm for $\ecat(\Abig)$, we may construct a finite presentation for $G$. Using our finite set of generators for $G$, we may construct the image $Q$ of $G$ in our permutation group. By Lemma~\ref{finite actions}, we can compute a finite set $\theta_i$ such that $G=\sqcup_i \theta_i G'$. Hence, $\phi$ and $\psi'$ are conjugate in $\Out(\f)$ iff $\phi$ and some $\theta_i^{-1}\psi'\theta_i$ are conjugate in $\Out(\f)$ iff $\phi$ and some $\theta_i^{-1}\psi'\theta_i$ are conjugate in $G'$. Since $G'<\X_\fc(\phi) < \Out(\f)$ for all $\fc$, $\phi$ and $\psi'$ are conjugate in $\Out(\f)$ iff $\phi$ and some $\theta_i^{-1}\psi'\theta_i$ are conjugate in $\X_\fc(\phi)$. We may use the supposed algorithm of Proposition~\ref{p:conjugacy in X} to decide whether or not this is the case and return a conjugator if it is.  The returned conjugator allows us to compute a conjugator for $\phi$ and $\psi$.
\end{proof}

\section{Staple Pairs} \label{sec:all staple pairs} 
\subsection{Limit lines $\accnr(\phi, \ti r) \subset \ti {\cal B}$} \label{sec:second limit lines}
    In Section~\ref{s:limit lines},  we associated a finite set $\accnr(r) \subset \cal B$ of $\phi$-invariant non-periodic lines to each  $r \in \cR(\phi)$.    In this section we  associate, to each lift $\ti r$ of $r$,  a subset $\accnr(\phi, \ti r) \subset \ti \B$ of the full pre-image  of   $\accnr(r)$ and then establish properties of $\accnr(\phi,\ti r)$ that will be needed later in the paper.

\begin{definition}  \label{defn:lines in ti r}
Choose a marked graph $K$.  For each lift  $\ti r \in \partial F_n$  of   $r \in \cR(\phi)$, let $\Phi_{\ti r}$ be the  unique lift of $\phi$ that fixes $\ti r$ and let $\ti R \subset \ti K$ be a ray with terminal end  $\ti r$. If $\ti L$ is a lift of $L \in \accnr(r)$ then  $L$ is $\phi$-invariant by Corollary~\ref{cor:limit lines}\pref{item:ell decomposes} and so each $\Phi_{\ti r}^j(\ti L)$ is a translate of $\ti L$, say $\Phi_{\ti r}^j(\ti L) = T_j(\ti L)$,  for some unique $T_j$. Define  $\ti L$ to be in $\accnr(\phi, \ti r)$ if for every finite subpath $\ti \beta$ of $\ti L$ there exists $ J(\ti \beta)$ such that $T_j(\ti \beta) \subset\ti R$ (equivalently $\ti \beta \subset T_j^{-1}(\ti R)$), for all   $j\ge J(\ti \beta)$. 
 \end{definition} 
 
\begin{remark}
As defined, $\accnr(\phi, \ti r)$ depends on $\Phi_{\ti r}$ and hence on $\phi$, in contrast to $\accnr(r)$ which  is independent of $\phi$.
\end{remark}

\begin{lemma} \label{accnr by conjugacy} $\accnr(\phi,\ti r)$ is well defined and $\Phi_{\ti r}$-invariant.   Moreover, if   $\psi = \theta \phi \theta^{-1}$ for some $\theta \in \Out(F_n)$ and if $\Theta$ is a lift of $\theta$ then $\Theta(\accnr(\phi, \ti r)) = \accnr(\psi, \partial \Theta(\ti r))$.
\end{lemma} 

\proof
Replacing $\ti R$ by a subray does not change  $\accnr(\phi,\ti r)$.  Since any two rays with terminal end $\ti r$  share a common subray, $\accnr(\phi,\ti r)$ is independent of the choice of $\ti R$.

As defined above,   $\accnr(\phi,\ti r)$ depends on the marked graph $K$ so we write $\accnr(\phi,\ti r, K)$ to make this explicit. We will prove
\begin{description}
\item ($*$) $\Theta(\accnr(\phi,\ti r, K)) = \accnr( \theta \phi \theta^{-1},\Theta(\ti r), K')$ for any marked graphs $K$ and $K'$ and any  $\Theta \in \Aut(F_n)$ representing any $\theta \in \Out(F_n)$.
\end{description}
  Applied with $\Theta = $ identity, ($*$) proves that $\accnr(\phi,\ti r, K)$ is independent of $K$ and hence that $\accnr(\phi,\ti r)$ is well defined.  The moreover statement is equivalent to $(*)$ and  $\Phi_{\ti r}$-invariance of $\accnr(\phi,\ti r)$ is  an immediate consequence of the definitions.  Thus the proof of the lemma will be complete once we prove ($*$). 

Assume the notation of Definition~\ref{defn:lines in ti r}.  Let $\ti r' =  \Theta(\ti r)$ and  $\psi  = \theta \phi \theta^{-1}$;  note that $\Psi_{\ti r'} = \Theta \Phi_{\ti r} \Theta^{-1}$.  Choose a homotopy equivalence $g : K \to K'$  of marked graphs that represents $\theta$ when $\pi_1(K)$ and $\pi_1(K')$ are identified with $F_n$ via their markings.  Let $\ti g : \ti K \to \ti K'$ be the lift  satisfying $ \ti g | \partial \f =  \Theta | \partial \f$, let $\ti R' = \ti g_\#(\ti R) \subset \ti K'$, let $\ti L' = \Theta(\ti L) = \ti g_\#(\ti L)$ and let $T'_j : \ti K' \to \ti K'$ be the covering translation satisfying $ T'_j |\partial \f= (\Theta   T_j   \Theta^{-1})|\partial \f$.     
Then   
$${\Psi^j_{\ti r'}}(\ti L') \cap \ti R' =  {\Psi^j_{\ti r'}}(\ti g_\#(\ti L)) \cap  \ti g_\#(\ti R) = \ti g_\#({\Phi_{\ti r}^j}(\ti L)) \cap \ti g_\#(\ti R)$$
By \cite{co:bcc} (see also \cite[Lemma~3.1]{bfh:tits0}), there is a constant $C$, depending only on $g$,  such that  $\ti g_\#({\Phi_{\ti r}^j}(\ti L)) \cap \ti g_\#(\ti R)$ contains the subpath of  $\ti g_\#({\Phi_{\ti r}^j}(\ti L) \cap  \ti R)$ obtained by $C$-trimming  (i.e.\ removing the first and last $C$ edges) and so contains the subpath of  $\ti g_\#(T_j(\ti \beta)) = T'_j \ti g_\#(\ti \beta)$   obtained by $C$-trimming for any chosen $\ti\beta$ and all $j \ge J(\ti \beta)$.     Given  a finite subpath $\ti \beta'$ of $\ti L'$ choose a finite subpath $\ti \beta$ of $\ti L$ such that the $C$-trimmed subpath of $\ti g_\#(\ti \beta)$ contains $\ti \beta'$. Then ${\Psi^j_{\ti r'}}(\ti L') \cap \ti R' \supset T'_j (\ti \beta')$ for all $j \ge J(\ti \beta)$.  Letting $J(\ti \beta') = J(\ti \beta)$,   we conclude that $\ti L' \in \accnr( \phi,\ti r)$.  By symmetry, we have proved ($*$).
\endproof  

Our goal in the remainder of this subsection is to understand $\accnr(\phi,\ti r)$ from the \ct\ point of view.

\begin{notn}\label{notn:elli} Choose $r \in \cR(\phi)$ and a \ct\ $\fG$ representing $\phi$; let  $E \in \E_f$ correspond to $r$ as in Lemma~\ref{identifying Fix+}.    Following the proof of Corollary~\ref{cor:limit lines},   let 
\[R_E =E \cdot \rho_0 \cdot \sigma_1 \cdot \rho_1 \cdot \sigma_2 \cdot \ldots \]
be the {\em coarsened  complete splitting} of $R_E$ where each $\sigma_i$ is a single growing term in the complete splitting of $R_E$ and each $\rho_i$ is a  (possibly trivial) Nielsen path.     For future reference, note that if $f(E) =  E\cdot u$ then $Eu$ is an initial subpath of $R_E$ whose terminal endpoint is a splitting vertex in the complete splitting of $R_E$ and hence is contained in some $\rho_p$.  

Following Notation~\ref{notn: f infinity} and Lemma~\ref{paths converge},  define, for all $i \ge 1$, 
 \[R_i^- =  f_\#^{\infty}(\bar \sigma_{i}) \qquad \qquad R_i^+ = f_\#^{\infty}(\sigma_{i}) \qquad  \qquad\ell_i =   (R_i^-)^{-1}\rho_{i} (R_{i+1}^+)\] 
Choose a lift $\ti r$ of $r$, let $\Phi_{\ti r}$ be the automorphism representing $\phi$ that fixes $\ti r$ and let $\ti f : \ti G \to \ti G$ be the lift corresponding to $\Phi_{\ti r}$.    Let $\ti R_{\ti E}$ be the lift of $R_E$ whose terminal end converges to $\ti r$ and whose initial edge is denoted $\ti E$, let  
  \begin{equation*} \label{splitting}  \ti R_{\ti E} =\ti E \cdot \ti \rho_0 \cdot \ti \sigma_1 \cdot \ti \rho_1 \cdot \ti \sigma_2 \cdot \ldots \end{equation*}  be the induced decomposition and let $\ti \ell_i$ be the lift of $\ell_i $ in which $ \rho_{i}$ lifts to $\ti \rho_{i}$.    Thus \[\ti\ell_i =   (\ti {R_i}^-)^{-1}\ti \rho_{i} \ti R_{i+1}^+\] where $\ti \sigma_{i+1}$ and $\ti R_{i+1}^+$ have the same initial endpoint and likewise for $\ti \sigma_{i}^{-1}$ and $\ti R_{i}^-$. We say that lines $\ti \ell_1, \ti \ell_2,\ldots$ are {\em visible} in $\ti R_{\ti E}$.  
\end{notn}
 
\begin{lemma} \label{visible lines occur}  Assume Notation~\ref{notn:elli}. 
\begin{enumerate}
\item
Each $\ell_i \in \acc(r)$ (Definition~\ref{defn:limit line downstairs}).
\item \label{ti ell occurs in ti r}
If $\ell_i \in \accnr( r)$ then $\ti \ell_i \in \accnr(\phi,\ti r)$.
\end{enumerate}
 \end{lemma}

\proof
Item (1) follows from  Lemmas~\ref{two definitions} and \ref{adjacency term} applied with $\alpha = \sigma_i$ and $\beta=\rho_i\sigma_{i+1}$. 

When verifying that $\ti \ell_i$ satisfies Definition~\ref{defn:lines in ti r}, it suffices to consider finite subpaths $\ti  \beta =  \ti \mu^{-1}\ti  \rho_{i} \ti  \nu$  of $\ti \ell_i$ with projections   $ \beta =  \mu^{-1} \rho_{i}  \nu$ where $\mu$ is an initial segment of $R_i^- = f_\#^{\infty}(\bar \sigma_i)$ that is a concatenation of terms in the coarsened  complete splitting of $R_i^-$ and $\nu$ is an initial segment of $R_{i+1}^+ = f_\#^{\infty}( \sigma_{i+1})$ that is a concatenation of terms in the coarsened complete splitting of $R_{i+1}^+$.  It follows from the definition of $f_\#^\infty$ (Notation~\ref{notn: f infinity}) that for all sufficiently large $j$, the lift of $\rho_i$ to $\ti f^j_\#(\ti \rho_i)$ extends to a lift of $ \beta$  to  a path 
\[\ti \beta_j  \subset      \ti f^j_\#( \ti \sigma_ i) \cdot  \ti f^j_\#(\ti  \rho_{i}) \cdot  \ti f^j_\#(\ti \sigma_{i+1}) =  \ti f^j_\#(\ti\sigma_ i \cdot \ti \rho_{i}\cdot \ti \sigma_{i+1}) \subset \ti R_{\ti E} \] 
    Since $f^j_\#$ preserves $\rho_i, R_i^-$ and $R_{i+1}^+$ there is a   covering translation $T_j$ such that \[T_j(\ti \rho_i) = \ti f^j_\#(\ti \rho_i) \qquad \qquad T_j(\ti R_i^-) = \ti f^j_\#(\ti R_i^-) \qquad \qquad T_j(\ti R_{i+1} ^+) = \ti f^j_\#(\ti R_{i+1}^+)\]   and so \[T_j(\ti \ell_i) =  \ti f^j_\#(\ti \ell_i)\] From $T_j(\ti \rho_i) = \ti f^j_\#(\ti \rho_i)$ we conclude that  $T_j(\ti \beta) = \ti \beta_j$  and so $T_j(\ti \beta) \subset \ti R_{\ti E}$.  This completes the proof of \pref{ti ell occurs in ti r}.
\endproof

  Our next result is a weak   converse of Lemma~\ref{visible lines occur}\pref{ti ell occurs in ti r}, namely that if $\ti L \in \accnr(\phi,\ti r)$ then $\ti L$ is in the $\Phi_{\ti r}$-orbit of some $\ti \ell_i$.  

\begin{prop} \label{prop:accnr}
Assume Notation~\ref{notn:elli}.
\begin{enumerate}
\item
For each $\ti L \in \accnr(\phi,\ti r)$ there exists $K$ such that $\ti f^k_\#(\ti L) \in \{ \ti \ell_i\}$ for all $k \ge K$.   Moreover,  $\accnr(\phi,\ti r) = \bigcup \Phi^m_{\ti r}(\ti \ell_i)$ where the union varies over all $\ti \ell_i \in \accnr(\phi, \ti r)$ and all $m \in \Z$. 
\item
For each $L \in \accnr(r)$ there is a lift   $\ti L \in  \accnr(\phi, \ti r)$.
\end{enumerate}
\end{prop}
 
We delay the proof of Proposition~\ref{prop:accnr} for   two needed lemmas.

\begin{lemma} \label{bounded twistpath}  Given a \ct\ $\fG$, there exists $M \ge 1$ so that the following holds for each \twistpath\ $w$, each non-fixed edge $E$ and each $k \ge 0:$    If $|m| \ge M$ and $\alpha_0 =  w^m$ is a subpath of $f^k_\#(E)$ then $\alpha_0$ extends to a subpath $\alpha_1$ of  $f^k_\#(E)$ satisfying the following two properties.
\begin{enumerate}
\item \label{M prop1}  $\alpha_1 = E'w^{q}$ or $\alpha_1  =  w^{q}\bar E'$ for some $E' \in \lin_w(f)$.
\item  \label{M prop2}  $\alpha_1$ is not contained in any Nielsen subpath of $f^k_\#(E)$.
\end{enumerate}
\end{lemma}

\proof  Let us first note that if the conclusions of the lemma hold for a subpath of $\alpha_0 = w^m$ of the form $w^{t}$ then they also hold for $\alpha_0$.  We may therefore  shorten $\alpha_0$ whenever  it is convenient. After replacing $E$ by $\bar E$ if necessary, we may assume that $E \in \E_f \cup \lin(f)$.

Choose $M''> 0$ so that  if $w_1 \ne w_2$ are twist paths then $w_1^{M''}$ is not a subpath of $w_2^m$ for any $m \in \Z$.   Items \pref{M prop1} and \pref{M prop2} hold for $M = M''$  and  $E \in \lin(f)$.   We may therefore assume that $E \in \E_f$ and that there exists $M' \ge M''$ so that \pref{M prop1} and \pref{M prop2} hold for $M=M'$ and for all edges $E' \in \E_f$ with height   less than that of $E$.

There is a path $u$ with height less than that of $E$ and a  complete splitting \[u = \tau_1 \cdot \ldots \cdot \tau_s\] such that 
  \[f^k(E) = E \cdot u \cdot f_\#(u) \cdot \ldots \cdot f_\#^{k-1}(u)\] for all $k \ge 1$.  Assuming without loss that $M'$ is greater than the length of any $\tau_j$, choose $M_1 \ge sM'$.  
  
    As a special case, we prove the lemma when $|m|\ge M_1$  and when   $\alpha_0 =  w^m$ is contained in   some $f_\#^l(u)$.  In this case  there exists $1 \le j \le s$ and $|m'| \ge M'$ and a subpath $\alpha_0' = w^{m'}$ of $\alpha_0$  such that $\alpha_0' \subset f_\#^l(\tau_j)$ for some $1 \le l \le k-1$.  As observed above, we can replace $\alpha_0$ with $\alpha_0'$. Since   the length of $\tau_j$ is less than $M'$ and the length of  $f_\#^l(\tau_j)$ is at least $M'$, $\tau_j$   is not a Nielsen path and so is either exceptional or an edge $E' $ with height less than that of $ E$.
 If $\tau_j$ is exceptional then its linear edges must be in the family determined by $w$ and  we take $\alpha_1$ to be all of $\tau_j$ except for the terminal edge.   If $\tau_j = E'$, then the inductive hypothesis implies that $\alpha_0$ extends to a   subpath $\alpha_1$ of $f_\#^l(E')$ that is not contained in a Nielsen subpath of $f_\#^l(E')$ and that satisfies \pref{M prop1}.    The hard splitting property of a complete splitting (Lemma 4.11(2) of \cite{fh:recognition}) implies that an  \iNp\ in a completely split path is contained in a single term of that splitting.  Thus  $\alpha_1$ is not contained in a Nielsen subpath of $f_\#^k(E)$ and so \pref{M prop2} is satisfied and we have completed  the proof of the special case.
 
Now choose $M$ so large that if $|m| \ge M$ and $\alpha_0 =  w^m$ is a subpath of $f^k_\#(E)$ then there is a subpath $\alpha_0' = w^{m'}$ of $\alpha_0$   with $m' \ge M_1$ so that $\alpha_0' \subset f_\#^l(u)$ for some $l$.   The existence of $M$ follows from the fact that the length of  $ f_\#^l(u)$ goes to infinity with $l$.   Replacing $\alpha_0$ with $\alpha'_0$, we are reduced to the special case.
\endproof

We choose a \lq central\rq\ subpath  $ \tau_L $ of  $L \in \accnr(r)$ as follows.   By Corollary~\ref{cor:limit lines}, $L =  (R^-)^{-1}\cdot \rho \cdot R^+$  where $R^{\pm}$ satisfy 1(a), 1(b) or 1(c) of Lemma~\ref{paths converge}.   In all three cases we will choose $ \tau =   \tau_L   = \tau_-^{-1} \rho \tau_+$ where $\tau_\pm$ is an initial segment of $R^\pm$.    Let $M$ be the constant from Lemma~\ref{bounded twistpath}.  \begin{itemize}
\item In case 1(a),   $R^+ = R_{E'}$ for some $E' \in \E_f$ and we take $\tau_+ = E'$.  
\item  In case 1(b), $R^+ = E'w^{\pm \infty}$  for some   $E' \in \lin_w(f)$ and we take $\tau_+ = E'w^{\pm M}$. 
\item   In case 1(c), $R^+ = w^{\pm \infty}$ and  we take $\tau_+ =  w^{\pm M}$.  
\end{itemize}The subpath $\tau_-$ is defined symmetrically.

\begin{lemma} \label{all are visible} Assume the notation of Notation~\ref{notn:elli} and of the previous paragraph.   Suppose that $\widetilde{  \tau_L}  \subset \ti L$ is a lift of  $\tau_L \subset L$  and that $\widetilde{  \tau_L} \subset \ti R_{\ti E}$. Then $\ti L = \ti \ell_i$ for some $i$.  
\end{lemma}

\proof
As a first case, suppose that $\tau_+ =  E' \in \E_f$ and so $R^+  = R_{E'}$.   Then $\ti \tau_+$ is a term $\ti \sigma_{i+1}$  in the coarsened complete splitting of $\ti R_{\ti E}$ by Lemma~\ref{not crossed} and $ R^+ = f_\#^{\infty}( \sigma_{i+1})$ by Example~\ref{example:converges}.  There are three subcases to consider, the first  being that  $\tau_- = E''  \in \E_f$.  In this subcase, $\ti \tau_-^{-1}$ is also a term  $\ti \sigma_j$ in coarsened complete splitting.  Since $\ti \tau_-$ is separated from $\ti \sigma_{i+1} = \ti \tau_+$ by  the Nielsen subpath $\ti \rho$, we have $\ti \tau_-^{-1} = \ti \sigma_{i}$ and $\ti \rho = \ti \rho_i$.   Thus  $f_\#^{\infty}(\bar \sigma_{i}) =  R^-$     and  $\ti  L = \ti \ell_i$.
 
 The second subcase is that $\tau_- =  E''w^{\pm M}$ where $E'' \in \lin_w(f)$ .  By Lemma~\ref{bounded twistpath}\pref{M prop2},  $\tau_-^{-1}$ is not contained in a Nielsen subpath of $R_E$.  It follows that the    terminal edge $\ti {E''}^{-1}$ of $\ti \tau_-^{-1}$ is contained in a $\ti \sigma_j$ that is either a single edge or an exceptional path.  As in the previous subcase, $j =i$.   Also as in the previous subcase,   $\ti \rho = \ti \rho_i$, $f_\#^{\infty}(\bar \sigma_{i}) = \ R^-$ and $\ti \ell_i = \ti L$.  
 
 The third and final subcase is that $\tau_- = w^{\pm M}$.  Since $\ti \tau_-^{-1}$ is followed in $\ti R_{\ti E}$ by $\ti \rho \ti E'$, it is not contained in a subpath of $\ti R_{\ti E}$ of the form $\ti w^m  \ti E_1^{-1}$ where $E_1  \in \lin_w(f)$.  Items \pref{M prop1} and \pref{M prop2} of Lemma~\ref{bounded twistpath} imply that $\ti \sigma_i = \ti E_1$ where $E_1 \in \lin_w(f)$  and that $\ti \rho_i =  \ti w^t \ti \rho$ for some $t$.      Example~\ref{example:converges} implies that $f_\#^{\infty}(\bar \sigma_{i}) =  w^{\pm \infty} =  R^-$ and so $\ti \ell_i = \ti L$.   We have now completed the proof in the case that $\tau_+ =  E' \in \E_f$.  Symmetric arguments apply in the case that $\tau_- =  E' \in \E_f$.

 Our next case is that $\tau_+ =  E'w^{\pm M}$ where $E' \in \lin_w(f)$ and   $R^+  =  E'w^{\pm \infty}$.    Lemma~\ref{bounded twistpath}\pref{M prop2} implies that the initial edge $\ti E'$ of $\ti \tau_+ $ is not contained in a Nielsen subpath of $\ti R_{\ti E}$ and so is either equal to some $\ti \sigma_{i+1}$ or is the first edge in some $\ti \sigma_{i+1}$ that projects to an exceptional path.  In either case $ f_\#^{\infty}( \sigma_{i+1}) =  E' w^{\pm \infty} =  R^+$.  The remainder of the proof in this second case  is exactly the same as in the first case.  Symmetric arguments apply in the case that $\tau_- =  E'w^{\pm M}$ with $E' \in \lin_w(f)$.
 
 We are now reduced to the case that $\tau_+ = w_2^{\pm M}, \  R_+  =  w_2^{\pm \infty}, \tau_- = w_1^{\pm M}$ and $R_-  =  w_1^{\pm \infty}$.   Thus $\ti \tau = \ti w_1^{\mp M} \ti \rho \ti w_2^{\pm M}$.   If $w_1 = w_2$ then $\rho$ is not an iterate of $w_1 = w_2$ because  $L$ is not periodic.    Lemma~\ref{bounded twistpath}\pref{M prop1} implies that  $\ti \tau$ extends to a subpath $\ti E_1\ti  w_1^p \ti \rho \ti w_2^q \ti E_2^{-1}$ where $E_i \in \lin_{w_i}(f)$ and $p,q \in \Z$.  It follows that  $\ti E_1  = \ti \sigma_i$, $\ti \rho_i = \ti w_1^p\ti \rho\ti w_2^q$ and $ \ti E_2^{-1} =  \ti \sigma_{i+1}$ for some choice of $i$.  As in the previous cases,   $\ti \ell_i = \ti L$.
 \endproof
 
 \medspace
 \noindent\emph {Proof of Proposition~\ref{prop:accnr}:} \ \   The first statement of (1) follows from Lemma~\ref{all are visible} and the definition of $\accnr(\phi, \ti r)$.  The moreover statement of (1) follows from the first statement and $\Phi_{\ti r}$-invariance of $\accnr(\phi, \ti r)$. 
 
For (2), let $E \in \E_f$ correspond to $r$.  Let $\tau_L \subset L$ be as in Lemma~\ref{all are visible}.  Since $L \in \accnr(r)$, $\tau_L$  lifts to a subpath $\ti \tau_L \subset \ti R_{\ti E}$.    Lemma~\ref{all are visible}   implies that the lift of $\tau_L$ to $\ti \tau_L$ extends to a lift of $L$ to an element of $\accnr(\phi,\ti r)$.
\qed

\begin{lemma} \label{forward invariance} Continue with  Notation~\ref{notn:elli}.   \begin{enumerate}  
\item   For all $i \ge 1$, there exists $j = j(i)> i$ such that $\ti f_\#(\ti \ell_i) = \ti \ell_j$.  More precisely, there exists $j > i$ such that  $\ti f_\#(\ti \rho_{i}) \subset \ti \rho_{j}$ and  $\ti f_\#(\ti \ell_i) = \ti \ell_j$ and there is  a covering translation $T$ such that   $T( \ti \rho_i) = \ti f_\#(\ti \rho_{i}) \subset \ti \rho_j$  and $T(\ti \ell_i) = \ti \ell_j$.
\item The assignment $i \mapsto j(i)$ is order preserving and $j(1) > p$. 
\end{enumerate}
\end{lemma}

\proof It suffices to prove (2) and  the \lq more precisely\rq\ statement of (1).  We begin with the latter. Since $\ti f_\#(\ti \rho_i)$ is a Nielsen path that is a concatenation of terms in the complete splitting of $\ti R_{\ti E}$, there exists $j$ such that $\ti f_\#(\ti \rho_{i}) \subset \ti \rho_{j}$.   Let $T$ be the unique covering translation satisfying $T( \ti \rho_i) = \ti f_\#(\ti \rho_{i})$.  It suffices to prove that  $\ti f_\#(\ti \ell_i) = \ti \ell_j$ and $T(\ti \ell_i) = \ti \ell_j$.

From  $\ti f_\#(\ti \rho_{i}) \subset \ti \rho_{j}$ it follows that  
\[\ti \rho_{j} =\ti \alpha \cdot  \ti f_\#(\ti \rho_{i}) \cdot \ti \beta\]
  where $\ti \alpha$ and $\ti \beta$ are, possibly trivial, Nielsen paths.  Since $\ti \sigma_{i}$ and $\ti \sigma_{i+1}$ are growing,  
  \[\ti \sigma_j  \cdot \ti \rho_{j}\cdot  \ti \sigma_{j+1}  = \ti \sigma_j  \cdot \ti \alpha \cdot \ti f_\#(\ti \rho_{i})\cdot \ti \beta \cdot  \ti \sigma_{j+1} \subset \ti f_\#(\ti \sigma_i ) \cdot \ti f_\#(\ti \rho_{i})\cdot \ti f_\#(  \ti \sigma_{i+1})\]
By Lemma~\ref{paths converge}\pref{item:first growing} 
   \[R_{i+1}^+ =  f_\#^\infty( \sigma_{i+1}) =   \beta  f_\#^\infty( \sigma_{j+1})  = \beta \cdot R_{j+1}^+\] 
    and  
  \[R_i^- =  f_\#^\infty( \bar \sigma_{i}) =   \bar \alpha f_\#^\infty( \bar \sigma_{j}) = \bar \alpha \cdot R_j^-\]
which implies that 
\[\ti f_\#(\ti \ell_i) = \ti f_\#((\ti R_i^-)^{-1} \cdot \ti \rho_{i} \cdot \ti R_{i+1}^+) = (\ti R_j^-)^{-1} \cdot\ti \alpha \cdot \ti f_\#(\ti \rho_{i})\cdot  \ti \beta \cdot  \ti R^+_{j+1} = (\ti R_j^-)^{-1}  \cdot \ti\rho_{j} \cdot    \ti R^+_{j+1}  =  \ti \ell_j\]
This completes the proof that $\ti f_\#(\ti \ell_i) = \ti \ell_j$. 

The covering translation $T$ that carries $\ti \rho_i$ to $\ti f_\#(\ti \rho_{i}) \subset \ti \rho_j$ also carries $\ti R_{i+1}^+$ to $ \ti \beta \cdot \ti R_{j+1}^+$ and $R_i^-$ to $  \ti \alpha^{-1} \cdot \ti R_j^-$.  Thus $T(\ti \ell_i) = \ti \ell_j$.

Finally, note that $j(i+1) - j(i)$ is equal to the number of growing terms in $\ti f_\#( \ti\sigma_i)$.  This implies (2) and hence also $j(i) > j$. Since $\ti E \cdot \ti f_\#(\ti u) \cdot  \ti f_\#(\ti \rho_0) \cdot f_\#(\ti \sigma_1)  \cdot \ti f_\#(\ti \rho_1)$ is an initial segment of $\ti R_E$ and since $ \ti f_\#(\ti \rho_0) \subset \ti \rho_p$, it follows that $j(1) > p$.
\endproof

We conclude this subsection by  defining a total order on $\accnr(\phi,\ti r)$.

\begin{definition}  \label{prec}Continue with  Notation~\ref{notn:elli}.  Given distinct $\ti L_1,  \ti L_2 \in \accnr(\phi,\ti r)$,   choose $k\ge 0$ so that $\ti f^k_\#(\ti L_1) = \ti \ell_i$ and $\ti f^k_\#(\ti L_2) = \ti \ell_{j}$ for some $i \ne j$.  (The existence of $k$ is guaranteed by   Proposition~\ref{prop:accnr}.)   Define $\ti L_1 \prec \ti L_2$ if $i < j$.   
\end{definition}

\begin{lemma}  \label{lem:prec}$\prec$ is a well defined,   $\Phi_{\ti r}$-invariant total order on $\accnr(\phi, \ti r)$ that is independent of the choice of $\fG$ representing $\phi$.  Moreover,  if    $\psi = \theta \phi \theta^{-1}$ for some $\theta \in \Out(F_n)$ and if $\Theta$ is a lift of $\theta$ then $\Theta :  \accnr(\phi, \ti r) \to  \accnr(\psi, \Theta(\ti r))$ preserves $\prec$.
\end{lemma}

We delay the proof of Lemma~\ref{lem:prec} to state and prove a technical lemma that allows us to redefine $\prec$ with less dependence on the location of $\ti \rho_i$ and $\ti \rho_j$ in $\ti R_{\ti E}$.

\begin{lemma} \label{equivalent to prec}
Continue with  Notation~\ref{notn:elli}.  Suppose that  $\ti \ell_i$ and $\ti \ell_j$ are distinct non-periodic visible lines. For all $k \ge 0$, let $\ti \ell_{i_k} = \ti f^k_\#(\ti \ell_i)$ and  $\ti \ell_{j_k} = \ti f^k_\#(\ti \ell_j)$   and let   $\ti y_{i,k}$ and  $\ti y_{j,k}$ be the   terminal endpoints of $\ti \ell_{i_k} \cap \ti R_{\ti E}$  and $\ti \ell_{j_k} \cap \ti R_{\ti E}$  respectively.   Then $i < j$ if and only if the following two conditions are satisfied.
\begin{enumerate}
\item
$\ti \ell_{i} \not \in \accnr(\phi,\partial_+ \ti \ell_j)$. 
\item
One of the following is satisfied.
\begin{enumerate}
\item 
$\ti \ell_{j}  \in \accnr(\phi, \partial_+ \ti \ell_{i})$
\item
$y_{j,k} - y_{i,k} \to \infty$ where $y_{j,k} - y_{i,k}$  $ = \pm$ the number of edges in the subpath connecting $  \ti y_{i,k}$ to $\ti  y_{j,k}$ and the sign is $+$ if and only if  $\ti y_{,jk} >  \ti  y_{i,k}$ in the orientation on $\ti R_{\ti E}$.
\end{enumerate}
\end{enumerate}
\end{lemma}

\proof
For the only if direction, assume that $i < j$. Lemma~\ref{forward invariance}, and an obvious induction argument  imply that  $i_k < j_k$ and that there are  unique covering translations  $T_k$  satisfying 
     $$T_k(\ti \rho_i) \subset \ti \rho_{i_{k}} \qquad \text{and} \qquad T_k(\ti\ell_i) = \ti f^k_\#(\ti \ell_i) = \ti  \ell_{i_k}$$     and $S_k$ satisfying 
     $$S_k(\ti \rho_j) \subset \ti \rho_{j_{k}} \qquad \text{and} \qquad S_k(\ti\ell_j) = \ti f^k_\#(\ti \ell_j) = \ti  \ell_{j_k}$$
Note that  $\ti h_k := S_k^{-1} \ti f^k$ is the lift of $f^k$ that preserves $\ti \ell_j$ and  so corresponds to  the automorphism  $\Phi_{\partial_+ \ti \ell_j}^k$.  Note also that $S_k^{-1}T_k(\ti \ell_i) = S_k^{-1}\ti f^k_\#(\ti \ell_i) = (\ti h_k)_\#(\ti \ell_i)$ and that $T_k(\ti \rho_i)  \subset \ti \rho_{i_k}$ is disjoint from $\ti \rho_{j_k}\ti R^+_{j_k+1}$.  The latter implies that  $S_k^{-1}T_k(\ti \rho_i)$ is disjoint from $\ti \rho_j \ti R_{j+1}^+$. It now follows from the definition that $\ti \ell_i \not \in \accnr(\phi, \partial_+ \ti \ell_j)$.  This completes the proof of  (1).  

For (2), we assume that $\ti \ell_{j} \not  \in \accnr(\phi, \partial_+ \ti \ell_{i})$ and prove that $\ti y_{j,k} - \ti y_{i,k} \to \infty$.    Continuing with the above notation,  $\ti g_k := T^{-1}_k\ti f^k$ corresponds to $\Phi^k_{\partial_+\ti \ell_i}$ and $T_k^{-1}S_k(\ti \ell_j)=\ti g_k(\ti \ell_j)$.  We claim that  there is a finite subpath  $\ti \beta \subset \ti \ell_j$    so that  for all $k \ge 0$, $T_k^{-1}S_k(\ti \beta) \not \subset \ti R_{i+1}^+$ and hence $S_k(\ti \beta) \not \subset T_k (\ti R_{i+1}^+)$.  If $\ell_j \not \in {\accnr}(\partial_+\ell_i)$ then this follows from Definition~\ref{defn:limit line downstairs}  and the fact that $T_k^{-1}S_k(\ti \ell_j)$ is a lift of $\ell_j$.   If $\ell_j \in {\accnr}(\partial_+\ell_i)$ then this follows from Lemma~\ref{all are visible} and Lemma~\ref{visible lines occur}.
      
On the other hand, $S_k(\ti \beta) \subset \ti R_{\ti E}$ for all sufficiently large $k$.  It follows that $\ti y_{i,k}$ precedes the terminal endpoint of $S_k(\ti \beta)$ in $\ti R_{\ti E}$.  Since  the number of edges in $S_k (\ti R_{j+1}^+) \cap \ti R_{\ti E}$ goes to infinity with $k$,  $\ti y_{j,k} - \ti y_{i,k} \to \infty$.

For the  if direction, we assume that $j < i$ and  we prove that either (1) or    (2) fails.  From the only if direction we know that (1) with $i$ and $j$ reversed is satisfied.  Thus (2a) fails.   Similarly, either (2a) or (2b), with the roles of $i$ and $j$ reversed is satisfied.  If the former holds then (1) fails and we are done.  Suppose then that (2b) with the roles of $i$ and $j$ reversed is satisfied.   Then $\ti y_{i,k} - \ti y_{j,k} \to \infty$ so (2b) fails. 
\endproof

\medskip
\noindent\emph{Proof of Lemma~\ref{lem:prec}:}  To make the dependence of $\prec$ on $\fG$ explicit we will write $\prec_f$.   Lemma~\ref{forward invariance} implies that $\prec_f$ is a well defined,   $\Phi_{\ti r}$-invariant total order on $\accnr(\phi, \ti r)$.  

Suppose that $\theta, \psi$ and $\Theta$ are as in the moreover statement, that $f' :G' \to G'$ is a \ct\ representing $\psi$, that $g : G \to G'$ is a homotopy equivalence representing $\theta$ and that $\ti g : \ti G \to \ti G'$ is the lift corresponding to $\Theta$.    Letting $\ti r' = \Theta(\ti r)$, we have $\Theta \Phi_{\ti r} = \Psi_{\ti r'} \Theta$.  By Lemma~\ref{accnr by conjugacy}, $\Theta(\accnr(\phi, \ti r)) = \accnr(\psi, \ti r')$. Let $\ti f'$ be the lift of $f'$ corresponding to $\Psi_{\ti r'}$.    

Given  $\ti L_1,  \ti L_2 \in \accnr(\phi, \ti r)$ such that $\ti L_1 \prec_f \ti L_2$,  we must show that    $\ti L_1' \prec_{f'} \ti L_2'$ where $\ti L_1'= \Theta(\ti L_1) = \ti g_\#(\ti L_1) $ and  $\ti L_2'= \Theta(\ti L_2) = \ti g_\#(\ti L_2)$.   We may replace $\ti L_1$ and $\ti L_2$   with $\Phi^k_{\ti r} \ti L_1$ and $\Phi^k_{\ti r} \ti L_2$ for any $k \ge 1$.   This follows from the $\Phi_{\ti r}$-invariance of $\prec_f$, the $\Psi_{\ti r'}$-invariance of $\prec_{f'}$ and the fact that  $\Theta \Phi_{\ti r} = \Psi_{\ti r'} \Theta$.    In particular, we may assume that there exists $i < j$ and $i' \ne j'$ such that $\ti L_1 = \ti \ell_i, \ti L_2 = \ti \ell_j, \ti L'_1 = \ti \ell'_i$ and $\ti L_2' = \ti \ell'_j$ where the $\ti \ell'_i$ and $\ti \ell'_j$  are  visible lines determined for $\accnr(\psi, \ti r')$ defined with respect to $f'$.  

To prove that $i' < j'$,  and thereby complete the proof of the lemma, we will verify items (1) and (2) of Lemma~\ref{equivalent to prec} in the prime system, which we will call $(1)'$ and $(2)'$.   Items $(1)'$ and $(2a)'$ follow from (1), (2a)  and Lemma~\ref{accnr by conjugacy}. Item $(2b)'$  follows from (2b) and the bounded cancellation lemma applied to $g$. 
\qed
\vspace{.1in}

{We conclude this section with a result that will be used in Lemma~\ref{new rays from lines}.

\begin{lemma}\label{l:in F}  
Suppose that $F$ is a free factor, that $\ti r \in \partial F$ is a lift of $r \in \cR(\phi)$  and that   $[F]$ is $\phi$-invariant.  Then each endpoint of each $\ti \ell \in \accnr(\phi,\ti r)$ is contained in $\partial F$.
\end{lemma}

\proof
Choose a \ct\ $\fG$ representing $\phi$ in which $ F$ is realized by a component $C$ of a core filtration element $H$.  By assumption, there is a ray $R$  in $C$ with terminal end $r$.  For each $\ell \in \accnr(r)$,  each finite subpath of $\ell$ is contained in $C$  and hence $\ell$ is contained in $C$.  Let $\ti C$ be the unique lift of $C$ whose boundary contains $\ti r$ and note that $\partial \ti C = \partial F$. Let $\ti R \subset \ti C$ be the   lift of $R$ with terminal endpoint $\ti r$ and let $\Phi_{\ti r}$ be the  automorphism representing $\phi$ that fixes $\ti r$.  From  uniqueness of $\ti C$, it follows that $\partial \ti C$ is $\Phi_{\ti r}$-invariant.  For all sufficiently large $j$, $\Phi_{\ti r}^j(\ti \ell) \cap \ti R \ne \emptyset$.  Since $\ell \subset C$ and distinct lifts of $C$ are disjoint, $\Phi_{\ti r}^j(\ti \ell) \subset \ti C$.  It follows that  the endpoints of $\Phi_{\ti r}^j(\ti \ell)$,  and hence  the endpoints of   $\ti \ell $, are contained in $\partial F$.
\endproof

\subsection{Topmost lines,  translation numbers and offset numbers}  \label{sec:topmost}
We continue with Notation~\ref{notn:elli} and with the partial orders $<$ on $\cR(\phi)$ and $\E_f$ given in Notation~\ref{weaker po} and Lemma~\ref{preserves partial order}.    

\begin{definition}\label{defn:topmost}An element  $L \in \accnr( r)$ is {\em $\phi$-topmost} if one of the following mutually exclusive properties is satisfied for the partial order $<$  on $\cR(\phi)$.
\begin{enumerate}
\item   $r$ is minimal in the partial order $<$.
\item  \label{item:topmost 2}  $L$ has an end $r_1 \in \cR(\phi)$ such that   $r_1<_c r$.  
\end{enumerate}
If $\ti L \in \accnr(\phi,\ti  r)$ projects to a $\phi$-topmost element of $\accnr(r)$  then $\ti L$ is a {\em topmost} element of $\accnr(\phi, \ti r)$.  Let ${\cal T}_{\phi,\ti r}$ be the {\em set   of  topmost elements of $\accnr(\phi, \ti r)$}. 
\end{definition}

\begin{lemma} \label{topmost is invariant}   ${\cal T}_{\phi,\ti r}$ is non-empty and $\Phi_{\ti r}$-invariant.\end{lemma}

\proof 
 Lemma~\ref{accnr by conjugacy} implies that $\accnr(\phi, \ti r)$ is $\Phi_{\ti r}$-invariant.  Since each element of $\cR(\phi)$ is $\phi$-invariant,  $\Phi_{\ti r}$-invariance of ${\cal T}_{\phi,\ti r}$  follows from the definitions.  If $r$ is minimal with respect to $<$ then every element of $\accnr(\phi, \ti r)$ is topmost and we are done.  Otherwise,  apply Lemma~\ref{preserves partial order} to choose $E'\in \E_f$ such that $E' <_c E$.  Either $E'$ or $\bar E'$ occurs as a term $\sigma_j$  in the coarsened complete splitting of $R_E$.  In the former case, $\ti  \ell_{j-1}$ is topmost in $\accnr(\phi, \ti r)$; in the latter case $\ti \ell_j$ is topmost in $\accnr(\phi, \ti r)$.  
\endproof

\begin{lemma}
There is an algorithm that lists the $\phi$-topmost elements of ${\accnr( r)}$.
\end{lemma}

\proof  Recall that the elements of ${\accnr( r)}$ can be enumerated by Lemma~\ref {cor:limit lines}\pref{item:algorithmic} and that the partial order on $\accnr(r)$ can be computed  by Notation~\ref{weaker po}.
 If $r$ is minimal then every element of ${\accnr( r)}$ is topmost.  Otherwise inspect the elements of ${\accnr( r)}$ to see which satisfy \ref{defn:topmost}.\pref{item:topmost 2}.
\endproof

Recall {from Notation~\ref{notn:elli}} that $p$ is chosen so that $\ti f_\#(\rho_0) \subset \ti \rho_p$.
\begin{lemma}  \label{finding topmost lines}     Each   $\ti L \in {\cal T}_{\phi,\ti r}$ is in the $\Phi_{\ti r}$-orbit of $\ti \ell_j$ for some $1 \le j \le p$.     \end{lemma}

\proof   By  Proposition~\ref{prop:accnr} and Lemma~\ref{topmost is invariant}, we may assume that $\ti L = \ti \ell_i$ for some $i > p$.  By Lemma~\ref{forward invariance}, it suffices to show that there exist $1 \le j \le p$ and $k \ge 1$ such that $\ti f^k_\#(\ti \rho_j) \subset \ti \rho_i$.   If this fails then there exists $1 \le j' \le p$ and $k' \ge 1$ such that $\ti \rho_i$ separates $\ti f^{k'}_\#(\ti \rho_{j'-1})$ from $\ti f^{k'}_\#(\ti \rho_{j'})$.  Assuming this we argue to a contradiction by showing that neither (1) nor (2) in Definition~\ref{defn:topmost} is satisfied.  First note that $\ti \sigma_i \ti \rho_i \ti \sigma_{i+1} \subset \ti f^{k'}_\#(\ti \sigma_{j'})$.  It follows that $\sigma_{j'}$ is not a linear term and so $\sigma_{j'} = E'$ or $\bar E'$ for some  $E'\in \E_f$. Since $E' < E$,  (1) is not satisfied.  If  an end $r''$ of $\ti \ell_i$ corresponds to an element $E'' \in \E_f$ then $E'' < E' < E$  and so (2)   is not satisfied.
\endproof

\begin{notn}   \label{notn:translation number}
The total order $\prec$ on $\accnr(\phi, \ti r)$ given in Definition~\ref{prec}   induces a total  order (also called) $\prec$ on ${\cal T}_{\phi,\ti r}$. Let $\ti L_1,\ldots,\ti L_{\tau(\phi,\ti r)}$ be, in order, the elements of $\{\ti\ell_1,\ldots,\ti \ell_p\} \cap {\cal T}_{\phi,\ti r}$. For $k \in \Z$ and $1 \le j \le \tau(\phi,\ti r)$, define $\ti L_{j+k\tau(\phi, \ti r)} = \Phi_{\ti r}^k(\ti L_j)$.
 
The following lemma allows us to change our notation from $\tau(\phi,\ti r)$ to $\tau(\phi,r)$.
\end{notn}

\begin{lemma}\label{lem:depends on r} $\tau(\phi,\ti r)$ depends only on $\phi$ and $r$ and not on the choice of $\ti r$.
\end{lemma}

\proof  The definition of $\tau(\phi,\ti r)$  uses the lift $\ti f :\ti G \to \ti G$  corresponding to $\Phi_{\ti r}$, the lift $\ti R_{\ti E}$ of $R_E$ whose terminal endpoint is $\ti r$, the lines $\{\ti \ell_i\} $ determined by $\ti R_{\ti E}$ as described in Notation~\ref{notn:elli} and the integer $p$, which depends only on $E$ and $f$.  If $a \in F_n$ and $T_a : \ti G \to \ti G$ is the corresponding covering translation, then the data associated to  $\ti r' = a  \ti r$ is $\ti f' = T_a \ti f T_a^{-1}$, $\ti R_{\ti E'} = T_a R_{\ti E}$, $\ti \ell_i' = T_a \ti \ell_i$ and $p$.  Since $\ti \ell_i$ and $\ti \ell_i'$  are lifts of the same line,  $\ti \ell_i \in {\cal T}_{\phi,\ti r}\Longleftrightarrow\ti \ell_i' \in   {\cal T}_{\phi,\ti r'}$.  This proves that $\tau(\phi,\ti r) =  \tau(\phi,\ti r')$ as desired.
\endproof

\begin{lemma} \label{lemma:translation} With notation as above:
\begin{enumerate}
\item  $s \mapsto \ti L_s$ defines an order preserving bijection between $\Z$ and  ${\cal T}_{\phi,\ti r}$.
\item  \label{item:translation}   ${\Phi_{\ti r}}(\ti L_s) = \ti L_{s+\tau(\phi,r)}$ for all $s$
\item  \label{item:visible}  $\ti L_s$ is visible if and only if $s \ge 1$.   
\end{enumerate}
\end{lemma}
 
\proof  The map $s \mapsto \ti L_s$ is surjective by Lemma~\ref{finding topmost lines} and is order preserving (and hence injective) because $\ti f_\#$ preserves $\prec$ and because $\ti L_1\prec \ti L_2 \prec \ldots\prec \ti L_{\tau(\phi,r)} \prec \ti f_\#(\ti L_1)$ where the last inequality follows {Lemma~\ref{forward invariance}, which implies that} $ \ti f_\#(\ti L_1) = \ti \ell_j$ for some $j > p$.  Item \pref{item:translation} follows from  the definitions.  Item \pref{item:visible} follows from Lemma~\ref{forward invariance}.  
 \endproof
 
 For the next lemma, we must choose a \ct\ $f': G' \to G'$ representing $\psi$ and then define ${\cal T}_{\psi, \ti r'}$  and $\tau(\psi,r')$ with respect to $f' : G' \to G'$.

 \begin{lemma}  \label{translation}  Suppose that  $\theta \in \Out(F_n)$ conjugates $\phi$ to $\psi$,  that $\theta(r) = r' \in \cR(\psi)$, that $\ti r, \ti r' \in \partial F_n$ represent $r$ and $r'$ respectively  and that $\Theta$ is  the lift of $\theta$ such that $\Theta(\ti r) = \ti r'$. Then:
\begin{enumerate}
\item \label{item:same translation}  $\tau(\phi,r) = \tau(\psi,r')$. 
 \item \label{item:offset} There is an integer $\offset(\theta,r)$ such that $ \Theta(\ti L_s) =\ti L'_{s + \offset(\theta,r)}  $ for all $s$. 
\end{enumerate}
\end{lemma}

\proof  Lemmas~\ref{accnr by conjugacy} and \ref{lem:prec}  imply that $\Theta $ induces a  $\prec$-preserving bijection between $\accnr(\phi,\ti r)$ and $\accnr(\psi,\ti r')$.  Lemma~\ref{weaker po conjugacy} implies that this bijection restricts to a bijection between  ${\cal T}_{\phi,\ti r}$ and  ${\cal T}_{\psi,\ti r'}$.  Since the only order preserving bijections of $\Z$ are translations, there is an integer $\offset(\theta,\ti r, \ti r')$ such that $ \Theta(\ti L_s) =\ti L'_{s + \offset(\theta,\ti r, \ti r')}  $ for all $s$.    If we replace $\ti r$ by another lift $\ti r^* = i_a \ti r$ then $\Theta$ is replaced by $\Theta^* = \Theta i_a^{-1}$ and $\ti L_i \in {\cal T}_{\phi,\ti r}$ is replaced by $\ti L^*_i = i_a \ti L_i  \in {\cal T}_{\phi,\ti r^*}$;  see the proof of Lemma~\ref{lem:depends on r}.  It follows that  $\Theta^*(L_i^*) = \Theta(\ti L_i)$ and hence that $\offset(\theta,\ti r, \ti r')$ is independent of the choice of lift $\ti r$.  The symmetric argument implies that $\offset(\theta,\ti r, \ti r')$ is also independent of the choice of $\ti r'$.  This completes the proof of \pref{item:offset}.

  Item \pref{item:same translation} therefore follows from
$$\ti L'_{s+  \tau(\phi,r)+\offset(\theta,r) } = \Theta \Phi_{\ti r}\ti L_s =  \Psi_{\ti r'}\Theta\ti L_s= \ti L'_{s+ \offset(\theta,r) + \tau'(\psi,r')}$$
\endproof

\begin{remark}
The bijection between $\Z$ and  ${\cal T}_{\phi,\ti r}$ depends on the notion of visible lines and so depends on the choice of \ct.  On the other hand,  Lemma~\ref{lemma:translation}\pref{item:translation} implies that  $\tau(\phi,r)$ depends only on $\phi$ and $r$ and not the choice  of a \ct. As such it can be computed from any \ct\ for $\phi$.  
The integer $\offset(\theta, r)$  depends on the choices of \ct s. 
\end{remark}

\subsection{ Staple pairs}\label{sec:staples}   
We continue with  Notation~\ref{notn:elli}.  We set further notation as follows.

\begin{notn}
If $E_i$ and  $E_j $ are distinct elements of $ \lin_w(f)$ then there exist non-zero $d_i \ne d_j$ so that $f(E_i) = E_i w^{d_i}$ and $f(E_j) = E_j w^{d_j}$.  Recall that a path of the form $E_iw^p\bar  E_j$ is called {\em exceptional} if $d_i$ and $d_j$ have the same sign.   If $d_i$ and $d_j$ have different signs then we say $E_iw^p\bar  E_j$  is {\em quasi-exceptional}.
\end{notn}

\begin{notn} \label{notn:staple pairs} 
We write $L \in \cS(\phi)$ and say that $L$ is a {\em staple} if  $L \in \accnr(\phi)$ has at least one periodic end; if both ends of $L$ are periodic then $L$ is a {\em linear staple}.
If $\ti L \in \accnr(\phi, \ti r)$ projects to an element of $\cS(\phi)$ for  $r \in \cR(\phi)$ and lift $\ti r$, then we write $\ti L \in \cS(\phi, \ti r)$ and $L \in \cS(\phi, r)$ and we say that $L$ and $\ti L$ {\em occur} in $ r$ and $\ti r$ respectively.

For each  $r \in \cR(\phi)$, an ordered pair  $b = ( L_1,  L_2)$ of elements of $\cS(\phi,r)$ is a {\em staple pair}  if there are   lifts $\ti L_1, \ti L_2 \in \cS(\phi,\ti r)$ and  a  periodic line $\ti A$ such that $\{\partial_+\ti L_1, \partial_- \ti L_2\}  \subset \{\partial_- \ti A, \partial_+ \ti A\}$. We write $b \in \cS_2(\phi, r)$ and $\ti b = (\ti L_1,\ti L_2) \in \cS_2(\phi, \ti r)$ and say that $b$ and $\ti b$ {\em occur} in $r$ and $\ti r$ respectively and that $\ti A$ is the {\em common axis} of $\ti b$.   By Corollary~\ref{cor:limit lines},   $\ti A$ corresponds to an element of $\A(\phi)$.   Define $\cS_2(\phi)= \cup \cS_2(\phi,r)$ where the union is taken over all $r \in \cR(\phi)$.
\end{notn}

\begin{lemma} \label{lem:S2 invariance} Each  $b \in \cS_2(\phi, r)$ is $\phi$-invariant.   The set $\cS_2(\phi, \ti r)$ is $\Phi_{\ti r}$-invariant.
\end{lemma}

\proof The first statement follows from the second and the fact (Lemma~\ref{cor:limit lines}\pref{item:ell decomposes}) that each element of $\accnr(\phi)$ is $\phi$-invariant. The second follows from the $\Phi_{\ti r}$-invariance of $\accnr(\phi, \ti r)$ (Lemma~\ref{accnr by conjugacy}) and the definition of $\cS_2(\phi, \ti r)$. 
\endproof

\begin{excont*}\label{ex.l}
In our example, $\cS(\phi)=\{a^\infty R_c, a^\infty ba^\infty\}$ and $\cS_2(\phi)=\{(a^\infty b a^\infty, a^\infty b a^\infty), (a^\infty b a^\infty, a^\infty R_c)\}$.
\end{excont*}

Throughout this section, $M$ is the stabilization constant defined in Notation~\ref{notn:stabilization}.    

Our next  lemma explains how staple pairs occur in an eigenray $\ti R_{\ti E}$.

\begin{lemma}  \label{staple pair examples}
Assume Notation~\ref{notn:elli}.   
\begin{enumerate}
\item \label{qe} If  $\sigma_i \rho_i \sigma_{i+1}$ is quasi-exceptional then 
 $  (\ti \ell_{i-1},\ti \ell_{i+1}) \in \cS_2(\phi, \ti r)$ with common axis  $\ti \ell_i$. 
  \item \label{three cases} If   one of the following hold 
 \begin{enumerate} 
 \item \label{exc}   $\sigma_i $ is exceptional;
 \item \label{lin} $\sigma_i \in \lin_{{w}}(f)$ and $\ell_{i}$ is not periodic;  
 \item  \label{bar lin}  $ \bar \sigma_{i} \in \lin_{{w}}(f)$ and $\ell_{i-1}$ is not periodic;
\end{enumerate}
then $  (\ti \ell_{i-1},\ti \ell_{i}) \in \cS_2(\phi, \ti r)$. 
\item \label{must be qe}If $\ti \ell_i$ is periodic and neither $\sigma_i$ nor $\bar \sigma_{i+1}$ is in $\E_f$  then $\sigma_i \rho_i \sigma_{i+1}$ is quasi-exceptional and so $  (\ti \ell_{i-1},\ti \ell_{i+1}) \in \cS_2(\phi, \ti r)$ with common axis  $\ti \ell_i$.  See also Remark~\ref{rem:periodic line}.
\item \label{all b}   For each $\ti b\in \cS_2(\phi, \ti r)$ there exists $K =  K(\ti b)$ such that $\Phi_{\ti r}^k(\ti b)$ is as in \pref{qe} or \pref{three cases} for all $k \ge K$.  {Moreover, in case \pref{lin}, $R_i^- = w^{\pm \infty}$ and in case \pref{bar lin}, $R_{i-1}^+ =  w^{\pm \infty}$.
}
\end{enumerate}
\end{lemma}

\proof
If  $\sigma_i \rho_i \sigma_{i+1}$ is quasi-exceptional then there are a twist curve $w$ and edges $E', E'' \in \lin_w(f)$ such that $ \sigma_i =  E'$, $\rho_i = w^q$ for some $q\in \Z$ and $\sigma_{i+1} =  \bar E''$.
Moreover, $f(E') = E'w^{d'}$ and $f(E'') = E''w^{d''}$ where $d'$ and $d''$ have opposite signs.  If $\ti \sigma_i = \ti E'$,   let  $\ti w$ be the lift of $w$ that begins with the terminal endpoint $\ti x$ of $\ti E'$.  Extend $\ti w$ to a periodic line $\ti A$ that projects bi-infinitely to the circuit determined by $w$ and is oriented consistently with   $\ti w$ . Let $\ti y \in \ti A$ be the terminal endpoint of the lift of $w^q$  that begins at $\ti x$ and let $\ti E''$ be the lift of $E''$ that ends at $\ti y$ .   
  Then  $\ti R^+_i $ is the concatenation of $\ti E'$ and a ray in $\ti A$ beginning at $\ti x$ and terminating at $\partial_+\ti A$ if $d' > 0$ and at $\partial_-\ti A$ if $d' < 0$.  Similarly, $\ti R^-_{i+1} $ is the concatenation of $\ti E''$ and a ray in $\ti A$ beginning at $\ti y$ and terminating at $\partial_+ \ti A$ if $d'' > 0$ and at $\partial_-\ti A$ if $d'' < 0$.  Neither $\ti \ell_{i-1} $ nor $\ti \ell_{i+1}$ is periodic.  Up to a change of orientation, $\ti \ell_i = \ti A$.   Thus  $  (\ti \ell_{i-1},\ti \ell_{i+1}) \in  \cS(\phi, \ti r)$ with common axis  $\ti \ell_i$ and (1) is proved. 
  
 If $\sigma_i $ is exceptional then $\sigma_i =E'w^q\bar E''$ where $E', w$ and $E''$ are as above except that $d'$ and $d''$ have the same sign.  Following the above notation, $\ti R^+_i $ begins with $\ti E'$, $\ti R^-_i$ begins with $\ti E''$ and both rays terminate at the same endpoint of $\ti A$. Neither $\ti \ell_{i-1} $ nor $\ti \ell_{i}$ is periodic. This completes the proof of (2a).
 
 If $\sigma_i = E' \in \lin_{{w}}(f)$, then following the above notation, $\ti \ell_{i-1}$ is   non-periodic (because it crosses $\ti E'$) with terminal endpoint in $\{\partial_-\ti A, \partial_+\ti A\}$   and $ \ti R^-_i $ has terminal endpoint in $\{\partial_-\ti A, \partial_+\ti A\}$.  If $\ti \ell_{i}$ is non-periodic then  $  (\ti \ell_{i-1},\ti \ell_{i}) \in  \cS_2(\phi, \ti r)$.  This completes the proof of (2b).    The proof of (2c) is similar.   
 
 Suppose  that $\ti \ell_i$ is as in (3).   If  $\bar \sigma_i \in \E_f$ then $ \ti R^-_i$ is not asymptotic to a periodic line in contradiction to the assumption that $\ti \ell_i$ is periodic.     If  $\bar \sigma_i \in \lin(f)$ or if $\sigma_i$ is exceptional then $ R^-_i = E_iw^{\pm\infty}$ where $E_i \in \lin_w(f)$, again in contradiction to the assumption that $\ti \ell_i$ is periodic.  We conclude that  $ \sigma_i = E' \in \lin(f)$.  The symmetric argument shows that $  \sigma_{i+1} = \bar E'' $ for some $E'' \in \lin(f)$.   Thus $  \ell_i$ has the form $(w')^{\pm \infty} \rho_i (w'')^{\pm \infty}$ where $w'$ is the \twistpath\ for $E'$ and $w''$ is the \twistpath\ for $E''$. Since $\ti \ell_i$ is a periodic line, $w' = w''$ and $\rho_i = (w')^q$ for some $q \in \Z$. This proves that $\sigma_i \rho_i \sigma_{i+1}$ is quasi-exceptional, which in conjunction with (1), completes the proof of (3).

For (4), suppose that $\ti b \in \cS_2(\phi, \ti r)$.   After replacing $\ti b$ with some $\Phi_{\ti r}^k(\ti b)$, we may assume by Proposition~\ref{prop:accnr} that $\ti b = (\ti \ell_{i-1},\ti \ell_j)$ for some {$i -1 \ne j$}. After replacing $\ti b$ with  $\Phi_{\ti r}^M(\ti b)$, we may assume that $\bar \sigma_i \not \in \E_f$.  (To see this note that if $\ti f_\#^M(\ti \rho_{i-1}) \subset \ti \rho_{s-1}$ then $\ti f_\#^M(\ti \ell_{i-1}) = \ti \ell_{s-1}$ and  $\ti \sigma_s$ is the first growing term of $\ti f_\#^M(\ti \sigma_i)$.)  By assumption, $\partial_+\ti R^+_i$  is an endpoint of the common axis $\ti A$ of $\ti b$.  Lemma~\ref{paths converge} therefore implies that $\sigma_i \not \in \E_f$ and hence that $\sigma_i$ is linear. In other words, $\sigma_i = E_i$ or $\sigma_i = \bar E_i$ or $\sigma_i =  E_i w_i^* \bar E_l$ for some \twistpath\ $w_i$ and for some $E_i, E_l \in \lin_{w_i}(f)$.  In all three cases, the terminal endpoint of $\ti E_i$ is contained in $\ti A$.      For the same reasons, we may assume that $\sigma_j = E_j$ or $\sigma_j = \bar E_j$ or $\sigma_j =  E_m w_j^* \bar E_j$ for some \twistpath\ $w_j$ and for some $E_j, E_m \in \lin_{w_j}(f)$; moreover,     the terminal endpoint of $\ti E_j$ is in $\ti A$. 

The proof now proceeds by a case analysis.  If $\sigma_i =  E_i w_i^* \bar E_l$ then the  midpoint of $\ti E_i$ [resp. $\ti E_l^{-1}$] separates $\ti A$ from $\ti \sigma_q$ for all $q< i$ [resp. $q >i]$ so  $j = i$ and we are in case \pref{exc}. The same argument, with the same conclusion, applies if $\sigma_j =  E_m w_j^* \bar E_j$.    
 We may now assume that $\sigma_i$ is either $E_i$ or $\bar E_i$ and that $\sigma_j$ is either $E_j$ or $\bar E_j$.  By considering the midpoints of $\sigma_i$ and $\sigma_j$ as in the previous case we see that:
 \begin{enumerate}[(a)]
 \item  If $\sigma_i = E_i$ then $j \ge i$.
 \item If  $\sigma_i = \bar E_i$ then $j \le i$.
 \item If $\sigma_j = E_j$ then $i \ge j$.
  \item If  $\sigma_j = \bar E_j$ then $i\le j$.
\end{enumerate}
If (a) and (c) are satisfied then   $j = i$,   we are in case \pref{lin} and $R_i^- = w^{\pm \infty}$.  Similarly,  if (b) and (d) are satisfied then  $j = i$,  we are in case \pref{bar lin}  and $R_{i-1}^+ = w^{\pm \infty}$.  Suppose next that (a) and (d) are satisfied.  In this case, $j \ge i$,  $w_i =  w_j$ and  the interval $\ti \tau$ of $\ti A$ bounded by the terminal endpoints of $\ti E_i$ and $\ti E_j$ equals $\ti w_i^q$ for some $q \in \Z$;  in particular, $\tau$  is a Nielsen path.   It must be that $\tau =  \rho_i$ and $j=i+1$, which is \pref{qe}.  Finally, suppose that (b) and (c) are satisfied.  Then $j \le i$,  $w_i =  w_j$ and  the interval $\ti \tau$ of $\ti A$ bounded by the terminal endpoints of $\ti E_j$ and $\ti E_i$ equals $\ti w_i^q$ for some $q \in \Z$;  in particular, $\tau$  is a Nielsen path.   It must be that $\tau =  \rho_{i-1}$ and $j=i-1$ which contradicts the fact that $i-1 \ne j$.  Thus this last case does not happen and we are done.
\endproof
    
\begin{notn} \label{notn:visible staples} We say that the staple pairs  $(\ti \ell_{i-1},\ti \ell_{i+1})$ and  $(\ti \ell_{i-1},\ti \ell_{i})$   that occur in items (1) and (2) of Lemma~\ref{staple pair examples} are {\em visible with index $i$} or just {\em visible} if the index is not explicitly given. Note that if $\ti b$ is visible then $\Phi_{\ti r}^k(\ti b)$ is visible for all $k \ge 0$.  \end{notn}
    
\begin{corollary}
The set of visible elements of $\cS_2(\phi, \ti r)$ is infinite.
\end{corollary} 

\proof
From $\Phi_{\ti r}$-invariance of $\cS_2(\phi, \ti r)$ (Lemma~\ref{lem:S2 invariance}) and  Lemma~\ref    {forward invariance},  we need only show that  $\cS_2(\phi, \ti r)$ contains a visible element.  There are always linear edges crossed by   $R_E$.  We are therefore reduced, by   Lemma~\ref{staple pair examples}, to the case that  some $\ti \ell_i$ is periodic.    If     $\ti f^M_\#(\ti \rho_i) \subset \ti \rho_j$ then $\ti \sigma_{j}$ is the last growing term in $\ti f^M_\#(\ti \sigma_i) $ and so $\sigma_j \not \in\E_f$.  Similarly, $\ti \sigma_{j+1}$ is the first growing term in $\ti f^M_\#(\ti \sigma_{i+1})$ and so $\bar \sigma_{j+1} \not \in\E_f$.    Lemma~\ref{staple pair examples}\pref{must be qe} implies that    $ (\ti \ell_{j-1}, \ti \ell_{j+1}) \in \cS_2(\phi, \ti r)$ and we are done. 
\endproof

Recall from Notation~\ref{notn:elli} that the $\ti \ell_i$'s are said to be {\it visible}.

\begin{lemma}  \label{visible soon}Suppose that  $\ti b = (\ti L_1,\ti L_2) \in \cS_2(\phi, \ti r)$ with common axis $\ti A$ and that one of the following two conditions are satisfied.
\begin{enumerate} [(a)]
\item Either $\ti L_1$  or  $\ti L_2$ is visible and there exist $k\ge 0$  such that    $\Phi^k_{\ti r}(\ti L_1,\ti L_2) = (\ti \ell_{i-1}, \ti \ell_i)$ for some $i$.
\item  Either $\ti L_1, \ti L_2$ {or the common axis} of  $\ti A$ is visible and there exist $k\ge 0$  such that    $\Phi^k_{\ti r}(\ti L_1,\ti L_2) = (\ti \ell_{i-1}, \ti \ell_{i+1})$ with common axis $\ti \ell_i$ for some $i$.
\end{enumerate}
Then    $\Phi^{2M}_{\ti r}(\ti b)$ is visible.
\end{lemma}
 
 \proof  We begin by establishing the following properties for each {visible line} $\ti \ell_j$. 
  \begin{enumerate}
 \item Suppose that $\partial_+\ti \ell_j$ is periodic and that $\ti \ell_s = \Phi_{\ti r}^{M}(\ti \ell_j)$.  Then  $\Phi_{\ti r}^{m}(\ti \ell_s)$ and  $\Phi_{\ti r}^{m}(\ti \ell_{s+1})$
are consecutive  (i.e. their indices differ by $1$) for all $m \ge 0$.
\item Suppose that $\partial_-\ti \ell_j$ is periodic and that $\ti \ell_s =  \Phi_{\ti r}^{M}(\ti \ell_j)$.  Then  $\Phi_{\ti r}^{m}(\ti \ell_{s-1})$ and  $\Phi_{\ti r}^{m}(\ti \ell_{s})$
are consecutive for all $m \ge 0$.
   \end{enumerate}
For (1), Lemma~\ref{forward invariance} implies that  $\ti f_\#^{M}(\ti \rho_{j}) \subset \ti \rho_{s}$ and our choice of $M$ implies that $\ti \sigma_{s+1} \not \in \E_f^{-1}$.  Since    $\partial_+\ti \ell_j$ is periodic, the same is true for $\partial_+\ti \ell_s$ and so  $\ti \sigma_{s+1} \not \in \E_f$.  We conclude that $\sigma_{s+1}$ is linear.  In particular,    $\ti f^{m}_\#(\ti \sigma_{s+1})$ has exactly one growing term.   If $\ti f^{m}_\#(\ti \rho_{s}) \subset \ti \rho_a$  then $\ti f^{m}_\#(\ti \rho_{s+1}) \subset \ti \rho_{a+1}$.  Lemma~\ref{forward invariance} implies that $\Phi_{\ti r}^{m}(\ti \ell_s) =\ti \ell_a$ and  $\Phi_{\ti r}^{m}(\ti \ell_{s+1}) = \ti \ell_{a+1}$.   This completes the proof of (1).  Item (2) is proved by the symmetric argument.

  We now apply  (1) and (2) to prove the lemma, assuming without loss that $k > M$.  In case (a),   we will show that      $\Phi^{M}_{\ti r}(\ti b)$ is visible. 
      If $\ti L_1$ is visible  let $  \ti \ell_{s-1} = \Phi_{\ti r}^M(\ti L_1)$.  Since $\Phi_{\ti r}^{k-M} (\ti \ell_{s-1})  = \Phi_{\ti r}^{k} (\ti L_1) = \ti \ell_{i-1}$, (1), applied with $m = k-M$,  implies that  $\Phi_{\ti r}^{k-M} (\ti \ell_{s}) =  \ti \ell_{i}$.  Since  $\Phi_{\ti r}^{k-M}(\Phi_{\ti r}^M(\ti L_2)) = \ti \ell_i$, we have      $\Phi_{\ti r}^M(\ti L_2) = \ti \ell_s$.  Thus $\Phi_{\ti r}^M(\ti b) = (\Phi_{\ti r}^M(\ti L_1), \Phi_{\ti r}^M(\ti L_2)) = (\ti \ell_{s-1}, \ti \ell_s)$ is visible.  This completes the proof when $\ti L_1$ is visible.  When  $\ti L_2$ is visible, a symmetric argument, using (2) instead of (1) shows that    $\Phi_{\ti r}^M(\ti L_1)$  and hence $\Phi_{\ti r}^M(\ti b)$ is visible.

In case (b), {note that } $\partial_+\ti L_1, \partial_-\ti L_2$ and both ends of $\ti A$ are periodic.   If   $\ti L_1$ is visible then   the above argument shows that the common axis of $\Phi_{\ti r}^M(\ti b)$ is visible and a second application shows that $\Phi_{\ti r}^{2M}(\ti L_2)$ is visible. The other cases are similar.
\endproof 

\begin{notn} \label{notn:topmost staple pair}
Suppose that $\ti b = (\ti L_1,\ti L_2) \in \cS_2(\phi, \ti r) $ projects to  $b \in \cS_2(\phi, r)$. If $\ti \ell_j \prec \ti L_1$ (see Definition~\ref{prec}) then we write  $\ti \ell_j \prec \ti b$.  We say that $b$  and $\ti b $ are {\em topmost} elements of $\cS_2(\phi,r)$ and $\cS_2(\phi, \ti r)$ respectively if for all $r_1 < r$ (see Definition~\ref{weaker po}) neither $b$ nor $b^{-1} :=  (L_2^{-1}, L_1^{-1})$ is an element of $\cS_2(\phi, r_1)$.  Since $\ti b$ and $\Phi_{\ti r}(\ti b)$ project to the same element of $\cS_2(\phi, r)$ and since $\cS_2(\phi, r)$ is $\Phi_{\ti r}$-invariant,
 it follows that the set of topmost element of $\cS_2(\phi, \ti r)$ is $\Phi_{\ti r}$-invariant.  
\end{notn}

\begin{lemma} \label{new bound for staple pairs}    The   set of topmost elements of $\cS_2(\phi, \ti r)$ is the union of a finite number of $\Phi_{\ti r}$-orbits.      Moreover, there exists a computable $B(r)> 0$ so that each of these orbits has a visible representative     with index at most $B(r)$.
 \end{lemma}
 
\proof
Define $B(r) > 2M$ by $\Phi^{2M}_{\ti r}(\ti \ell_p) = \ti \ell_{B(r)}$.
 
 Suppose that $\ti b$ is a topmost element of $\cS_2(\phi, \ti r)$.  After replacing $\ti b$ with some $\Phi^{k}_{\ti r}(\ti b)$, we may assume  by Lemma~\ref{staple pair examples}\pref{all b}  that  $\ti b = (\ti \ell_{i-1}, \ti \ell_i)$ or $\ti b = (\ti \ell_{i-1}, \ti \ell_{i+1})$ with common axis $\ti l_i$.  We consider the  $\ti b = (\ti \ell_{i-1}, \ti \ell_i)$ case first, assuming without loss that $i > 2M$.   The proof {below} is similar to that of Lemma~\ref{finding topmost lines}.

Suppose that  there exists $1 \le j' \le p$ and $k' > 0$ so that $\ti \sigma_{i-1} \ti \rho_{i-1} \ti \sigma_{i} \ti \rho_{i}\ti \sigma_{i+1} \subset \ti f^{k'}_\#( \ti \sigma_{j'})$. Then $\sigma_{j'} = E'$ or $\bar E'$ for some $E' \in \E_f$ so either $ \sigma_{i-1}   \rho_{i-1}   \sigma_{i}  \rho_{i} \sigma_{i+1}$ or $ \bar \sigma_{i+1}  \bar  \rho_{i}  \bar \sigma_{i} \bar\rho_{i-1} \bar\sigma_{i-1}$ occurs as a concatenation of terms in the coarsening of the complete splitting of $R_{E'}$.  Letting $r' \in \cR(\phi)$ correspond to $E'$, Lemma~\ref{staple pair examples}  implies that either  $b$ or $b^{-1}$ is an element of $\cS_2(\phi,r')$  in contradiction to the assumption that $b$ is topmost in $\cS_2(\phi, r)$.  Thus no such $j'$ and $k'$ exist. It follows  that there exists $1 \le j \le p$ and $k> 0$ such that $\ti f^{k}_\#(\ti \rho_{j})$ is contained in either $\ti \rho_{i-1}$ or $\ti \rho_{i}$.    Equivalently, $\Phi_{\ti r}^{k}(\ti \ell_{j})$ is equal to either $\ti \ell_{i-1}$ or $\ti \ell_i$. Since $\ti \ell_j$ is one of the lines comprising the pair $ \Phi_{\ti r}^{-k}(\ti b)$, Lemma~\ref{visible soon} implies that $\Phi_{\ti r}^{2M-k} (\ti b) =  \Phi_{\ti r}^{2M}(\Phi_{\ti r}^{-k}(\ti b)) $ is visible with index at most $B(r)$. 

In the remaining case, $\ti b = (\ti \ell_{i-1}, \ti \ell_{i+1})$ with common axis $\ti \ell_i$.  Arguing as in the first case, we conclude that there exists $k \ge 0$ and $1 \le j \le p$ such that $\Phi_{\ti r}^k(\ti \ell_j)$ is equal to either $\ti \ell_{i-1}$ or $\ti \ell_i$ or $\ti \ell_{i+1}$.  The proof then concludes as in the first case.
\endproof

\begin{remark}  \label{visible staple pairs} The $\Phi_{\ti r}$-image  of a visible topmost staple pair is a visible topmost staple pair.  It follows that if a topmost staple pair $\ti b$ occurs in $\ti r$ and if   $\ti \ell_{B(r)} \prec \ti b$ then $\ti b$ is visible. \end{remark}   

\begin{remark}
The set of topmost elements of $\cS_2(\phi, r)$ could be empty.
\end{remark}  
 
 \begin{lemma}\label{lower rays and staple pairs}  If $\hat r < r$ and $b \in \cS_2(\phi,\hat r)$ then either $b \in \cS_2(\phi,r)$  or $b^{-1} \in \cS_2(\phi,r)$. 
 \end{lemma}

\proof
The proof is similar to that of Lemma~\ref{new bound for staple pairs}.  Let $\hat E$ be the higher order edge corresponding to $\hat r$, let $R_{\hat E} =  \hat \rho_0\cdot \hat \sigma_1 \cdot \hat \rho_1 \cdot \ldots$ be the  coarsening of the complete splitting into single growing terms and Nielsen paths and let $\hat \ell_1,\hat \ell_2,\ldots$ be the associated visible lines.     By Lemma \ref{staple pair examples}\pref{all b} and L  emma~\ref{lem:S2 invariance}, there exists $i \ge 1$ so that $b = (\hat \ell_{i-1}, \hat \ell_i)$ or $b = (\hat \ell_{i-1}, \hat \ell_{i+1})$.     Since $\hat r < r$, there exists $j > 1$ so that either $\sigma_j = \hat E$ or $\sigma_j = \hat E^{-1}$.  The cases are symmetric so we assume that  $\sigma_j = \hat E^{-1}$ and leave the $\sigma_j = \hat E$ case to the reader.  Since $  \ell_j \in \accnr( r)$,  the inverse of every finite subpath of $R_{\hat E}$ occurs as  subpath of $R_E$.  In particular, the inverse of $\hat \rho_{i-2} \cdot \hat \sigma_{i-1} \ldots \cdot \hat \sigma_{i+2} \hat \rho_{i+2}$ occurs as a concatenation of terms in $R_E$.   Lemma~\ref{staple pair examples} therefore implies that $b^{-1} \in \cS_2(\phi,r)$.
\endproof
 
\begin{lemma}  \label{topmost is natural}
Suppose that  $\theta \in \Out(F_n)$ conjugates $\phi$ to $\psi$,  that $\theta(r) = r' \in \cR(\psi)$, that $\ti r, \ti r' \in \partial F_n$ represent $r$ and $r'$ respectively  and that $\Theta$ is  the lift of $\theta$ such that $\Theta(\ti r) = \ti r'$.  Then $\Theta$ induces a bijection $\cS_2(\phi,\ti r) \mapsto\cS_2(\psi, \ti r')$ that restricts to a bijection on topmost elements.
\end{lemma}
  
\proof
This follows from Lemma~\ref{accnr by conjugacy},  which provides a bijection between $\accnr(\phi, \ti r)$ and $\accnr(\psi, \ti r')$, and the definitions.
\endproof

\begin{definition} \label{defn:m} Given $b = (L_1,L_2) \in \cS_2(\phi,r)$,  choose  lifts $\ti L_1, \ti L_2$ and  a  periodic line $\ti A$ such that $\{\partial_+\ti L_1, \partial_- \ti L_2\}  \subset \{\partial_- \ti A, \partial_+ \ti A\}$.  Orient $\ti A$ to be consistent with  the \twistpath\ $w$ to which it projects and  let $a \in F_n$ be the root-free element of $F_n$ that stabilizes $\ti A$ and satisfies $a^+ = \partial_+\ti A$.  Each $\theta \in \X_\fc(\phi)$ (Definition~\ref{d:X}) satisfies $\theta(\sHsub(L_i))=\sHsub(L_i)$ for $i=1,2$.   Lemma~\ref{l:unique lift} therefore implies that there  are unique $\Theta_i \in \theta$ such that $\Theta_i(\sHsub(\ti L_i))=\sHsub(\ti L_i)$.   Since both $\Theta_1$ and $\Theta_2$   represent $\theta$ and fix $a$ there exists $m_b(\theta) \in \Z$ such that $\Theta_1=i_a^{m_b(\theta)}\Theta_2$.
\end{definition}

\begin{excont*} \label{ex.m} See Figure~\ref{Figbarestaple}.
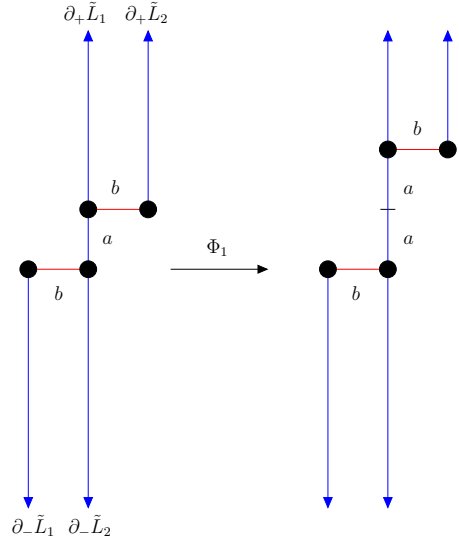
\begin{figure}[h!] 
\centering
\resizebox{6cm}{!}{\begin{tikzpicture}[y=-1cm]
\sf
\draw[semithick,red] (6.19125,12.54125) -- (7.46125,12.54125);
\draw[black] (7.3025,11.27125) -- (7.62,11.27125);
\draw[semithick,arrows=triangle 45-triangle 45,blue] (7.46125,17.62125) -- (7.46125,7.46125);
\draw[red] (7.46125,11.27125) -- (8.73125,11.27125);
\draw[arrows=-triangle 45,black] (9.2075,12.54125) -- (11.27125,12.54125);
\draw[semithick,arrows=-triangle 45,blue] (8.73125,11.27125) -- (8.73125,7.46125);
\draw[semithick,red] (12.54125,12.54125) -- (13.81125,12.54125);
\draw[red] (13.81125,10.00125) -- (15.08125,10.00125);
\draw[black] (13.6525,11.27125) -- (13.97,11.27125);
\draw[semithick,arrows=-triangle 45,blue] (15.08125,10.00125) -- (15.08125,7.46125);
\draw[semithick,arrows=-triangle 45,blue] (12.54125,12.54125) -- (12.54125,17.62125);
\draw[semithick,arrows=triangle 45-triangle 45,blue] (13.81125,17.62125) -- (13.81125,7.46125);
\draw[semithick,arrows=-triangle 45,blue] (6.19125,12.54125) -- (6.19125,17.62125);
\path (6.6,13.2) node[text=black,anchor=base west] {$b$};
\path (7.62,12) node[text=black,anchor=base west] {$a$};
\path (7.8,10.95375) node[text=black,anchor=base west] {$b$};
\path (14.2,9.68375) node[text=black,anchor=base west] {$b$};
\path (14,12) node[text=black,anchor=base west] {$a$};
\path (14,10.95375) node[text=black,anchor=base west] {$a$};
\path (12.9,13.2) node[text=black,anchor=base west] {$b$};
\filldraw[black] (6.19125,12.54125) circle (0.1778cm);
\filldraw[black] (7.46125,12.54125) circle (0.1778cm);
\filldraw[semithick,black] (8.73125,11.27125) circle (0.1778cm);
\filldraw[black] (7.46125,11.27125) circle (0.1778cm);
\filldraw[black] (12.54125,12.54125) circle (0.1778cm);
\filldraw[black] (13.81125,10.00125) circle (0.1778cm);
\filldraw[semithick,black] (15.08125,10.00125) circle (0.1778cm);
\filldraw[black] (13.81125,12.54125) circle (0.1778cm);
\node [above] at (10.2,12.4) {$\Phi_1$};
\node [above] at (7.4,7.5) {$\partial_+\tilde L_1$};
\node [above] at (8.7,7.5) {$\partial_+\tilde L_2$};
\node [below] at (7.5,17.6) {$\partial_-\tilde L_2$};
\node [below] at (6.3,17.6) {$\partial_-\tilde L_1$};
\end{tikzpicture}%

}
\caption{The graphs here are parts of $\ti G$ with each horizontal segment a lift of the edge $b$ and where vertical segments project into $a^\infty$. $\tilde L_1$ and $\tilde L_2$ are lifts of the staple $a^\infty ba^\infty$ and $(L_1,L_2)$ is a staple pair. $\Phi_1$ is the lift of $\phi$ that fixes $\ti L_1$. Intuitively $\ti L_2$ slides away from $\ti L_1$ under the action of $\Phi_1$ by 1 period and so ${m_{(L_1,L_2)}\big(\phi \big)=1}$}.\label{f:staple}
\label{Figbarestaple}
\end{figure}
\end{excont*}

\begin{lemma}  \label{m is a homomorphism}
For each $b=(L_1, L_2)\in \cS_2(\phi)$, \   $m_b : \X_\fc(\phi)  \to \Z$ is a well defined homomorphism.
\end{lemma}

\proof
We first check  that  $m_b(\theta)$ is independent of the choice of $\ti L_1$ and $\ti L_2$ and so is well defined.  Suppose that $\ti L_1'$ and $\ti L_2'$ are another choice with corresponding   $\ti A', a'$ and $\Theta_i'$.         Choose $c \in \f$ so that $i_c(a) = a'$.     For $i=1,2$,  $\ti L_i'$ and $i_c(\ti L_i)$ are lifts of $L_i$ with an endpoint in $\{{a'}^-, {a'}^+\}$   and so there exists $n_i$ such that \[\ti L_i' = i_{a'}^{n_i}i_c(\ti L_i) = i_ci_{a}^{n_i} (\ti L_i)\]   By uniqueness,   \[\Theta'_i = (i_ci_{a}^{n_i}) \Theta_i (i_ci_{a}^{n_i})^{-1} = i_ci_{a}^{n_i} \Theta_i i_{a}^{-n_i}i_c^{-1}= i_c \Theta_i i_c^{-1}\]
so \[\Theta'_1 {\Theta'}_2^{-1} = i_c \Theta_1\Theta_2^{-1}i_c^{-1} =i_c i_a^{m_b(\theta)} i_c^{-1} = i_{a'}^{m_b(\theta)}\]
as desired.

 To prove  that $m_b(\theta)$ defines a homomorphism, suppose that $\psi \in \X_{\fc}(\phi) $ and  $\Psi_i$ satisfies $\Psi_i(\sHsub(\ti L_i))=\sHsub(\ti L_i)$.   Then   $\Psi_i \Theta_i (\sHsub(\ti L_i))=\sHsub(\ti L_i)$ and $$ i_a^{m_b(\psi\theta)} =  \Psi_1 \Theta_1 \Theta_2^{-1}\Psi_1^{-1} =  \Psi_1  i_a^{m_b(\theta)} \Psi_2^{-1}=  \Psi_1   \Psi_2^{-1} i_a^{m_b(\theta)}=i_a^{m_b(\psi)}  i_a^{m_b(\theta)} =  i_a^{m_b(\psi)+ m_b(\theta)} $$ so 
$m_b(\psi\theta)=m_b(\psi)+ m_b(\theta)$.
\endproof

\begin{remark}   \label{rem:broader def of m}The same proof shows that $m_b$ defines a homomorphism on  both $\{\theta \in \Out(F_n):  \theta( L_i)  = L_i \text{ for }  i=1,2\}$  and  $\{\theta \in \Out(F_n):  \theta(\sHsub(L_i))=\sHsub(L_i ) \text{ for }  i=1,2\}$.  The former is the stabilizer of $b$ and the latter can be thought of as   the \lq weak stabilizer\rq\ of $b$.   
\end{remark}

The next lemma relates $m_b(\phi)$ to the twist coefficients of $\phi$.

\begin{lemma} \label{lem:m(phi)} Suppose that $b = (L_1,L_2) \in \cS_2(\phi,r)$ where $L_1 = ({R_1^-})^{-1} \rho_1 R_1^+$  and $L_2 = ({R_2^-})^{-1} \rho_2 R_2^+$  are the decompositions of Corollary~\ref{cor:limit lines}\pref{item:ell decomposes}.  Suppose also that $w$ is a \twistpath\ and that $E', E'' \in \lin_w(f)$ satisfy $f(E') = E'w^{d'}$ and $f(E'') = E''w^{d''}$. 
\begin{enumerate}
\item  If $R_1^+ = E' w^{\pm \infty}$ and $R_2^- = E'' w^{\pm \infty}$ then $m_{b} (\phi) = d'-d''$.
\item If  $R_1^+ = E' w^{\pm \infty}$ and $R_2^- =  w^{\pm \infty}$ then $m_{b} (\phi) = d'$.
\item If $R_1^+ = w^{\pm \infty}$ and $R_2^- = E'' w^{\pm \infty}$ then $m_{b} (\phi) = -d''$.
\end{enumerate} 
In particular, $m_b(\phi) \ne 0$ for   all $b \in \cS_2(\phi)$.
\end{lemma}

\proof  Choose lifts $\ti L_1 = ({\ti R_1^-})^{-1} \ti \rho_1 \ti R_1^+, \ \ti L_2 = ({\ti R_2^-})^{-1} \ti \rho_2 \ti R_2^+$ and $\ti A  =\ti w^\infty$ so that $\partial \ti R_1^+, \partial \ti R_2^{-} \in \{\partial_- \ti A, \partial_+ \ti A\}$.   Denote the initial endpoints of $\ti R_1^+$ and $\ti R_2^{-}$ by $\ti x$ and $\ti y$ respectively.  There exist  $\Phi_{1}, \Phi_{2} \in \phi$ such that $\Phi_{1}$ fixes the endpoints of $\ti L_1$ and $\Phi_{2}$ fixes the endpoints of $\ti L_2$.   The corresponding lifts  $\ti f_{1}$ and $\ti f_{2}$ fix $\ti x$ and $\ti y$ respectively.    In particular, $ \ti f_1 (\ti y) = i_a^{m_b(\phi)} \ti f_2(\ti y) =  i_a^{m_b(\phi)} \ti y$. In  case (1), the path $\ti \tau$ connecting $\ti x$ to $\ti y$ equals  $\ti E'\ti w^p (\ti  E'')^{-1}$ for some $p\in \Z$ and   $(\ti f_1)_\#(\ti E'\ti w^p (\ti  E'')^{-1}) = \ti E'\ti w^{p+d'-d''} (\ti  E'')^{-1}$.  It follows that $\ti f_1(\ti y) = i_a^{d' - d''} \ti y$ and hence that $m_b(\phi) = d' - d''$.  In case (2), $\ti \tau = \ti E'\ti w^p$ and $(\ti f_1)_\#(\ti E'\ti w^p) = \ti E'\ti w^{p+d'} $.  Thus, $\ti f_1(\ti y) = i_a^{d'} \ti y$ and $m_b(\phi) = d'$.  Case (3) is proved similarly.  
 
{ 
    Lemma~\ref{staple pair examples} implies that if $b$ is as in case (1) then either $b =  (\ell_{i-1},\ell_{i+1})$ where $\sigma_i \rho_i \sigma_{i+1}$ is quasi-exceptional or $b =  (\ell_{i-1},\ell_{i})$ where $\sigma_i $ is  exceptional.  In either case, $E' \ne E''$ so $d' \ne d''$. This completes the proof that $m_b(\phi) \ne 0$ and hence the proof of the lemma.
 }
\endproof

\subsection{Spanning Staple Pairs}\label{s:induction 1}
\def\ccs{coarsened complete splitting}
We continue with the notation of the preceding subsections{; in particular, see Notation~\ref{notn:elli}}.  In addition, we let   $\ti \ellprime_0, \ti \ellprime_1\dots$ be the sequence of lines obtained from $\ti \ell_0,\ti \ell_1,\dots$  by removing all periodic lines.  In other words,  $\ti \ellprime_0, \ti \ellprime_1\dots$ is the set of visible lines in $\accnr(\phi,\ti r)$.

If $E$ has quadratic growth (equivalently, each $\sigma_i$ is linear) then every $(\ti \ellprime_t,\ti \ellprime_{t+1})$ is an element of $\cS_2(\phi,\ti r)$ by Lemma~\ref{staple pair examples}; see Figure~\ref{f:staples}.
\begin{figure}[h!] 
\centering
\includegraphics[width=4cm,height=4cm,angle=0]{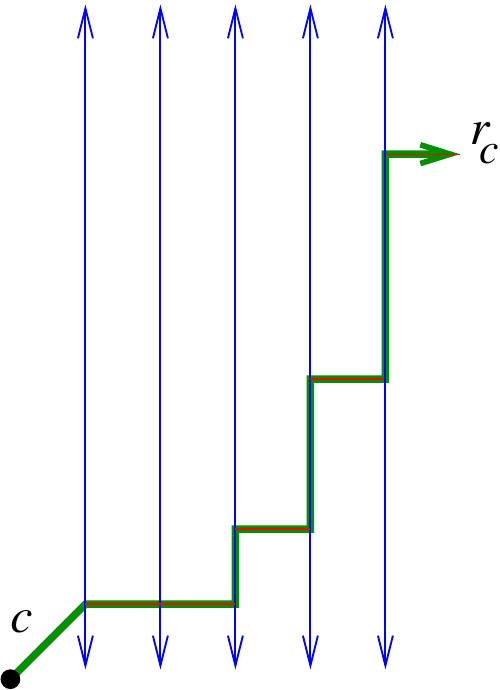}
\caption{An end of $R_c=cbbaba^2ba^3\dots$ is in a union of staples pairs, cf.\ Figure~\ref{f:staple}. The relation between staple pairs and $R_E$ for $E$ of higher than quadratic growth is more complicated.}
\label{f:staples}
\end{figure}
This is not the case when $E$ has higher order.   We now define a related but weaker property that does hold for every   $(\ti \ellprime_t,\ti \ellprime_{t+1})$.  Its utility is illustrated in    the proof of Lemma~\ref{new rays from lines} which is applied in the proof of Lemma~\ref{check any theta}.

\begin{definition}
We say that an ordered pair  $(\ti \eta_1,\ti \eta_2)$ of elements of $\accnr(\phi,\ti r)$ is {\em spanned by  a  staple pair $\ti b = (\ti L_1,\ti L_2) \in \cS_2 (\phi,\ti r)$} if the following two conditions are satisfied.
\begin{itemize}
\item  either $\ti L_1 = \ti \eta_1$ or $\ti L_1 \in \accnr(\phi,\partial_+\ti \eta_1)$.   
\item either $\ti L_2 = \ti \eta_2 $   or  $\ti L_2^{-1}  \in \accnr(\phi,\partial_-\ti \eta_2)$.
\end{itemize}
Note that if $(\ti \eta_1,\ti \eta_2)$ is spanned by a staple pair then  $(\Phi_{\ti r}^k\ti \eta_1,\Phi_{\ti r}^k\ti \eta_2)$ is spanned by a staple pair   for all $k \in \Z$.
\end{definition}

Our next result uses techniques from the proofs of  Lemma~\ref{equivalent to prec} and  Lemma~\ref{new bound for staple pairs}.

\begin{lemma}  \label{subpath} \begin{enumerate} \item  Suppose that $\sigma_i \in \E_f$ and that $\ti r_i = \partial_+ \ti \ell_{i-1} = \partial \ti R^+_{i}$.  If  $ \ti \sigma_j \ti \rho_j \ti \sigma_{j+1}$ is a subpath of $\ti f^m_\#(\ti \sigma_i)$ for $m > 0$  and if $\ti \ell_j$   is non-periodic  then ${\Phi^{-m}_{\ti r}}(\ti \ell_j) \in\accnr(\phi,\ti r_i)$.
\item  Suppose that $\sigma_i \in \E_f^{-1}$ and that $\ti r_i = \partial_- \ti  \ell_{i} = \partial \ti R_i^-$.  If $ \ti \sigma_j \ti \rho_j \ti \sigma_{j+1}$ is a subpath of $\ti f^m_\#(\ti \sigma_i)$ for $m > 0$   and if $\ti \ell_j$   is non-periodic  then ${\Phi^{-m}_{\ti r}}(\ti \ell_j^{-1}) \in \accnr(\phi,\ti r_i)$.
\end{enumerate}
\end{lemma}

\proof The two cases are symmetric so we prove (1) and leave (2) to the reader.    Assuming that $\sigma_i \in \E_f$, let $T : \ti G \to \ti G$ be the covering translation that carries   $\ti \sigma_i$ to the initial edge of $\ti f^m_\#(\ti \sigma_i)$ and hence satisfies $T(\ti R^+_i) =  \ti f^m_\#(\ti R^+_i) = {\Phi_{\ti r}^m}(\ti R^+_i)$.  Then $  T^{-1} \ti f^m$ is the lift of $f^m$ that preserves the terminal endpoint $\ti r_i$ of $\ti R_i^+$ and so $i_c^{-1} \Phi_{\ti r}^m = \Phi_{\ti r_i}^m$ where $i_c$ is the inner automorphism corresponding to $T$.   Note that $\ti f^m_\#(\ti \sigma_i)$ is a concatenation of terms in the complete splitting of 
$\ti R_{\ti E}$ whose first edge  $T(\ti \sigma_i)$ projects into $\E_f$.    Thus $\ti f^m_\#(\ti \sigma_i) = \ti \sigma_a  \cdot \ti \rho_a \cdot\ldots \cdot \ti \rho_{b-1}  \cdot\ti \sigma_b  \cdot\ti  \tau$  for some $a \le j < b$ and some (possibly trivial) Nielsen path $\ti \tau$ that is an initial segment of $\ti \rho_b$.    Note also that $T^{-1}\ti f^m_\#(\ti \sigma_i)$ is a concatenation of terms in the complete splitting of 
$\ti R_i^+$.    It follows that  $T^{-1}\ti \sigma_a \cdot  T^{-1} \ti \rho_a \cdot \ldots \cdot  T^{-1}\ti  \rho_{b-1}\cdot   T^{-1}\ti \sigma_b$   is a concatenation of terms in the coarsened complete splitting of $\ti R_i^+$. In particular, $T^{-1} \ti \sigma_j \cdot T^{-1} \ti \rho_j  \cdot T^{-1}\ti \sigma_{j+1}$  is a concatenation of terms in the coarsened complete splitting of $\ti R_i^+$ .  Since $\ti \ell_j$ is non-periodic, the same is true for $T^{-1}(\ti \ell_j)$ and we conclude that   $T^{-1}(\ti \ell_j) \in \accnr (\phi, \ti r_i)$. 
Proposition~\ref{prop:accnr} implies that ${\Phi^{-m}_{\ti r}}(\ti \ell_j) = {\Phi_{\ti r_i}^{-m}}i_c^{-1}(\ti \ell_j) = {\Phi_{\ti r_i}^{-m}}(T^{-1}(\ti \ell_j))$ is an element of  $ \accnr(\phi, \ti r_i)$.
\endproof 

\begin{prop}\label{l:chaining}  Each ordered pair $(\ti \ellprime_t,\ti \ellprime_{t+1})$ is  spanned by an element $(\ti L_1, \ti L_2) \in \cS_2(\phi,\ti r)$. If $\partial_+ \ti \ellprime_t$ $[{resp. }\   \partial_-\ti \ellprime_{t+1}]$ is periodic then $\ti L_1 = \ti \ellprime_t$ $[ {resp. } \  L_2 = \ti \ellprime_{t+1}]$.
\end{prop}

\proof  Set $\Phi = \Phi_{\ti r}$.
We first show that if  $\ti \ell_i$ is periodic then $\ti \ell_{i-1}$ and $\ti \ell_{i+1}$ are non-periodic and  $(\ti \ellprime_t,\ti \ellprime_{t+1}) := (\ti \ell_{i-1},\ti \ell_{i+1})$ satisfies the conclusions of the lemma.  Let $M$ be the stabilization constant for $f$ (Notation~\ref{notn:stabilization}).   If $\ti f^M_\#(\ti \rho_i) \subset \ti \rho_j$ then $\ti \sigma_{j}$ is the last growing term in $\ti f^M_\#(\ti \sigma_i)$ and $\ti \sigma_{j+1}$ is the first growing term in $\ti f^M_\#(\ti \sigma_{i+1})$.  Moreover, Lemma~\ref{forward invariance} implies that $\ti \ell_j = \ti f^M_\#(\ti \ell_i)  $  and so $\ti \ell_j$ is periodic.         By our choice of $M$,  $\ti \sigma_{j} \not \in \E_f$  and  $  \ti \sigma_{j+1} \not \in \E_f^{-1}$. Lemma~\ref{staple pair examples}\pref{must be qe} implies that  $\ti \sigma_{j}\ti \rho_j \ti \sigma_{j+1}$ is quasi-exceptional and  $ (\ti \ell_{j-1}, \ti \ell_{j+1}) \in \cS_2(\phi,\ti r)$. 
  Since $\sigma_j \in \lin(f)$, it follows that either $\sigma_i \in \lin(f)$ or $ \sigma_i \in \E_f$.  In the former case,    $\ti f^M_\#(\ti \rho_{i-1}) \subset \ti \rho_{j-1}$ and  $\ti f^M_\#(\ti \ell_{i-1}) = \ti \ell_{j-1}$ so $ \Phi^{-M}(\ti \ell_{j-1}) =  \ti \ell_{i-1} $; in particular $\ti \ell_{i-1}$ is non-periodic and $\partial_+\ti \ell_{i-1}$ is periodic.  In the latter case, $\ti \sigma_{j-1} \ti \rho_{j-1}\ti \sigma_{j}$ is a terminal subpath of    $\ti f^M_\#(\ti \sigma_{i})$ so Lemma~\ref{subpath} implies that $\Phi^{-M}(\ti \ell_{j-1}) \in \accnr(\phi, \partial _+\ti \ell_{i-1})$;  note that in this latter case, $\partial_+ \ti \ellprime_t = \partial_+\ti \ell_{i-1}$ is not periodic.    A symmetric argument shows that either $ \Phi^{-M}(\ti \ell_{j+1}) =  \ti \ell_{i+1} $ or $\Phi^{-M}(\ti \ell_{j+1}^{-1})  \in \accnr(\phi,\partial _-\ti \ell_{i+1})$. Thus $(\ti \ell_{i-1},\ti \ell_{i+1})$ is  spanned by $\Phi^{-M}(\ti \ell_{j-1}, \ti \ell_{j+1})$ and the if statement of the lemma is satisfied.
  
  \begin{remark}\label{rem:periodic line} The above argument includes a proof that if $\ti \ell_i$ is periodic and if $\ti \ell_j = \ti f^M_\#(\ti \ell_i)  $ then $\ti \sigma_{j}\ti \rho_j \ti \sigma_{j+1}$ is quasi-exceptional and  $\ti \ell_j$ is the common  axis of $ (\ti \ell_{j-1}, \ti \ell_{j+1}) \in \cS_2(\phi,\ti r)$. 
    \end{remark}

{Continuing with the proof, we}  now know that for each $(\ti \ellprime_t,\ti \ellprime_{t+1})$ there exists $i$ so that $(\ti \ellprime_t,\ti \ellprime_{t+1})$ is equal to either $(\ti \ell_{i-1},\ti \ell_{i})$ or $(\ti \ell_{i-1},\ti \ell_{i+1})$. Moreover, the conclusions of the lemma hold in the latter case so we may assume that   $(\ti \ellprime_t,\ti \ellprime_{t+1}) =  (\ti \ell_{i-1},\ti \ell_{i})$.  Lemma~\ref{staple pair examples}\pref{qe} implies that $\sigma_i \rho_i \sigma_{i+1}$ is not quasi-exceptional.   If $\sigma_i$ is linear (i.e. $\sigma_i$ is exceptional or $\sigma_i \in \lin(f)$ or $\bar \sigma_i \in \lin(f)$)   then $(\ti \ellprime_t,\ti \ellprime_{t+1}) = (\ti \ell _{i-1},\ti \ell_{i}) \in \cS_2(\phi,\ti r)$ by  Lemma~\ref{staple pair examples}\pref{three cases}.  

It remains to consider the  $\sigma_i \in \E_f$ and  $ \sigma_i \in \E_f^{-1}$ cases.  These   are symmetric so we assume that  $\sigma_i \in \E_f$ and leave the $ \sigma_i \in \E_f^{-1}$ case to the reader.   For the if statement, note that $\partial_+\ti \ell_{i-1}$  is non-periodic.  As above, there exists $j > i$ such that $\ti f^M_\#(\ti \rho_i) \subset \ti \rho_j$ and $\ti \ell_j = \ti f^M_\#(\ti \ell_i)$; in particular, $\ti \ell_j$ is not periodic.   Since    $\ti \sigma_{j}$ is the last growing term in $\ti f^M_\#(\ti \sigma_i)$ and since $\ti f^M_\#(\ti \sigma_i)$ contains at least two growing terms, Lemma~\ref{subpath} implies that  if $\ti \ell_{j-1}$ is non-periodic then $\Phi^{-M}(\ti \ell_{j-1}) \in   \accnr(\phi, \partial_+ \ti \ell_{i-1})$.   If  $ \sigma_j$ is exceptional or $\sigma_j \in \lin(f)$ then another application of Lemma~\ref{staple pair examples} \pref{three cases} shows that $(\ti \ell_{j-1}, \ti \ell_j) \in \cS_2(\phi,\ti r)$. In this case,   $(\ti \ell_{i-1}, \ti \ell_i)$ is spanned by $\Phi^{-M}(\ti \ell_{j-1}, \ti \ell_j) = (\Phi^{-M}(\ti \ell_{j-1}), \ti \ell_i)$ and we are done.   The same argument works if  $\bar \sigma_j \in \lin(f)$ and $\ti \ell_{j-1}$ is not periodic.  Suppose then that $\bar \sigma_j \in \lin(f)$ and $\ti \ell_{j-1}$ is  periodic.    There exists $s > j$ such that $\ti f^M_\#(\ti \rho_j) \subset \ti \rho_s$ and $\ti \ell_s = \ti f^M_\#(\ti \ell_j) = \ti f^{2M}_\#(\ti \ell_i)$.  Since  $\ti \sigma_j$ is linear, $\ti f^M_\#(\ti \sigma_{j})$ contains a single growing term and so $\ti f^M_\#(\ti \rho_{j-1})\subset \ti \rho_{s-1}$.  Thus  $ \ti \ell_{s-1} = \ti f^M_\#(\ti \ell_{j-1})$ is periodic.  Remark~\ref{rem:periodic line}  implies that  $\ti \sigma_{s-1}\ti \rho_{s-1} \ti \sigma_{s} $  is quasi-exceptional and  $\ti \ell_{s-1}$ is the common axis of $ (\ti \ell_{s-2}, \ti \ell_{s}) \in \cS_2(\phi,\ti r)$.  Moreover, $\ti \sigma_{s-2}\ti \rho_{s-2} \ti \sigma_{s-1}\ti \rho_{s-1} \ti \sigma_{s} \subset f^{2M}_\#(\ti \sigma_i)$. The proof now concludes as in the previous case with   $(\ti \ell_{i-1}, \ti \ell_i)$  spanned by $\Phi^{-2M}(\ti \ell_{s-2}, \ti \ell_s) = (\Phi^{-2M}(\ti \ell_{s-2}), \ti \ell_i).$

We may now assume that  $  \sigma_j$ has higher order and so $\sigma_j \in \E_f^{-1}$ by our choice of $M$.   In particular, $\partial_-\ti \ell_j$, and hence $\partial_-\ti \ell_i$, is non-periodic.

Choose $k > 0$ so that the coarsened complete splitting of  $\ti f^k_\#(\ti \sigma_j)$ has at least one linear term $\ti \sigma_s$ that is neither the first nor second term nor the last or next to last growing  term in that splitting.  Thus $\ti \sigma_{s-2}\cdot \ti \rho_{s-2}\cdot\ldots\cdot \ti \rho_{s+1} \cdot \ti \sigma_{s+2}$  is a subpath of $\ti f^{k}_\#(\ti \sigma_{j})$ and hence also a subpath of $\ti f^{M+k}_\#(\ti \sigma_{i})$.  By Lemma~\ref{staple pair examples}, two of the three lines $\ti \ell_{s-1}, \ti \ell_s, \ti \ell_{s+1}$ form an element of $\cS_2(\phi,\ti r)$   that we denote $ (\ti L_1,\ti L_2)$.    Lemma~\ref{subpath} implies that    
 $$\Phi^{-M-k}(\ti L_1) \in \accnr(\phi,\partial_+ \ti \ell_{i-1})   \qquad \text { and } \qquad   
 \Phi^{-k}(\ti L_2^{-1}) \in \accnr(\phi, \partial_- \ti \ell_{j})$$  
 It follows   that  $$\Phi^{-M-k}(\ti L_2^{-1})\in \accnr(\phi, \partial_-\ti \ell_i)$$ 
 and hence that $\Phi^{-M-k}((\ti L_1,\ti L_2))$  spans $(\ti \ell_{i-1}, \ti \ell_{i})$.
 \endproof 

In Lemma~\ref{new rays from lines} below we use Proposition~\ref{l:chaining} to give conditions on $\arb \in \Out(F_n)$ which imply that $\arb$ fixes $r$.  The proof of Lemma~\ref{new rays from lines} is inductive and it is useful in the induction step to know that $\arb$ strongly fixes $r$ in the following sense.

\begin{definition}  We say that $\arb\in\Out(F_n)$ {\it strongly fixes} $r \in \cR(\phi)$ if for some (and hence every) lift $\ti r$ there is a lift $\Arb \in  \arb$ that fixes each element of $\accnr(\phi,\ti r)$. 
\end{definition}

\begin{lemma}\label{new rays from lines}
Suppose that $[F]$ is a $\phi$-invariant free factor conjugacy class,     that $r \in \cR(\phi)$ is carried by $[F]$ and that $\arb\in\Out(F_n)$ satisfies:
\begin{enumerate}
\item \label{item:preserves filtration}
$[F]$ is $\arb$-invariant.
\item
$\arb$ fixes each element of $\accnr(r)$ and each $r' < r$ (as defined in  Notation~\ref{weaker po}). 
\item
The restriction $\arb| F$  commutes with $\phi|F$.
\item \label{item:Q is zero}
$m_b(\arb) = 0$ for all $b\in\cS_2(\phi,r)$. (See Definition~\ref{defn:m} and Remark~\ref{rem:broader def of m}.)
\end{enumerate}
Then $\arb$ strongly fixes $r$.
\end{lemma}

\proof 
Given a lift $\ti r$, we continue with the $\ti \sigma_i, \ti \rho_i, \ti \ell_i, \ti \mu_t$ and $\Phi_{\ti r}$ notation. We may assume without loss that $\ti r \in \partial F$.  By uniqueness,  $F$   is $\Phi_{\ti r}$-invariant.    By Lemma~\ref{l:in F} each $\ti \mu_t$ has endpoints in $\partial F$.

For each $t\ge 1$, item (2) implies the existence of a (necessarily unique) lift $\Arb_t$ of $\arb$ that fixes $\ti \mu_t$.  
We show below that $\Arb_t$ is independent of $t$, say $\Arb_t = \Arb$ for all $t$.  Assuming this for now, the proof concludes as follows.
Since the endpoints of the $\ti \mu_t$'s limit on $\ti r$, we have $\Arb(\ti r) = \ti r$.   From this and  (1) it follows that  $\Arb(F) = F$.  Item (3) implies that the commutator  $[\Phi_{\ti r}| F,\Arb| F]$ is  inner.  Since {the commutator} $[\Phi_{\ti r}| F,\Arb| F]$ fixes $\ti r$,  it must be trivial.  Thus, $\Phi_{\ti r}|F$ and $\Arb|F$ commute and   the    same is true for $ \Phi_{\ti r} | \partial F$ and $ \Arb | \partial F$.   Given $\ti L \in \accnr(\phi,\ti r)$, there exist $m \ge 0$ and $t \ge 1$  such that  ${\Phi_{\ti r}^m}(\ti L) = \ti \mu_t$ by  Proposition~\ref{prop:accnr}. Since $\ti \mu_t$ has endpoints in $\partial F$, the same is true for $\ti L$.  Thus 
     $$\Arb(\ti L) = (\Phi_{\ti r}^{-m} \Arb \Phi_{\ti r}^m)(\ti L) = { \Phi_{\ti r}^{-m}} \Arb({\Phi_{\ti r}^m}(\ti L)) =   {\Phi_{\ti r}^{-m}} \Arb(\ti \mu_t) = {\Phi_{\ti r}^{-m}}(\ti \mu_t) = \ti L$$
as desired.

\smallskip
It remains to prove that $\Arb_t= \Arb_{t+1}$ for all $t \ge 1$.   
 
 The proof is by induction on the height of $r$ in the partial order $<$ on $\cR(\phi)$.  In the base case, $r$ is a minimal element of $\cR(\phi)$ so each $\sigma_i$ is linear by Lemma~\ref{preserves partial order} and Lemma~\ref{not crossed}.  In this case,  each $\ti \mu_t\in \cS(\phi,\ti r)$ and each  $(\ti \mu_t, \ti \mu_{t+1}) \in \cS_2(\phi, \ti r)$ by Lemma~\ref{staple pair examples}.  Item (4) completes the proof.

For the inductive step, we   use   Proposition~\ref{l:chaining}.   Let $\mu_t = \ell_{i}$ and $\mu_{t+1} = \ell_j$. As a first case, suppose that   $\ \partial_+ \ti \ell_{i} = \ti r' $  for some $r' \in \cR(\phi)$.  Let   $E' \in \E_f$ be the higher order edge corresponding to $r'$.  Then $r > r'$ and  Lemma~\ref{preserves partial order} and Lemma~\ref{not crossed} imply that either $E'$ or $\bar E'$ occurs as a term in the complete splitting of some $f^k_\#(E)$.   Thus either $\Acc(E') \subset \Acc(E)$  or $\Acc(\bar E') \subset \Acc(E)$.    Lemma~\ref{two definitions} therefore implies that if $L' \in \accnr(r')$ then either $L' \in \accnr(r)$ or ${L'}^{-1} \in \accnr(r)$.  From (2) we see that $\arb$ fixes each element of $\accnr(r')$.  Lemma~\ref{lower rays and staple pairs} and (4) imply that $m_b(\arb) = 0$ for all $b \in \cS_2(\phi, r')$  so $r'$ and $[F]$ satisfy the hypotheses of this lemma.  By the inductive hypothesis,  there is a (necessarily unique) lift $\Arb'$ that fixes each element of $\accnr(\phi,\ti r')$.  As noted above, it follows that $\Arb'$ fixes $\ti r'$ and so $\Arb' = \Arb_t$. 

 There are two subcases.  The first is that   $\partial_- \ti \ell_{j} = \ti r'' $ for some $r''\in \E_f$.  Arguing as in the previous paragraph we see that     $\Arb_{t+1}$  fixes each element of $\accnr(\phi,\ti r'')$.  By Proposition~\ref{l:chaining}, there exist $\ti L' \in \accnr(\phi,\ti r')$ and $\ti {L''}^{-1} \in \accnr(\phi,\ti r'')$  such that $(\ti L', \ti L'') \in \cS_2(\phi,\ti r)$.  Item (4) implies that $\Arb_t = \Arb_{t+1}$.
 
The second subcase is that    $\partial_- \ti \ell_{j}$  is the end of a periodic line and hence $\ell_j \in \cS(\phi,r)$.  By Proposition~\ref{l:chaining},  there exists $\ti L' \in \accnr(\phi, \ti r')$  such that $(\ti L', \ti \ell_{j}) \in\cS_2(\phi,\ti r)$. Once again, (4) implies that $\Arb_t = \Arb_{t+1}$.  This completes the proof when  $\partial_+ \ti \ell_{i}$ projects into $\cR(\phi)$.  
  
A symmetric argument handles the case that $\partial_- \ti \ell_{j}$ projects into $\cR(\phi)$.  The remaining case is that both $\partial_+\ti \ell_{i}$  and $\partial_- \ti \ell_{j}$  are endpoints of  periodic lines.  Proposition~\ref{l:chaining} implies that $(\ti \ell_i, \ti \ell_{j}) \in\cS_2(\phi,\ti r)$ so (4) completes the proof as in previous cases.   \endproof

\section{$\bar Q$} \label{s:XD}

Recall that our main theorem is reduced to    Proposition~\ref{p:conjugacy in X} by Lemma~\ref{l:reduction}.

For the rest of the paper, we assume  the hypotheses of Proposition~\ref{p:conjugacy in X},   i.e.\ $\phi, \psi\in\upgn$, $\fc$ is a special chain for $\phi$ and $\psi$, and ${\sf I}_\fc(\phi)={\sf I}_\fc(\psi)$. Our goal is to find a conjugator $\theta\in\X_\fc(\phi)$ {or prove that no such conjugator exists.} 

Because $\phi$ and $\fc$ are fixed for the rest of the paper, we will often write $\X$ for $\X_\fc(\phi)$. In fact, we will often suppress $\fc$ when it appears as a decoration.

 \begin{lemma} \label{lem:fixes each r} For each   $\F_i \in \frak c$ and       each $r \in \cR(\phi|\F_i)$ there exists $r' \in \cR(\psi|\F_i)$   such that $\theta(r) = r'$ for each $\theta \in \X$  whose restriction $\theta | \F_i $ conjugates $\phi|\F_i$ to $\psi | \F_i$. 
 \end{lemma} 
\proof Let $\fe = \F^- \sqsubset \F^+ \in \frak c$ be the one-edge extension with respect to which $r$ is new (Definition~\ref{d:added lines}).  In other words, $r \in \cR^+(\fe, \phi) := \cR(\phi|\F^+) \setminus \cR(\phi|\F^-)$.  Since $\theta$ fixes $\frak c$,  it follows that   $\theta | \F^\pm$ conjugates $\phi | \F^\pm$ to $\psi| \F^\pm$ and so  $\theta$ induces a bijection between  $\cR^+(\fe, \phi)$ and $\cR^+(\fe, \psi)$.  This completes the proof if $r$ is the only element of  $\cR^+(\fe, \phi)$.  Otherwise $\cR^+(\fe, \phi) = \{r,s\}$ and  $\cR^+(\fe, \psi) = \{r',s'\}$ and   we are in case \sH\sH.   By definition,      
  $\LW_\fe(\phi)=\{L, L^{-1}\}$ where $\partial_-L = r$ and $\partial_+L = s$.      By Lemma~\ref{LW is natural},    $\theta(L)\in\LW_\fe(\psi)$. Since $\theta\in\X$, $\theta(\sHsub(L))=\sHsub(L)=\sH_{\psi,\frak c}\big(\theta(L)\big)$. By Lemma~\ref{just one eigengraph line}, there is a unique $L'\in\LW_\fe(\psi)$ that is in $\sH_{\psi,\frak c}\big(\theta(L)\big)$. Hence $\theta(L)=L'$ and $\theta(r) = r' := \partial_-L'$.
\endproof

We continue with the notation of Section~\ref{sec:all staple pairs} and also assume that a \ct\ $f' :G ' \to G'$  representing $\psi$  has been chosen that realizes $\fc$. We use prime notation when working with $\psi$ and $r'$; for example, $E' \in \E_{f'}$ is the edge corresponding to $r'$ and   $\ti \ell_1',\ti \ell_2',\ldots$ are the visible lines in $\ti R'_{\ti E'}$  and $\Psi_{\ti r'}$ is the lift $\Psi \in \psi$ that fixes $\ti r'$. 

\begin{definition} \label{d:Q}Recall from Corollary~\ref{cor:limit lines}, Definition~\ref{defn:m} and Lemma~\ref{lem:m(phi)}  that $\cS_2(\phi)$ is finite,    that for all $b \in \cS_2(\phi)$ there is a homomorphism $m_b : \X  \to \Z$  and  that $m_b(\phi) \ne 0$.  Define a homomorphism $Q^\phi:\X\to \Q^{\cS_2(\phi)}$ by letting the $b$-coordinate  of $Q^\phi(\theta)$ be $Q^\phi_b(\theta) = m_b(\theta)/m_b(\phi)$.
\end{definition}

\begin{definition}\label{d:pair equivalence}
 Let $\sim$ be the equivalence relation on $\cS_2(\phi)$ generated by $b\sim b'$ if $b$ and $b'$ occur in the same $r\in\cR(\phi)$ (as defined in  Notation~\ref{notn:staple pairs}) and let 
 $$\cS_2(\phi)=\cS_2^1(\phi)\sqcup\cS_2^2(\phi)\sqcup\dots$$ be the decomposition of $\cS_2(\phi)$ into $\sim$-equivalence classes. For each $i$, consider the diagonal action of $\Z$ on $\Q^{\cS_2^i(\phi)}$, i.e.\ $k\vec{s}=\vec{s}+k(1,1,\dots)$. Let $\bar Q^\phi$ denote the homomorphism
$$\X_\fc(\phi)\overset{Q^\phi}{\to}\Q^{\cS_2(\phi)}{\to}  \bar\Q^{\cS_2(\phi)}:=(\Q^{\cS_2^1(\phi)}/\Z)\oplus(\Q^{\cS_2^2(\phi)}/\Z)\oplus\dots
$$
\end{definition} 

For the rest of the paper, $Q$ and $\bar Q$ will always denote $Q^\phi$ and $\bar Q^\phi$. 

We can now state the second reduction of the conjugacy problem for $\upgn$ in $\Out(\f)$.

\begin{prop} \label{last step}
{\it There is an algorithm that takes as input  $\phi, \psi \in \upgn$ and a chain $\fc$ such that 
\begin{itemize}
\item
$\fc$ is a special chain for $\phi$ and $\psi$ and 
\item
${\sf I}_\fc(\phi) = {\sf I}_\fc(\psi)$
\end{itemize}
and that outputs {\tt YES} or {\tt NO} depending on whether or not there is  $\theta \in \Ker(\bar Q^\phi)$ conjugating  $\phi$ to $\psi$.  Further, if {\tt YES} then  such a $\theta$ is produced. 
}

\end{prop}

\begin{lemma} \label{p:reduction2}  Proposition~\ref{last step} implies {Proposition~\ref{p:conjugacy in X}} and hence Theorem~\ref{t:main}.
\end{lemma} 

Proposition~\ref{last step} is proved in Section~\ref{s:algorithm}. 

Lemma~\ref{p:reduction2} is proved by  applying the following technical proposition, whose proof takes up the rest of this section.
 
\begin{prop}  \label{prop:barQ}
 We have  an algorithm that produces a finite set $\{\eta_i\} \subset  \X$ so that the union of the cosets of $\Ker(\bar Q)$ determined by the $\eta_i$'s contains each $\theta \in \X$ that conjugates $\phi$ to $\psi$.
\end{prop}  

 \noindent {\it Proof of Lemma~\ref{p:reduction2} (assuming Proposition~\ref{prop:barQ}).} \ \   
Let $\{\eta_i\}$ be the finite set produced by Proposition~\ref{prop:barQ}  and let $\psi_i = \psi^{(\eta_i^{-1})}$.  It follows that   $\phi^\theta = \psi$ if and only if $\phi^{\theta_i'} = \psi_i$ where $\theta_i' = \eta_i^{-1}\theta$ and that  $\theta$ is in the coset represented by $\eta_i$ if and only if $\theta_i' \in \Ker(\bar Q)$.   Thus, by applying Proposition~\ref{last step} to $\phi$ and $\psi_1$, we can decide if there exists $\theta$ in the coset represented by $\eta_1$ that conjugates $\phi$ to $\psi$.  If {\tt YES} then return {\tt YES} and one such $\theta$.  Otherwise move on to $\eta_2$ and repeat.  If {\tt NO} for each $\eta_i$, then return {\tt NO}.
  \qed  
  
  \medskip

The following two lemmas  are proved in Section~\ref{proof of 7.3} and Section~\ref{proof of 7.4} respectively.  In the remainder of this  section we use them to prove Proposition~\ref{prop:barQ}.    
The definition  of topmost staple pair appears in Notation~\ref{notn:topmost staple pair}.  
The definition  of  $\offset(\theta,r)$   is given in Lemma~\ref{translation}\pref{item:offset}.  The partial order $<$ on $\cR(\phi)$ is defined in Notation~\ref{weaker po}.

\begin{lemma}  \label{new finding b}   Suppose that $b \in \cS_2(\phi,r)$ is topmost    and that $\theta \in \X$ conjugates $\phi$ to $\psi$. Then  given an upper bound for $|\offset(\theta, r)|$  one can compute an upper bound for $|m_\theta(b)|$.    
\end{lemma}
 
\begin{lemma}\label{new extreme}  Suppose that $\theta \in \X$ conjugates $\phi$ to $\psi$ and that  $r,r_1 \in \cR(\phi)$ satisfy $r_1 < r$.
 Then given an upper bound for $|\offset(\theta,r)|$ one can compute  an upper bound for $|\offset(\theta,r_1)|$.  
\end{lemma}

 \medskip
 
 \noindent {\it Proof of Proposition~\ref{prop:barQ} (assuming Lemma~\ref{new finding b} and Lemma~\ref{new extreme})}.   
 We begin by computing 
 $D =  D(\phi,\psi)$ so that  $|Q_{b_1}(\theta)-Q_{b_2}(\theta)| < D$  for all $\theta\in \X$ that conjugate $\phi$ to $\psi$     and  all $b_1,b_2 \in \cS_2(\phi)$ satisfying $b_1 \sim b_2$. 
 
 Given $r \in \cR(\phi)$ we will find $D_{r}$ such that $|Q_{b_1}(\theta)-Q_{b_2}(\theta)| < D_{r}$  for all $\theta\in \X$ that conjugate $\phi$ to $\psi$ and all $b_1,b_2 \in \cS_2(\phi, r)$. We then take $D = |\cR(\phi)|  \max\{ D_{r}\}$ (where the $|\cR(\phi)|$ factor allows us to consider equivalent staple pairs that do not occur in the same ray).

For all $s \in \Z$,  $ \theta\phi^s$  is an element of $\X$ and conjugates $\phi$ to $\psi$; see Lemma~\ref{l:phi in X}.
  The translation number  $\tau(\phi,r)$ is defined in   Notation~\ref{notn:translation number}.    By definition and by Lemma~\ref{translation} we have    
 $$\offset(\theta\phi^s,r) = \offset(\theta,r)+  \tau(\phi^s, r)  =  \offset(\theta,r)+ s \tau(\phi, r)$$
Since
$$Q_{b_1}(\theta\phi^s) -Q_{b_2}(\theta\phi^s) =  (Q_{b_1}(\theta)+s)  - (Q_{b_2}(\theta)+s) = Q_{b_1}(\theta)  - Q_{b_2}(\theta)$$
 we may assume without loss that $$0 \le   \offset(\theta,r) \le \tau(\phi,r) $$   
 
 Using only this inequality we will produce an upper bound  $D_0$  for $|m_\theta(b)|$  when $b \in \cS_2(\phi, r)$.  
 This determines an  upper bound for $|Q_b(\theta)|$  when $b \in \cS_2(\phi, r)$, which when doubled gives  the desired upper bound $D_r$ for  $|Q_{b_1}(\theta)-Q_{b_2}(\theta)|$ when $b_1,b_2 \in \cS_2(\phi,r)$.
  
If $b$ is topmost in $r$ then  Lemma~\ref{new finding b} gives us $D_0$.   Otherwise, choose $r_1 < r$ so that $b$ is topmost in $r_1$.  Apply Lemma~\ref{new extreme} to find an upper bound for  $|\offset(\theta,r_1)|$ and then apply   Lemma~\ref{new finding b} to $b$ and $r_1$ to produce $D_0$ and hence $D$. 

To complete the proof {Proposition~\ref{prop:barQ}}, define $$\X(D):=\{\theta\in \X\mid |Q_{b_1}(\theta)-Q_{b_2}(\theta)| < D \mbox{ for all }b_1 \sim b_2 \in \cS_2(\phi)\}$$  Our choice of $D$ guarantees that   $\X(D)$ contains all $\theta \in \X$ that conjugate $\phi$ to $\psi$.  For each $i$, the image of $\X(D)$ by 
$$Q^i:\X\overset{Q}{\to}\Q^{\cS_2(\phi)}{\to} \Q^{\cS_2^i(\phi)}
$$ is discrete, $\Z$-invariant, and contained in a bounded neighborhood of the diagonal in $\Q^{\cS_2^i(\phi)}$. Hence the image of $\X(D)$ by
$$\bar Q^i:\X\overset{Q^i}\to \Q^{\cS_2^i(\phi)}\to\Q^{\cS_2^i(\phi)}/\Z$$ is finite and  $\X(D)$ is contained in finitely many cosets of $\Ker(\bar Q^i)$ and so also in finitely many cosets of $\Ker(\bar Q)$. 

To get representatives of these cosets we must find, for each   $\bar q \in \bar Q(\X(D))$, an element of $\X \cap \bar Q^{-1}(\bar q)$.  For this, it is enough to express $\bar q$ as a word in the $\bar Q$-image of the finite generating set $\mathcal G_{\X}$ for $\X=\Out_{\sf J}(\f)$ supplied by  Lemma~\ref{l:hat MW}}. To accomplish this, we find a finite subset $S\subset \Q^{\cS_2(\phi)}$ whose image in $\bar\Q^{\cS_2(\phi)}$ covers $\bar Q(\X(D))$ and then express the elements of $S$ in terms of the $Q(\mathcal G_\X)$. To find $S$, we first find finite $S^i\subset\Q^{\cS_2^i(\phi)}$ whose image in $\Q^{S^i_2(\phi)}/\Z$ covers $\bar Q^i(\X(D))$ and then take for $S$ the direct sum of the $S^i$'s, i.e.\ $$S:=\{q\in\Q^{\cS_2(\phi)}\mid \mbox{the projection of $q$ to $\Q^{\cS_2^i(\phi)}$ is in $S^i$}\}$$

We now find $S^i$. By definition of $Q$, the denominators of the coordinates of the image of $Q$ are bounded by $\max\{m_b(\phi)\mid b\in\cS_2(\phi)\}$. For convenience, we assume we have cleared denominators and all coordinates in the image of $Q$ are integers. Each $\bar q_i$ in the image of $\bar Q^i$ is represented by $q_i \in \Q^{\cS_2^i(\phi)}$ with first coordinate equal to $0$. Hence we may then take $S^i$ to be the set of vectors in $\Q^{\cS_2^i(\phi)}$ with integer coordinates of absolute value at most $D$ and $S$ to be the set of vectors in $\Q^{\cS_2(\phi)}$ with integer coordinates of absolute value at most $D$. 

The desired set of coset representatives can then be taken to be $\{\theta_s\mid s\in S\cap Q(\X(\phi))\}$ where by definition $\theta_s$ is a choice of element of $\X(\phi)$ satisfying $Q(\theta_s)=s$. We compute $S\cap Q(\X(\phi))$ and $\theta_s$ as follows. First compute $Q(\mathcal G_\X)$. It remains to check which elements of $S$ can be expressed as $\Z$-linear combinations of elements of this $Q(\mathcal G_\X)$ and to produce such a $\Z$-linear combination if it exists. For this, recall that given a finite set of vectors in $\Z^N$ it is standard (see for example \cite{vf:integers}) to find compatible bases $B_0$ for the free $\Z$-submodule they generate and $B$ for $\Z^N$. ($B_0$ and $B$ are {\it compatible} if there is a subset $\{b_m\}$ of $B$ and integers $n_m$ such that $B_0=\{n_m\cdot b_m\}$.) Without concern for efficiency, write each element of $S$ in terms of $B$ and check using divisibility of coordinates if it can be written in terms of $B_0$.
\qed

\subsection{Proof of Lemma~\ref{new finding b}} \label{proof of 7.3}

\begin{lemma} \label{compute m}  Assume that:
\begin{enumerate}
\item  $\theta \in \X$ conjugates $\phi$ to $\psi$.
\item  $b \in \cS_2(\phi,r)$   and  $b' =\theta(b)   \in \cS_2(\psi,r')$ where $r' =\theta(r)${}.  
  \item We are given
\begin{enumerate}
\item  [(a)]  a lift $\ti r$  of $r$ and a lift $\ti b$   of $b$ that is visible in $\ti r$ with index $i$
\item [(b)] a lift $\ti r'$  of $r'$ and a lift $\ti b'$   of $b'$ that is visible in $\ti r'$ with index $i'$
 \end{enumerate}
such that  $\Theta(\ti b) = \ti b'$ where $\Theta$ is the unique automorphism representing $\theta$ and  satisfying $\Theta(\ti r) =\ti r'$.  
\end{enumerate}
Then one can compute $m_b(\theta)$ up to an additive constant  that is independent of $\theta$. 
\end{lemma}

\proof   We  give a formula for $m_b(\theta)$ up to an error of at most one in terms of quantities $s$ and $s'$ (defined below) and then show how to compute $s$ and $s'$, up to a uniform additive constant, from $i$ and $i'$.

Let $\ti b = (\ti L_1, \ti L_2)$ where $\ti L_1 = \ti \ell_{i-1}$ and $\ti L_2 = \ti \ell_{i}$ or $\ti \ell_{i+1}$ and let  $\ti A$ be the common axis for $\ti b$.  By Corollary~\ref{cor:limit lines}, $\ti A$ projects to a \twistpath\  $w$ and we assume that the orientation on $\ti A$ agrees with that of $w$.   Similarly, $\ti b' = (\ti L'_1, \ti L'_2)$ where $\ti L'_1 = \ti \ell'_{{i'}-1}$ and $\ti L'_2 = \ti \ell'_{i'}$ or $\ti \ell'_{i'+1}$ and $\ti A'$ is the common axis for $\ti b'$. Let $\ti x_1$  be the nearest point on $\ti A$  to the initial end $\partial_- \ti L_1$. (See Figure~\ref{Bare1a}.) If $L_1$ is not a linear staple then $\sH_{\phi,\fc}(L_1)=[F(\partial_- \ti L_1), \partial_+ \ti L_1]$.   In this case, the ray from $\ti x_1$ to $\partial_-\ti L_1 $ contains an edge $\ti \sigma_{i-1}$ of height greater than that of $F(\partial_- \ti L_1)$ and so $\ti x_1$ is the nearest point on $\ti A$ to any point in $F(\partial_- \ti L_1)$.

\begin{figure}[h!] 
\centering
\resizebox{12cm}{!}{\begin{tikzpicture}[y=-1cm]
\sf
\filldraw[black] (12.7,24.13) circle (0.1143cm);
\filldraw[black] (12.7,26.67) circle (0.1143cm);
\filldraw[black] (12.7,29.21) circle (0.1143cm);
\draw[black] (12.28725,29.5275) -- cycle;
\draw[black] (8.89,26.51125) -- (8.89,26.82875);
\draw[black] (10.16,26.51125) -- (10.16,26.82875);
\draw[arrows=-triangle 45,black] (12.7,29.21) -- (22.86,29.21);

\draw[arrows=-triangle 45,black] (17.7,29.21) -- (22.86,33);

\draw[arrows=-triangle 45,black] (12.7,24.13) -- (22.86,24.13);
\draw[arrows=triangle 45-triangle 45,black] (12.7,22.86) -- (12.7,33.02);
\draw[arrows=-triangle 45,black] (6.35,26.67) -- (2.54,30.48);
\draw[arrows=triangle 45-,black] (2.54,26.67) -- (12.7,26.67);
\node [above left] at (10,26.5) {$\tilde\sigma_{i-1}$};
\node [above left] at (22.86,24.13) {$\tilde h^1(\partial_+\ti L'_2)$};
\node [above] at (12.7,22.86) {$\partial_+\ti A$};
\node [right] at (13,26.67) {$\tilde x_1$};
\node [left] at (12.5,24.13) {$\tilde y_2$};
\node [right] at (11.7,29.21) {$\tilde x_2$};
\node [below] at (12.7,33.02) {$\partial_-\tilde A$};

\node [below] at (22.86,33) {$\tilde h^2(\partial_+\ti L'_2)$};

\node [above left] at (22.86,29.21) {$\partial_+\ti L_2$};
\node [above left] at (2.54,30.48) {$\tilde h^1(\partial_-\ti L'_1)$};
\node [above left] at (2.54,26.67) {$\partial_-\ti L_1$};
\end{tikzpicture}%

}
\caption{ 
}
\label{Bare1a}
\end{figure}

By hypothesis, $\theta(L_1) = L'_1$.  Since $\theta \in \X$, it follows that $L'_1 \in \theta (\sH_{{\phi, \fc}}(L_1))=\sH_{{\phi, \fc}}(L_1)$.  Choose a     homotopy equivalence $h :G' \to G$ that preserves markings.  If $L_1$ is linear then $\sH_{{\phi, \fc}}(L_1) = \{L_1\}$  so $L_1'  = L_1$.  In this case,  we let $\ti h^1 : \ti G' \to \ti G$   be the lift of $h$ satisfying $\ti h^1(\ti L'_1) = \ti L_1$.
 If $L_1$ is not linear then $\sH_{{\phi, \fc}}(L_1) =  [ F_{{\fc}}(\partial_-\ti L_1), \partial_+\ti L_1]$.  In this case,  we let $\ti h^1 : \ti G' \to \ti G$   be the lift of $h$ satisfying $ \ti h_\#^1(\ti L'_1) \in (\partial  F_{{\fc}}(\partial_-\ti L_1), \partial_+\ti L_1)$.   Let $\Theta_1$ be the unique lift of $\theta$ satisfying $\Theta_1(\ti L_1) = \ti h_\#^1(\ti L_1') \in\sH_{{\phi, \fc}}(\ti L_1)$.

Let $\ti x_2$ and $\ti y_2$ be the  nearest points on $\ti A$  to the terminal ends $\partial_+ \ti L_2$ and $\ti h^1_\#(\partial_+\ti L'_2)$  respectively.    Arguing as above, there is a   lift $\ti h^2 : \ti G' \to \ti G$ such that $\ti h^2_\#(\ti A') = \ti A$ and such that $\ti x_2$ is the nearest point to  $\ti h^2(\partial_+\ti L'_2)$.    Moreover, there is a lift $\Theta_2$ of $\theta$ such that $\ti h^2_\#(\ti L_2') = {\Theta_2}(\ti L_2) \in (\partial_-L_2, \partial  F_{{\fc}}(\partial_+\ti L_2))$. It follows from  Definition~\ref{defn:m} that  the  oriented path $\ti \alpha \subset \ti A$ from    $\ti x_2$      to  $\ti y_2$  has the  form  $\ti w^{m_b(\theta)}$.
   
Let $\ti \beta$ and $\ti \beta'$    be the paths in $\ti A$ connecting   $\ti x_2$ to $\ti x_1$ and $\ti x_1$ to $\ti y_2$ respectively. Let $s$ and $s'$ be the number of complete copies of $\ti w$ (counted with orientation) crossed by the paths $\ti \beta$ and $\ti \beta'$ respectively.  Then   $|m_b(\theta) -(s' +s)| \le 1$.
 
Determining $s$ from the index $i$ is straightforward.  We  consider the cases  of Lemma~\ref{staple pair examples}.  In case \pref{qe}, $\sigma_i \rho_i \sigma_{i+1}$ is quasi-exceptional and $\rho_i = w^s$ where $w$ is the \twistpath\ for $\sigma_i$ and $\bar \sigma_{i+1}$.   In case \pref{exc}, $\sigma_i = E' w^s \bar E''$ for some $E',E'' \in \lin_w(f)$.  In case \pref{lin}, $\sigma_i = E'$ is linear with \twistpath\ $w$ and $\ell_{i}$ is not periodic.   If $\rho_i$ has an initial segment of the form $w^j$ for some $j >0$ then $s$ is the maximal such $j$; otherwise $-s$ is  the maximal $j \ge 0$ such that $\rho_i$ has an initial segment of the form $w^{-j}$.    In case \pref{bar lin}, $\bar \sigma_{i+1} = E'$ is linear with \twistpath\ $w$ and $\ell_{i-1}$ is not periodic.   In this case    $s$ is determined by the maximal initial segment of $\bar \rho_i$ of the form $w^{\pm j}$ as in the  case \pref{lin}.  

Let $\ti x_1'$ and $\ti x_2'$ be the nearest points on $\ti A'$ to $\partial_-\ti {L_1'}$ and $\partial_+\ti {L_1'}$ respectively.  Let $\ti w'$ be a    fundamental domain for the natural action of $\Z$ on  $\ti A'$. Arguing as above, using $G'$ in place of $G$, we can compute   the number $t'$ of complete copies of $\ti w'$ (counted with orientation) crossed by the path connecting $\ti x_2'$ to $\ti x_1'$.  We can also compute the  bounded cancellation constant $C'$ for $h$ \cite{co:bcc} (see also  \cite[Lemma~3.1]{bfh:tits0}).  Since $| s'-t'| < 2 C'$,    $m_b(\theta) =  t'+s$ up to the additive constant  $C = 2C'+1$ .     
\endproof
 
\medskip

\noindent{\it Proof of Lemma~\ref{new finding b}:}  By Lemma~\ref{new bound for staple pairs} and Remark~\ref{visible staple pairs}, applied to $\psi$, $b' = \theta(b)$ and $r' = \theta(r)$ we can find $B'$ so that each lift  $\ti b' \in \cS_2(\psi,\ti r')$ of $b'$ that satisfies $\ti \ell'_{B'} \prec \ti b'$   is visible in $\ti r'$.   After increasing $B'$ if necessary, we may assume that $\ti \ell'_{B'}$ is topmost in $\ti r'$. Now apply  Lemma~\ref{new bound for staple pairs} to find a lift $\ti b_0 \in \cS_2(\phi, \ti r)$ of $b$.  Using the given upper bound $C$ on $|\offset(\theta,r)|$, choose $q \ge 0$ so that $\ti \ell'_{B'} \prec \theta(\Phi^q \ti b_0)$.   Let $\ti b = \Phi^q \ti b_0$.  From   $C$ and $q$ we can compute an upper bound $I'$ for the index of $\theta(\ti b)$.  By Lemma~\ref{new bound for staple pairs}, we can list all visible  $\ti b'$  with index at most $I'$ and so have finitely many candidates for $\theta(\ti b)$.  Applying Lemma~\ref{compute m} to $b$ and each of these candidates gives us the desired upper bound for $m_b(\theta)$.
\qed

\subsection{Stabilizing a ray}

Suppose that $E_i$ is the unique edge of height $i> 0$ and that  $\sigma \subset G$ is  a path with  height $i $ whose endpoints, if any, are not contained in the interior of $E_i$.  Recall from Definition 4.1.3 and Lemma 4.1.4 of \cite{bfh:tits1} that  $\sigma$ has a unique splitting, called the {\em highest edge splitting of $\sigma$} whose splitting vertices are the initial endpoints of each occurence of $E_i$ in $\sigma$ and the terminal vertices of each occurence of $\bar E_i$ in $\sigma$.  In particular,  each term   in the splitting has  the form $E_i\gamma\bar E_i$, $E_i \gamma$, $\gamma \bar E_i$ or $\gamma$ for some $\gamma \subset G_{i-1}$.

The following lemma is used in the proof of Lemma~\ref{new extreme}.    We make implicit use of \cite[Lemma 4.6]{fh:recognition} which states that  if $\fG$ is completely split and  a path $\sigma \subset G$ is completely split then $f^k_\#(\sigma)$ is completely split for all $k \ge 0$.

\begin{lemma} \label{eventually split}
Suppose that  $\fG$ is a \ct\ representing $\phi$, that the edge $E$ corresponds to  some $r \in \cR(\phi)$, that $\xi$ is a finite subpath with endpoints at vertices and that $R=[\xi R_E]$.  Equivalently, $R = \tau R_1$ for some finite path $\tau$ with endpoints at vertices and some subray    $R_1$ of $R_E$.     Then there exists a computable   $k \ge 0$ so that   $f^k_\#(R)$ is completely split.  
\end{lemma} 
 
\proof
The proof is by induction on the height $h$ of $R$ with the base case being vacuous because the lowest stratum in the filtration is a fixed loop.
  
We are free to replace $R $ by an iterate $f^l_\#(R)$ whenever it is convenient.  We also have a less obvious replacement move.
\begin{enumerate}
\item   \label{cut off beginning}   If  there is a splitting $R = \nu \cdot R'$ where $\nu$ has endpoints at vertices then we may replace $R$ by $R' $. 
\end{enumerate}

This follows from :   
  \begin{itemize}
  \item   \cite[Lemma 4.25]{fh:recognition} For any finite path $\nu$ with endpoints at vertices, $f_\#^k(\nu)$ is completely split for all sufficiently large $k$.
   \item  \cite[Lemma 4.11]{fh:recognition}  If a path $\sigma$ has a decomposition $\sigma = \sigma_1 \sigma_2$ with $\sigma_1$ and $\sigma_2$ completely split and the turn $(\bar \sigma_1,\sigma_2)$ legal then $\sigma = \sigma_1 \sigma_2$ is a complete splitting.
   \item  One can check if a given finite path with endpoints at vertices has a complete splitting (because there are only finitely many candidate decompositions).
  \end{itemize}
  
  Let $h_1$ be the height of $R_1$. 
  Each splitting vertex $v$ for the highest edge splitting of $R_1$ is also a splitting vertex for the complete splitting of $R_E$ and so determines a splitting of  $R_1$ into a finite initial subpath followed by a completely split terminal ray $\gamma$.  If $h = h_1$ then $v$ determines a splitting of $R$ into a finite initial subpath followed by  $\gamma$.  In this case,  an application of \pref{cut off beginning} completes the proof. 

We may therefore assume that  $h_1 <h$ and so the highest edge splitting of $R$ is finite.  Applying \pref{cut off beginning}, we may assume that the highest edge splitting of $R$ has just one term.  Thus   $R =  E_h \mu R_1$ where $E_h $ is the unique edge with height $h$ and $\mu$ has height less than $h$.   Let $h_2 < h$ be the height of $R_2 = \mu R_1$.  (At various stages of the proof, we will let $R_2$ be the ray obtained from $R$ by removing its initial edge. The exact edge description of $R_2$ will vary with the context.)

Let   $u_h$ be the path satisfying $f(E_h) = E_h \cdot u_h$. If the height of $u_h$ is $> h_2$ then $f_\#(R) =  E_h \cdot [u_hf_\#(R_2)]$   is a splitting so we may replace $R$ by  $ [u_hf_\#(R_2)]$ which has height less than $h$.  In this case the induction hypothesis completes the proof.  If the height of $u_h$ is $< h_2$     and $R_2 = \sigma_1\cdot \sigma_2 \cdot \ldots$ is the highest edge splitting of $R_2$, then  $R =  [E_h u_h\sigma_1] \cdot  \sigma_2 \cdot \ldots$  is a splitting and the same argument completes the proof.  We are now reduced to the case that  the height of $u_h$ is $h_2$ and we make this assumption for the rest of the proof.

\begin{figure}[h!] 
\centering
\resizebox{12cm}{!}{\begin{tikzpicture}[y=-1cm]
\sf
\filldraw[black] (12.7,13.97) circle (0.1143cm);
\filldraw[black] (12.7,16.51) circle (0.1143cm);
\filldraw[black] (7.62,16.51) circle (0.1143cm);
\draw[black] (13.97,16.35125) -- (13.97,16.66875);
\draw[arrows=triangle 45-triangle 45,black] (5.08,16.51) -- (18.57375,16.51);
\draw[ultra thick,arrows=triangle 45-,black] (7.62,19.05) -- (7.62,16.51) -- (12.7,16.51) -- (12.7,13.97);
\draw[ultra thick,arrows=triangle 45-,black] (12.7,15.5) -- (12.7,13.97);
\draw[thick,arrows=triangle 45-,black] (7.62,19.05) -- (7.62,16.51) -- (12.7,16.51) -- (12.7,13.97);
\draw[arrows=-triangle 45,black] plot [smooth, tension=1] coordinates {(12.8,14.12875) (14,16) (17.4625,16.4)};
\node [below] at (5.08,16.51) {$P$};
\node [below] at (18.57375,16.51) {$Q$};
\node [left] at  (12.7,13.97) {$\tilde x$\ \ };
\node [left] at  (12.7,15.2) {$\tilde E_h$\ \ };
\node [below left] at (12.7,16.51) {$\tilde y$\ \ };
\node [below right] at (12.7,16.51) {\ \ \ $\tilde u_h$};
\node [right] at (7.62,19.05) {\  $\tilde r$};
\node [above] at (7.62,16.51) {$p(\tilde r)$};
\node [above right] at (7.62,16.51) {\ \ \ \ \ \ \ \ \  \ \ \ \ \  $\tilde R$};
\node [above right] at (14,16) {$\tilde R_{\tilde E_{h}}$};
\end{tikzpicture}

}
\caption{}
\end{figure}

We claim that there exists $k \ge 0$ so that $E_hu_h f_\#(u_h) $ is an initial segment of $f^k_\#(R)$. (Note that for any given $k$, one can check if $E_hu_h f_\#(u_h) $ is an initial segment of $f^k_\#(R)$ and so $k$ with this property can be computed once one knows that it exists.) Choose a lift $\ti E_h \subset \ti G$ of $E_h$,    let $\Gamma$ be the component of the full pre-image of $G_{h_2}$ that contains the terminal $\ti y$ endpoint of $\ti E_h$ and let $\ti f  : \ti G \to \ti G$ be the lift of $f$ that fixes the initial endpoint $\ti x$ of $\ti E_h$.  Then $\Gamma$ is $\ti f$-invariant and the lift of $R$ whose first edge is $\ti E_h$ decomposes as $\ti R =\ti E_h \ti R_2$ where $\ti R_2 \subset \Gamma$ is a lift of $R_2$.  Let $\ti u_h$ be the lift of $u_h$ with initial endpoint $\ti y$.  Then $\ti f(\ti E_h) = \ti E_h \cdot \ti u_h$ and $\ti R_{\ti E_h} \setminus \ti E_h = \ti u_h\cdot f_\#(\ti u_h) \cdot f^2_\#(\ti u_h)\cdot \ldots$  
      is a ray  of height $h_2$ that converges to an attracting fixed point $Q \in \partial \Gamma$ for the action of  $ \ti f$ on $\partial \Gamma$.  By Lemma~2.8(ii) of \cite{bfh:tits3} there is another fixed point   $P\ne Q \in \partial \Gamma$ for the action of $\partial \ti f$.     The line $  \overrightarrow{PQ} \subset \Gamma$ from $P$ to $Q$  is $\ti f_\#$-invariant  and has height $h_2$.      Let $\V$ be the set of highest edge splitting vertices of $\overrightarrow{PQ}$ with the order induced by the orientation on $\overrightarrow{PQ}$.  Then  $\ti f_\#$ preserves the highest edge splitting of $\overrightarrow{PQ}$ and so $\ti f$  induces an order preserving bijection of $\V$.  Our choice of $Q$ guarantees that $\ti f$ moves points in   $\V$  away from $P$ and towards  $Q$.  Since $\ti f$ induces an order preserving injection of the set $\V'$ of highest edge splitting vertices of $\ti R_{\ti E_h} \setminus \ti E_h$ into itself,  it follows that   $\V' \subset \V$. To see this, note that for each $\ti v' \in  \V'$ and all sufficiently large $m$, $\ti f^m(\ti v')$ is a highest edge splitting vertex for the common terminal ray of  $\overrightarrow{PQ}$ and $R_{\ti E_h} \setminus \ti E_h$ and so $\ti f^m(\ti v') \in \V$.  Since the restriction of $\ti f^m$ to the vertex set of $\Gamma$ and the restriction of $\ti f^m$ to $ \V$ are  bijections,  $\ti v' \in \V$. 

Since $\ti r$ is an attractor for $\Phi_{\ti r}$,   $\ti r \ne P$.  If $\ti r = Q$ then the lemma is obvious so we may assume that the nearest point projection $p(\ti r)$ of $\ti r$ to $ \overrightarrow{PQ}$ is well defined.   The line $\overrightarrow{ \ti rQ}$  intersects  $\overrightarrow{PQ}$ in the ray    $\overrightarrow{p(\ti r) Q}$.  The set of highest edge splitting vertices of  $\overrightarrow{p(\ti r) Q}$ equals the intersection of the set of highest edge splitting vertices  of $ \overrightarrow{PQ}$ and the set of highest edge splitting vertices  of $\overrightarrow {\ti r Q}$.    It follows that  the set of highest edge splitting vertices of  $\overrightarrow{p(\ti f_\#(\ti r)) Q}$  is the $\ti f_\#$-image of the set of highest edge splitting vertices of  $\overrightarrow{p(\ti r) Q}$.   Thus $p(f^k_\#(\ti r))  \to Q$ and, after replacing $R$ by some $\ti f^k_\#(R)$, we may assume that $p(\ti r)$ is contained in  $f^2_\#(\ti u_h)\cdot f^3_\#(\ti u_h)\cdot \ldots$. This completes the proof of the claim.

We now fix $k$ satisfying the conclusions of the above claim and replace $R$ by $f^k_\#(R)$.    Thus $R = E_hu_h f_\#(u_h) \dots$  and we let $R_2 = u_h f_\#(u_h) \dots$ be the terminal ray of $R$ obtained by removing its initial edge.   We will prove that the decomposition of $R$ determined by the highest edge splitting vertices of $R_2$   is a splitting of $R$.  The proof then concludes as in previous cases.   

We continue with the notation established in the proof of the  claim.  Choose $\ti v \in \V \cap \ti u_h$ and decompose $\ti R$ as $\ti R = \ti \alpha \ti \beta\ti  \gamma$ where $\ti \alpha =\overrightarrow{\ti x \ti v}$,   $\ti \beta= \overrightarrow{\ti v \ti f(\ti v)}$ and  $\ti \gamma = \overrightarrow{\ti f(\ti v) \ti r}$.  Since $\ti \alpha \ti \beta$ is a subpath of   $\ti E_h \cdot \ti u_h\cdot f_\#(\ti u_h) \cdot f^2_\#(\ti u_h)\cdot \ldots$,   no edges of height $h_2$ are cancelled when $\ti f(\ti \alpha \ti \beta)$ is tightened to $\ti f_\#(\ti \alpha \ti \beta)$.   Similarly, no edges of height $h_2$ are cancelled  when   $\ti f( \ti \beta\ti \gamma)$ is tightened to $\ti f_\#( \ti \beta\ti \gamma)$ because $\ti \beta \ti  \gamma$ is a concatenation of terms in the highest edge splitting of $\ti R_2$.     Since $\ti \beta$ contains at least one edge of height $h_2$, it follows that  no edges of height $h_2$ are cancelled  when   $\ti f(\ti R) = \ti f(\ti \alpha \ti \beta\ti  \gamma)$ is tightened to $\ti f_\#(R)$.   This proves that the highest edge splitting of  $\ti R_2$   is a splitting of $\ti R$ as desired.  
\endproof

\subsection{Proof of Lemma~\ref{new extreme}} \label{proof of 7.4}
Recall from Notation~\ref{notn:translation number} and Lemma~\ref{lem:fixes each r} that ${\cal T}_{\phi,\ti r}$ is the set   of  topmost elements of $\accnr(\phi, \ti r)$ and that $r'= \theta(r)$ and $r'_1 = \theta(r_1)$ are independent of $\theta\in \X$ that conjugates $\phi$ to $\psi$.

Suppose that $\ti r_1$ and $\ti r'_1$ are lifts  of $r_1$ and $r_1'$ respectively and  that $\Theta$ is the lift of $\theta$  satisfying $\Theta(\ti r_1) = \ti r_1'$.   If  $\Theta(\ti L) = \ti L'$ where $\ti L \in{\cal T}_{\phi,\ti r_1}$ has index $s$ and $\ti L' \in{\cal T}_{\psi,\ti r'_1}$ has index $s'$  then  $\offset(\theta,r_1) = s'-s$.  We will not be able to find   $\ti L$ and $\ti L'$ whose indices we know exactly but we will be able to find $\ti L$ and $\ti L'$ whose indices we know   up to a  uniform bound and this is sufficient.
    
Before beginning the formal proof, we introduce a way to find distinguished elements of ${\cal T}_{\phi,\ti r_1}$. 
\medskip

 \begin{notn}\label{defn:extreme line}   Suppose that $r_1 <_c r$ (Notation~\ref{weaker po}) and that $\ti r_1$ and $\ti r$ are   lifts   such that    ${\cal T}_{\phi, \ti r_1} \cap \accnr(\phi,\ti r) \ne \emptyset$.    The  {\em $(\ti r, \ti r_1)$-extreme} line is the element of ${\cal T}_{\phi,\ti r_1} \cap \accnr(\phi,\ti r) $ that  is maximal  in the order on   ${\cal T}_{\phi, \ti r_1}$ . 
 \end{notn}

The next lemma states that  extreme lines behave well with respect to conjugation.

  \begin{lemma}   \label{extreme is canonical} Suppose that  $\theta$ conjugates $\phi$ to $\psi$, that $\Theta \in  \theta$, that  $\ti r$ and $\ti r_1$ are lifts of $r >_c r_1 $  and that  $\ti L_2 \in {\cal T}_{\phi,\ti r_1}$ is  $(\ti r, \ti r_1)$-extreme.   Then $\Theta(\ti L_2)$ is   $( \Theta(\ti r),  \Theta(\ti r_1))$-extreme. 
  \end{lemma}
  
  \begin{proof} This follows from Lemma~\ref{translation}   and Lemma~\ref{accnr by conjugacy} which imply that $\Theta$ maps  ${\cal T}_{\phi,\ti r_1}$ to ${\cal T}_{\psi, \Theta(\ti r_1)}$ preserving order and maps $\accnr(\phi,\ti r)$ to $\accnr(\psi, \Theta(\ti r))$.
  \end{proof}

\begin{proof}[Proof of Lemma~\ref{new extreme}]  If $C$ is an upper bound for $|\offset(\theta,r)|$, it suffices to find, for each $|c| \le C$, an upper bound $C_{1,c}$ for $|\offset(\theta,r_1)|$ assuming that $\offset(\theta,r) = c$.  The desired upper bound $C_1$ for $|\offset(\theta,r_1)|$ is then $\max\{C_{1,c}\}$.  Going forward we may therefore assume that we know $|\offset(\theta,r)|$ exactly.

There is no loss in assuming $r_1<_cr$.   Let $E$ and $E_1$ be the elements of $\E_f$ corresponding to $r$ and $r_1$ respectively.  We will assume that $E_1$ occurs in $R_E$; the remaining case, in which $\bar E_1$ but not $E_1$ itself occurs in $R_E$, is  similar and is left to the reader.  Recall from Notation~\ref{notn:translation number} that  the visible elements of ${\cal T}_{\phi,\ti r}$ are enumerated $ \ti L_1, \ti L_2,\ldots$.   For $j \ge 0$, define $q_j \ge 0$ by $\ti L_j = \ti \ell_{q_j}$ and so $\ti L_j = (\ti R^-_{q_j})^{-1}\ti \rho_{q_j} \ti R^+_{q_j+1}$. 
\medskip

The first step of the proof is to show that 
\begin{itemize}
\item [(a)]
there is a computable $J > 0$ so that if $j \ge J$ and if $\ti r_{1,j} := \partial_+\ti L_j$ is a lift of  $r_1$ (equivalently, $\sigma_{q_j+1} = E_1$ and $ R^+_{q_j+1} = R_{E_1}$), then the line $S_j$ connecting  $\ti r =\partial_+\ti R_{\ti E}$ to $\ti r_{1,j}$ is completely split. See Figure~\ref{Bare2a}.
\end{itemize}
Lemma~\ref{lemma:translation} implies that $\ti f^k_\#(\ti L_j) = \ti L_{j +k \tau(\phi,r)}$  and hence $ \ti f^k_\#( S_j)  =   S_{j +k \tau(\phi,r)}$.  It therefore suffices to show that  for each $0 \le j \le \tau(\phi,r)$, there is a computable $K \ge 0$ so that   $ \ti f^K_\#(S_j)$ (and hence $ \ti f^k_\#(S_j)$ for all $k \ge K$)  is completely split.

 The line $S_j$ decomposes as a concatenation of (the inverse of) a ray in $\ti R_{\ti E} \setminus \ti E$ and a proper subray of a lift of  $R_{E_1}$. The height of the former is at least that of $E_1$ and the height of the latter is at most that of $E_1$.  Moreover, $\ti R_{\ti E_1} \setminus \ti E_1$ has height less than that of $E_1$. It follows that $\ti R_{\ti E} \setminus \ti E$ and $ S_j$ have the same height and that each    splitting vertex $\ti v$  for the highest edge splitting of $ S_j$ is contained in $\ti R_{\ti E} \setminus \ti E$ and is a splitting vertex for $\ti R_{\ti E}$.  Splitting $ S_j$ at one such $\ti v$ we write $ S_j = \ti A^{-1}_j \cdot \ti B_j$ where  $A_j$ is a concatenation of terms in the complete splitting of $R_{E}$ and $B_j$ has a subray in common with $R_{E_1}$. For all $k \ge 0$,  $  S_{j +k \tau(\phi,r)} = \ti f^k_\#( S_j) = \ti f^k_\#(\ti A^{-1}_j) \cdot \ti f^k_\#(\ti B_j)$.  By Lemma~\ref{eventually split},  we can find $K$ so that $ f^K_\#( B_j)$ is completely split. It follows (see the second bullet in the proof of Lemma~\ref{eventually split}) that  $ \ti f^K_\#(S_j) $ is completely split. This completes the first step.
\medskip

\begin{figure}[h!] 
\centering
\resizebox{12cm}{!}{\begin{tikzpicture}[y=-1cm]
\sf
\draw[thick,black] (11.35063,8.81063) +(87:1.35171) arc (87:3:1.35171);
\filldraw[black] (10.95375,10.16) circle (0.1143cm);
\filldraw[black] (10.16,10.16) circle (0.1143cm);
\filldraw[black] (9.36625,10.16) circle (0.1143cm);
\filldraw[black] (12.22375,9.8425) circle (0.1143cm);
\filldraw[black] (12.62168,9.28793) circle (0.1143cm);
\draw[thick,arrows=-triangle 45,black] (5.08,10.16) -- (17.78,10.16);
\draw[black] (6.985,10.00125) -- (6.985,10.31875);
\draw[black] (8.09625,10.00125) -- (8.09625,10.31875);
\draw[thick,arrows=-triangle 45,black] (5.08,8.5725) -- (5.08,10.16);
\draw[thick,arrows=-triangle 45,black] (12.7,8.89) -- (12.7,5.715);
\draw[arrows=-triangle 45,black] plot [smooth, tension=1] coordinates {(4.3,8.6) (7.1,11) (17.4625,10.7)};
\draw[arrows=-triangle 45,black] plot [smooth, tension=1] coordinates {(5.3,8.6) 
(10,9.3) (12.5,5.87375)};
\draw[arrows=-triangle 45,black] plot [smooth, tension=1] coordinates {(17.2,9.9) (14,9.2) (13,6.0325)};
\node [right] at (4.4,9.2) {$\tilde E$};
\node [below right] at (6.985,10.2) {$\tilde\sigma_{q_j+1}$};
\node [below] at (9.36625,10.3) {$\tilde w_0$};
\node [right] at (12.7,9.28793) {$\tilde w_4$};
\node [above right] at (14,9.2) {$S_j$};
\node [above left] at (10,9.3) {$T_j$};
\node [below] at (7.1,11) {$\tilde R_{\tilde E}$};
\node [right] at (17.78,10.16) {$\tilde r$};
\node [above] at (12.6,5.7) {$\tilde r_{1,j}$};
\end{tikzpicture}%

}
\caption{}
\label{Bare2a}
\end{figure}

Let ${\cal T}_{\psi,\ti r'} = \{\ti L'_1,\ti L'_2,\ldots\}$ be the set of  topmost elements of $\accnr(\psi,\ti r')$. By definition, $\Theta(\ti L_j) = \ti L'_{j +\offset(\theta,r)}$.  The following $\psi$ and $r'$ analog of (a) is verified by the same arguments given for (a).
 \begin{itemize}
\item [(b)]
there is a computable $J' > 0$ so that if $j \ge J'$ and if $\ti r'_{1,j} :=\partial_+\ti {L'_j} $   projects to  $r'_1$  then the line $S'_j$ connecting  $\ti r'$   to $ \ti r'_{1,j} $ is completely split.
\end{itemize}

Note also that 
\begin{itemize}
\item [(c)]
For all $j\ge 1$, the line $T_j$ connecting the initial vertex of $\ti R_{\ti E}$ to $\ti r_{1,j}$ is completely split and similarly for the line $T'_j$ connecting the initial vertex of $\ti R'_{\ti E'}$ to  $ \partial_+ \ti {L'_j}$.
\end{itemize}
    
For $j \ge 0$, let $\V_j$ be the set of highest edge splitting vertices of $\ti R^+_{q_j+1} \setminus \ti \sigma_{q_j+1}$    (which is a terminal ray of $\ti L_j$) and let $\V'_j$ be the set of highest edge splitting vertices of $\ti R'^+_{q'_j+1} \setminus \ti \sigma'_{q'_j+1}$. The second step of the proof is to choose an   index $j$ so that the following four  properties are satisfied.
\begin{description}
\item [(i)] $S_j$ is completely split.
\item [(ii)] There exist $\ti w \in \V_j$ such that $\ti w, \ti f(\ti w), \ti f^2(\ti w) \in \ti R^+_{q_j+1} \cap \ti R_{\ti E} $.
\item [(iii)]  Letting $j' = j +\offset(\theta,r)$, the line $S'_{j'}$ is completely split. 
\item [(iv)] There exist $\ti w' \in \V'_{j'}$ such that $\ti w', \ti g(\ti w'), \ti g^2(\ti w') \in  \ti {R'}^+_{q'_{j'}+1} \cap \ti R'_{\ti E'}$.  
\end{description}

Items (i) and (iii) hold for all $j \ge \max\{J,J' - \offset(\theta,r)\}$.   For (ii),  write $j = a_j + c_j \tau(\phi,r)$ where $0 \le a_j < \tau(\phi,r)$.  Then $\ti L_j =  \ti f^{c_j}_\#(\ti L_{a_j})$ and  $\ti L_j \cap \ti R_{\ti E}$  contains an initial segment of  $\ti {R}^+_{q_{j}+1}$ whose length goes to infinity with $j$.   
  If $c_j$ is sufficiently large then (ii) is satisfied.  Item (iv) is established in the same way, completing the second step.   

\medskip

We have $\Theta(\ti r) = \ti r'$ and $\Theta(\ti L_j) = \ti L'_{j'}$.  The latter implies that $\Theta(\ti r_{1,j}) = \ti r'_{1,j'}$.    Lemma~\ref{extreme is canonical}  implies that $\Theta$ maps the $(\ti r, \ti r_{1,j})$-extreme line to the $(\ti r', \ti r'_{1,j'})$-extreme line.  Let $s_j$ be the index  of  the $(\ti r, \ti r_{1,j})$-extreme line (as an element of ${\cal T}_{\ti r_{1,j}}$) and  let $s'_{j'}$ be the index  of  the $(\ti r', \ti r'_{1,j'})$-extreme line (as an element of ${\cal T}_{\ti r'_{1,j'}}$).  Then $\offset(\theta, r_1) = s'_{j'} - s_j$.  We will complete the proof by   finding $a_j  \le s_j \le  b_j$ such that $b_j - a_j \le 3 \tau(\phi, r_1)$ and $a'_{j'}  \le s'_{j'} \le  b'_{j'}$ such that $b'_{j'} - a'_{j'} \le 3 \tau(\psi, r'_1) = 3\tau(\phi, r_1)$.  These allow us to compute $\offset(\theta, r_1)$ with an error at most $ 6\tau(\phi, r_1)$ and hence compute an upper bound for $\offset(\theta, r_1)$.

\medskip

Let $h_2$ be the height of  $R_{E_1} \setminus E_1$ (which is the same as the height of $  \ti R^+_{q_j+1} \setminus \ti \sigma_{{q_j}+1}$) and  let $E_2$ be the unique edge of height $h_2$.      We claim that 
 \begin{itemize}
 \item [(d)] each $\ti w \in \V_j \cap \ti R_{\ti E}$ is a splitting vertex for the      complete splittings of  $T_j$ and  $\ti R_{\ti E}$.
 \end{itemize}
 It suffices to show that  $\ti w$ is not contained in the interior of a term $\ti \mu$ in one of those splittings.   Such a $\ti \mu$ would be an     \iNp\ or exceptional path   with height $\ge h_2$ and whose first edge is contained in  $  \ti R^+_{q_j+1} \setminus \ti \sigma_{{q_j}+1}$ (because $\ti \sigma_{{q_j}+1}$ is a term in both splittings)  and so has height at most $h_2$.  Thus $E_2$ would be a linear edge with twist path $w_2$ and $\mu$ would have one of the following forms: $E_2 w_2^p\bar E_2$, $E_2 w_2^p \bar E_3 $ or $E_3 w_2^p \bar E_2$ where $p \ne 0$ and where $E_3 \ne E_2$ is a linear edge of height $< h_2$ with twist path $w_2$.   In none of these cases does the interior of $\mu$ contain a vertex that is the initial endpoint of $E_2$ or the terminal endpoint of $\bar E_2$.   
 This completes the proof of (d).  
\medskip

A similar analysis shows that 
\begin{itemize}
\item [(e)]
each    $\ti w \in \V_j  $ that is disjoint from $\ti R_{\ti E}$ is a splitting vertex for the complete splittings of $S_j$ and   $T_j$.
\end{itemize}

\medskip

Let $\ti w_0$ be the last element of $\V_j$ such that $\ti w_1 = \ti f(\ti w_0)$ and $\ti w_2 = \ti f^2(\ti w_0)$ are contained in $ \ti R_{\ti E}$ (and hence contained in  
 $ \V_j \cap \ti R_{\ti E}$).   Item (d) implies that    
  the path $\ti \alpha$ connecting $\ti w_0$ to $\ti w_2$     inherits the same complete splitting from   $\ti R_{\ti E}$ and from $  \ti R^+_{q_j+1}$.  Thus the lift $\ti \sigma_a$ of $E_2$  or $\bar E_2$ with endpoint $\ti w_1$ determines   an element   $\ti L^1$  of ${\cal T}_{\phi,\ti r_{1,j}} \cap \accnr(\phi,\ti r) $. (If $\ti \sigma_a$ is a lift of $E_2$  
  then $\ti w_1$ is the initial endpoint of $\ti \sigma_a$ and $\ti L^1 = \ti \ell_{a-1}$; If $\ti \sigma_a$ is a lift of $\bar E_2$ then $\ti w_1$ is the terminal endpoint of $\ti \sigma_a$ and then $\ti L^1 = \ti \ell_{a_1}$.) In particular, the index $s_j$ of  the $(\ti r, \ti r_{1,j})$-extreme line (as an element of ${\cal T}_{\phi,\ti r_{1,j}} \cap \accnr(\phi, \ti r)$) is at least as big as that of $\ti L^1$.   

Let $\ti w_3=\ti f^3(\ti w_0)$ and $\ti w_4 = \ti f^4(\ti w_0)$, neither of which   is contained in $\ti R_{\ti E}$.  The  lift of $E_2$ or $\bar E_2$ with endpoint $\ti w_4$ determines  an element $\ti L^4$  in {$ {\cal T}_{\phi, \ti r_{1,j}}$}.  Item (e) and the hard splitting property of a complete splitting (Lemma 4.11 of \cite{fh:recognition}) implies that no point in the terminal ray of $  \ti R^+_{q_j+1}$ that begins with $\ti w_4$  is ever identified, under iteration by  $\ti f$,  with a point in $\ti R_{\ti E}$. It follows that $\ti L^4$ is not an element of $\accnr(\phi,\ti r) $ and so   $s_j$  is  less than the index   of $\ti L^4$.    

Combining the inequalities established in the preceding two paragraphs we are able to compute $s_j$ with an error of at most $3\tau(\phi, r_1) $.  The parallel argument allows us to   compute the index $s'_j$ of  the {$(\ti r', \ti {r'}_{1,j'})$}-extreme line (as an element of {$ {\cal T}_{\phi, \ti r_{1,j}}$}) with an error of at most $3 \tau(\psi, r'_1)=3 \tau(\phi,r_1)$.  As noted above, this completes the proof.
\end{proof}

\section{Proof of Proposition~\ref{last step}}\label{s:algorithm}

 {Some of out  arguments are} by induction up through the elements $\F_k$ of the chain $\fc$.  We write $\phi | \F_k = \psi | \F_k$ if $\phi | [F] = \psi | [F]$ for each component $[F]$ of $\F_k$.  Similarly, we say $\theta|\F_k$ conjugates $\phi|\F_k$ to $\psi | \F_k$ if $\phi^\theta |\F_k =  \psi |\F_k$.   If $G_s$ is the core filtration element corresponding to $\F_k$  and if $C$ is a component of $G_s$  with rank one then {$[C]$ is a component of $\F_0$} and we define $\stallings(f|C) = C$.  With this convention, $\stallings(f|G_s)$   is the disjoint union ${ \sqcup} \stallings(f|C_i)$ as $C_i$ varies over the components of $G_s$.  (See Section~\ref{sec:stallings}.)   

We show below that Proposition~\ref{last step}  is a consequence  of the following lemma and proposition. The former  addresses the restrictions to $\F_0$ and the latter provides the inductive step for the higher order one-edge extensions.  

\begin{lemma}  \label{base case}  
Suppose that $\phi, \psi \in \upgn$ share the special chain $\fc$ and satisfy ${\sf I}_\fc(\phi) = {\sf I}_\fc(\psi)$. Let $\F_0 = \F_0(\phi) = \F_0(\psi)$.  Then 
\begin{enumerate}
\item \label{item:good on lines} $\theta(L) = L$ for each $\theta \in \X$ and each $L \in \acc(\phi)$ that is carried by $\F_0$.
\item  \label{item:conjugate} If there exists $\theta_0 \in \X$ such that  $\phi^{\theta_0} | \F_0 = \psi | \F_0$  then $\phi | \F_0 = \psi | \F_0$ and $\phi^\theta | \F_0 = \psi | \F_0$ for all  $\theta \in \X$.
\end{enumerate}
\end{lemma}

\begin{prop}  \label{inductive step}  
Suppose that $\phi, \psi \in \upgn$ share the special chain $\fc$ and satisfy ${\sf I}_\fc(\phi) = {\sf I}_\fc(\psi)$ and that the   special one-edge extension $\fe=(\F^- \sqsubset \F^+)$ in $\fc$ satisfies 
\begin{enumerate}
\item $\phi | \F^- = \psi | \F^-$
\item  $\{L \in \acc(\phi): L \subset \F^-\} = \{L' \in \acc(\psi): L' \subset \F^-\}$
\end{enumerate}
Then there is an algorithm to decide if there exists  $\theta \in \Ker(\bar Q) < \X$ such that the following are satisfied.
\begin{enumeratecontinue}
\item  \label{item:ind step conjugate}$\phi^\theta | \F^+ = \psi | \F^+$
\item \label{item:ind step lines}   $\theta(\{L \in \acc(\phi): L \subset \F^+\}) = \{L' \in \acc(\psi): L' \subset \F^+\}$.  
\end{enumeratecontinue}
Moreover, if such an element $\theta$ exists then one is produced.
\end{prop} 
 
Before proving Lemma~\ref{base case} and Proposition~\ref{inductive step}, we  use them to prove Proposition~\ref{last step}.

\bigskip

\noindent {\it Proof of Proposition~\ref{last step} (assuming Lemma~\ref{base case} and Proposition~\ref{inductive step}):} \ \   If $\phi | \F_0 \ne \psi | \F_0$ then  no element of $\X$ conjugates $\phi$  to $\psi$ by Lemma~\ref{base case}\pref{item:conjugate} so we return {\tt NO} and {\tt STOP}.  Otherwise, $\phi^\theta | \F_0 = \psi | \F_0$ for all $\theta \in \X(\phi)$ and we define $\theta_0 = $identity and $\psi_0 =  \psi$.

  Suppose $\fc=(\F_0 \sqsubset \F_1 \sqsubset \ldots \sqsubset \F_t)$. Apply Proposition~\ref{inductive step} with $(\phi, \psi_0, \F_0, \F_1)$ in place of $(\phi, \psi, \F^-,\F^+)$.   If the \ref{inductive step}-algorithm returns {\tt NO} then  there is no $\theta$ as in the conclusion of Proposition~\ref{last step} because any such $\theta$  would satisfy items (3) and (4) of Proposition~\ref{inductive step}; we return {\tt NO} and {\tt STOP}.  Otherwise, Proposition \ref{inductive step} gives us an element $\theta_1 \in \Ker(\bar Q)$. Letting $\psi_1  = \psi_0^{(\theta_1^{-1})}$ we have that $\phi |  \F_1 = \psi_1 | \F_1$ and $\{L \in \acc(\phi): L \subset \F_1\} = \{L' \in \acc(\psi_1): L' \subset \F_1\}$.  From $\theta_1 \in \X$ and Lemma~\ref{l:X is natural}, it follows that ${\sf I}(\phi)={\sf I}(\psi_1)$. 

 Apply Proposition~\ref{inductive step} with $(\phi, \psi_1, \F_1, \F_2)$ in place of $(\phi, \psi, \F^-,\F^+)$.  Suppose that  the \ref{inductive step}-algorithm returns {\tt NO}. Then there are no elements of $\Ker(\bar Q)$ that conjugate $\phi|\F_2$ to $\psi_1|\F_2$, and so also no elements of $\Ker(\bar Q)$ that conjugate $\phi$ to $\psi_1$. It follows also then that there are no elements $\theta$ of $\Ker(\bar Q)$ that conjugate $\phi$ to $\psi$. Indeed for such a $\theta$, $\theta_1^{-1}\theta$ would conjugate $\phi$ to $\psi_1$.   We therefore return {\tt NO} and {\tt STOP}. Otherwise, Proposition \ref{inductive step} gives us an element $\theta_2 \in \Ker(\bar Q)$.     Letting $\psi_2  = \psi_1^{(\theta_2)^{-1}} $we have that $\phi|  \F_2 = \psi_2 | \F_2$ and $\{L \in \acc(\phi): L \subset \F_2\} = \{L' \in \acc(\psi_2): L' \subset \F_2\}$. As in the previous case, ${\sf I}(\phi) = {\sf I}(\psi_2)$.  Repeat this until either some application of Proposition~\ref{inductive step} returns {\tt NO} or until we reach $\psi_t = \psi^{(\theta_1 \ldots \theta_t)^{-1}}$   satisfying $\phi =  \phi| \F_t = \psi_t | \F_t = \psi_t$.   In the former case there  is no $\theta$ as in the conclusion of Proposition~\ref{last step} and we return {\tt NO}  and {\tt STOP}.  In the latter case $\theta = \theta_1 \ldots \theta_t$ conjugates $\phi$ to $\psi$ and is an element of $\Ker(\bar Q)$; we return {\tt YES} and $\theta$ and then {\tt STOP}.
\qed
 
\medskip

\noindent {\it Proof of Lemma~\ref{base case}:}\ \      
If  $L \in \acc(\phi)$ is carried by $\F_0$ then the ends of $L$ are periodic. If $L$ is periodic then $\ti L=a^\infty$ for some $[a]\in\A(\phi)$; see Corollary~\ref{cor:limit lines}\pref{item:ell decomposes}. By definition of $\X$, $\theta([a])=[a]$ and so $\theta(L)=L$. Otherwise, $L\in\accnr(\phi)$ has type {\sf P-P}, in which case $\sH(L)$ determines $L$; see Section~\ref{s:algebraic invariants}. Again by definition of $\X$, $\theta(\sH(L))=\sH(L)$ and so $\theta(L)=L$. This verifies \pref{item:good on lines}. 

It suffices to show that if a free factor $F$ represents a component of $\F_0$ then either $\phi^\theta | F = \psi | F$  for all $\theta \in \X(\phi)$ (and in particular for $\theta =$ identity) or $\phi^\theta | F = \psi | F$  is satisfied by no element of $\X(\phi)$.  

Let $\phi_F = \phi|F$ and $\psi_F = \psi | F$. If  $F$ has rank one then $\phi_F$ and $\psi_F$ are both the identity because $\phi$ and $\psi$  are rotationless.      We may therefore assume that $F$ has rank at least two.   Since $\cR(\phi_F) =  \emptyset$, Lemma~\ref{gjll}  implies that $\FixN(\Phi_F) = \partial \Fix(\Phi_F)$ for each $\Phi_F \in \cP(\phi_F)$.  Also,  $\FixN(\Phi_F)$ contains at least three points, so $ \Fix(\Phi_F)$  has rank at least two  and $\Fix(\Phi_F) \ne \Fix(\Phi'_F)$ for $\Phi_F \ne \Phi'_F \in \cP(\phi_F)$ by Lemma~\ref{malnormal}.      

There is  a  unique $\Phi \in \cP(\phi)$  such that $\Phi_F = \Phi | F$. From ${\sf I}(\phi) = {\sf I}(\psi)$, it follows that there exists $\Psi \in \cP(\psi)$ such that $\Fix(\Phi) = \Fix(\Psi)$.    Since $\theta \in \X$, there exists $\Theta$ representing $\theta$ such that $\Fix(\Phi)$   is $\Theta$-invariant. It follows that $\Theta(F) \cap F$ is non-trivial and hence that $\Theta(F) = F$ (because $F$ is a free factor and $\theta$ preserves $[F]$).  Letting $\Psi_F = \Psi | F$ and $\Theta_F = \Theta | F$, we have that $\FixN(\Phi_F) =  \partial \Fix(\Phi_F) = \partial \Fix(\Psi_F) = \FixN(\Psi_F)$ is $\Theta_F$-invariant.Lemma~\ref{lem:lifting} implies that  the eigengraphs for $\phi_F$ and for $\psi_F$ carry the same lines and that $\theta$ preserves this set of lines.      Thus $\phi_F,\psi_F$ and $\theta_F$ satisfy the hypotheses of Lemma~\ref{l:recognition}.

If $a \in F$ is fixed by distinct $\Phi_F, \Phi'_F \in \cP(\phi_F)$ then $[\Phi_F,a]$ is an element of $\sa(\phi_F)$ and $[\Phi, a]$ is an element of $\sa(\phi)$.  Lemma~\ref{l:recognition} implies that  $$[\Phi_F,a] \mapsto [\Psi_F, \Theta_F(a)]$$
defines a  bijection $\B_{\sa,F} : \sa(\phi_F) \to \sa(\psi_F)$ that is independent of the choice of $\Theta_F$ representing $\theta_F$ and preserving $\FixN(\Phi_F)$. 
Since $\theta \in \X$, by Definition~\ref{d:X}\pref{i:last} we have   $[\Fix(\Phi), a]  =  \theta([\Fix(\Phi),  a]) = [\Theta(\Fix(\Phi)),   \Theta(a)]= [\Fix(\Phi),   \Theta(a)]$.   Equivalently, there exists $c \in F_n$ such that  $i_c(\Fix(\Phi)) =  \Fix(\Phi)$ and $i_c \Theta(a) = a$.   Thus $c \in \Fix(\Phi)$ and after replacing $\Theta$ by $i_c\Theta$ we may assume that $\Theta(a) =  a$ and hence that $\Theta_F(a) = a$.       We conclude that $\B_{\sa,F}$ is independent of $\theta$.

Check by inspection if $\B_{\sa,F} $ preserves twist coordinates.  If it does then Lemma~\ref{l:recognition} implies that each $\theta \in \X(\phi)$ conjugates $\phi_F$ to $\psi_F$; if not, then no element of $\X(\phi)$ conjugates $\phi_F$ to $\psi_F$.  \qed

\bigskip

The rest of the paper is dedicated to the proof of   Proposition~\ref{inductive step}. 

\bigskip

 Set $\fc=(\F_0 \sqsubset \F_1 \sqsubset \ldots \sqsubset \F_t$) and thus $\fe\in\fc$ has the form $\F^-\sqsubset \F^+$ where $\F^- = \F_{k-1}$ and $\F^+ = \F_{k}$ for some $1\le k\le t$.   (We will use these notations interchangeably depending on the context.)

\begin{definition}\label{defn:Xepsilon}
For $\epsilon = \pm$, $\X^\epsilon$ is the set of $\theta\in  \X$ such that:
\begin{enumerate}[(a)]
\item   \label{item:KerBarQ} $\theta \in \Ker(\bar Q)$; and 
\item   \label{item:conjugates epsilon} $\theta | \F^\epsilon$ conjugates $\phi|\F^\epsilon$ to $\psi | \F^\epsilon$. 
\end{enumerate}
\end{definition}

By the next lemma, our goal is to produce an element of $\X^+$ or deduce that  $\X^+$ is empty. 

\begin{lemma}  \label{lem:X+ is good enough}
$\theta \in    \Ker(\bar Q)$ satisfies items \pref{item:ind step conjugate} and \pref{item:ind step lines} of Proposition~\ref{inductive step} if and only if $\theta \in \X^+$.
\end{lemma}

\proof
Comparing the definitions, it suffices to show that each $\theta \in \X^+$ satisfies Proposition~\ref{inductive step}\pref{item:ind step lines}; namely, $\theta(\{L \in \acc(\phi): L \subset \F^+\}) = \{L' \in \acc(\psi): L' \subset \F^+\}$. By symmetry, it suffices to show that if   $L \in \acc(\phi)$ is carried by $\F^+$  then $\theta(L) \in \acc(\psi)$ is carried by $\F^+$.   Since $\F^+$ is $\theta$-invariant, it suffices to show that $\theta(L)\in \acc(\psi)$.      If $L$ is periodic then $\ti L=a^\infty$ for some $[a]\in\A(\phi)$ by Corollary~\ref{cor:limit lines}\pref{item:ell decomposes}.  Since $\theta \in \X$,  $\theta([a])=[a]$ and   $\theta(L)=L \in \acc(\psi)$. Otherwise $L\in\accnr(\phi)$ and, as  ${\sf I}_\fc(\phi) = {\sf I}_\fc(\psi)$, there exists $L' \in \accnr(\psi)$ such that $ \sH(L) = \sH(L')$. Since $\theta| \F^+$ conjugates $\phi | \F^+$ to $\psi | \F^+$, Corollary~\ref{cor:limit lines}\pref{item:endpoints in FixN} and Lemmas~\ref{lem:lifting} and \ref{first theta} imply that $\theta(L)$ lifts into $\stallings(g_{u'})$.   Also,   $\theta(L) \in \theta\big(\sH(L)\big) = \sH(L) = \sH(L')$   because  $\theta \in \X$.  Lemma~\ref{just one eigengraph line} implies that $\theta(L) = L' \in \accnr(\psi)$ and we are done.   
\endproof

    By \cite[Theorem 7.4]{fh:CTconjugacy}
      we can choose \ct s   $\fG$ and $g:G' \to G'$   representing $\phi$ and $\psi$ respectively such that  each $\F_i$ is realized by a core filtration element and such that  the core filtration elements of $G$ and $G'$ realizing $\F^- = \F_{k-1}$ are identical as marked graphs and that after identifying them to a common subgraph $G_s$, the restrictions $f_s  = f|G_s$ and $g_s = g|G'_s$ are equal.  
In particular
  \begin{enumerate}
  \item
  $\stallings(f_s)=\stallings(g_s)$
 \end{enumerate}
 
 Before describing $f_s$ and $g_s$ in more detail, we record some  useful properties of   $\X^-$.  We define $\cR(\phi | \F_-) =  \cup \cR(\phi | [F])$ as $[F]$ varies over the components of $\F_-$.

 \begin{lemma}  \label{three properties} Each $\theta \in \X^-$ satisfies the following properties.
\begin{enumerate}
\item \label{item:commutes}$\theta | \F^-$ commutes with $\phi|\F^- = \psi|\F^-$.
\item   \label{item:lines from above} $\theta$ fixes each element of  
  $\acc(\phi)$ that is carried by $\F^-$.
\item \label{item:fixes r}$\theta$  fixes each element of $\cR(\phi | \F^-)$.
\end{enumerate}
\end{lemma}

\proof  \pref{item:commutes} follows from  Definition~\ref{defn:Xepsilon}(b) and the  hypothesis that $\phi|\F^- = \psi|\F^-$.  For \pref{item:lines from above},  note that if $L \in \acc(\phi)$ is carried by $\F^-$ then  $L$  lifts to $\stallings(f_s)$ by  Corollary~\ref{cor:limit lines}\pref{item:endpoints in FixN} so    \pref{item:commutes} and  Lemmas~\ref{lem:lifting} and \ref{first theta}  imply that $\theta(L) \in \stallings(g_{s})$.   Since $\theta(L)=L$ if $L$ is periodic and otherwise $\theta(L) \in \theta(\sH(L)) = \sH(L)$, \pref{item:lines from above} follows from  \pref{item:commutes} and  Lemma~\ref{just one eigengraph line}.  By \pref{item:commutes}  and Lemma~\ref{lem:fixes each r},  $\theta(r)$ is independent of $\theta \in \X^-$.   Item \pref{item:fixes r} therefore follows from the fact that $\X^-$ contains the identity.
\endproof

 Suppose that $G_u \subset G$ and $G'_{u'} \subset G'$ are the core filtration elements realizing $\F^+ = \F_{k}$.  Let   $f_{u} = f | G_u$ and $g_{u'} = g | G'_{u'}$.
  Since $\F^- \sqsubset \F^+$ is a special  one-edge extension, $G_u $ is obtained from $G_s$ by adding a single topological arc $E$ which is either a single edge $D$   or is the union $E = \bar C D$ of a pair of edges $C,D$ with common initial endpoint not in $G_s$.   (We have previously denoted edges in $G$ by $E$  and now we are using $C$ and $D$ instead and using $E$ for a topological arc.  This is more convenient for the current argument and should not cause confusion.) By Lemma~\ref{lemma:three extension types},   there are three possibilities.    In each case, there is one component $\stallings_*(f_u)$ of $\stallings (f_u)$ that is not a component of $\stallings (f_s)$.   \label{p:types}
\begin{description}
\item \label{higher higher} [{\sf HH}]:\ \ 
($E=\bar CD$ consists of two higher order edges) $\stallings (f_u)$ is obtained from $\stallings (f_{s})$ by adding a new component  $\stallings_*(f_u)$ which is a line labeled $ R^{-1}_{C}\cdot R_{D}$.  
\item\label{linear-higher order}[{\sf LH}]:\ \ 
($E=\bar CD$ where $C$ is linear and $D$ is higher order) $\stallings (f_u)$ is obtained from $\stallings (f_{s})$ by adding a new component $\stallings_*(f_u)$ which is a one point union of a lollipop corresponding to $C$ and a ray labeled $R_{D}$.
\item\label{higher order}[{\sf H}]:\ \ 
($E=D$, a higher order edge)  $\stallings (f_u)$ is a one-point union of $\Gamma(f_{s})$ and  a ray labeled $R_{D}$.    $\stallings_*(f_u)$ is the one-point union of a component $\stallings_*(f_{s})$ of $\stallings (f_{s})$ and  a ray labeled $R_{D}$.
 \end{description}

Similarly, we can orient the topological arc $E'$ that is added to $G_s = G'_s$ to form $G'_{u'}$ so that $\stallings (g_{u'})$ is obtained from $\stallings(g_{s})$ in one of these three ways.

 By Lemmas~\ref{l:types} and \ref{l:special is natural}, we may assume 
\begin{enumerate} 
\setcounter{enumi}{1}
\item \label{same type} The extensions $\stallings(f_s) \subset \stallings(f_u)$ and $\stallings(g_s) \subset\stallings(g_{u'})$ have the same type  {\sf HH}, {\sf LH} or {\sf H} 
\end{enumerate} 
 for if not, then  $\phi|\F^+$ and $\psi | \F^+$ are not conjugate by an element of $\X$ so we  return {\tt NO} and {\tt STOP}.

\begin{remark} \label{rem:structure of eigengraphs}    A  vertex $v$  in $G$ that is new in an \sHH\ extension,    is not incident to any fixed or linear edge. It therefore follows from the construction of {$\Gamma(f)$} given at the beginning of Section~\ref{sec:stallings}
 that the component $\Gamma(f,v)$ of {$\Gamma(f)$}  corresponding to $v$ is obtained from the disjoint  union of eigenrays $R_E$, one for each $E \in \E(f)$ with initial vertex $v$, by identifying their initial vertices.   Similarly, if $v$ is new in an \sLH\ extension, then $\Gamma(f,v)$ is the one-point union of the lollipop associated to the unique linear edge with $v$ as initial vertex and the eigenrays $R_E$ associated to $E \in \E_f$ with $v$ as initial vertex.   
\end{remark}

 \begin{lemma} {Suppose that $\frak e$ has type    ${\sf H}$  and that $\X^+ \ne \emptyset$.  Then  $\stallings_*(f_{s})=\stallings_*(g_{s})$.}  
 \end{lemma}
 \proof  Assume that $\theta \in \X^+$.  Denote the set of lines that lift into a Stallings graph $\stallings$ by $\Lambda(\stallings)$.   Lemma~\ref{first theta}  and Lemma~\ref{lem:lifting} imply that $\theta(\Lambda({\stallings_*}(f_{u}))) = \Lambda({\stallings_*}(f_{u'}))$.  By construction, $\Lambda(\stallings_*(f_{s})) = \{ L \in \Lambda({\stallings_*}(f_{u})): L \subset G_s\}$ and $\Lambda(\stallings_*(g_{s})) = \{ L \in \Lambda({\stallings_*}(g_{u})): L \subset G_s\}$.  Thus 
 $\theta(\Lambda(\stallings_*(f_{s}))) = \Lambda(\stallings_*(f_{s}))$.
 
 The proof now divides into cases.  If $\stallings_*(f_{s}))$ contains a ray corresponding to some $r \in \cR(\phi)$ then 
 $\Lambda(\stallings_*(f_{s}))$ contains a line that ends at $r$.  Lemma~\ref{three properties}\pref{item:fixes r} then implies that  $\Lambda(\stallings_*(g_{s}))$ contains a line that ends at $r$ and hence that $\stallings_*(f_{s}))$ contains a ray corresponding to $r \in \cR(\phi)$. This proves that $\stallings_*(f_{s})=\stallings_*(g_{s})$.
 
 We may now assume  that  { $\stallings_*(f_{s})$}  is compact.  If $\stallings_*(f_{s})$ has rank at least two then $\pi_1(\stallings_*(f_{s}))$ is a component of $\Fix(\phi)$ and is hence $\theta$-invariant.   In this case, $\pi_1(\stallings_*(f_{s})) =  \pi_1(\stallings_*{(g_{s})})$.  Lemma~\ref{malnormal}\pref{i:cap} implies that $\stallings_*(f_{s})=\stallings_*(g_{s})$.   
 The final case is that  $\stallings_*(f_{s}))$ has rank one and so is a topological circle labeled by a component $Y$ of $G_0$ consisting of a  single edge $e$.  In this case, $\Lambda(\stallings_*(f_{s})) = \{e^\infty, e^{-\infty}\}$, which is $\theta$-invariant.  It follows that $\Lambda(\stallings_*(g_{s}))   = \{e^\infty, e^{-\infty}\}$ and hence that $\stallings_*(f_{s})=\stallings_*(g_{s})$.
 \endproof

We may therefore assume that 
\begin{enumeratecontinue}
\item \label{item:stallingsk} in case {\sf H},  $\stallings_*(f_{s})=\stallings_*(g_{s})$. 
\end{enumeratecontinue}

We next apply the Recognition Theorem to give criteria for an  element  in $\X^-$ to be in $\X^+$. 

\begin{lemma} \label{lem:twist index counts}The following are equivalent for each $\theta \in \X^-$.
\begin{enumerate}
\item \label{item:conjugates}
$\theta \in \X^+$; equivalently $\theta|\F^+$ conjugates $\phi | \F^+$ to $\psi | \F^+$.
\item \label{item:recog}
\begin{enumerate}
\item \label{item:lifts} a line $L$  lifts into $\Gamma(f_u)$ if and only if $\theta(L)$  lifts into $\Gamma(g_{u'})$. 
\item \label{item:LH index} If $\F^- \sqsubset \F^+$ has type {\sf LH} then the twist index for $C$ with respect for $f$  equals the twist index for $C'$ with respect to $g$. Equivalently, if $f(C) = Cw^d$ then $g(C') = C'w^d$.
\end{enumerate} 
\end{enumerate}
\end{lemma}

\proof
\pref{item:conjugates}  implies  \pref{item:lifts}  by  Lemmas~\ref{lem:lifting} and \ref{first theta}.  We may therefore    assume that  \pref{item:lifts} is satisfied and  prove that  \pref{item:conjugates} is equivalent to \pref{item:LH index}.

If $[F]$ is a component of $\F^+$ that is also a component of $\F^-$ then $\theta|F$ conjugates $\phi | F$ to $\psi | F$ because $\theta \in \X^-$.  We may therefore restrict our attention to the unique component of $\F^+$ that is not also a component of $\F^-$.  In other words, we may assume that $G_u $ is connected and so may assume that $G_u= G$ and $\F^+ = \{[F_n]\}$.

By Lemma~\ref{l:recognition}, there is  a  bijection $B : \sa(\phi) \to  \sa(\psi)$  that preserves twist coordinates if and only if  $\theta$ conjugates $\phi$ to $\psi$.  By definition of $\X$, $\theta$ preserves each element of $\A_{\both}(\phi)$. We are therefore reduced to showing that \pref{item:LH index} is satisfied if and only if the following is satisfied for each $[a] \in \A_\both(\phi)$:
\begin{description}
\item($*$) 
the restricted bijection $B : \sa(\phi,[a]) \to  \sa(\psi,[a])$ preserves twist coordinates.
\end{description}  

Since ($*$) is satisfied for  $[a]$ if and only if it is satisfied for $[\bar a]$, we may  assume that the \twistpath\ $w$ for $[a]_u$ satisfies $[a] = [w]$. Extending Notation~\ref{notn:base lift},   we define
$$\cP(\phi,a):= \{\Phi_{a,0},\ldots, \Phi_{a,m-1}\}$$ 
In particular, $\Phi_{a,0}$ is the base principal lift for $a$ (with respect to $f$) and    there  is an order-preserving bijection between the set   $\{E^1,\ldots,E^{m-1} \}$ of linear edges with axis $[a]$ and $\{\Phi_{a,1},\ldots,\Phi_{a,m-1}\}$.   For $1 \le j \le m-1$, there exist distinct twist indices $d_j \ne 0$ such that $f(E^j) = E^j w^{d_j}$.  Define $d_0 = 0$.     Lemma~\ref{lem:ct principal lifts} and Lemma~\ref{lem:strong axis} imply that  
$$\sa(\phi,[a]) = \{[\Phi_{a,0},a],\ldots ,[\Phi_{a,m-1},a]\}$$
and  that the twist coefficient for the pair $([\Phi_{a,i},a],[\Phi_{a,j},a])$ is   $d_i -d_j$. 

We consider two cases. In the first, we assume that either: 
\begin{itemize}
\item
$\F^- \sqsubset \F^+$ has  type {\sf LH}  and $C \not \in \{E^1,\ldots,E^{m-1} \}$; or 
\item
$\F^- \sqsubset \F^+$ does not have  type {\sf LH}
\end{itemize}
and we prove that ($*$) is satisfied.

In this case, 
      $$\cP(\psi,a) = \{\Psi_{a,0},\ldots, \Psi_{a,m-1}\}$$
       and    
$$\sa(\psi, [a]) = \{[\Psi_{a,0},a],\ldots ,[\Psi_{a,m-1},a]\}$$
 with the same sequence  of linear edges  $\{E^1,\ldots,E^{m-1} \}$ and the same sequence of twist indices $\{d_0,\ldots,d_{m-1}\}$. The bijection $B : \sa(\phi) \to  \sa(\psi)$ induces a permutation $\pi $ of $\{0,\ldots,m-1\}$ satisfying $B([\Phi_{a,i},a]) = [\Psi_{a,\pi(i)},a]$. We will show that $\pi$ is the identity and hence that $B : \sa(\phi,[a]) \to  \sa(\psi,[a])$ preserves twist coordinates.

Choose an automorphism $\Theta$ representing $\theta$ and fixing $a$. By Lemma~\ref{l:recognition},  
$$\Theta(\FixN(\Phi_{a,i})) =\FixN(\Psi_{a,\pi(i)})$$
Let $C_s$ be the component of $G_s$ that contains  $w$, and hence contains each $E^i$,   and let $[F]$ be the corresponding component of $\F^-$; we may assume without loss that $a \in F$.  Applying Notation~\ref{notn:base lift} to $\phi | F = \psi | F$ represented by the \ct\ $f |C_s$, we see that 
$$\cP(\psi | F,a) = \cP(\phi | F,a) = \{\Phi_{a,0}|F ,\ldots, \Phi_{a,m-1}|F\}$$
       and    
$$\sa(\psi | F,[a]) =\sa(\phi | F,[a]) = \{[\Phi_{a,0}|F,a],\ldots ,[\Phi_{a,m-1}| F,a]\}$$
 with the same sequence  of linear edges  $\{E^1,\ldots,E^{m-1} \}$ and the same sequence of twist indices $\{d_0,\ldots,d_{m-1}\}$.   Since $C_s$ is $f$-invariant and $[F]$ is $\theta$-invariant, \pref{item:lifts} implies that the set of lines that lift to $\Gamma(f_s)$ is $\theta$-invariant.  Applying Lemma~\ref{l:recognition} produces a  permutation $B_F$ of $\sa(\phi | F,[a]) $  and an induced permutation $\pi_F$ of $\{0,\ldots,m-1\}$.   Since $\theta | F$ commutes with  $\phi|F$, $B_F$ preserves twist-coordinates.  Thus, $d_i -d_j = d_{\pi_F(i)} - d_{\pi_F(j)}$ for all $i$ and $j$.  The only possibility is that $\pi_F$ is the identity and so
 
 $$\FixN(\Psi_{a,\pi(i)})\cap  \partial F = \Theta(\FixN(\Phi_{a,i})) \cap \partial F $$ 
 $$
 =  (\Theta| F)(\FixN(\Phi_{a,i}| F)) = \FixN(\Psi_{a,i} | F) = \FixN(\Psi_{a,i})\cap \partial F$$
 It follows that   $   \FixN(\Psi_{a,\pi(i)}) \cap \FixN(\Psi_{a,i})$ contains $ \FixN(\Psi_{a,i})\cap \partial F$ which  has cardinality at least three. Lemma~\ref{two automorphisms} implies that $\pi(i) = i$ as desired.  This completes the first case.

For the second case, we assume that:
\begin{itemize}
\item
$\F^- \sqsubset \F^+$ has  type {\sf LH} and   $C \in \{E^1,\ldots,E^{m-1} \}$
\end{itemize}
and prove that ($*$) is equivalent to \pref{item:LH index}.

Assuming without loss that $C =  E^{m-1}$,   the sequence of linear edges for $\psi$ is $\{E^1,\ldots,E^{m-2}, C' \}$ with twist indices $\{d_0,\ldots,d_{m-2}, d_{m-1}'\}$.   Thus \pref{item:LH index} is the statement that $d_{m-1} = d'_{m-1}$    and we are reduced to showing that $\pi$ is the identity.

If $m >2$ then  $\cP(\phi | F,a) = \cP(\psi | F,a)$ is indexed by $\{E^1,\ldots,E^{m-2} \}$ and the above analysis applies to show that $\pi$ restricts to the trivial permutation of $\{0,\ldots,m-2\}$.  It then follows that $\pi$ must fix the one remaining element $m-1$ of $\{0,\ldots,m-1\}$.  
  
We are now reduced to the case that $m=2$.  In particular, $[a] \not \in \A(\phi | F)$.   By construction, the base lift $\Phi_{a,0}$ restricts to an element of $\cP(\phi | F,a)$. It follows that $\Phi_{a,1}$ does not restrict to an element of $\cP(\phi | F,a)$.   The same holds for $\Psi_{a,0}$ and $\Psi_{a,1}$.   Since conjugation by $\theta | F$ preserves $\cP(\phi | F,a)$, it must be that $\Phi^\Theta_{a,0} = \Psi_{a,0}$ and $\Phi^\Theta_{a,1} = \Psi_{a,1}$.  This completes the proof of the lemma.
\endproof 

The next step in the algorithm is to check if the following condition is satisfied.
 
\begin{enumerate}
\setcounter{enumi}{3}
\item \label{item:twist consistency}
If $\F^- \sqsubset \F^+$ has type {\sf LH} then the twist index  for $C$ with respect to $f$ equals the twist index for $C'$ with respect to $g$.
\end{enumerate}

If not, return {\tt NO} and {\tt STOP}.   This is justified by Lemma~\ref{lem:twist index counts}.

\begin{lemma} \label{check any theta}    $\theta(r) = r$ for all   $\theta \in \X^-$ and $r \in \newrays(\phi)=\cR(\phi|\F^+) \setminus \cR(\phi| \F^-)$.  
\end{lemma}    

\proof  Since $\theta \in \Ker(\bar Q)$, there exists $p \in \Z$ such that $Q_b(\theta) = p$ for all $b \in \cS_2(\phi)$ that occur in $r$.  Letting $\arb=\theta^{-1} \phi^{p}$, it follows that  $Q_b(\arb) =0$ for all $b \in \cS_2(\phi)$ that occur in $r$.  Thus $\arb$ satisfies   Lemma~\ref{new rays from lines}\pref{item:Q is zero}.   Lemma~\ref{new rays from lines} \pref{item:preserves filtration}  is obvious and  the two remaining  items in the hypotheses of that lemma follow from Lemma~\ref{three properties}.  We may therefore apply  Lemma~\ref{new rays from lines}    to conclude that $\arb(r)=r$ and hence that $\theta(r) = \theta \arb(r)=  \phi^{p}(r)=r$.
\endproof

\begin{corollary} \label{same rays}
If $\X^+ \ne \emptyset$ then $\newrays(\phi) = \newrays(\psi)$.
\end{corollary}

\proof
If $\theta \in \X^+$ then $\newrays(\phi) = \theta\big(\newrays(\phi)\big) = \newrays(\psi)$   where the first equality follows from $\X^+ \subset \X^-$ and   Lemma~\ref{check any theta} and the second equality follows from Definition~\ref{defn:Xepsilon}\pref{item:conjugates epsilon} and Lemma~\ref{first theta}(3).
\endproof

\begin{notn}
$f(D) =  D\cdot \sigma$ for some completely split  path $\sigma \subset G_s$, and letting   $\afterR_D=  \sigma\cdot f_\#(\sigma) \cdot \ldots\cdot f^j_\#(\sigma) \cdot \ldots$,   the eigenray $ R_{D}$ determined by $D$ decomposes as  $R_D =   D\afterR_D$.   In the {\sf HH} case, $\afterR_C$ is defined analogously and  $R_C = C \afterR_C$.    The  rays $R'_{D'}, \afterR'_{D'}, R'_{C'}$ and $\afterR'_{C'}$ are defined similarly using $g : G' \to G'$ in place of $f : G \to G$.  
\end{notn}

 Each element of $\newrays(\phi)$ is represented by $R_D=DS_D$ or $R_C =CR_C$ and similarly for each element of $\newrays(\psi)$.    \cite[Lemma 6.3]{fh:CTconjugacy} therefore supplies an algorithm to  decide if a given  $r \in\newrays(\phi)$  and a given  $r' \in \newrays(\psi)$ are equal.  Applying this up to three times, we can     decide if  $\newrays(\phi)= \newrays(\psi)$.  If  $\newrays(\phi)\ne  \newrays(\psi)$ then $\X^+ = \emptyset$ by Corollary~\ref{same rays};  return {\tt NO} and {\tt STOP}.    We may therefore assume   that 
  \begin{enumeratecontinue}
  \item \label{r equals r'} $\newrays(\phi) = \newrays(\psi)$.  Denote this common set by $\newrays$.  In the {\sf H} and {\sf LH} cases, the unique element of $\newrays$ corresponds to $D$ and $D'$ and is denoted $r_D$.  In the {\sf HH} case, $\newrays = \{r_C, r_D\}$ where $r_C$ corresponds to $C$ and $C'$ and $r_D$ corresponds to $D$ and $D'$; this may require reversing the orientation on $E'$.    Remark~\ref{ray carried},    implies that $S_D$ and $S'_{D'}$ are contained in the core filtration element $G_p$ that realizes $F(r_D)$ and that, in the {\sf HH} case, $S_C$ and $S'_{C'}$ are contained in the core  filtration element $G_q$ realizing $F(r_C)$.
 \end{enumeratecontinue}

\cite[Lemma 6.3]{fh:CTconjugacy} also gives us initial subpaths of $S_D$ and $S'_{D'}$ whose terminal complements are equal. We may therefore assume 

\begin{enumeratecontinue}
\item \label{common terminal ray}
There is a  finite path $\kappa_D \subset G_p$  such that $S'_{D'}$ is obtained from $\kappa_D S_D$ by tightening.  Similarly, in case {\sf HH},  there is a  finite path $\kappa_C \subset G_q$  such that $S'_{C'}$ is obtained from $\kappa_C S_C$ by tightening. 
\end{enumeratecontinue}
 
We record the following for convenient referencing.  

\begin{lemma} \label{quotable}
Suppose that $\fG$ is a \ct\ representing $\phi$ and realizing $\fc$ and that either $L$ is  an element of $\acc(\phi)$ or $L$ is an element of $\LW_{\fe}(\phi)$ where $\fe\in\fc$ is not \llarge.  Let $\sigma$ be a line in $ 
\nbhd(L)$.  Then one of the following (mutually exclusive) properties is satisfied.    
 \begin{itemize}
\item
$L$ does not cross any higher order edges;    $\sigma = L$. 
\item
$L = \beta^{-1} R_e$  [resp. $  R_e^{-1} \beta$] for some higher order edge $e$ and ray $\beta$ that does not cross any higher order edges;    $\sigma = \beta^{-1} e \tau$  [resp.\ $\tau^{-1} \bar e \beta$]   where $\tau$ is a ray in the core filtration element  that realizes $F(r_e)$.
\item
$L = R_{e_1}^{-1}\rho R_{e_2}$ where $e_1,e_2$ are higher order edges and $\rho$ is a Nielsen path;
$\sigma = \tau_1^{-1}e_1^{-1}\rho e_2\tau_2$ where  $\tau_1$ is a ray in   the core filtration element   that realizes $F(r_{e_1})$ and  $\tau_2$ is a ray in   the core filtration element that realizes $F(r_{e_2})$.  
\end{itemize}
\end{lemma}

\proof
The description of $L$ comes from Lemma~\ref{lem:lines in gf}  and  the fact (Lemma~\ref{just one eigengraph line})  that each such $L$ is carried by $\gf$.  The description of $\sigma$ is immediate from the definitions of $\sH(L)$.
\endproof
     
\begin{notn}
Represent the trivial element of $\Out(F_n)$ by a  homotopy equivalence  $h:G\to G'$  that restricts to the identity on $G_{s}$. We may assume that $h(G_u) = G'_{u'}$ because $G_u$ and $G'_{u'}$ are core graphs that  represent $\F^+$.  Recall from \pref{r equals r'} that $G_p$ is the core filtration element  that realizes $F(r_D)$ and that in the {\sf HH} case, $G_q$ is   the core  filtration element  realizing $F(r_C)$.
\end{notn}

\begin{remark} \label{rem:only one crossing}
If the endpoint set of  $E$ is equal to the endpoint set  of $E'$ then $G$ and $G'$ differ only by a marking change so one can view  $h$, combinatorially,  as a homotopy equivalence from $G$ to $G$ (that does not preserve markings). In this case, \cite[Corollary 3.2.2]{bfh:tits1} implies that $h(E) = \bar \mu E' \nu$  or $h(E) = \bar \mu \bar E' \nu$ for some paths $\mu, \nu \subset G_s$. { The same conclusion holds if the endpoint sets of $E$ and $E'$ are not equal because} one can fold initial and terminal segments of $E'$ into $G_s$ to arrange  that    the endpoint set of  $E$ is equal to the endpoint set of $E'$.
\end{remark}

The next step in the algorithm is to check if the following statement is satisfied.
\begin{enumeratecontinue}
\item \label{contractible consistency} If $\fe$ has type \sH\sH\   then $h(E) = \bar \mu E' \nu$ for some paths $\mu \subset G_q$ and $\nu \subset G_p$.
\end{enumeratecontinue}

If \pref{contractible consistency} fails then  we  return {\tt NO} and {\tt STOP}.  We justify this by the following lemma.  

\begin{lemma} \label{l:check orientation}
If $\fe$ has type \sH\sH\ and  $\X^+ \ne\emptyset$    then $h(E) = \bar \mu E \nu$ for some paths $\mu, \nu \subset G_s$.  Moreover,  $\mu \subset G_q$ and $\nu \subset G_p$.
\end{lemma}
  
\proof
By Remark~\ref{rem:only one crossing},   $h(E) =\bar  \mu E' \nu$  or $h(E) = \bar \mu \bar E' \nu$ for some paths $\mu,\nu \subset G_s$ so for the main statement,we just want to rule out the latter possibility.   Let $L = R_C^{-1}R_D$ and $L' = R_{C'}^{-1}R_{D'}$.   Then  $\LW_{\fe}(\phi) = \{L,L^{-1}\}$ and $\LW_{\fe}(\psi) = \{L', {L'}^{-1}\}$. Since $\X^+ \ne \emptyset$, there exists   $\theta_0 \in \X^-$ such that $\theta_0 |\F^+$ conjugates $\phi | \F^+$ to $\psi | \F^+$. Thus  $\theta_0(\LW_{\fe}(\phi)) = \LW_{\fe}(\psi)$ and, by     Lemma~\ref{check any theta},   $L$ and $\theta_0(L)$ have the same initial ends and the same terminal ends. It follows that   $\theta_0(L) = L'$. Since $h$ represents an element of $\X^-$,  $h_\#(L) \in \sH(L')$.   In particular,  $h_\#(R_{C}^{-1}R_D) $ does not cross $\bar E'$ which implies that $h(E)$ does not cross $\bar E'$. This completes the proof of the main statement.  

From $h_\#(L)  = h_\#(S_C^{-1}\bar C D S_D) =
  [\bar S_C \bar \mu] \bar C'D' [\nu S_D]$ it follows that $[\nu S_D] \subset G_p$ and $[\mu S_C] \subset G_q$. Thus  $\nu \subset G_p$ and $\mu \subset G_q$.
\endproof
 
\medspace

The remainder of the proof of Proposition~\ref{p:reduction2} is the construction of an element  $\theta^+ \in\X^+$.  By Lemma~\ref{lem:twist index counts} and \pref{item:twist consistency}, it suffices to find $\theta \in \X^-$ that induces a bijection between lines that lift to   $\stallings(f_u)$ and lines that lift to  $\stallings(g_{u'})$.

The next lemma states that the conclusions of Lemma~\ref{l:check orientation} are satisfied in the \sH\ and {\sf L}\sH\ cases without the assumption that $\X^+ \ne \emptyset$.

\begin{lemma} \label{no need for kappa}
$h(E) =\bar  \mu E' \nu$ for some $\mu \subset G_s$ and  $\nu \subset G_p$.  In case {\sf HH},   $\mu \subset G_q$. 
\end{lemma}

\proof The \sH\sH\ case follows from \pref{contractible consistency} so we consider only the {\sf H} and {\sf LH} cases.

By Remark~\ref{rem:only one crossing},  $h(E) =\bar  \mu E' \nu$  or $h(E) = \bar \mu \bar E' \nu$ for some paths $\mu,\nu \subset G_s$. Each $L\in \sH_{\fe}(\phi)$ (realized in $G$)   decomposes as $L =\bar \alpha E \beta$ for some rays $\alpha  \subset G_s$ and $\beta \subset G_p$.     Likewise each $L'\in \sH_{\fe}(\psi)$ (realized in $G'$)  decomposes as $L' =\bar \alpha' E' \beta'$ for some rays $\alpha'  \subset G_s$ and some $\beta' \subset G_p$.   Since $h$ represents an element of $\X$, Lemma~\ref{consistency} implies that $ h_\#(L) \in \sH_{\fe}(\psi)$.   It follows that $h_\#(L)$ does not cross $\bar E'$ and hence that $h(E)$ does not cross $\bar E'$.  This proves that $h(E) \ne  \bar \mu \bar E' \nu$  and so $h(E) =  \bar\mu  E' \nu$.  Note also that $h_\#(L) = [\bar \alpha \bar \mu]E' [\nu \beta]$  which implies that $[\nu \beta] \subset G_p$ and hence that $\nu \subset G_p$.
\endproof

\bigskip 

\begin{lemma}  \label{controlling mu}
In case {\sf H}, $\mu$ is a Nielsen path for $f | G_s = g|G_s$.
\end{lemma} 

\proof
There are three cases to consider, depending on  the rank   of $\Gamma_*(f_u)$, the component of $\Gamma(f_u)$ containing the ray labeled $R_{D}$.  Let $x$ be the initial vertex of $D$.  Recall that $\F^- = \F_{k-1}$ and $\F^+ = \F_{k}$.
 
In the first case, $\rk(\Gamma_*(f_u))=0$ and we will show that $\mu$ is trivial. By Remark~\ref{rem:structure of eigengraphs}, there  exists edges $  E_1,\ldots,E_m \in \E_f$, $m \ge 2$, such that  $\Gamma_*(f_u)$  is obtained from $R_D \sqcup R_{E_1} \sqcup\ldots \sqcup R_{D_m}$  by identifying the initial endpoints of all of these rays.         By construction, $\LW_{\fe}(\phi) = \{L_1,\ldots,L_m\}$   where $L_i = R_{E_i}^{-1} R_D$.      The description of  $\Gamma_*(g_{u'})$
and  $\LW_{\fe}(\psi)$ is similar with $R_D$ replaced by $R'_{D'}$. There is a permutation $\pi$ of $\{1,\ldots,m\}$ such that   $\sH(L_i) = \sH(L_{\pi(i)}')$.  
Writing $R_{E_i} = E_i S_i$ for some ray $S_i$ with height less than that of $E_i$, we have 
 $$h_\#(L_i) = [S_i^{-1} \bar E_i   \bar \mu D' \nu S_{D'}]  = [S_i^{-1} \bar E_i   \bar \mu] D'[ \nu S_{D'}] $$
 where $[ \cdot]$ is the tightening operation.  On the other hand, letting $l = \pi(i)$, we have  
$h_\#(L_i)  = \alpha^{-1}\bar E_l D' \beta $ where $\alpha$ [resp. $\beta$] has height less than that of $E_l$ [resp. $D'$]  and  so 
$$[\mu E_i S_i] = E_l \alpha$$
Note that $S_i$ has a subray that is disjoint from $E_l$.  Since $S_i =  u_i \cdot f_\#(u_i) \cdot f^2_\#({u_i}) \cdot \dots$ is a coarsening of the complete splitting of $S_i$, it follows that $S_i$ is disjoint form $E_l$; see Remark~\ref{ray carried}.   If $i \ne l$ then $E_l$ is the first edge of $\mu$.  If $i = l$ then  $\mu$ is either trivial or has the form $E_i \sigma \bar E_i$.  In either case,  $\mu$ is either trivial or begins with $E_l$.  As this is true for all $1 \le i \le m$, we conclude that $\mu$ is trivial.

If $\rk(\Gamma_*(f_u))=1$ then either   $x$ is contained in a  circle component $B$ of the core filtration element realizing $\F_0$  or    there exists $j < k$ such that $\F_{j-1} \sqsubset \F_{j}$ is an {\sf LH} extension realized by adding  a linear edge $C_j$ and a higher order edge $D_j$ with \lq new\rq\   initial endpoint $x$; in the former case, we say that $D$ is type (i)  and in the latter case  we say that $D$ is type (ii).  If $D$ is type (i) then $B$ is a single  edge $e$ (by the (Periodic Edges) property of a \ct) and   $ \LW_{\fe}(\phi) = \{  e^{\infty} D S_D, e^{-\infty}D S_D\}$ and likewise $ \LW_{\fe}(\psi) = \{  e^{\infty} D' S_{D'}, e^{-\infty}D' S_{D'}\}$.  Lemma~\ref{quotable} implies that  $h_\#(e^{\infty} D S_D) = e^{\pm\infty} D' \beta'$ for some ray $\beta' \subset G(r_D)$. We also have $h_\#(e^{\infty} D S_D) = [e^{\infty}\bar \mu] D' [\nu S_D]$. We conclude that $\mu = e^m$ for some $m$ and in particular $\mu$ is a Nielsen path. If $D$ is type (ii) then   $ \LW_{\fe}(\phi) = \{  w_j^{-\infty} \bar C_jDS_D, w_j^{\infty} \bar C_jDS_D\}$ and similarly for $\LW_{\fe}(\psi)$. As in the previous case, $h_\#(w_j^{\infty} \bar C_jDS_D)$ is equal to both  $ w_j^{\pm \infty} \bar C_jD'\beta'$ and $   [w_j^{\infty} \bar C_j \bar \mu] D' [\nu S_D]$.  It follows that $\mu C_j w_j^\infty = C_j w_j^{\pm\infty}$ which implies that $\mu = C_j w_j^m \bar C_j$.  This completes the proof if $\rk(\Gamma_*(f_u))=1$.
     
For the final case, assume that $\rk(\Gamma_*(f_u))\ge 2$ and hence that $x \in G_s$.   Given  a lift $\ti x \in  \ti G$ of the initial endpoint $x$ of $D$, we set notation as follows:     $\ti f : \ti G \to \ti G$ is the principal lift that fixes $\ti x$;  $\Phi \in \cP(\phi)$ is the principal automorphism satisfying $ \Phi | \partial \f =  \ti f | \partial \f$; $\ti D$ is the lift of $D$ with initial endpoint $\ti x$; $\ti R_{\ti D}$ is the lift of $R_{D}$ whose first edge is $\ti D$; $\ti r_{\ti D}$ is the terminal endpoint of $\ti R_{\ti D}$; and $\ti N$ is the set of lines $(\Fix(\Phi),F(\ti r_D))$ and so is a lift of $\sH_{\fe}(\phi) = \{[\Fix(\Phi),F(\ti r_D)]\}$. Similarly, given a lift $\ti y\in \ti G'$ of the initial endpoint $y$ of $D'$, we have: $\ti g : \ti G' \to \ti G'$, $\Psi$, $\ti D'$, $\ti R_{\ti D'}$, $\ti r'_{\ti D'}$ and $\ti N'$. By Lemma~\ref{consistency},    $\sH_{\fe}(\phi) =  \sH_{\fe}(\psi)$ so we may choose $\ti y$ so that $\Fix(\Phi) = \Fix(\Psi)$ and $F(\ti r_{\ti D})=F(\ti r'_{\ti D'})$.  In particular, $\ti N  = \ti N'$.
   
Let $C_s$ be the component of $G_s$ that contains  both $x$ and $y$ (which is possible because they are the endpoints of $\mu \subset G_s$)  and let  $\ti C_s \subset \ti G$ be the lift  that contains $\ti x$.   Then $\ti C_s$ is $\ti f$-invariant and $\Fix(\ti f) \subset \ti C_s$ because $G_s$ contains all Nielsen paths in $G$ with an endpoint at $x$.  There is a  free factor $F$   representing $[C_s]$ such that $\partial F = \partial \ti C_s$.   Since $\partial \Fix(\Phi)$ is contained in the closure of $\Fix(\ti f)$, we have $\partial \Fix(\Phi) \subset \partial \ti C_s = \partial F$.   Letting $\ti C'_s \subset \ti G'$ be the lift  of $C_s$  that contains $\ti y$ and   $F'$   the free factor satisfying $\partial F' = \partial \ti C'_s$, the same argument shows that    $\partial \Fix(\Psi) \subset \partial F'$.    Since $\Fix(\Phi) = \Fix(\Psi)$ is non-trivial, $F = F'$.   Since $\phi | F = \psi | F$ and $\Fix(\Phi) = \Fix(\Psi)$ has rank at least two,  $\Phi | F= \Psi | F$.

Let $\ti h : \ti G \to \ti G'$ be the lift of $h : G\to G'$ that acts as the identity on $\partial F_n$ and let $p : \ti G \to G$ and $p' : \ti G' \to G'$ be the covering projections.    Then   $\ti h|\ti C_s : \ti C_s \to  \ti C'_s$   is a homeomorphism satisfying $p' \ti h(\ti z) = p(\ti z)$ for all $\ti z \in \ti C_s$ where, as usual, we are viewing $G_s$ as a subgraph of  both $G$ and $G'$.     Moreover, $\ti h \ti f \ti h^{-1} | \ti C'_s =  \ti g | \ti C'_s$     
 because they both project to $g|C'_s$  and induce $\Psi | F'$.  In particular $\ti g$ fixes $\ti h (\ti x)$.  Choose $\ti L \in \ti N$ that decomposes as  $\ti L = \ti \alpha \ti R_{\ti D}$.  Then  $h_\#( L) = [h ( \alpha)  {\mu}^{-1}]  D' \tau'$  for  some ray $ \tau'$ with height lower than that of $\D'$. Since $\ti h_\#(\ti L)  \in  \ti N'$  we have $\ti h_\#(\ti L) = [\ti h (\ti \alpha) \ti {\mu}^{-1}] \ti D' \ti \tau'$ for some lift $\ti \mu \subset C'_s$   and   some lift $\ti \tau'$ with height lower than that of $\ti D'$.  In particular,  $ \ti \mu $ connects $\ti y$ to $\ti h (\ti x)$ and so is a Nielsen path for $\ti g$. Thus $\mu$ is a Nielsen path for $g |G_s = f|G_s$.
\endproof
  
We 
 will complete the proof  by constructing a homotopy equivalence $d :G' \to G'$ such that the outer automorphism $\theta^+$ determined by $d h:G \to G'$ is an element of $\X^+$.    By construction,  $d$ will be the identity on the complement of $E'$    and satisfy  $d(E') =  \bar \mu' E' \nu'$  where $\mu', \nu'  \subset G_s$ are closed paths.  In cases {\sf H} and {\sf LH}, $\mu'$  will be trivial.  

\medskip

\begin{definition} \label{d1 part 1}We define $\nu'$, which  always corresponds    to $D'$, as follows.   By \pref{common terminal ray}, there is a finite path $\kappa_D \subset G_p$ such that $S'_{D'}$ is obtained from  $\kappa_D S_D$ by tightening.    By construction, $h_\#(ES_D) $ is obtained from $  \bar \mu E' \nu S_D$   by tightening. Letting 
$$\nu' = [\kappa_D\bar \nu]$$
 it follows that 
 $(dh)_\#(ES_D)  = [\bar \mu \bar \mu']E' [\kappa_D \bar \nu \nu S_{D}] =   [\bar \mu \bar \mu'] E'S'_{D'}$. 
 Thus, in cases {\sf HH} and {\sf LH} we have
 $$(dh)_\#( R_{D}) = R'_{D'}$$
 and in case ${\sf H}$ we have 
 $$(dh)_\#(R_{D}) =[\bar \mu \bar \mu']R'_{D'} =  \bar \mu  R'_{D'}$$
 where the second equality comes from the fact that $\mu'$ is trivial (see below) in case {\sf H}.
Since $\nu, \kappa_D \subset G_p$, we have 
\begin{description}
\item [(Control of $\nu'$):]  $\nu'\subset G_p$.  
\end{description}

\medskip   

In cases {\sf LH} and {\sf H}, $\mu'$ is defined to be trivial.  In case  {\sf HH} we choose $\mu'$ as we did $\nu'$ replacing $D$ with $C$.  The result is that in case {\sf HH} $$(d h)_\#(R_{C}) = R'_{C'}$$
and so
 $$(dh)_\#(R^{-1}_{C}R_{D}) = {R'_{C'}}^{-1}R'_{D'}$$
 Also, 
\begin{description}
\item [(Control of $\mu'$):]  In case {\sf HH},   $\mu'\subset G_q$.  
\end{description}
This completes the definition of $d$.
\end{definition}

\begin{lemma}  \label{d in Ker Q}
The outer automorphism $\delta$ represented by $d :G' \to G'$ is an element of $\Ker(Q) \subset \X$.  
\end{lemma}

\proof    We  use the following properties of  $d :G' \to G'$  to prove that $\delta \in \X$.
\begin{itemize}
\item [(a)] $d$ preserves every component of every filtration element of $G'$. In particular, $\delta$ preserves $\fc$ and every $[F]\in\F\in\fc$. 
\item [(b)] If $e'$ is not a higher order edge in $E'$ then $d(e') = e'$.  
\end{itemize}
These are both obvious from the definition.
\begin{itemize}
\item [(c)]  $d_\#$ fixes each Nielsen path of $g$.
\end{itemize}
This follows from (b) and the fact that Nielsen paths do not cross higher order edges.
\begin{itemize}
\item  [(d)] If $e'$ is a higher order edge and $G'_{p'}$ is the filtration element that realizes $F(r'_{e'})$ 
then the set of rays of the form $e'\beta'$  with $\beta' \subset G'_{p'}$ is mapped into itself by $d_\#$.
\end{itemize}
If $e'$ is neither $C$ nor $D$ then this follows from (a) and (b). Otherwise it follows from (a),  (Control of $\nu'$) and (Control of $\mu'$). 

\medspace

Item (c) implies that $\delta$ fixes each component of $\Fix(\phi)$ and every element of $\A(\phi)=\A(\psi)$.  Item (b) implies that $d | G'_s = $identity and hence that $\delta$ fixes each element of $\sa(\phi|\F_0)$ and so satisfies defining property \pref{i:last} of $\X$.   Suppose that either $L'$ is  an element of   $\acc(\psi)$  or $L'$ is an element of  $\LW_{\fe'}(\psi)$ where $\fe'\in\fc$ is not \llarge. Then (b), (d) and Lemma~\ref{quotable}  imply that  $\sH(L')$ is $\delta$-invariant.   Similarly (c) and (d) imply that if $\fe '$ is \llarge\ then $\sH_{\fe'}(\phi)$ is $\delta$-invariant. We have now shown that $\delta \in \X$. In particular, $\delta$ is in the domain of $\bar Q$.
 
To prove that $\delta \in \Ker(Q)$,  suppose that $\ti b' = (\ti L_1',\ti L_2')$ is a staple pair for $\psi$ with common axis $\ti A'$.    By (b), there is a lift  $\ti d : \ti G' \to \ti G'$  of $d$ that pointwise fixes $\ti A'$.  We claim that $\ti d$ preserves both $\nbhd(\ti L'_1)$ and $\nbhd(\ti L'_2)$.   The $\ti L'_1$ and $\ti L'_2$  cases are symmetric so we will consider $L_2'$.  If both ends of $L_2'$ are periodic then $\nbhd(L'_2) = \{L'_2\}$ by Lemma~\ref{quotable}.  Moreover,  $\ti L'_2$ does not cross any higher order edges and so is pointwise fixed by $\ti d$.  We may therefore assume that there is a decomposition $\ti L'_2 = \ti \alpha \ti R'_{\ti e'_2}$ where $\ti e'_2$ is a higher order edge and $\ti \alpha$ does not cross any higher order edges.  It follows that  $\ti d$ pointwise fixes $\ti \alpha$ and that $\ti e_2$ is the initial edge of $\ti {d}_\#(\ti R'_{\ti e'_2})$.  Item (d) now implies that $\ti d$ preserves   $\nbhd(\ti L'_2)$.   This completes the proof of the claim.   It now follows from the definitions that $m_{b'}(\delta) = 0$ and hence $Q_{b'}(\delta)=0$.  Since $b'$ is arbitrary, $\delta \in \Ker(Q)$. 
\endproof

The final step in the algorithm that proves Proposition~\ref{last step} is to return {\tt YES} and the outer automorphism $\theta^+$ represented by $dh : G \to G'$.  In conjunction with Lemma~\ref{lem:X+ is good enough}, the following lemma justifies this step.

\begin{lemma}  \label{last lemma} $dh:G \to G'$ represents an element   $\theta^+ \in \X^+$.
\end{lemma}

\proof
Since $h$ represents an element of $\Ker(\bar Q)$,    Lemma~\ref{d in Ker Q} implies that $\theta^+\in \Ker(\bar Q) \subset \X$.   We are therefore reduced to showing that $\theta^+ | \F^+$ conjugates $\phi | \F^+$ to $\psi | \F^+$.  By Lemma~\ref{lem:twist index counts} and \pref{item:twist consistency}, it suffices to prove that:
 
\begin{description}
\item  [(b1)]  a line $L\subset G$ lifts to  $\Gamma(f_u)$   if and only if $\theta^+(L) \subset G'$ lifts to $\Gamma(g_{u'})$.
\end{description}

Recall that  $\Gamma(f_u)$  is obtained from  $\Gamma(f_{s})$  by either adding a single new component (the {\sf HH} and {\sf LH} cases) or by adding a ray  to one of the components $\Gamma_*(f_{s})$ of $\Gamma(f_{s})$ (the {\sf H} case).  The same is true for $\Gamma(g_{u'})$.   The component of  $\Gamma(f_u)$ that is not a component of $\Gamma(f_{s})$ is denoted $\Gamma_*(f_u)$; it is the unique component that contains a ray labeled $R_{D}$.  Likewise, the \lq new\rq\ component $\Gamma_*(g_{u'})$ of $\Gamma(g_{u'})$ is the  one that contains a ray labeled $R'_{D'}$.     Recall also that $f_{s} = g_{s}$, that $\Gamma(f_{s})=\Gamma(g_{s})$,  that $\Gamma_*(f_{s}) = \Gamma_*(g_{s})$.  

Item (b1) is obvious if $L$ lifts into a component of $\Gamma(f_s)$ so we may  assume, after reversing orientation on $L$ if necessary, that $R_D$ is a terminal ray of $L$.    

In case {\sf HH},   $L = R^{-1}_{C} R_{D}$ so  (b1)  follows from  $(dh)_\#(R^{-1}_{C}R_{D}) = {R'_{C'}}^{-1}R'_{D'}$; see
  Definition~\ref{d1 part 1}.  
  
In case ${\sf H}$, $(dh)_\#(R_{D}) = \bar \mu R'_{D'}$ by Definition~\ref{d1 part 1}.  Let $\hat x $ be the initial endpoint of the lift of $R_D$ into $\stallings(f_u)$ and let $x$ be its projection into $G_s$.  Define $\hat y$ and $y$ similarly using $R'_{D'}$ in place of $R_D$. If $R_D$ is also a terminal ray of $L^{-1}$,    then $L = R_{D}^{-1}  \xi R_{D}  $ for some Nielsen path $\xi$ and  $(dh)_\#(L) =  R'^{-1}_{D'} [\mu \xi \bar   \mu] R'_{D'}$ which lifts into   $\stallings(g_{u'})$ because $ [\mu \xi \bar   \mu]$ is a Nielsen path by Lemma~\ref{controlling mu}.  The remaining case is that   $L = \beta^{-1} R_{D}$ for some ray $\beta \subset G_s$   that lifts to a ray in $\Gamma(f_s)$ based at $\hat x$.  In this case, $(d h)_\#(L) = \beta^{-1} \bar  \mu R'_{D'}$  which lifts into $\Gamma^0(g_{u'})$ because $\mu$ is a Nielsen path that lifts to a path in $\stallings(g_s)$ connecting $\hat y$ to $\hat x$.  

In case  {\sf LH},   $\LW_\fe(\phi) = \{ L_+, L_-\}$ and  $\LW_\fe(\psi) = \{ L'_+, L'_-\}$ where 
   $L_\pm = w^{\pm \infty} \bar C R_D$ and $L'_\pm = w^{\pm \infty} \bar C' R'_{D'}$.  Since   ${\sf I}_\fc(\phi) = {\sf I}_\fc(\psi)$  and $\theta^+ \in \X$, we  have that  $\theta^+(L_+)$  is contained in either $\nbhd(L_+')$ or $\nbhd(L_-')$. By construction,    $\theta^+(L_+) = (dh)_\#(L_+) = [w^{\infty} \bar \mu] \bar C'R'_{D'}$.  Thus, $\theta^+(L_+) \in \nbhd(L_+')$  and  $[w^{\infty} \bar \mu] = w^\infty$.   It follows that $\mu = w^p$ for some $p \in \Z$ and hence that $\theta^+(L_\pm) = L_\pm'$ and    $\theta^p(R_D^{-1}w^kR_D) =  R'^{-1}_{D'}w^kR'_{D'}$ for all $k$.   This completes the proof of (b1) and hence the proof of the lemma.
   \endproof
   
\begin{appendices}
\section{More on  $\Ker \bar Q$}
The main results of this appendix are that $m_b(\theta)$ can be computed for $\theta\in\X$, that $\Ker\bar Q$ is of type $\mathsf{VF}$ (Definition~\ref{d:vf}), and that a finite presentation for $\Ker \bar Q$ can be computed. This section is needed for future work and is not used in the proof of the main theorem of this paper.

\subsection{A Stallings graph for $\sHsub(\ti b)$}\label{s:graph}
We will need the following remark.

\begin{remark} \label{rem:overlap} Suppose that $G$ is a marked graph and that for $i=1,2$, $\ti A_i$ is the axis of a covering translation $T_i: \ti G \to \ti G$ and that the number of edges in a fundamental domain for $\ti A_i$ is $s_i$.  If  $\ti A_1 \cap \ti A_2$ contains at least $s_1+s_2+1$ edges then $\ti A_1 = \ti A_2$.   To see this,  decompose     $\ti A_1 \cap \ti A_2 =e_1 e_2 \ldots$ into edges  and note that   $T_1T_2(e_1) = e_{s_1+s_2+1} =T_2T_1(e_1)$.  Since $T_1 T_2$ and $T_2T_1$ agree on an edge they are equal and so $T_1$ and $T_2$ have the same axes.
\end{remark}

\begin{notn}\label{algebraic staple pair}
Let $\fG$ be a CT for $\phi$ with $\ti f:\ti G\to\ti G$ a lift to the universal cover.  Assume notation as in the definition of $m_b$ (Definition~\ref{defn:m}).  In particular, $b=(L_1,L_2)\in \cS_2(\phi)$, $\ti b = (\ti L_1, \ti L_2)$ are lifts of $(L_1, L_2)$ such that $\ti L_1^+, \ti L_2^- \in \{\ti A^-, \ti A^+\}$ where $\ti A$ is the common axis of $\ti b$ and $a \in \f$ is a root-free element with axis $\ti A$ and orientation chosen as in the definition. For $\theta\in \X$, $\Theta_i\in\theta$ is defined uniquely by $\Theta_i(\sHsub(\ti L_i))=\sHsub(\ti L_i)$ and $m_b(\theta)$ is defined so that $\Theta_1=i_a^{m_b(\theta)}\Theta_2$.

If $\ti L_2^+ = \ti r_2 $ for some $r_2  \in  \cR(\phi)$ then define $\sHsub^2(\ti b) =  F_\fc(\ti r_2)$. Otherwise, $\ti L_2^+ \in \{c_2^-,c_2^+\}$ for some root-free $c_2 \in \f$ representing an element of $\A(\phi)$ and  $\sHsub^2(\ti b) := \langle c_2 \rangle$.  Define $\sHsub^1(\ti b)$ similarly using $\ti L_1^-$ in place of $\ti L_2^+$.  Finally, define $\sHsub(\ti b) = \langle  \sHsub^1(\ti b),  \sHsub^2(\ti b)\rangle$.

The covering transformation corresponding to $a$ is denoted $\tau$. Additionally, $H^i:=\sHsub^i(\ti b)$, $T^i$ denotes the minimal subtree for $H^i$, $\Gamma^i$ denotes the Stallings graph for $H^i$, and $\ti L(k)$ denotes the line $[\ti L_1^-,\tau^k(\ti L_2^+)]$.
\end{notn}

\begin{remark} \label{r:compare}
Comparing definitions of $\sHsub(\ti L_i)$ and $H^i$, $\Theta_i$ is the unique $\Theta\in\theta$ fixing $\langle a\rangle$ and $H^i$.
\end{remark}

\begin{lemma}\label{l:k}
There is $k\ge 0$ such that:
\begin{itemize}
\item
$T^1\cap \tau^k(T^2)=\emptyset$ and
\item
The arc $\ti \mu$ spanning between $T^1$ and $\tau^k(T^2)$ contains more than two fundamental domains of $\ti A$ with orientation agreeing with that of $\ti \mu$. 
\end{itemize}
\end{lemma}

\begin{proof}
The ends $\ti A$ are not ends of $T^i$. Indeed, if $\ti r_i$ is ray, then the associated higher order edge separates $T^i$ and the ends of $\ti A$. If not, then $T^i$ is the axis corresponding to the end of $\ti L_i$ that is not an end of $\ti A$. Hence there is a neighborhood of $\ti A^+$ that is disjoint from $T_i$. Therefore, the conclusion holds for all large $k$.
\end{proof}

\begin{corollary} \label{c:compute graph}
We may compute:
\begin{enumerate}
\item\label{i:compute graph}
for all $l\in \Z$, the Stallings graph for $\langle H^1, i_a^l(H^2)\rangle$;
\item\label{i:k}
an integer $k\ge 0$ as in Lemma~\ref{l:k}; and
\item\label{i:m}
for all $\theta\in\X$, $m_b(\theta)$.
\end{enumerate}
\end{corollary}
  
\begin{proof}
\pref{i:compute graph}: By Bass-Serre theory, for $k$ as in Lemma~\ref{l:k}, the Stallings graph for $\langle H^1, i_a^k(H^2)\rangle$ is obtained by attaching at its endpoints a copy of the arc spanning between $T^1$ and $\tau^k(T^2)$ to the Stallings graphs for $[H^1]$ and $[i_a^kH^2]=[H^2]$.

These latter graphs can be computed. Indeed, if $\ti L_1^-$ is an eigenray, then, by definition, $[H^1]$ has as its Stallings graph a component of a stratum of $G$, and otherwise is a circle representing $\langle c_1\rangle$. There is a symmetric argument for $[H^2]$. $\ti L(k)$ spans between $T^1$ and $\tau^k(T^2)$. Hence the desired Stallings graph, for large $k$, is the result of immersing the ends of $\ti L(k)$ into the Stallings graphs and then performing any folding. By Lemma~\ref{l:k}, folding stops when the copy of $\ti\mu$ is the spanning arc.

We see that, for large $k$, $\langle H^1, i_a^k(H^2)\rangle$ is an internal free product. By Remark~\ref{r:compare}, $\Phi_1^s(\langle H^1, i_a^k(H^2)\rangle)=\langle H^1, i_a^{sm_b(\phi)+k}(H^2)\rangle$ 
is also a free product and that its Stallings graph can be computed as above (but perhaps the spanning arc is folded away). Recall (Lemma~\ref{lem:m(phi)}) that $m_b(\phi)\not= 0$.  We note in passing that therefore $[H^1, i_a^l(H^2)]$ is good (Definition~\ref{d:good}).

\pref{i:k}:
For $l=0, 1, 2, \dots$, iteratively start computing Stallings graphs for  $\langle H^1, i_a^l(H^2)\rangle$. When, after folding in $\ti L(l)$, more than two correctly oriented fundamental domains of $\ti A\cap \ti L(l)$ remain in the spanning arc, stop and set $k=l$.

\pref{i:m}: Let $\theta\in \X$ and $m:=m_b(\theta)$. If $\Theta_1(H^1)=H^1$, then by definition, $\Theta_1(H^1, i_a^k(H^2))=(H^1, i_a^{k+m}(H^2))$. Hence, $m$ can be read off by comparing the Stallings graphs for $[\langle H^1, i_a^k(H^2)\rangle]$ and $[\langle H^1, i_a^{k+m}(H^2)\rangle]$ for large enough $k$. The latter, being the Stallings graph for $\theta[ \langle H^1, i_a^k(H^2)\rangle ]$, can be computed by representing $\theta$ as a topological representative $g:G\to G$, applying $g$ to the Stallings graph for $[\langle H^1, i_a^k(H^2)\rangle]$, tightening, and taking the core.
\end{proof}

\begin{corollary}
\begin{enumerate}
\item \label{i:good}
$[H_1, i_a^lH_2]$ is good for all $l\in \Z$.
\item \label{i:no a}
No non-trivial power of $a$ is in $\langle H_1, i_a^lH_2\rangle$.
\end{enumerate}
\end{corollary}
  
\begin{proof}
\pref{i:good} was noted during the proof of Corollary~\ref{c:compute graph}.

\pref{i:no a}: Since $\Theta_1(a)=a$, for the second item it is enough to show that $$a\notin\Theta^l_1(\langle H^1, i_a^l(H^2)\rangle)=\langle H^1, i_a^{l+m}(H^2)\rangle$$ for large $l+m$.
So assume that $k>0$ is as in Lemma~\ref{l:k}. If $a\in \langle H^1, i_a^k(H^2)\rangle$, then there is an immersion of $\ti A$ with image a closed loop into the Stallings graph for $\langle H^1, i_a^k(H^2)\rangle$ and that overlaps the copy of $\ti\mu$ in its intersection with $\ti A$. The immersion crosses this spanning arc at most once. Indeed, otherwise there would be a covering translation of $\ti G$ taking $\ti A$ to itself reversing orientation, but $a$ and $a^{-1}$ are not conjugate; see Remark~\ref{rem:overlap}. Hence the image of $\ti A$ is not a closed loop.
\end{proof}

\begin{corollary}  \label{c:ker Q equivalence} For $\theta \in \X$, $m_b(\theta) = 0 \Longleftrightarrow 
[\sHsub(\ti b)]$ is $\theta$-invariant.
\end{corollary}

\begin{proof}
($\implies$): Suppose $m:=m_b(\theta)=0$. Using Remark~\ref{r:compare} we have 
\begin{align*}
\Theta_1(\sHsub(\ti b))=&\Theta_1(\langle H^1, H^2\rangle)\\
=&\langle \Theta_1(H^1), \Theta_1(H^2)\rangle\\ 
=&\langle H^1, i_a^m(H^2)\rangle\\
=&\langle H^1, H^2\rangle
\end{align*}

($\Longleftarrow$): Suppose $m> 0$, the case for $m<0$ being similar. Choose $l$ so that $k:=ml$ is as in Lemma~\ref{l:k}. We saw in Corollary~\ref{c:compute graph} that the Stallings graph for $[\langle H^1, i_a^{ml}(H^2)\rangle]$ has an arc spanning between Stallings graphs for $[H^1]$ and $[H^2]$ and that the Stallings graph for $[\Theta_1^{l+1}(\langle H^1, H^2 \rangle)]$ is obtained by inserting $m$ copies of a fundamental domain for $\ti A$ into this spanning arc. Since these Stallings graphs are not equal, $\theta([\langle H^1,H^2\rangle])\not=[\langle H^1,H^2\rangle]$.
\end{proof}

\subsection{$\Ker Q$}
Recall (Definition~\ref{d:Q}) the homomorphism $Q=Q^\phi:\X\to\Q^{\cS_2(\phi)}$ given by setting the $b$th coordinate of $Q(\theta)$ equal to $m_b(\theta)/m_b(\phi)$.

\begin{proposition}\label{p:ker Q}
A finite presentation for $\Ker Q$ can be computed. $\Ker Q$ is of type {\sf VF}.  
\end{proposition}

\begin{proof}
We use the notation as in Section~\ref{s:graph}. For each $b\in\cS_2(\phi)$, choose $\ti b$ and replace it with $\Phi_1^k(\ti b)$ and compute $[\sHsub(\ti b)]$ where $k$ is as in Lemma~\ref{l:k}. By Lemma~\ref{c:ker Q equivalence}, $\theta\in \Ker Q$ iff $\theta\in\X$ and $\theta([\sHsub(\ti b)])=[\sHsub(\ti b)]$ for each $b$. Hence $\theta\in\Ker Q$ iff $\theta$ fixes the concatenation of the sequences $\big( \sf J \big)$   and $\big([\sHsub(\ti b_1)], \dots, [\sHsub(\ti b_N)]\big)$ where $(b_1, \dots b_N)$ is an ordering of $\cS_2(\phi)$ and $\sf J$   is as in Definition~\ref{d:X}. This concatenation is an element of $\ecat(\Abig)$ (see Notation~\ref{n:seq}). By Lemma~\ref{l:hat MW},  we can compute a finite presentation for $\Ker Q$ and, by Proposition~\ref{p:vf}, $\Ker Q$ is of type {\sf VF}.
\end{proof}

\subsection{$\Ker \bar Q$}
$\bar Q$ is defined in Definition~\ref{d:pair equivalence}.

\begin{proposition}
A finite presentation for $\Ker\bar Q$ can be computed.
$\Ker \bar Q$ is of type {\sf VF}.  
\end{proposition}

\begin{proof}  
Let $\pi$ denote the quotient map $Q(\X)\to \bar Q(\X)$. Since we have a finite generating set for $\X$, we can compute a finite presentation for the free abelian group $\Ker\pi$. Using Proposition~\ref{p:ker Q}, Lemma~\ref{l:bw}, and 
$$1\to \Ker Q\to \Ker \bar Q\to \Ker \pi\to 1$$
we can compute a finite presentation for $\Ker \bar Q$.

 The above short exact sequence shows that $\Ker\Bar Q$ is an extension of a group of type ${\sf VF}$ (Proposition~\ref{p:ker Q}) by a finitely generated abelian group. It follows from the proposition below that $\Ker \bar Q$ is of type {\sf VF}.
\end{proof}

  The following proposition is from Moritz Rodenhausen's thesis. As far as we know, it is not published and so for the reader's convenience we copy his proof here.

\begin{proposition}[{\cite[Proposition~13.18]{mr:thesis}}]
Suppose that in the short exact sequence $1\to G'\overset{i}{\to} G\overset{\pi}{\to}G''\to 1$, the group $G'$ is of type {\sf VF} and $G''$ is finitely generated abelian. Then $G$ is of type {\sf VF}.
\end{proposition}

\begin{proof}
Suppose first that $G''$ is infinite cyclic. Let $H'$ be a subgroup of some finite index $d$ in $G'$ such that $H'$ is of type {\sf F} . The the intersection $K'$ of all (finitely many) index $d$ subgroups of $G'$ also is of type {\sf F}. Furthermore, the group $K'$ is a characteristic subgroup of $G'$. Let now $t\in G$ be an element such that $\pi(t)$ generates $G''$. We denote by $K$ the subgroup of $G$ generated $t$ and $i(K')$. It has finite index in $G$ and fits into a short exact sequence 
$$1\to K'\to K\to G''\to 1$$
We see that  $K$ is an extension of groups of type {\sf F} and so has type {\sf F}; see \cite[Theorem~7.3.4]{rg:methods}. Hence $G$ is of type {\sf VF}.

The case where $G''$ is isomorphic to $\Z^m$ is proved by induction on $m$. Let $A''$ be an infinite cyclic summand of $G''$, so $G''/A''\cong \Z^{m-1}$. The short exact sequences
$$1\to G'\to \pi^{-1}(A'')\to A''\to 1$$
$$1\to\pi^{-1}(A'')\to G\to G''/A''\to 1$$
and the induction hypothesis completes the proof in case $G''\cong \Z^m$.

For the general case, let $H''$ be a free abelian subgroup of $G''$ of finite index and $H=\pi^{-1}(H'')\subset G$. We obtain a short exact sequence
$$1\to G'\to H\to H''\to 1$$
Hence $H$, and so also $G$, is of type {\sf VF}.
\end{proof}

\end{appendices}

\bibliographystyle{amsalpha}
\bibliography{ref}

\end{document}